\documentclass[a4paper,10pt,twoside]{amsart}
\usepackage[foot]{amsaddr}
\usepackage{amsmath,amssymb,amsthm,mathtools,esint,enumitem}
\usepackage[svgnames]{xcolor}
\usepackage{hyperref}
\hypersetup{linktoc=all,colorlinks=true,linkcolor=Maroon,citecolor=Maroon,urlcolor=DarkBlue} 

\usepackage[margin=1.5in, marginparwidth=1.2in]{geometry}

\DeclareMathAlphabet{\mathlcal}{U}{dutchcal}{m}{n} 

\usepackage{tikz-cd, float}
\usepackage{istgame}
\usepackage{cases}
\makeatletter
\@namedef{subjclassname@2020}{\textup{2020} Mathematics Subject Classification}
\makeatother
\usepackage[capitalize]{cleveref}
\newtheorem{theorem}{Theorem}[section]
\newtheorem*{theorem*}{Theorem}
\newtheorem{proposition}[theorem]{Proposition}
\newtheorem{corollary}[theorem]{Corollary}
\newtheorem{lemma}[theorem]{Lemma}
\theoremstyle{remark}
\newtheorem{remark}[theorem]{Remark}
\theoremstyle{definition}

\newtheorem{assumption}[theorem]{Assumption}

\newtheorem*{duplicate*}{{\hypersetup{hidelinks}Assumption~\ref{ass}}}

\newtheorem*{nota*}{Notation}

\usepackage{autonum} 

\numberwithin{equation}{section}

\usepackage{etoolbox} 
\expandafter\patchcmd\csname\string\proof\endcsname
  {\normalparindent}{0pt }{}{}

\renewcommand{\a}{\mathsf{a}}
\newcommand{\A}{\mathsf{A}}
\renewcommand{\b}{\mathsf{b}}
\newcommand{\C}{\mathbb{C}}
\renewcommand{\d}{{\rm d}}
\newcommand{\E}{\mathbb{E}}
\newcommand{\eps}{\varepsilon}
\newcommand{\G}{\Gamma^*}
\newcommand{\Gtilde}{\widetilde{\Gamma}^*}
\newcommand{\id}{{\rm id}}
\newcommand{\kay}{\mathlcal{k}}
\renewcommand{\l}{\ell}
\newcommand{\length}[1]{| #1 |_{\prec}}
\newcommand{\m}{{\bf m}}
\newcommand{\n}{{\bf n}}
\newcommand{\N}{\mathbb{N}}

\newcommand{\p}{\mathsf{p}}
\newcommand{\pp}{\mathrm{p.p.}}
\renewcommand{\preceq}{\preccurlyeq}

\newcommand{\R}{\mathbb{R}}
\newcommand{\RR}{\R[[\a_k,\b_\ell,\p_\n]]}
\newcommand{\s}{\mathfrak{s}}
\newcommand{\T}{\mathsf{T}}
\newcommand{\z}{\mathsf{z}}
\newcommand{\Z}{\mathbb{Z}}
\newcommand{\0}{{\bf 0}}
\newcommand{\+}{\hspace{-.5ex}+\hspace{-.5ex}}
\renewcommand{\-}{\hspace{-.25ex}-\hspace{-.25ex}}

\DeclarePairedDelimiter\abss{\lvert}{\rvert_{\s}}
\DeclareMathAlphabet{\mathup}{OT1}{\familydefault}{m}{n}
\newcommand{\dx}[1]{\mathup{d} #1 \,}

\title[Stochastic estimates for the thin-film equation]{Stochastic estimates for the thin-film equation\\ with thermal noise}
\author[R. S. Gvalani]{Rishabh S. Gvalani$^1$}
\address{$^1$University of Edinburgh, United Kingdom}
\author[M. Tempelmayr]{Markus Tempelmayr$^2$}
\address{$^2$EPFL, Lausanne, Switzerland and Universität Münster, Germany.}
\date{\textcolor{Maroon}{\today}}
\allowdisplaybreaks
\begin{document}


\begin{abstract}
We construct and derive uniform stochastic estimates on the renormalised model for a class of fourth-order conservative quasilinear singular SPDEs in arbitrary dimension $d\geq 1 $ and in the full subcritical regime of noise regularity. The prototype of the class of equations we study is the so-called thin-film equation with thermal noise, also commonly referred to in the literature as the stochastic thin-film equation. We derive an explicit expression for the form of the counterterm as a function of the film mobility which is in surprising agreement with the form conjectured in~\cite[Remark 9.1]{GGKO22}. 
\end{abstract}



\maketitle
\tableofcontents

\section{\texorpdfstring{\bf Introduction}{Introduction}}
\noindent
We would like to study the following class of quasilinear singular stochastic partial differential equations (SPDEs) posed on $\mathbb{R}^{1+d}$
		\begin{equation}
			L u =  \nabla \cdot(a(u) \nabla \Delta u)  + \nabla \cdot( b(u) \xi) \, ,  
		\label{eq:spde-intro}
		\end{equation}
	where $L$ is given by the following fourth-order parabolic operator
		\begin{equation}
			L:= \partial_0 + \Delta^2, \qquad \Delta := \sum_{i=1}^d \partial_i^2 \, ,
		\end{equation}
	with the gradient $\nabla$ and divergence $\nabla \cdot$ also defined with respect to $x_i$ for $i= 1,\dots ,d$, and where $a,b$ are sufficiently smooth prescribed scalar-valued nonlinearities such that the equation is uniformly parabolic and $\xi$ is some $\R^d$-valued rough, random forcing. We denote by $x= (x_0, \dots,x_d)$ a typical element of $\R^{1+d}$ with $x_0$ denoting the time-like coordinate and $x_i$ for $i =1,\dots,d$ the space-like coordinates. One can check that $L$ satisfies some natural scaling invariance with respect to the scaling $\mathfrak{s} = (4, 1, \dots ,1)\in \N^{1+d}$, i.e.~given some Schwartz function $f$, we have for any $\eps>0$
		\begin{equation}
			(L f^\eps) (x) = \eps^{|L|}(L f) (\hat x) \, ,
			\label{eq:L-invariance}
		\end{equation}
	where $|L|=4$ is the order of the operator $L$, 
		\begin{equation}\label{xhat}
			\hat x = 
			(\hat x_0, \dots, \hat x_{d}) := (\eps^{\mathfrak{s}_0}x_0,\dots,\eps^{\mathfrak{s}_{d}}x_{d}) \, ,
		\end{equation}
	and $f^{\eps}(x):=f(\hat x)$. 
 One should think of $\xi$ as being some ensemble of tempered distributions with a prescribed law. We are specifically interested in the case in which the solution $v$ of 
		\begin{equation}
			L v = \nabla \cdot \xi \, ,
			\label{eq:lspde-intro}
		\end{equation}
	is almost surely $C^\alpha$ with $\alpha \in (0,1)$. Since  $L$ is a fourth-order operator, by Schauder theory, this is the case if $\nabla \cdot \xi\in C^{\alpha -4}$, which in turn is true if $\xi \in C^{\alpha-3}$. Note that we measure H\"older regularity with respect to the appropriately $\mathfrak{s}$-scaled Carnot--Carath\'eodory metric on $\R^{1+d}$ associated to the operator $L$, see \eqref{metric}. We also impose that the ensemble $\xi$ satisfies the following scaling invariance in law
		\begin{equation}
			\hat\xi(x):= \eps^{3-\alpha}\xi(\hat x) \sim \xi(x) \, ,
			\label{eq:noise-invariance}
		\end{equation}
		where $\sim$ denotes equality in law.
	The equation~\eqref{eq:spde-intro} is singular in the sense that given some $u \in C^{\alpha}$ with $\alpha <3/2$ the products $a(u)\nabla \Delta u$ and $b(u)\xi$ cannot be defined in a canonical manner; 
 to avoid case distinctions we will restrict ourselves to the more singular case $\alpha<1$. What gives us some hope is the fact that~\eqref{eq:spde-intro} is locally \emph{subcritical} for $\alpha>0$. To be more precise, consider $u^\eps(x)= \eps^{-\alpha} u(\hat x)$. Then, using~\eqref{eq:L-invariance} and~\eqref{eq:noise-invariance}, we formally compute
		\begin{align}
			(L u^{\eps})(x) = \eps^{4-\alpha}(L u)(\hat x) &= \nabla \cdot (a(\eps^\alpha u^\eps) \nabla \Delta u^\eps)(x)  + \nabla \cdot (b(\eps^{\alpha} u^\eps) \hat\xi)(x) \label{sinv1} \, .
		\end{align}
	From the above rough calculation, we can see that, if $a$ vanishes and $b$ is order one for small $u$, then as $\eps \to0$, $u^{\eps}$, the \emph{fine-scale} version of $u$, solves the the linear equation~\eqref{eq:lspde-intro}.  

\medskip

There has been a flurry of research activity in recent years in the study of subcritical singular SPDEs starting with the seminal works of Hairer~\cite{Hai14}, who developed the theory of regularity structures to treat such equations and Gubinelli, Imkeller, and Perkowski~\cite{GIP15} who studied these equations using the approach of paracontrolled calculus. 
Focusing on regularity structures, by now there exists a more or less automated machinery to deal with any semilinear and subcritical singular SPDE with its different aspects contained in the works by Hairer~\cite{Hai14}, Bruned, Hairer, and Zambotti~\cite{BHZ19}, Bruned, Chandra, Chevyrev, and Hairer~\cite{BCCH21}, and Chandra and Hairer~\cite{CH16}. 
The contents of this paper are focused on obtaining uniform estimates for the so-called model which is covered in the semilinear case, using a Feynman-diagrammatic approach, by~\cite{CH16} and, using the spectral gap inequality and Malliavin calculus, by Hairer and Steele~\cite{HS23} and by Bailleul and Hoshino \cite{BH23}. We also mention the work by Kunick and Tsatsoulis \cite{KT22} which, to our knowledge, is the first paper to use a spectral gap-based approach to derive these stochastic estimates in the tree-based setting, albeit in the very specific case of the dynamical $\varphi^4_2$-model.
The spectral gap inequality has furthermore been used to establish stochastic estimates for the one dimensional generalised KPZ equation by Bailleul and Bruned \cite{BB24}, and in the context of rough paths by Gassiat and Klose \cite{GK24}.

In the direction of quasilinear singular SPDEs, the first results go back to Otto and Weber~\cite{OW19} who used a rough-paths based approach to treat certain quasilinear singular SPDEs, 
and to Bailleul, Debussche, and Hofmanov\'a~\cite{BDH19}, and Furlan and Gubinelli~\cite{FG19},
which are both based on paracontrolled calculus. 
These works are limited to a small range of the subcritical regime, 
which has been successively extended by Gerencs\'er and Hairer~\cite{GH19}, 
Otto, Sauer, Smith, and Weber~\cite{OSSW18},
and Bailleul and Mouzard~\cite{BM23}. 
All of these works share the drawback of not being able to treat the full subcritical regime of regularity and often (as in~\cite{BDH19,GH19}) rely on transforming the quasilinear SPDE to a semilinear one and using the, by now well-developed, semilinear theory. 
Actually \cite{GH19} treats the full subcritical regime, but due to the performed transformation, the obtained counterterm can only be shown to be local (as expected) in a smaller regime. 
The follow-up work by Gerencs\'er~\cite{Ger20}, and the work by Bruned, Gerencs\'er, and Nadeem~\cite{BGN24} show that the counterterm is actually local in a larger regime. 
After the first submission of this manuscript, Bruned and Dotsenko~\cite{BD24} extended this to the full subcritical regime using tools from operad theory and homological algebra.
We also refer the reader to the recent work by Bailleul, Hoshino and Kusuoka \cite{BHK22} who study the solution theory for quasilinear SPDEs by working in a non-translation invariant setting. 
This has the drawback that the stochastic estimates are not known to be true in the full subcritical regime, and the expected translation-invariant form of the counterterm is only recovered in a subset of the full subcritical regime.

The first works to treat a quasilinear SPDE, namely the equation
\begin{equation}
\partial_t u + a(u) \partial_x^2 u = \xi \, , 
\label{eq:qSPDE}
\end{equation} 
in the full subcritical regime $\alpha \in (0,1)$ and without transforming the SPDE are contained in the four papers by Otto, Sauer, Smith, and Weber~\cite{OSSW21},  the second author together with Linares and  Otto~\cite{LOT23} and Linares, Otto, and Tsatsoulis~\cite{LOTT21}, respectively, and by the second author~\cite{Tem23}. The first of them derives a priori estimates for~\eqref{eq:qSPDE} in the full subcritical regime. The second one constructs the necessary algebraic objects, specifically the so-called structure group, and connects them to the ones introduced in~\cite{BHZ19} (see also the recent work of Bruned and Linares \cite{BL23} for an extension of \cite{LOT23} to general semilinear equations). The third one derives uniform stochastic estimates on the associated model, using Malliavin calculus based tools and the spectral gap inequality, and the last one studies convergence and universality of the renormalised model.
It is still an open problem to combine the uniform stochastic estimates and convergence of the model with the a priori estimate to construct a solution of the equation. 
This has been carried out after the first submission of this manuscript by Broux, Otto, and Steele~\cite{BOS25} for a semilinear equation based on a continuity method, 
and it is expected that this method can be adapted to the quasilinear case.

In \cite{OSSW21, LOT23, LOTT21, Tem23}, the authors work with a regularity structure indexed by multiindices unlike the tree-based one used in~\cite{Hai14,BHZ19,BCCH21,CH16,HS23}. 
We point the reader to the lecture notes~\cite{LO22,OST23} for an introduction to various aspects of the multiindex-based approach to regularity structures. 
This is also the setting we will adopt in this paper, 
and a first contribution can be thought of as generalising \cite{LOTT21} to a larger class of singular SPDEs, demonstrating the robustness of the method, see Subsection~\ref{ss:lott} below for a more detailed comparison.

\medskip

We now briefly discuss the main example of equation that we cover in the general class of equations of the form~\eqref{eq:spde-intro}. Consider the so-called thin-film equation with thermal noise\footnote{Also known in the literature as the stochastic thin-film equation.}, i.e. 
		\begin{equation}
			\partial_0 u = -\nabla \cdot (M(u) \nabla \Delta u ) +\nabla \cdot(M^{\frac{1}{2}}(u)\xi) \, ,
			\label{eq:tfe-intro}
			\tag{TFE}
		\end{equation}
	where $M$ is some sufficiently nice function of $u$ and $\xi$ is some rough, random forcing, typically space-time white noise. This equation governs the evolution of the height $u$ of a thin, viscous film driven by capillarity, limited by viscosity, and forced by thermal fluctuations. It can be formally derived either in an ad-hoc manner by applying a fluctuation-dissipation ansatz to the deterministic thin-film equation (see~\cite{GGKO22,DMS05}) or from first principles by considering an appropriate rescaling of the equations of fluctuating hydrodynamics (see~\cite{GMR06}). Here the function $M$ is referred to as the mobility of the film and is typically chosen to be a power law, i.e.~$M(u)=u^m$, $m \geq 1$ or more precisely $M(u)= u^3 + \lambda^{3-m} u^m$, with the most physically interesting cases corresponding to $0 < m < 3$. The choice $\lambda=0$ corresponds to imposing the no-slip boundary condition at the liquid-solid interface, while the choice $\lambda>0$ and $0<m <3$ corresponds to the Navier-slip boundary condition at the interface (see, for example, the discussion in~\cite{GO02}). One can check that~\eqref{eq:tfe-intro} can be cast into the form of~\eqref{eq:spde-intro} choosing $a(u)=1-M(u)$ and $b(u)=M^{\frac12}(u)$. 
	However, our results only apply to the non-degenerate setting, i.e.~to $a=1-M<1-\delta$ for some $\delta>0$. Furthermore, for a solution theory to be expected to apply, sufficient regularity of $M$ and $M^\frac12$ is required.

The thin-film equation with thermal noise has received a lot of attention in the recent years not only from the mathematics community but from the condensed and soft matter physics community, see \cite{DMS05, MR05, ZLLS22, LZLS23, RSEHDS24} and references therein. However, all previous literature on this equation from the mathematics as well as from the physics community has not argued for the existence of a counterterm as the equation was only studied 
    outside the singular regime, i.e.~with the noise $\xi$ regular enough so that all products on the right hand side of~\eqref{eq:tfe-intro} can be defined in a canonical manner. For $d=1$, the first construction of non-negative martingale solutions to~\eqref{eq:tfe-intro} in the regular regime with It\^o noise and in the presence of an interface potential (compensating the degeneracy of the equation) was given by Fischer and Gr\"un in~\cite{FG18}.  Gess and Gnann~\cite{GG20} constructed non-negative martingale solutions for~\eqref{eq:tfe-intro} with regular Stratonovich noise, quadratic mobility ($m=2$), and no interface potential. The question of existence for regular Stratonovich noise and $m=3$ was settled in~\cite{DGGG21} by Dareiotis, Gess, Gnann, and Gr\"un. We refer the reader to~\cite{MG22,Sau21} for constructions in higher dimensions with regular noise, to~\cite{GK22,DGGS23} for the study of solutions with compactly supported initial data, and to~\cite{GMR06,DOGKP19,GGKO22} for the study of numerical discretisations of this equation. To our knowledge, this work is the first to rigorously study the~\eqref{eq:tfe-intro} in the physically interesting singular, but non-degenerate regime. 
In this setting we derive upper bounds of the renormalisation constants, 
as well as matching lower bounds to rule out any hidden cancellations. 
Our result thus gives the first rigorous evidence of the existence of a counterterm for this equation,
which can be seen as a second main contribution of the present work.
The subtle dependence on the choice of mollifier makes it necessary to carry out some computations by brute force.
Moreover, in the quasilinear setting the choice of mollifier influences the
analytic function showing up as a prefactor in the counterterm (unlike in the semilinear setting, where it only influences a constant prefactor). 
Our computations also exhibit how exactly the mollifier influences this functional dependence of the counterterm which may be of independent interest. 
We believe that our results open the door to studying the physical meaning of the counterterm which shows up in the thin-film equation.
Regarding the physical relevance of this term, we choose to focus on the case $\alpha=1/2$ in Subsection~\ref{sec:structure}, since this choice of noise corresponds to the physically relevant regime in one spatial dimension where the noise arises from microscopic thermal fluctuations which influence the height of the film.
We also mention that since the first submission of this manuscript, some progress has been made for the simpler stochastic porous medium equation, see \eqref{eq:PME} below, in a singular and degenerate case \cite{TW25}. However, the methods used there do not seem to be applicable to the fourth order \eqref{eq:tfe-intro}.

		\begin{remark}\label{rem:comp}
			We note that we can just as easily consider a more general form of nonlinearity, for example, by assuming $a$ and $b$ to be matrix-valued. As will become clear in the later sections, the main arguments for the model estimates are insensitive to the exact form of the right hand side of~\eqref{eq:spde-intro}. What would be affected would be the exact form of the counterterm which is sensitive to the symmetries of the equation (see the discussion in~\cref{sec:counterterm}). Thus, for the sake of both brevity and notational convenience, we work with scalar-valued nonlinearities.

			Another rather straightforward generalisation is to replace our choice of $L$ with an arbitrary parabolic operator. Choosing $L$ as
				\begin{equation}
					L:= \partial_0 -\Delta \, ,
				\end{equation} 
            and replacing accordingly the nonlinear term $a(u)\nabla\Delta u$ by $a(u) \nabla u$ allows us to consider stochastic porous medium type equations of the form 
            \begin{equation}\label{eq:PME}\tag{PME}
                (\partial_0-\Delta)u = \nabla\cdot(a(u)\nabla u + b(u)\xi) \, .
            \end{equation}
			 In this case, all the proofs carry over \emph{mutatis mutandis}. 

			We could also drop the divergence on the right hand side of~\eqref{eq:spde-intro} and replace $a(u) \nabla u$ by $a(u)\Delta u$ which would allow us to treat the quasilinear multiplicative stochastic heat equation 
				\begin{equation}
					(\partial_0 -\Delta)u = a(u) \Delta u  + b(u) \xi
					\tag{qSHE}
					\label{eq:qSHE}
				\end{equation}
			or, by choosing $a\equiv0$, the so-called generalised parabolic Anderson model (gPAM). Again, the form of the estimates does not change but the form of the counterterm does. 
            In the later sections of the paper, we will remark, whenever possible, on the modifications in our arguments necessary to treat these other equations.
		\end{remark}

\subsection{Comparison to \texorpdfstring{\cite{LOTT21}}{[LOTT24]}}
\label{ss:lott}
While in \cite{LOTT21} many aspects were specifically set up for a one dimensional equation with the heat operator, a simple quasilinear nonlinearity, additive noise, and negative regularity of the Cameron--Martin space, 
all these restrictions have been resolved in the present work.
We still choose as a main working example the relatively simple 
\eqref{eq:spde-intro},
on the one hand to keep the already involved notation as simple as possible, 
and on the other hand because it is structurally close to the thin-film equation we are ultimately interested in and which we believe to be physically interesting.
To emphasize the flexibility of the approach and its implementation in the present work, several other equations are considered in parallel throughout the article (see~\cref{rem:comp,rem:oe1,rem:oe2,rem:oe3}).
In the following we briefly remark on the main technical differences between this work and \cite{LOTT21}.

The overall strategy of the proof outlined in Subsection~\ref{sec:strategy} 
is the same in both works. 
Minor differences are that, as mentioned previously,
$\partial_t+\Delta^2$ can be replaced by any differential operator $L$ satisfying a scaling invariance like~\eqref{eq:L-invariance} and such that $LL^*$ is elliptic, 
and that the spectral gap assumption is formulated for vector valued functionals (cf.~Assumption~\ref{ass}~(iii) and the discussion before). 
Related to this assumption, the restrictions on and relations between $\alpha$, $D$, and $s$ are made systematic (cf.~Remark~\ref{rem:essentialAssumption} and \eqref{kappa}). 

The algebraic framework is extended to handle several nonlinearities and arbitrary dimension.
This includes the algebraic definitions \eqref{piminus} of $\Pi_x^-$, \eqref{D0} and \eqref{Dn} of $D^{(\0)}$ and $D^{(\n)}$, and \eqref{exp} of $\Gamma^*$, the population conditions \eqref{pop1} and \eqref{pop2}, and the definition of the homogeneity in \eqref{homogeneity}. 
Heuristic arguments are given throughout which are robust and yield appropriate definitions for other equations. 

The inductive structure of the proof outlined in Subsection~\ref{sec:inductive_proof} 
is adapted to this extended algebraic setup. 
In particular, the ordering relation $\prec$ defined via \eqref{order} is adapted to any degree of modelledness obtained via $\mathrm{d}\Gamma^*$ defined in \eqref{dGamma}, 
as explained in Remark~\ref{rem:general_order}. 
Furthermore, the dependence of $\Pi^-_{x\beta}$ on $c_{\beta'}$ is more involved than in \cite{LOTT21} due to divergence form of the equation, 
which is taken care of in Lemma~\ref{lem:tri1}~(i).

The BPHZ-choice of the counterterm explained in Subsection~\ref{sec:bphz} requires a  further important innovation compared to \cite{LOTT21}, again due to the divergence form of the equation.
Although it may seem to be a minor difference, the divergence form makes the top-down postulate of the counterterm much more involved as 
the counterterm could a priori be a higher order tensor. 
That it is actually scalar valued (as desired) relies on the fact that the noise is invariant under orthogonal transformations. Proving this is subtle and relies on non-trivial invariance properties of model components contained in~\cref{prop:symmetries}~(4). 
We also refer the reader to the discussion in~\cref{rem:divcount} for more details on how the divergence form structure of the equation affects the choice of counterterm.
Similar symmetry arguments have been used e.g.~in \cite[Sections 4 and 5.1]{CCHS24} to reduce the form of the counterterm to a mass renormalisation in the context of Yang--Mills--Higgs in three spatial dimensions.

Also the annealed Schauder theory of Subsection~\ref{sec:schauder} is affected by the divergence form of the equation. 
More precisely, the power counting in the obtained estimates differs from \cite{LOTT21}. 
Since working on the full space requires sufficiently strong estimates on large scales for the solution formulas \eqref{eq:solutionformula} and \eqref{eq:representation} to hold, we carry out all arguments to ensure that this is still the case. 

It is for the reconstruction arguments of Subsection~\ref{sec:recon}, 
that the modelled distribution $\mathrm{d}\Gamma^*$ is such that the Malliavin derivative of the model is automatically modelled to the correct (high enough) order, 
as mentioned in the paragraph on Subsection~\ref{sec:inductive_proof} above. 
This manifests itself in the identity~\eqref{magic},
which here also takes care of several structurally different nonlinearities, 
and suitable definitions of the weights $\bar w$ and $w$ defined in~\cref{sec:application_SG} allowing for this high degree of modelledness. 
Furthermore, the definitions of the weights are made so that they hold for any degree of regularity of the Cameron--Martin space, 
whereas they are restricted to negative regularity in \cite{LOTT21}.
A final rather minor difference is that we choose to incorporate the projection $Q$ into the definition of $\mathrm{d}\Gamma^*$ (cf.~\eqref{dGamma}), which slightly simplifies the proof of Lemma~\ref{rec2} (Reconstruction II).

The price to pay is that the same projection $Q$ makes the proof of Lemma~\ref{alg3} (Algebraic argument III) in Subsection~\ref{sec:algebraic} 
more involved than in \cite{LOTT21}. 
Since this argument is far less complex than reconstruction, we think that this is a good tradeoff. 
The rest of the section is identical to \cite{LOTT21}, however we provide a proof of Lemma~\ref{alg1} (Algebraic argument I) as it was left to the reader in \cite{LOTT21}.

Some of the three-point arguments presented in Subsection~\ref{sec:3point} 
slightly change, again due to the projection $Q$, and due to the multidimensionality of the problem. The respective proofs which require an adaptation are given.

Averaging of Subsection~\ref{sec:average} potentially differes from \cite{LOTT21} as the weights differ. However, the presented proof only depends on the moment bounds \eqref{average_psi_w} and \eqref{average_psi_wx} and thus carries over from \cite{LOTT21}.

\subsection{Outline of the paper}
As already hinted at in the introduction, in this paper we focus on deriving uniform stochastic estimates on the model of equations of the form~\eqref{eq:spde-intro} in the framework of multiindex-based regularity structures as introduced in~\cite{OSSW21} and studied in~\cite{LOT23,LOTT21}. In~\cref{sec:setup}, we introduce numerous objects needed to define the model associated to~\eqref{eq:spde-intro}, starting by imposing a form for the counterterm based only on symmetries of the equation in~\cref{sec:counterterm}. Once we have introduced the model, which can be realised as an infinite hierarchy of linear PDEs (see~\eqref{piminuscomponents}), we present the main results of the paper in~\cref{thm:main} which contains the uniform estimates on the model. As a consequence of our model estimates, we also obtain as a corollary (using the results of~\cite{Tem23}) in~\cref{cor:uniqconvscal}, convergence of the model as the mollification parameter goes to $0$ and uniqueness of the limiting model. In~\cref{sec:divergingbounds}, we carefully study the counterterm  associated to~\eqref{eq:spde-intro} and provide diverging lower bounds on the renormalisation constants in~\cref{prop:scaling_c}. As a consequence, we show in~\cref{sec:structure} that, under appropriate conditions, the form of the counterterm agrees with the one conjectured in~\cite[Remark 9.1]{GGKO22}. 

\cref{sec:proof} is dedicated to the proof of the main result. We start by providing a bird's eye view of the structure of the proof in~\cref{sec:strategy} by introducing the numerous intermediate objects we need to derive the estimates on the model. Since the model is represented by an infinite hierarchy of linear PDEs, we need to derive our estimates inductively. To this end, in~\cref{sec:inductive_proof}, we introduce the ordering with respect to which we perform induction, while in~\cref{sec:bphz} we explain how we choose our renormalisation constants in a manner which is consistent with this ordering. 
\cref{sec:schauder} is dedicated to the proofs of the various integration arguments which involve inverting the linear operator $L$, while~\cref{sec:recon} provides proofs of the reconstruction arguments needed to make sense of the, a priori, singular products. We conclude with~\cref{sec:algebraic,sec:3point,sec:average} where we provide the algebraic, three-point, and averaging arguments needed for the proof. 

In \cref{sec:explicit} we provide a proof of the form of the counterterm stated in \cref{prop:scaling_c}. 

\cref{app:smooth,app:analytic} contain auxiliary results which are essential for the main result, but are mainly technical and distract for the main ideas of the proof.

\medskip
\paragraph{\bf Acknowledgements}
The authors would like to thank Lucas Broux, Benjamin Gess, Florian Kunick, and Felix Otto for many useful discussions during the course of this work. 
MT acknowledges funding from the Deutsche Forschungsgemeinschaft (DFG, German Research Foundation) under Germany's Excellence Strategy EXC 2044–390685587, Mathematics Münster: Dynamics–Geometry–Structure, 
and from the European Research Council (ERC) under the European Union’s Horizon 2020 research and innovation programme (Grant agreement No.~101045082). 
This work was partially written while both authors held positions at the Max–Planck Institute for Mathematics in the Sciences; funding and working conditions are gratefully acknowledged.

\section{\texorpdfstring{\bf Set up and main result}{Set up and main result}}\label{sec:setup}
\subsection{Ansatz for the counterterm}\label{sec:counterterm}
Since equation \eqref{eq:spde-intro} is expected to be in need of a renormalisation, we a priori postulate a counterterm 
on the level of the equation. To this end, we proceed as in \cite{LOTT21}:
we start from a general form of the counterterm,
and successively reduce the number of degrees of freedom by imposing suitable and natural postulates on the solution. The difficulty lies then in showing that what remains after such a reduction is rich enough to allow us to obtain uniform (in a mollification parameter) stochastic estimates.

\medskip

We will adhere to the following guiding principles: firstly, we aim for a deterministic counterterm that only depends on the law of the noise.
Secondly, since equation \eqref{eq:spde-intro} is local and in conservative form, 
it is desirable to obtain a counterterm that is a local function of the solution $u$ and conservative.
As we expect a solution $u$ to have H\"older regularity $\alpha\in (0,1)$, the counterterm can be a function of the solution $u$ and the space-time point $x$, and a polynomial in its derivatives. Moreover, a meaningful counterterm should be of lower order.
In particular we do not allow for derivatives $\partial_0$.  The most general counterterm of this form is the following
\begin{equation}
    \nabla\cdot \Big(\sum_{\beta} h_\beta(u,x) \bigotimes_{k\in\N} (\nabla^k u)^{\otimes\beta(k)}\Big),
\end{equation}
where $\beta$ is a multiindex over $k\in\N$ restricted to 
\begin{equation}
    \sum_{k\in\N} k\beta(k)<|L|-1=3, 
\end{equation}
$h_\beta$ is a $(1+\sum_k k\beta(k))$-tensor applied to a $(\sum_k k\beta(k))$-tensor, 
and $\nabla^k$ denotes the $k$-tensor $(\partial_{i_1}\cdots\partial_{i_k} u)_{i_1,\dots,i_k=1}^d$. 
In our setting the application of an $m$-tensor $H$ to an $n$-tensor $U$ with $m\geq n$, 
results in an $(m-n)$-tensor  given by 
\[
\Big(\sum_{i_1,\dots,i_n=1}^d H_{i_1,\dots,i_m} U_{i_1,\dots,i_n} \Big)_{i_{n+1},\dots,i_m=1}^d \, .
\]
To restrict the counterterm to fewer degrees of freedom, we will now take symmetries of the law of the noise into account. As $L$ is a constant coefficient operator, a solution $u$ of \eqref{eq:spde-intro} satisfies for all $v\in\R^{1+d}$ 
\begin{equation}
    \xi\sim\xi(\cdot+v) \quad\implies\quad u\sim u(\cdot+v) \, , 
\end{equation}
where we recall that $\sim$ denotes equality in law.
For a solution of the renormalised equation to preserve this property,
we would need to  restrict to functions $h_\beta$ with no explicit space-time dependence. We now turn to reflection invariance. Observe that for the spatial reflection $Rx:=(x_0,-x_1,\dots,-x_{d})$, any solution $u$ of \eqref{eq:spde-intro} has the property
\begin{equation}
    \xi\sim-\xi(R\cdot) \quad\implies\quad u\sim u(R\cdot).
\end{equation}
For a solution of the renormalised equation to preserve this property,
we need to restrict the counterterm to multiindices $\beta$ such that
\begin{equation}\label{eq:refl2}
    \sum_{k\in\N} k \beta(k) \quad\textnormal{is odd.}
\end{equation}
Putting these two properties together, this leads to the reduced form of the counterterm
\begin{equation}
    \nabla\cdot (h(u) \nabla u)
\end{equation}
for a matrix-valued function $h$. 
Another invariance of solutions $u$ of \eqref{eq:spde-intro} is the following generalisation of the previous spatial reflection: 
consider space-like orthogonal transformations of $\R^{1+d}$ of the form 
$Ox:=(x_0,\bar{O}(x_1,\dots,x_d))$ for an orthogonal matrix $\bar{O}\in\R^{d\times d}$. 
Since $a$ and $b$ are scalar-valued, any solution $u$ of \eqref{eq:spde-intro} has the property
\begin{equation}
    \xi\sim \bar{O}^T \xi(O\cdot) \quad\implies\quad u\sim u(O\cdot).
\end{equation}
This is preserved on the level of the renormalised equation, provided $h=\bar{O}^T h \bar{O}$. 
Since $\bar{O}$ is an arbitrary orthogonal matrix, $h$ has to be scalar-valued which we therefore assume. The last and most crucial postulate connects the counterterm with the nonlinearities of the equation. To do so, we no longer fix a pair of nonlinearities $a,b$, 
but consider all nonlinearities simultaneously.
This point of view allows for the following invariance of equation \eqref{eq:spde-intro}: 
for any shift $v\in\R$,
\begin{equation}
    (u,a,b) \textnormal{ satisfies } \eqref{eq:spde-intro}
    \quad\implies\quad
    (u-v,a(\cdot +v),b(\cdot +v)) \textnormal{ satisfies } \eqref{eq:spde-intro}.
    \label{shiftcovariance}
\end{equation}
By looking at all nonlinearities at once, 
the counterterm inherits a functional dependence on $a, b$, i.e.~$h[a,b](u)$. 
Preserving the above invariance on the level of the renormalised equation 
is guaranteed by postulating the following shift covariance:
for any shift $v\in\R$,
\begin{equation}
    h[a,b](u) = h[a(\cdot+v),b(\cdot+v)](u-v).
\end{equation}
This is equivalent to the fact that the counterterm $h$ coincides with a functional $c$ of the nonlinearities $a,b$ only, i.e. 
\begin{equation}\label{ct}
    h[a,b](u) = c[a(\cdot+u),b(\cdot+u)].
\end{equation}
Informally speaking, this expresses the idea that the form of the counterterm should not depend on the choice of origin in $u$-space. Finally, since the deterministic dynamics of~\eqref{eq:tfe-intro} are locally well-posed, it is natural to ask that the equation has no counterterm if we set $b \equiv 0$. Given the shift covariance~\eqref{ct}, this is tantamount to setting $c[a,0]=0$.

\begin{remark}[Form of the counterterm for other equations]
For the quasilinear stochastic heat equation \eqref{eq:qSHE} in $1+d$ dimensions, we have $L=(\partial_0-\Delta)$, $|L|=2$ and $\mathfrak{s}=(2,1,\dots,1)$. The restriction that the counterterm should be deterministic, local in $u$, and of lower order, leads to the following form
\begin{equation}
    \sum_{\beta} h_\beta(u,x) \bigotimes_{k\in\N} (\nabla^k u)^{\otimes\beta(k)},
\end{equation}
where $h_\beta$ is a $\sum_k k\beta(k)$-tensor and $\beta$ is restricted to $\sum_k k\beta(k)<|L|=2$. 
To exploit the reflection invariance of the noise, we observe that $u\sim u(R\cdot)$ provided $\xi\sim\xi(R\cdot)$.
Preserving this restricts to $\sum_{k} k\beta(k)$ being even, 
hence together with stationarity the counterterm reduces to exactly $h(u)$. 

Similarly, for the stochastic porous medium equation \eqref{eq:PME} 
in $1+d$ dimensions, 
we have $L=(\partial_0-\Delta)$, $|L|=2$, and $\mathfrak{s}=(2,1,\dots,1)$. 
The restriction that the counterterm should be deterministic, local in $u$, conservative, and of lower order leads to the following general form,
\begin{equation}
    \nabla\cdot \Big(\sum_{\beta} h_\beta(u,x) \bigotimes_{k\in\N} (\nabla^k u)^{\otimes\beta(k)}\Big),
\end{equation}
where we must have
\begin{equation}
    \sum_{k\in\N} k\beta(k)<|L|-1=1, 
\end{equation}
Imposing the same symmetries as for~\eqref{eq:tfe-intro}, leads us to the conclusion that the counterterm must be $0$.
\label{rem:oe1}
\end{remark}

\medskip

The above discussion motivates the following assumption on the ensemble $\xi$.
\begin{assumption}[Part I]\label{ass}
The law $\E$ of the tempered distribution $\xi$ satisfies 
\begin{enumerate}
\item[(i)] $\xi(\cdot)\sim\xi(\cdot+v)$ for all $v\in\R$,
\item[(ii)] $\xi(\cdot)\sim \bar{O}^T \xi(O\cdot)$ for any space-like orthogonal transformation \[
O x = (x_0, \bar{O}(x_1,\dots,x_d)) \, ,\] 
for some orthogonal $\bar{O} \in \R^{d\times d}$.
\end{enumerate}
\end{assumption}
\begin{remark}\label{rem:centered}
Combining (i) and (ii) with $\bar{O}=-{\rm id}$ of Assumption~\ref{ass} we obtain $\xi\sim-\xi$, in particular $\mathbb{E}\xi=0$. 
\end{remark}
Note that if an ensemble $\xi$ satisfies Assumption~\ref{ass}, 
then $\xi*\rho$ still satisfies the assumption, 
provided 
$\rho=\rho(O\cdot)$. 

\medskip

We can summarise the renormalisation problem as follows. 
We consider a mollified noise $\xi_\tau:=\xi*\psi_\tau$ for a suitably\footnote{we will choose a specific $\psi$ in \eqref{kernel}} rescaled mollifier $\psi$ satisfying Assumption~\ref{ass}, 
and we aim to find a scalar-valued function $h$ (depending on $\tau$ and the choice of the mollifier $\psi$) satisfying \eqref{ct} such that 
the solution manifold of the renormalised equation 
\begin{equation}\label{spde}
    L u =  \nabla \cdot(a(u) \nabla \Delta u)  + \nabla \cdot( b(u) \xi_\tau) - 
    \nabla\cdot (h(u) \nabla u) 
\end{equation} 
stays under (quantitative) control as the mollification parameter $\tau$ tends to $0$ \footnote{$u$ clearly depends on the mollification parameter $\tau>0$, but we suppress this dependence for the sake of notational convenience}. 
We will make more precise what we mean by controlling the solution manifold in \cref{sec:main}, see in particular Theorem~\ref{thm:main}. 

\medskip

To obtain this quantitative control on the solution manifold we will need an appropriate mixing assumption on the noise ensemble, 
which takes the form of a spectral gap (SG) inequality. 
We follow the discussion in \cite[Section 2.1]{LOTT21} and recall the main objects involved, for a more in-depth discussion we refer the reader to the aforementioned reference. We measure distances with the parabolic Carnot--Carath\'eodory distance
\begin{equation}\label{metric}
    \abss{x-y}:=\sum_{i=0}^d |x_i-y_i|^{\frac{1}{\mathfrak{s}_i}}, 
\end{equation}
where $\mathfrak{s}\in\N^{1+d}$ is the scaling associated to the operator $L$. Equipped with this notion of distance, the effective dimension $D$ of $\R^{1+d}$ is given by
\begin{equation}\label{dim}
    D=\sum_{i=0}^d \mathfrak{s}_i, 
    \label{eq:effectiveD}
\end{equation}
and we may define (parabolic) H\"older spaces with respect to this distance in the usual manner. Additionally, we can define anisotropic versions of Sobolev norms $\| \cdot\|_{\dot H^s}$ for $s\in\R$,
with the help of the space-time elliptic operator $LL^*$, as follows
\begin{align}
    \| G \|_{\dot H^{s}}:&= \left( \int_{\R^{1+d}} \dx{x} \, \big| (LL^*)^{\frac{s}{2|L|}}G(x)\big|^2 \right)^{\frac{1}{2}} \, .
    \label{eq:Sobolev}
\end{align}
Here, $G$ is allowed to be vector-valued (or even matrix-valued, cf. \eqref{sg}). Furthermore, we define cylindrical functionals 
\begin{equation}
    F[\xi] = f(\langle\xi,\zeta_1\rangle,\dots,\langle\xi,\zeta_N\rangle)
\end{equation}
for some $f\in C^\infty(\R^N;\R^n)$ and $\R^d$-valued Schwartz functions $\zeta_1,\dots,\zeta_N$, where $\langle\cdot,\cdot\rangle$ denotes the pairing between a tempered distribution and a Schwartz function. 
For such cylindrical functionals we may define
\begin{equation}\label{derivative}
    \frac{\partial F}{\partial\xi}[\xi] = \sum_{i=1}^N \partial_i f(\langle\xi,\zeta_1\rangle,\dots,\langle\xi,\zeta_N\rangle) \otimes \zeta_i 
\end{equation}
as a map from $\R^{1+d}$ to $\R^{n\times d}$, and for suitable $\delta\xi:\R^{1+d}\to\R^d$ 
\begin{equation}
    \R^n\ni\delta F(\delta\xi):= \Big\langle\frac{\partial F}{\partial\xi}[\xi],\delta\xi \Big\rangle
    := \int_{\R^{1+d}}\dx{x} \, \frac{\partial F}{\partial\xi}[\xi](x) \, \delta\xi(x).
\end{equation}
As will become clear in the later sections, we will only consider $n\in\{1,d\}$. Having introduced this notion of derivative, we are in a position to formulate our final assumption on the ensemble $\xi$.

\begin{duplicate*}[Part II]
The law $\E$ of the tempered distribution $\xi$ satisfies
\begin{enumerate}
\item[(iii)] for $\alpha\in(\max\{0,\frac32-\frac{D}{4}\},1)\setminus\mathbb{Q}$ and $s:= \alpha -3 +D/2$
the spectral gap inequality 
\begin{equation}\label{sg}
    \E\left|F-\E F\right|^2 \leq \E\left\|\frac{\partial F}{\partial\xi}\right\|_{\dot H^{-s}}^2,
\end{equation}
for all integrable cylindrical functionals $F$. 
In addition, we assume that the operator \eqref{derivative}, which is defined on cylindrical functions, is closable with respect to the topologies of $\E^\frac{1}{2}|\cdot|^2$ and $\E^\frac{1}{2}\|\cdot\|_{\dot H^{-s}}^2$.
\end{enumerate}
\end{duplicate*}
\smallskip

Assuming that the constant in \eqref{sg} is equal to $1$ is no restriction by a suitable rescaling of space-time.
Note that if an ensemble $\xi$ satisfies \eqref{sg} with constant $1$, 
then $\xi*\rho$ satisfies \eqref{sg} with constant $\|\rho\|_{L^1}$. 

\begin{remark}
The spectral gap inequality \eqref{sg} implies the corresponding $p$-version 
$\E|F-\E F|^p \lesssim_p 
\E \|\frac{\partial F}{\partial\xi}\|_{\dot H^{-s}}^p$
for any $p\geq2$, 
which we will frequently use in form of 
\begin{equation}\label{sgp}
    \E^\frac{1}{p}\left|F\right|^p 
    \lesssim_p |\E F| + 
    \E^\frac{1}{p} \left\|\frac{\partial F}{\partial\xi}\right\|_{\dot H^{-s}}^p .
\end{equation}
Indeed, the $p$-version follows formally by applying \eqref{sg} to $F^{p/2}$ and using the chain rule; 
for a rigorous proof see e.g.~\cite[Proposition~5.1]{IORT23}.
Furthermore, the closability of \eqref{derivative} extends to the topologies of $\E^\frac{1}{p}|\cdot|^p$ and $\E^\frac{1}{p}\|\cdot\|_{\dot H^{-s}}^p$.
\end{remark}

\begin{remark}
Let $\xi_t(y):=\xi*\psi_t(y)$ with $\psi_t$ defined in \eqref{kernel}. 
Then $\xi_t$ is a cylindrical functional with derivative $\psi_t(y-\cdot)$, 
which is centered by \cref{rem:centered}, 
and an application of \eqref{sgp} yields 
\begin{equation}\label{bound_xi}
\E^\frac{1}{p}|\xi_t(y)|^p\lesssim_p \|\psi_t\|_{\dot H^{-s}} \lesssim (\sqrt[8]{t})^{\alpha-3}.
\end{equation}
Hence, the Kolmogorov's continuity theorem tells us that indeed $\xi$ has a modification which has H\"older continuous realisations for any exponent less than $\alpha-3$. Thus, this motivates our choice of $s=\alpha-3+D/2$ in the spectral gap inequality \eqref{sg}.
\label{rem:reg}
\end{remark}

\begin{remark}\label{rem:essentialAssumption}
The only assumptions which are essential for our proof are ~\cref{ass} (i) and (iii). ~\cref{ass}~(ii) is made mainly for the sake of convenience to reduce the complexity of the counterterm. 
A careful inspection of our proof shows that the more complex setting can also be treated in this framework. 
Let us briefly comment on the restriction of $\alpha$ in Assumption~\ref{ass}~(iii): 
$\alpha>0$ is dictated by subcriticality and used several times in the proof, 
$\alpha>3/2-D/4$ is necessary for reconstruction, see \eqref{kappa}, and
$\alpha\not\in\mathbb{Q}$ is related to the failure of Schauder theory for integer exponents, see \cref{lem:int1} and \cref{rem:failureSchauder}, 
while $\alpha<1$ is assumed just for convenience to simplify the norms we work with and avoid case distinctions. 
\end{remark}

\begin{remark}[Spectral gap assumption for other equations]
For the stochastic porous medium equation \eqref{eq:PME} we would choose 
$s=\alpha-1+D/2$ in the spectral gap assumption \eqref{sg}. 
By a similar argument to the one in~\cref{rem:reg}, the resulting H\"older regularity of the noise would be arbitrarily close to but less than $\alpha-1$. 
Similarly, for the quasilinear stochastic heat equation~\eqref{eq:qSHE}, we would choose $s=\alpha-2+D/2$ and obtain noise of H\"older regularity arbitrarily close to but less than $\alpha -2$. 
In the former case, $\alpha$ would be restricted to 
$\alpha\in(0,1)\setminus\mathbb{Q}$, 
while in the latter it would be restricted to 
$\alpha\in(\max\{ 0, 1-\frac{D}{4} \},1)\setminus\mathbb{Q}$; 
in both cases we have $D=2+d$.\footnote{
As in \cref{rem:essentialAssumption}, the restriction $\alpha>0$ is for subcriticality, 
$\alpha<1$ to avoid case distinctions, and 
$\alpha\not\in\mathbb{Q}$ due to the failure of Schauder theory;
the analogous consideration that leads to \eqref{kappa} yields 
$2\alpha>1-D/2$ for the former case 
(which is weaker than $\alpha>0$ due to $D=2+d\geq3$), 
and $2\alpha>2-D/2$ for the latter case 
(which is more stringent than $\alpha>0$ in $D=2+d=2+1$).}
\label{rem:oe2}
\end{remark}

\subsection{The centered model}\label{sec:model}

In this section, we are after a parameterisation of the whole solution manifold. This is tantamount to defining the so-called centered model in the language of regularity structures.
Much of what is discussed in this section is not justified rigorously, and should rather be seen as providing some intuition. For the reader's convenience, we have marked these not fully rigorous arguments starting with ``Heuristic argument of...'', and ending with the symbol  $/\!\!/$. The reader familiar with the multiindex-based approach to regularity structures can continue directly with Section~\ref{sec:main}, where the objects informally introduced in this section are rigorously constructed and uniformly (in the mollification parameter $\tau$) estimated.

Again, we will closely follow the strategy of \cite[Section 2.2]{LOTT21} and start with an informal discussion.
If $a$ and $b$ are constant, we obtain from \eqref{ct} that the corresponding counterterm $h$ is constant. 
As we shall see in Lemma~\ref{lem:int1}, 
for fixed $x\in\R^{1+d}$ and under a suitable growth condition there is a unique solution $v$ of the linear equation \eqref{eq:lspde-intro} satisfying $v(x)=0$. Therefore, a canonical parameterisation for solutions $u$ of \eqref{spde} for $a=0$, $b=1$ is given by 
$u=v+p$, where $p$ satisfies $L p = 0$. 
Such $p$ are analytic, and it will be convenient to actually consider all analytic $p$. 
This can be done at the expense of relaxing \eqref{spde} to hold only modulo analytic functions.
The hope is (and this has been made rigorous in case of a semilinear equation in \cite{BOS25}), 
that this parameterisation persists for sufficiently smooth $a$ and $b$ that are sufficiently close to $0$ and $1$, respectively. 
By the shift covariance \eqref{shiftcovariance} and the corresponding property of the counterterm, it is even expected that $p$ with $p(x)=0$ for fixed $x\in\R^{1+d}$ provide a sufficiently rich parameterisation.
A choice of coordinates on this space is given by
\begin{equation}\label{p_n}
    \p_\n[p]:=\frac{1}{\n!}\,\frac{\partial^\n p}{\partial x^\n}(0), 
    \quad
    \n\in\N_0^{1+d}\setminus\{\0\}, 
\end{equation}
which together with
\begin{equation}
    \a_k[a]:=\frac{1}{k!}\,\frac{\d^k a}{\d u^k} (0)
    \quad\textnormal{and}\quad 
    \b_\l[b]:=\frac{1}{\l!}\,\frac{\d^\l b}{\d u^\l} (0),
    \quad
    k,\l\in\N_0,
    \label{coordinates}
\end{equation}
is expected to provide a complete parameterisation of the above mentioned solution manifold.
From now on we shall always assume $k,\l\in\N_0$ and $\n\in\N_0^{1+d}$ and we will usually refrain from writing the corresponding set. Additionally, we will write $\n \neq \0$ for $\n \in \N_0^{1+d}\setminus \{\0\}$. 

\medskip

To approach the centered model $\Pi_x$, the previous discussion suggests, still on an informal level, to make the ansatz
\begin{align}
    &u(y)-u(x) \\
    &= 
    \sum_\beta \Pi_{x\beta}(y)
    \prod_k \Big(\frac{1}{k!}\,\frac{\d^k a}{\d u^k}(u(x))\Big)^{\beta_a(k)}
    \prod_\l \Big(\frac{1}{\l!}\,\frac{\d^\l b}{\d u^\l}(u(x))\Big)^{\beta_b(\l)}
    \prod_{\n\neq\0} \Big(\frac{1}{\n!}\,\frac{\partial^\n p}{\partial x^\n}(x) \Big)^{\beta_p(\n)} \, , 
\end{align}
where we sum over multiindices\footnote{where for functions $f,g$, $f\cup g$ denotes their set-theoretic union.} $\beta=\beta_a\cup\beta_b\cup\beta_p: \N_0 \,\dot\cup\, \N_0 \,\dot\cup\, (\N_0^{1+d}\setminus\{\0\}) \to \N_0$, i.e.~$\beta(k)=\beta_a(k)$ for $k$ in the first copy of $\N_0$, $\beta(\ell)=\beta_b(\ell)$ for $\ell$ in the second copy of $\N_0$, and $\beta(\n)=\beta_p(\n)$ for $\n\in\N_0^{1+d}\setminus\{\0\}$.

The idea is now to separate the construction of the coefficients $\Pi_{x\beta}$ of the above ansatz, 
from the construction of an actual solution $u$ of \eqref{spde}.
The former is based on probabilistic tools and is the content of this work, 
while the latter is based on analytic tools and is yet to be done for \eqref{spde} (as mentioned earlier this has been carried out for a semilinear equation in \cite{BOS25}).
We therefore need to get more information on the coefficients $\Pi_{x\beta}$ of the above ansatz, which we continue with after the following notational remark.

\smallskip
\begin{nota*} 
We remark here that we shall denote by $k$ an element of the first copy of $\N_0$ corresponding to the nonlinearity $a$, and by $\l$ an element of the second copy of $\N_0$ corresponding to the nonlinearity $b$.
The unit vectors in directions $k,\l,\n$ are denoted by $e_k, f_\l, g_\n$, respectively, 
i.e.~$e_{k}: \N_0 \,\dot\cup\, \N_0 \,\dot\cup\, (\N_0^{1+d}\setminus\{\0\}) \to \N_0$ satisfies 
$e_k(k')=\delta_k^{k'}$ for $k'$ in the first copy of $\N_0$, 
$e_k(\l')=0$ for $\l'$ in the second copy of $\N_0$, 
and $e_k(\n')=0$ for $\n'\in\N_0^{1+d}\setminus\{\0\}$.
\end{nota*}
\smallskip

In the case $a=0$ and $b=1$ the above ansatz is consistent with $u-u(x)=v+p$, provided we choose $\Pi_{x f_0} = v$ and 
\begin{equation}\label{polypart:poly} 
\Pi_{x g_\n} = (\cdot-x)^\n \, .
\end{equation}
To get our hands on $\Pi_{x\beta}$ for all other $\beta$ we proceed as follows.
By using the monomials
\begin{equation}
    \z^\beta:=\prod_k \a_k^{\beta_a(k)} \prod_\l \b_\l^{\beta_b(\l)} \prod_{\n\neq\0} \p_\n^{\beta_p(\n)}, 
\end{equation}
the above power series ansatz can be more compactly written as
\begin{equation}\label{ansatz}
    u(y)-u(x) = 
    \sum_\beta \Pi_{x\beta}(y) \, \z^\beta[a(\cdot+u(x)),b(\cdot+u(x)),p(\cdot+x)]. 
\end{equation}
This allows to work with the space of formal power series $\R[[(\a_k)_k,(\b_\l)_\l,(\p_\n)_{\n\neq\0}]]$, 
and define $\Pi_x=\sum_\beta \Pi_{x\beta} \z^\beta$. 
Recall that this space comes with the following algebra structure: for $\pi'=\sum_\beta\pi'_\beta\z^\beta$ and $\pi''=\sum_\beta\pi''_\beta\z^\beta$ it holds $\pi'\pi''=\sum_\beta(\pi'\pi'')_\beta\z^\beta$ with $(\pi'\pi'')_\beta=\sum_{\beta'+\beta''=\beta}\pi'_{\beta'}\pi''_{\beta''}$.
Also $c$ from \eqref{ct} as a functional of $a,b$ can be (formally) identified with a power series $c=\sum_\beta c_\beta \z^\beta$. 
From the equation \eqref{spde} (relaxed to hold only modulo analytic functions), one can then (formally!) derive the following hierarchy of PDEs for the coefficients $\Pi_{x\beta}$:
\begin{subequations}
\label{modelhierarchy}
\begin{align}
    L\Pi_{x\beta} &\, = \nabla\cdot\Pi^-_{x\beta} \quad\textnormal{mod analytic functions, where} \label{hierarchy} \\
    \Pi^-_{x\beta} &:= \Big(
    \sum_k \a_k \Pi_x^k \nabla\Delta\Pi_x 
    + \sum_\l \b_\l \Pi_x^\l \xi_\tau - 
    \sum_{m\in\N_0} \tfrac{1}{m!} \Pi_x^m\nabla\Pi_x (D^{(\0)})^m c \Big)_\beta\, . \label{piminus}
\end{align}  
\end{subequations}
Here $D^{(\0)}$ is given by the derivation 
\begin{equation}\label{D0}
    D^{(\0)}\coloneqq \sum_k (k+1) \a_{k+1} \partial_{\a_k} 
    + \sum_\l (\l+1) \b_{\l+1} \partial_{\b_\l}.
\end{equation}
Note that its matrix components $(D^{(\0)})_\beta^\gamma$, defined by 
\begin{equation}
    D^{(\0)} \z^\gamma = \sum_\beta (D^{(\0)})_\beta^\gamma \, \z^\beta ,
\end{equation}
are given by 
\begin{equation}\label{D0comp}
    (D^{(\0)})_\beta^\gamma = \sum_k (k+1) \gamma_a(k) \delta_\beta^{\gamma-e_k+e_{k+1}} 
    + \sum_\ell (\ell+1) \gamma_b(\ell) \delta_\beta^{\gamma-f_\ell+f_{\ell+1}},
\end{equation}
and that the sums over $k,\ell$ are finite for fixed $\beta$. 
Furthermore, for fixed $\beta$ there are only finitely many $\gamma$ with $(D^{(\0)})_\beta^\gamma\neq0$, 
hence \eqref{D0} 
is well defined on $\RR$.
For later use, we mention the following consequence of \eqref{D0comp}
\begin{equation}\label{D0props}
    (D^{(\0)})_{\beta}^{\gamma}\neq 0  
    \hspace{1ex}\implies\hspace{1ex}\left\{
    \begin{array}{l}
    \sum_k \beta_a(k) = \sum_k \gamma_a(k), \\
    \sum_\ell \beta_b(\ell)=\sum_\ell \gamma_b(\ell), \\
    \sum_k k\beta_a(k)+\sum_\ell \ell\beta_b(\ell)
    =1+\sum_k k\gamma_a(k)+\sum_\ell \ell\gamma_b(\ell),\\
    \beta_p(\n)= \gamma_p(\n) \text{ for all }\n\neq\0.
    \end{array}\right.
\end{equation}

\medskip

\noindent

\textit{Heuristic argument for \eqref{modelhierarchy}.}
We first note that, with the shorthand notation $a':=a(\cdot+u(x))$, $b':=b(\cdot+u(x))$ and $p':=p(\cdot+x)$,  the above ansatz~\eqref{ansatz} can be rewritten as $u(y)-u(x)=\Pi_x[a',b',p'](y)$. Then, clearly the left hand side of \eqref{spde} equals $L\Pi_x[a',b',p']$. 
For the first term on the right hand side of \eqref{spde}, we note that 
$a(u)=a'(u-u(x))=a'(\Pi_x[a',b',p'])$, which by \eqref{coordinates}, 
yields $a(u)=(\sum_k \a_k\Pi_x^k)[a',b',p']$. Hence, this term can be written as  
\[
(\nabla\cdot\sum_k \a_k\Pi_x^k\nabla\Delta\Pi_x)[a',b',p'] \, .
\]
For the second term on the right hand side of \eqref{spde}, we proceed in a similar manner to obtain that it equals
\[
(\nabla\cdot\sum_\ell \b_\ell\Pi_x^\ell\xi)[a',b',p'] \, .
\]
For the last term on the right hand side of \eqref{spde}, we have to work a little bit harder.  Using~\eqref{ct}, we know that the counterterm is of the form $c[a(\cdot+u),b(\cdot+u)]$. 
To express this as a functional of $a$ and $b$ 
we first identify $D^{(\0)}$ with the infinitesimal generator of $u$-shift on $(a,b)$-space given by  
\[
\frac{\d}{\d v}\Big|_{v=0} c[a(\cdot+v),b(\cdot+v)] \, .
\]
To see this identification, observe that the infinitesimal generator maps by \eqref{coordinates} $\a_k \mapsto (k+1)\a_{k+1}$ and $\b_\ell \mapsto (\ell+1)\b_{\ell+1}$. 
Moreover, it maps $\p_\n\mapsto0$ and it is a derivation, 
on $\R[\a_k,\b_\ell,\p_\n]$ it therefore has to coincide with the derivation $D^{(\0)}$.

By iteration we obtain
\[((D^{(\0)})^m c)[a,b]=\frac{\d^m}{\d v^m}\Big|_{v=0} c[a(\cdot+v),b(\cdot+v)] \, ,\] 
and hence by Taylor's theorem
\begin{equation}
    c[a(\cdot+v),b(\cdot+v)]=\Big(\sum_{m\in\N_0} \tfrac{1}{m!} v^m (D^{(\0)})^m c\Big)[a,b].
\end{equation}
Since $h[a,b](u)=c[a(\cdot+u),b(\cdot+u)]=c[a'(\cdot+\Pi_x[a',b',p']),b'(\cdot+\Pi_x[a',b',p'])]$, we obtain
\begin{equation}
    h[a,b](u) = \Big(\sum_{m\in\N_0} \tfrac{1}{m!} \Pi_x^m (D^{(\0)})^m c\Big)[a',b',p'],
\end{equation}
which finally tells us that the last term on the right hand side of \eqref{spde} equals
\begin{equation}
    \Big(\nabla\cdot\sum_{m\in\N_0} \tfrac{1}{m!} \Pi_x^m\nabla\Pi_x (D^{(\0)})^m c\Big)[a',b',p'].
\end{equation}
Since $a,b,p$ were arbitrary, this concludes the argument for \eqref{modelhierarchy}.
\hfill /\!\!/

\medskip

\begin{remark}\label{rem:hierarchy}
Let us point out that for fixed $\beta$ the sums over $k,\ell,m$ in \eqref{piminus} 
are finite sums and are thus well-defined. 
Although \eqref{modelhierarchy} looks like a nonlinear equation, it is, in fact, an infinite hierarchy of linear equations, 
\begin{align}
    L\Pi_{x\beta} &\, = \nabla\cdot\Pi^-_{x\beta} \quad\textnormal{mod analytic functions,} \\
    \Pi^-_{x\beta} &= \sum_k\sum_{e_k+\beta_1+\cdots+\beta_{k+1}=\beta}
    \Pi_{x\beta_1}\cdots\Pi_{x\beta_k}\nabla\Delta\Pi_{x\beta_{k+1}} \label{piminuscomponents} \\
    &\ + \sum_\ell\sum_{f_\ell+\beta_1+\cdots+\beta_\ell=\beta} 
    \Pi_{x\beta_1}\cdots\Pi_{x\beta_\ell}\xi_\tau \\
    &\ - \sum_m\tfrac{1}{m!}\sum_{\beta_1+\cdots+\beta_{m+2}=\beta} 
    \Pi_{x\beta_1}\cdots\Pi_{x\beta_m}\nabla\Pi_{x\beta_{m+1}}
    ((D^{(\0)})^m c)_{\beta_{m+2}}. 
\end{align}
As follows from Lemma~\ref{lem:tri1}~(i), this is indeed a hierarchy. 
To illustrate the complexity of this hierarchy, 
we enumerate a few examples\footnote{we list those components that are relevant for $\alpha>1/2$, see \eqref{badMultiindices}} of the equations solved by components $\Pi_{x\beta}$: 
\begin{align}
    L\Pi_{x f_0} &= \nabla\cdot \xi_\tau, \\
    L\Pi_{x f_0+f_1} &=\nabla\cdot \big(\Pi_{x f_0}\xi_\tau
    \textcolor{Maroon}{ \ 
    - \, \nabla\Pi_{xf_0} c_{f_1} }
    \big) , \\
    L\Pi_{x f_1+g_\n} &= \nabla\cdot \big((\cdot-x)^\n\xi_\tau
    \textcolor{Maroon}{ \ 
    - \, \nabla (\cdot-x)^\n c_{f_1} }
    \big) , \\
    L\Pi_{x e_1+2f_0} 
    &=\nabla\cdot \big(\Pi_{x f_0}\nabla\Delta\Pi_{x f_0}
    \textcolor{Maroon}{ \ 
    - \, \nabla\Pi_{xf_0} c_{e_1+f_0} }
    \big) , \\
    L\Pi_{x e_1+f_0+g_\n} 
    &= \nabla\cdot \big(\Pi_{x f_0}\nabla\Delta (\cdot-x)^\n 
    + (\cdot-x)^\n\nabla\Delta\Pi_{x f_0} 
    \textcolor{Maroon}{ \ 
    - \, \nabla(\cdot-x)^\n c_{e_1+f_0} }
    \big) , \\[1.5ex]
    L\Pi_{x 2f_1+g_\n} &= \nabla\cdot \big(\Pi_{x f_1+g_\n} \xi_\tau
    \textcolor{Maroon}{ \ 
    - \, \nabla (\cdot-x)^\n c_{2f_1}
    - \nabla\Pi_{x f_1+g_\n} c_{f_1} }
    \big) , \\
    L\Pi_{x f_0+f_2+g_\n} 
    &= \nabla\cdot \big(2(\cdot-x)^\n \Pi_{x f_0} \xi_\tau \textcolor{Maroon}{ \ 
    - \, \nabla (\cdot-x)^\n c_{f_0+f_2} } \\
    &\textcolor{Maroon}{ \hphantom{=\nabla\cdot\big(}
    - \, \Pi_{xf_0} \nabla(\cdot-x)^\n \underbrace{(D^{(\0)} c)_{f_2}}_{=2c_{f_1}}
    - \, (\cdot-x)^\n \nabla\Pi_{xf_0} \underbrace{(D^{(\0)} c)_{f_2}}_{=2c_{f_1}} }
    \big) , \\
    L\Pi_{x e_2+2f_0+g_\n} 
    &= \nabla\cdot \big(\Pi_{x f_0}^2\nabla\Delta (\cdot-x)^\n 
    + 2(\cdot-x)^\n\Pi_{x f_0} \nabla\Delta \Pi_{x f_0}
    \textcolor{Maroon}{ \ 
    - \, \nabla (\cdot-x)^\n c_{e_2+2f_0} } \\
    &\textcolor{Maroon}{ \hphantom{=\nabla\cdot\big(}
    - \, \Pi_{xf_0}\nabla(\cdot-x)^\n 
    \underbrace{(D^{(\0)}c)_{e_2+f_0}}_{=2c_{e_1+f_0}} 
    - \, (\cdot-x)^\n \nabla\Pi_{xf_0} 
    \underbrace{(D^{(\0)}c)_{e_2+f_0}}_{=2c_{e_1+f_0}} }
    \big) , \\
    L\Pi_{x 2e_1+2f_0+g_\n} 
    &= \nabla\cdot \big(\Pi_{x e_1+2f_0}\nabla\Delta (\cdot-x)^\n 
    + (\cdot-x)^\n\nabla\Delta\Pi_{x e_1+2f_0} \\
    &\hphantom{=\nabla\cdot\big(} 
    + \Pi_{x e_1+f_0+g_\n}\nabla\Delta\Pi_{x f_0} 
    + \Pi_{x f_0}\nabla\Delta\Pi_{x e_1+f_0+g_\n} \\
    &\textcolor{Maroon}{\hphantom{=\nabla\cdot\big(}
    - \, \nabla(\cdot-x)^\n c_{2e_1+2f_0} 
    - \nabla\Pi_{xe_1+f_0+g_\n} c_{e_1+f_0} } \\
    &\textcolor{Maroon}{\hphantom{=\nabla\cdot\big(}
    - \, \Pi_{xf_0}\nabla(\cdot-x)^\n 
    \underbrace{(D^{(\0)}c)_{2e_1+f_0}}_{=c_{e_0+e_1+f_0}}
    - \, (\cdot-x)^\n \nabla\Pi_{xf_0} 
    \underbrace{(D^{(\0)}c)_{2e_1+f_0}}_{=c_{e_0+e_1+f_0}} }
    \big) , \\
    L\Pi_{x e_1+f_0+f_1+g_\n} 
    &= \nabla\cdot \big(\Pi_{x f_0+f_1} \nabla\Delta (\cdot-x)^\n 
    + (\cdot-x)^\n \nabla\Delta\Pi_{x f_0+f_1} \\
    &\hphantom{=\nabla\cdot\big(} 
    + \Pi_{x f_0}\nabla\Delta\Pi_{x f_1+g_\n} 
    + \Pi_{x f_1+g_\n}\nabla\Delta\Pi_{x f_0} 
    + \Pi_{x e_1+f_0+g_\n}\xi_\tau \\
    &\textcolor{Maroon}{\hphantom{=\nabla\cdot\big(}
    - \, \nabla(\cdot-x)^\n c_{e_1+f_0+f_1} 
    - \, \nabla\Pi_{xf_1+g_\n} c_{e_1+f_0}
    - \, \nabla\Pi_{xe_1+f_0+g_\n} c_{f_1} } \\
    &\textcolor{Maroon}{\hphantom{=\nabla\cdot\big(}
    - \, \Pi_{xf_0}\nabla(\cdot-x)^\n 
    \underbrace{(D^{(\0)}c)_{e_1+f_1}}_{=c_{e_0+f_1}+c_{e_1+f_0}}
    - \, (\cdot-x)^\n \nabla\Pi_{xf_0} 
    \underbrace{(D^{(\0)}c)_{e_1+f_1}}_{=c_{e_0+f_1}+c_{e_1+f_0}} } 
    \big) .
\end{align}
\end{remark}

From Remark~\ref{rem:hierarchy}, we can already notice that for some multiindices $\beta$ we expect $\Pi_{x\beta}=0$, for example $\beta\in\{0,2f_0,\dots\}$.  This motivates the following definition. We call a multiindex \emph{populated}, if and only if 
\begin{equation}\label{pop1}
\begin{split}
    1 + \sum_k k\beta_a(k) + \sum_\l \l \beta_b(\l) = \sum_\l \beta_b(l) + \sum_{\n\neq\0} \beta_p(\n) \quad\text{and} \hspace{12.5ex}\\
    \Big( \beta \text{ is \emph{purely polynomial}, i.e. } \beta=g_\n \text{ for some }\n\neq\0,
    \quad\text{or}\quad
    \sum_\l \beta_b(\l)>0 \Big).
\end{split}
\end{equation}
We can motivate the above condition through the following formal scaling argument. 

\medskip

\noindent
\textit{Heuristic argument for \eqref{pop1}.}
Consider~\eqref{eq:tfe-intro} with some smooth ensemble $\xi$ and define $u_{\lambda}= \lambda u$ for some $\lambda>0$. Then, it is easy to check that $u_\lambda$ solves the same equation as $u$ but with nonlinearities $a_\lambda =a(\lambda^{-1}\cdot)$ and $b_\lambda= \lambda b(\lambda^{-1}\cdot)$ and  parameterisation $p_\lambda = \lambda p$. Thus, using the formal power series expansion~\eqref{ansatz} for the solution, we have
\begin{align}
&\lambda(u(y)-u(x))
=u_\lambda(y)-u_\lambda(x) \\ 
&= \sum_{\beta}\Pi_{x\beta}(y)\mathsf{z}^\beta[a_\lambda(\cdot + u_\lambda(x)) ,b_\lambda(\cdot + u_{\lambda}(x)),p_\lambda(\cdot + x)] \\
&= \sum_{\beta} 
\lambda^{-\sum\limits_{k}k\beta_a(k)
-\sum\limits_{\ell}(\ell-1)\beta_b(l) 
+\sum\limits_{\mathbf{n}\neq \0} \beta_p(\mathbf{n})} 
\Pi_{x\beta}(y)\mathsf{z}^\beta[a(\cdot + u(x)) ,b(\cdot + u(x)),p(\cdot + x)] \, .
\end{align}
The first part of the population condition~\eqref{pop1} follows from equating the powers of $\lambda$ of the above expression and \eqref{ansatz} multiplied by $\lambda$. 

For the second part of~\eqref{pop1}, we impose that $\Pi_{x\beta}$ is a multilinear map of the noise of rank at least $1$, 
unless $\beta$ is purely polynomial. 
Let $u_\lambda$ denote the solution obtained by choosing the noise $\lambda \xi$ for some $\lambda>0$. Clearly, this is the same as considering the solution obtained by choosing the nonlinearity $b_\lambda = \lambda b$. Using the power series expansion of the solution, we have
\begin{align}
\sum_{\beta}\Pi_{x\beta}[\lambda\xi](y)\mathsf{z}^\beta[a,b,p]
=\sum_{\beta}\Pi_{x\beta}[\xi](y)\mathsf{z}^\beta[a,b_\lambda,p] 
= \sum_{\beta} \lambda^{\sum\limits_{\ell}\beta_b(\ell)} \Pi_{x\beta}[\xi](y)\mathsf{z}^\beta[a,b,p] .
\end{align}
From the above expression, clearly $\sum_{\ell}\beta_b(\ell)>0$ for all $\beta$ not purely polynomial since, otherwise, the associated $\Pi_{x\beta}$ is not multilinear with rank at least $1$.
\hfill /\!\!/

\medskip

Analogous, we will restrict $c\in\R[[\a_k,\b_\ell]]$ a priori by the following population condition
\begin{equation}\label{pop2}
    c_\beta\neq 0
    \quad\implies\quad
    \sum_k k\beta_a(k)+\sum_\ell \ell\beta_b(\ell)=\sum_\ell \beta_b(\ell)
    \quad\text{and}\quad \sum_\ell\beta_b(\ell)>0.
\end{equation}
We will see in \cref{sec:bphz} that $c$-components violating this condition will not play any role in renormalisation. 
One can also motivate this population constraint using the same scaling argument as for $\Pi_{x\beta}$. 

\medskip

\noindent
\textit{Heuristic argument for \eqref{pop2}.}
If we insist that, even in the presence of the counterterm, $u_\lambda$ as defined earlier is a solution 
(with $a_\lambda,b_\lambda,p_\lambda$), then we must have 
\begin{equation}
h[a_\lambda,b_\lambda](u_\lambda)= h[a,b](u)  \, .
\end{equation}
The first part of condition~\eqref{pop2} then follows by using the power series expansion for the counterterm and enforcing the above identity. 
For the second part of~\eqref{pop2}, we consider the $c[a,b_\lambda]$ for $b_\lambda=\lambda b$ with $\lambda>0$. 
Then, using the power series expansion of $c$, we have
\begin{equation}
c[a,b_\lambda] = \sum_{\beta} c_\beta \mathsf{z}^\beta[a,b_\lambda ]= \sum_{\beta}\lambda^{\sum\limits_{\ell}\beta_b(\ell) } c_\beta \mathsf{z}^\beta[a,b ] \, .
\end{equation}
Since by assumption $c[a,0]=0$, each component of the above power series 
for which $c_\beta\neq0$ must converge to $0$ as $\lambda \to 0$. Thus, we must have $\sum\limits_{\ell}\beta_b(\ell)>0$.
\hfill /\!\!/

Summarising, we therefore postulate 
\begin{equation}\label{pop3}
    \Pi_{x\beta}\neq0\quad\implies\quad\beta\textnormal{ populated},
\end{equation}
and consider $\Pi_x$ as taking values in 
\begin{equation}
    \T^*:=\big\{\pi\in\RR \,\big|\, \pi_\beta\neq0\implies\beta\textnormal{ populated}\big\}.
\end{equation}
For later use, we introduce the polynomial part $\bar{\T}^*$ of $\T^*$ by 
\begin{equation}
    \bar{\T}^*:=\big\{\pi\in\RR \,\big|\, \pi_\beta\neq0\implies\beta\textnormal{ purely polynomial}\big\}.
\end{equation}
This induces the decomposition of $\T^*$ into 
\begin{equation}
\T^*=\bar{\T}^*\oplus\widetilde{\T}^*.
\end{equation}
Analogous to $\Pi_x$, we want to consider $\Pi^-_x$ as a $\T^*$ valued map, where we note the following:
For $\pi,\pi'\in \T^*$, one can check that $\sum_\ell \b_\ell\pi^\ell$ is again in $\T^*$, 
and the same holds true for $\sum_m \pi^m\pi'(D^{(\0)})^m c$ due to the population constraint \eqref{pop2} of $c$ 
and the mapping properties \eqref{D0props} of $D^{(\0)}$.
Moreover, due to the presence of the factors $\b_\ell$ and $c$, 
these products belong in fact to $\widetilde{\T}^*$.
As opposed to that, $\sum_k \a_k \pi^k \pi'$ is in general not\footnote{consider e.g.~$\p_{\n_1},\p_{\n_2}\in \T^*$, then $\a_1 \p_{\n_1}\p_{\n_2}\not\in \T^*$} an element of $\T^*$.
However, in case it is an element of $\T^*$, then due to the presence of the factor $\a_k$ it is automatically contained in $\widetilde{\T}^*$.
We therefore introduce the projection $P$ from $\RR$ to $\widetilde{\T}^*$ in the definition of $\Pi^-_x$, 
to obtain the $\widetilde{\T}^*$ valued map 
\begin{equation}
\Pi^-_x = 
P\sum_k \a_k \Pi_x^k \nabla\Delta\Pi_x 
+ \sum_\l \b_\l \Pi_x^\l \xi_\tau 
- \sum_{m\in\N_0} \tfrac{1}{m!} \Pi_x^m\nabla\Pi_x (D^{(\0)})^m c \, , 
\label{eq:piminusttilde}
\end{equation}
which is consistent with \eqref{piminus}.

\medskip

We turn to the \emph{homogeneity} $|\beta|$ of a multiindex $\beta$ which we define as follows
\begin{equation}\label{homogeneity}
    |\beta|:=\alpha(1+[\beta])+|\beta|_p,
\end{equation}
where 
\begin{equation}\label{homogeneity2}
    [\beta]:=\sum_k k\beta_a(k)+\!\sum_\ell \ell\beta_b(\ell)-\!\sum_{\n\neq\0}\beta_p(\n),
    \ \
    |\beta|_p:=\sum_{\n\neq\0} |\n|\beta_p(\n),
    \ \
    |\n|:=\sum_{i=0}^d \mathfrak{s}_i \n_i.
\end{equation}
The appearance of the homogeneity is best seen from the following formal\footnote{but it can be made rigorous, see Corollary~\ref{coro}~3.d) below} scaling argument.

\medskip

\noindent
\textit{Heuristic argument for \eqref{homogeneity}.}
Recall from \eqref{sinv1} that if $u$ is a solution to \eqref{eq:spde-intro}, 
then $u^\eps$ is a solution to \eqref{eq:spde-intro} provided $a$, $b$ and $\xi$ are replaced by $\hat a:=a(\eps^\alpha\cdot)$, 
$\hat b:=b(\eps^\alpha\cdot)$ and $\hat\xi$ given by \eqref{eq:noise-invariance}.
Notice that this persists for the renormalised equation, provided $h$ is replaced by $\hat h:=\eps^2 h(\eps^\alpha\cdot)$. 
On the parameterisation $p$, we now impose the same scaling as on $u$, i.e.~$\hat p(x):=\eps^{-\alpha}p(\hat x)$ with $\hat x$ given in \eqref{xhat}. 
From this we obtain $\eps^{-\alpha}u[a,b,p,\xi](\hat y)=u^\eps(y) = u[\hat a,\hat b,\hat p, \hat\xi](y)$. 
Using $\eps^\alpha \z^\beta[\hat a,\hat b,\hat p] = \eps^{|\beta|} \z^\beta[a,b,p]$ in \eqref{ansatz}, we read off
\begin{equation}
    \Pi_{\hat x \beta}[\xi](\hat y) = \eps^{|\beta|} \Pi_{x\beta}[\hat\xi](y) \, . 
\end{equation}
\hfill /\!\!/

\medskip

    We denote the set of all homogeneities by 
    \begin{equation}
        \A:=\{|\beta| \,\big|\, \beta \text{ populated}\}.
    \end{equation}
    As a subset of $\alpha\N_0+\N_0$ this set is bounded from below and locally finite.
    Furthermore, by $\alpha\not\in\mathbb{Q}$ from Assumption~\ref{ass}~(iii), we have
    \begin{equation}\label{homogeneity_integer_polynomial}
    |\beta|\in \A\cap\N_0 \quad\implies\quad \beta \textnormal{ is purely polynomial.}
    \end{equation}

\begin{remark}
Via the hierarchy \eqref{hierarchy} we can associate $\Pi_{x\beta}$ to trees, 
as is done in the theory of regularity structures developed in \cite{Hai14}. 
Neglecting the counterterm, 
$\beta_a(k)$ equals the number of nodes without decoration 
and with $k+1$ outgoing edges, 
one of them with $\nabla\Delta$-decoration,
$\beta_b(\ell)$ equals the number of nodes with a noise decoration 
and with $\ell$ outgoing edges, 
and $\beta_p(\n)$ equals the number of nodes with an $\n$-th monomial decoration
and without children. 
All edges not connecting to a monomial node carry an additional $L^{-1}\nabla\cdot$ decoration.
Hence the total number of nodes is given by 
$\sum_k \beta_a(k) + \sum_\ell \beta_b(\ell) + \sum_{\n\neq\0} \beta_p(\n)$, 
while the number of edges is given by 
$\sum_k (k+1)\beta_a(k) + \sum_\ell \ell\beta_b(\ell) + \sum_{\n\neq\0}0\beta_p(\n)$.
The population condition 
$1+\sum_k k\beta_a(k)+\sum_\ell \ell\beta_b(\ell) 
= \sum_\ell \beta_b(\ell)+\sum_{\n\neq\0}\beta_p(\n)$ 
is then equivalent to saying that the number of edges 
differs from the number of nodes by $1$, 
i.e.~$\beta$ corresponds to a tree, 
and $\Pi_{x\beta}$ equals the linear combination of all trees 
with this given configuration. 
E.g.~the multiindices $e_1+2f_0$ and $2e_1+2f_0+g_\n$ are associated to 
\begin{equation}
\begin{istgame}
\xtdistance{3mm}{3mm}
\setistgrowdirection{north}
\istroot(0)[null node]
\istb*[double] \istb* \endist
\end{istgame}
\quad\textnormal{and} \quad
\begin{istgame}
\xtdistance{3mm}{3mm}
\setistgrowdirection{north}
\istroot(0)[null node]
\istb[dotted]{\mbox{\tiny$\nabla\!\Delta\! X^{\mathbf{n}}$}}[ar] \istb \endist
\istroot(a)(0-2)[null node] \istb*[double] \istb* \endist
\end{istgame}
+
\begin{istgame}
\xtdistance{3mm}{3mm}
\setistgrowdirection{north}
\istroot(0)[null node]
\istb*[double] \istb \endist
\istroot(a)(0-2)[null node]
\istb[dotted]{\mbox{\tiny$\nabla\!\Delta\! X^{\mathbf{n}}$}}[ar] \istb* \endist
\end{istgame}
+
\begin{istgame}
\xtdistance{3mm}{3mm}
\setistgrowdirection{north}
\istroot(0)[null node]
\istb*[double] \istb \endist
\istroot(a)(0-2)[null node]
\istb*[double] \istb[dotted]{\mbox{\tiny$X^{\mathbf{n}}$}\!}[al] \endist
\end{istgame}
+
\begin{istgame}
\xtdistance{3mm}{3mm}
\setistgrowdirection{north}
\istroot(0)[null node]
\istb[double] \istb* \endist
\istroot(a)(0-1)[null node]
\istb[dotted]{\mbox{\tiny$\nabla\!\Delta\! X^{\mathbf{n}}$}}[ar] \istb* \endist
\end{istgame}
+
\begin{istgame}
\xtdistance{3mm}{3mm}
\setistgrowdirection{north}
\istroot(0)[null node]
\istb[double] \istb* \endist
\istroot(a)(0-1)[null node]
\istb*[double] \istb[dotted]{\mbox{\tiny$X^{\mathbf{n}}$}\!}[al] \endist
\end{istgame}
+
\begin{istgame}
\xtdistance{3mm}{3mm}
\setistgrowdirection{north}
\istroot(0)[null node]
\istb[double] \istb[dotted]{\mbox{\tiny$X^{\mathbf{n}}$}\!}[al] \endist
\istroot(a)(0-1)[null node]
\istb*[double] \istb* \endist
\end{istgame}
,
\end{equation}
where a dot represents the noise decoration, 
a single solid line represents the integration decoration (i.e., $L^{-1}\nabla\cdot$), 
a double line represents the $\nabla\Delta$-decoration combined with the integration decoration, 
and a dotted line is just to visualise to which node the monomial decoration belongs.

It can also be checked, that the notion of homogeneity defined in \eqref{homogeneity} is consistent with the one defined in the tree-based setting of regularity structures. 
For more details and proofs we refer the reader to \cite[Section~7]{LOT23}.
\end{remark}

\begin{remark}
Note that by \eqref{pop1} and \eqref{homogeneity2} we have for populated multiindices $1+[\beta]=\sum_\ell \beta_b(\ell)>0$, 
in particular $[\beta]\geq0$. This is exactly the population condition in \cite[(2.23)]{LOTT21}. 
\end{remark}

\begin{remark}[Hierarchy, population and homogeneity for other equations]
For the quasilinear heat equation \eqref{eq:qSHE}, 
\eqref{polypart:poly} still holds true.
Similarly, one can obtain the hierarchy
\begin{equation}
    (\partial_0-\Delta)\Pi_x = \Pi^-_x = 
    P \sum_k \a_k \Pi_x^k \Delta\Pi_x 
    + \sum_\l \b_\l \Pi_x^\l \xi_\tau - 
    \sum_m \tfrac{1}{m!} \Pi_x^m (D^{(\0)})^m c.
\end{equation}
The population condition \eqref{pop1} is the same, 
however \eqref{pop2} changes to $c_\beta=0$ unless $\beta$ is populated.
The reason is that there is no additional term $\nabla\Pi_x$ multiplying $c$ on the right hand side of the hierarchy of equations. 
Also the homogeneity given by \eqref{homogeneity} stays the same. 
For the generalised porous medium equation \eqref{eq:PME}, 
\eqref{polypart:poly} stays the same, and the hierarchy is given by 
\begin{equation}
(\partial_0-\Delta)\Pi_x
= \nabla\cdot\Pi^-_x 
= \nabla\cdot\big(P \sum_k \a_k \Pi_x^k \nabla\Pi_x
+ \sum_\ell \b_\ell \Pi_x^\ell \xi_\tau\big).
\end{equation}
The population condition \eqref{pop1} as well as the homogeneity \eqref{homogeneity} persist.
\label{rem:oe3}
\end{remark}

\medskip

Before stating the main theorem, we introduce the recentering maps $\G_{xy}$. 
Since the following is not equation dependent at all, 
we just collect the main properties needed from \cite[Section 2.2.6]{LOTT21}. 
In Section~\ref{sec:group} we construct a group $\mathsf{G}^*$ that contains these maps $\G_{xy}$.
Let us also mention that it is possible to find a space $\T$ and a group $\mathsf{G}$, 
such that $(\A,\T,\mathsf{G})$ is a regularity structure in the sense of \cite[Definition~2.1]{Hai14}, 
such that $\T^*$ is the algebraic dual of $\T$, 
and such that $\G_{xy}$ is dual to some $\Gamma_{yx}\in \mathsf{G}$.
A detailed discussion of this can be found in \cite[Section~5.3]{LOT23} and \cite[Section~2.6]{LOTT21}.

\medskip

We aim for linear maps $\G_{xy}\in{\rm End}(\T^*)$ that recenter the model in the sense of
\begin{equation}
    \Pi_x = \G_{xy}\Pi_y + \Pi_x(y), \label{recenter}
\end{equation}
and satisfy
\begin{equation}
    \G_{xy} = \G_{xz}\G_{zy} 
    \quad\textnormal{and}\quad
    \G_{xx} = \id. \label{transitive}
\end{equation}
Moreover, we impose triangularity with respect to the homogeneity
\begin{equation}
    (\G_{xy}-\id)_\beta^\gamma \neq 0 
    \quad\implies\quad 
    |\gamma|<|\beta|, \label{Gtri}
\end{equation}
and for purely polynomial multiindices
\begin{equation}
    (\G_{xy})_{g_\n}^\gamma = \left\{
    \begin{array}{cl}
        \tbinom{\n}{\m} (y-x)^{\n-\m} & \textnormal{if }\gamma=g_\m \textnormal{ for some } \0\neq\m\leq\n, \\
        0 & \textnormal{otherwise},
    \end{array}\right. \label{recenterPoly} 
\end{equation}
where $\m\leq\n$ has to be understood componentwise.

\subsection{Main result}\label{sec:main}

The main result~\cref{thm:main} establishes the existence of $\Pi_x$ and $\G_{xy}$ 
that satisfy, along with all the postulates from Section~\ref{sec:model}, suitable stochastic estimates which are uniform in the mollification parameter $\tau$.
For convenience, we choose to mollify by $\xi_\tau:=\xi*\psi_\tau$ with the semigroup $\psi_\tau$ defined in \eqref{kernel}, 
however no substantial changes occur when choosing a different kernel $\rho$, 
as long as $\rho$ satisfies 
$\rho=\rho(O\cdot)$ with $O$ given in \cref{ass}.

Analogous to \cite{OSSW21}, we expect that this provides exactly the right construction to feed into an a priori estimate and develop a solution theory for \eqref{eq:spde-intro}, which we aim to address in future work. 

\begin{theorem}\label{thm:main}
Under Assumption~\ref{ass} (i)--(iii) the following holds for every $\tau>0$ 
and $\xi_\tau:=\xi*\psi_\tau$ with $\psi_\tau$ defined in \eqref{kernel}. 

\medskip

There exists a deterministic $c\in\R[[\a_k,\b_\ell]]$ satisfying \eqref{pop2} and 
\begin{equation}\label{pop4}
    c_\beta\neq0 \quad\implies\quad 
    |\beta|<2+\alpha \quad\textnormal{and}\quad
    [\beta] \textnormal{ is even}, 
\end{equation}
such that for every populated $\beta$ 
and for every $x\in\R^{1+d}$ there exists a random $\Pi_{x\beta}\in C^4(\R^{1+d})$ such that almost surely
\begin{equation}
    L\Pi_{x\beta} = \nabla\cdot\Pi^-_{x\beta} \quad\textnormal{unless $\beta$ is purely polynomial},
    \label{eq:PDEPibeta}
\end{equation}
with $\Pi^-_{x}$ defined in \eqref{eq:piminusttilde}, and which is given by \eqref{polypart:poly} for $\beta$ purely polynomial.

\medskip

Moreover, for every $x,y\in\R^{1+d}$ there exists a random $\G_{xy}\in{\rm End}(\T^*)$ such that almost surely we have \eqref{recenter},  \eqref{transitive}, \eqref{Gtri} and \eqref{recenterPoly}. 

\medskip

Finally, we have for all $p<\infty$
\begin{align}
    \E^\frac{1}{p} |\Pi_{x\beta}(y)|^p &\lesssim |x-y|_\s^{|\beta|}, \label{estPi}\\
    \E^\frac{1}{p} |(\G_{xy})_\beta^\gamma|^p &\lesssim |x-y|_\s^{|\beta|-|\gamma|} \label{estGamma},
\end{align}
where here and in the sequel, $\lesssim$ means $\leq C$ with a constant $C$ only depending on $\alpha,$ $\beta$, $p$ and\footnote{$\psi$ is introduced in Section~\ref{sec:schauder}} $\|\psi\|_{L^1}$, 
but being independent of $x,y$ and $\tau>0$.
\end{theorem}

As a consequence of the results and techniques in~\cite{Tem23} (see in particular~\cite[Remark 1.6]{Tem23}), the estimates of~\cref{thm:main} imply the following result.
\begin{corollary}[Uniqueness, convergence, and invariance]\label{coro}
We have the following:
\begin{enumerate}
\item[\emph{1.}] Existence and uniqueness: Given a noise $\xi$ which satisfies~\cref{ass}, there exists a unique model $(\Pi,\Gamma^*)$ for~\eqref{eq:tfe-intro} in the sense of~\cite[Definition 1.1]{Tem23}.
\item[\emph{2.}] Convergence and universality: Given a sequence of noises $\xi_n$ which satisfy~\cref{ass} uniformly in $n$ and that converge in law (resp. in $L^p$, almost surely) to $\xi$, the corresponding models $(\Pi_n,\G_n)$ converge component-wise in law (resp. in $L^p$, almost surely) to $(\Pi,\G)$, the unique limiting model associated to $\xi$. \label{conv}
\item[\emph{3.}] Invariance: Given a noise $\xi$ which satisfies~\cref{ass}, the corresponding model satisfies almost surely the following natural invariances for all populated $\beta$:
\begin{enumerate}
\item[\emph{a.}] $\Pi_{x}[\xi(\cdot +h)](y)=\Pi_{x+h}[\xi](y+h)$,
\item[\emph{b.}] $\Pi_{x\beta}[-\xi(R\cdot)](y)=(-1)^{|\beta|_p}\Pi_{Rx\beta}[\xi](Ry)$,
\item[\emph{c.}] $\Pi_{x\beta}[-\xi](y)=(-1)^{\sum_{\ell}\beta_b(\ell)}\Pi_{x\beta}[\xi](y)$,
\item[\emph{d.}] $\Pi_{x\beta}[\hat{\xi}](y) = \eps^{-|\beta|}\Pi_{\hat{x}\beta}[\xi](\hat{y})$ and 
$(\Gamma^*_{ x y}[\hat\xi])^\gamma_\beta 
= \eps^{-|\beta| + |\gamma|}(\Gamma^*_{\hat x \hat y}[\xi])^\gamma_\beta$, for all $\eps>0$, where $\hat{x}, \hat{y}$ and $\hat{\xi}$ are defined in~\eqref{xhat} and ~\eqref{eq:noise-invariance}, respectively.
\end{enumerate}
\end{enumerate}
\label{cor:uniqconvscal}
\end{corollary}
The notion of convergence for Item~2 of~\cref{cor:uniqconvscal} is described more precisely in~\cite[Theorem 1.4]{Tem23}.

\begin{remark}[Qualitative smoothness]\label{rem:smoothness}
We stress once more that the estimates \eqref{estPi} of $\Pi_x$ and \eqref{estGamma} of $\G_{xy}$ in Theorem~\ref{thm:main}
are uniform in the mollification scale $\tau>0$ from \eqref{spde}, 
and even carry over to the limiting model, cf. \cref{coro}. 
As long as $\tau>0$, we have additional qualitative smoothness properties 
that degenerate as $\tau\to0$, 
but which are useful to prove Theorem~\ref{thm:main}. 
More precisely, the counterterm $c$ is bounded by\footnote{the presence of $\sqrt[8]{\cdot}$ is due to the scaling of the specific choice of mollifier $\psi_\tau$} 
\begin{equation}\label{boundedness_c}
|c_\beta|\lesssim (\sqrt[8]{\tau})^{|\beta|-\alpha-2}, 
\end{equation}
which matches the lower bound obtained in \cref{prop:scaling_c}
in the case of $d=1$, $\alpha\in(\frac{1}{2},1)$ and for a special choice of mollifier.
In line with this, we have boundedness of up to fourth-order derivatives of $\Pi_x$,
\begin{equation}\label{boundedness_Pi}
\E^\frac{1}{p}|\partial^\n \Pi_{x\beta}(y)|^p
\lesssim (\sqrt[8]{\tau})^{\alpha-|\n|}
(\sqrt[8]{\tau}+|x-y|_\s)^{|\beta|-\alpha}
\quad\textnormal{for all }1\leq|\n|\leq4.
\end{equation}
Furthermore, we have the following annealed and weighted $C^{4,\alpha}$-estimate on $\Pi_{x}$, 
\begin{align}\label{smoothness_Pi}
&\E^\frac{1}{p}|\partial^\n \Pi_{x\beta}(y)-\partial^\n\Pi_{x\beta}(z)|^p \\
&\lesssim (\sqrt[8]{\tau})^{-|\n|}
(\sqrt[8]{\tau}+|x-y|_\s+|x-z|_\s)^{|\beta|-\alpha}
|y-z|_\s^{\alpha}
\quad\textnormal{for all } |\n|\leq 4,
\end{align}
and the analogous annealed and weighted $C^{1,\alpha}$-estimate on $\Pi_x^-$, 
\begin{align}\label{smoothness_Pi-}
&\E^\frac{1}{p}|\partial^\n\Pi^-_{x\beta}(y)-\partial^\n\Pi^-_{x\beta}(z)|^p \\
&\lesssim (\sqrt[8]{\tau})^{-3-|\n|}
(\sqrt[8]{\tau}+|x-y|_\s+|x-z|_\s)^{|\beta|-\alpha}
|y-z|_\s^{\alpha}
\quad\textnormal{for all }|\n|\leq 1.
\end{align}
The former yields by an application of Kolmogorov's continuity theorem the in Theorem~\ref{thm:main} claimed regularity $\Pi_{x\beta}\in C^4(\R^{1+d})$ almost surely.
\end{remark}

The proof of Remark~\ref{rem:smoothness} is a generalisation of the one of 
\cite[Remark~2.3]{LOTT21}; 
for completeness we provide the proof in Appendix~\ref{app:smooth}.


\begin{remark}[Analyticity in $a_0$]
\label{rem:analyticity}
The constants $c_\beta$ from \cref{thm:main} give via \eqref{ct} and \eqref{coordinates} back the counterterm $h$, 
\begin{equation}
h(u(x)) = \sum_\beta c_\beta \Big(\prod_{k\geq0} \frac{1}{k!}\frac{\d^k a}{\d u^k}(u(x))\Big)^{\beta_a(k)}\Big(\prod_{\l\geq0} \frac{1}{\l!}\frac{\d^\l b}{\d u^\l}(u(x))\Big)^{\beta_b(\l)},
\end{equation}
where due to \eqref{pop4} the sum is restricted to multiindices $|\beta|<2+\alpha$. 
Despite this restriction, some care has to be taken in this expression:
for fixed $\beta$, the products $\prod_{k\geq0}$ and $\prod_{\l\geq0}$ are effectively finite and thus well defined, 
since $\beta$ vanishes for all but finitely many $k,\l$;
however the sum over $\beta$ is infinite due to the degeneracy of $[\cdot]$ (and hence $|\cdot|$) and the degeneracy of the population constraint \eqref{pop2} in $e_0$. 
By a simple resummation, we observe
\begin{equation}
h(u(x)) = \!
\sum_{\hat\beta: \hat\beta_a(k=0)=0}
\sum_{\hat k\geq0} c_{\hat\beta +\hat k e_0} a(u(x))^{\hat k} \!
\Big(\!\prod_{k\geq1} \frac{1}{k!}\frac{\d^k a}{\d u^k}(u(x))\!\Big)^{\beta_a(k)}\Big(\!\prod_{\l\geq0} \frac{1}{\l!}\frac{\d^\l b}{\d u^\l}(u(x))\!\Big)^{\beta_b(\l)}\!,
\end{equation}
where $\hat\beta$ is again restricted to $|\hat\beta|<2+\alpha$ and the sum over $\hat\beta$ is thus finite. 
It is therefore left to argue why the sum over $\hat k$ is convergent, 
which we do in the following.

\medskip

Instead of deriving the model $\Pi_x$ from the renormalised equation \eqref{spde}, 
we consider 
\begin{equation}
(\partial_0 + (1-a_0)\Delta^2) u 
= \nabla\cdot\big( (a(u)-a_0)\nabla\Delta u
+ b(u)\xi_\tau - h(u)\nabla u \big)
\end{equation}
with $a_0=a(0)$. 
In the power series ansatz \eqref{ansatz}, 
this amounts to restricting to multiindices $\hat\beta$ satisfying $\hat\beta_a(k=0)=0$, 
and the coefficients $\hat\Pi_{x\hat\beta}$ inherit a dependence on $a_0$ through
\begin{equation}
(\partial_0+(1-a_0)\Delta^2) \hat\Pi_{x\hat\beta}
= \nabla \cdot \hat\Pi_{x\hat\beta}^- \, ,
\label{eq:pibetaanalytic}
\end{equation}
where $\hat\Pi_{x\hat\beta}^-$ is defined as in \eqref{piminus} with the difference that the sum over $k$ starts from $k=1$ 
and $D^{(\0)}$ is replaced by 
$\hat D^{(\0)} 
:= \a_1\partial_{a_0} 
+\sum_{k\geq1}(k+1)\a_{k+1}\partial_{\a_k}
+\sum_{\ell\geq0}(\ell+1)\b_{\ell+1}\partial_{\b_\ell}$.
Hence for all the $\hat{\cdot}\,$-objects the coordinate functional $\a_0$ 
is replaced by an additional parameter $a_0<1$ through the differential operator in \eqref{eq:pibetaanalytic}. 
We now show that this dependence of $\hat\Pi_{x\hat\beta}$ (and 
$\hat c_{\hat\beta}$) on $a_0$ is analytic as long as $a_0<1$. 
For this, it is convenient to allow for complex $a_0\in\C$ and show differentiability in the parameter $a_0$ in the half plane $\mathrm{Re}(a_0) <1$. 
Furthermore, we introduce yet another model $\bar\Pi$: 
it is defined in complete analogy with the model $\Pi$
(thus containing an $\a_0$ component), 
with the only difference that $\bar\Pi$ and $\bar\Pi^-$ are related by the same differential operator as are $\hat\Pi$ and $\hat\Pi^-$, i.e.
\begin{equation}\label{eq:barPi}
(\partial_0+(1-a_0)\Delta^2)\bar\Pi_{x\beta} = \nabla\cdot\bar\Pi^-_{x\beta}.
\end{equation}
The reason to introduce this further model is, 
that on the one hand we clearly have
\begin{equation}\label{mt04}
\bar\Pi_{x}(a_0=0) = \Pi_{x}, 
\quad
\bar c(a_0=0) = c, 
\end{equation}
and on the other hand, as we shall argue below, it relates to the model $\hat\Pi$ by 
\begin{subnumcases}
{
\tfrac{1}{\hat k !}\partial_{a_0}^{\hat k} \hat\pi_{\hat\beta} 
= \bar\pi_{\hat\beta+\hat k e_0}
\quad \textnormal{for} \quad 
\pi=
\label{eq:analytic}
}
c \, , \label{eq:analytic_c} \\ 
\Pi_x \, , \label{eq:analytic_Pi} 
\end{subnumcases}
for all $\hat k \in\N_0$ , 
where here and in the following we understand \eqref{eq:analytic_Pi} with respect to the norm 
\begin{equation}\label{eq:norm_Pi_analytic}
\sup_{y: y\neq x} |x-y|_\s^{-|\hat\beta|} \E^\frac{1}{p} |\Pi_{x\hat\beta}(y)|^p \, .
\end{equation}
Hence, the $\hat\cdot\,$-objects are indeed analytic in $a_0$ by \eqref{eq:analytic}, 
and the combination of \eqref{eq:analytic} and 
\eqref{mt04} shows that the above mentioned sum over $\hat k\geq0$ is indeed convergent, and moreover 
\begin{equation}
\sum_{\hat k \geq0} c_{\hat\beta+\hat k e_0} a(u(x))^{\hat k}
= \hat c_{\hat\beta} \big(a(u(x))\big).
\end{equation}
The proof of \eqref{eq:analytic} is again a generalisation of the one of \cite[Remark~2.7]{LOTT21}; 
for completeness we provide the proof in Appendix~\ref{app:analytic}. 
\end{remark}

\subsection{Analysis of the counterterm}\label{sec:divergingbounds}
In this section, we will perform a more careful analysis of the counterterm needed to renormalise the thin-film equation~\eqref{eq:tfe-intro}. For the sake of simplicity, we focus on the case $d=1$ and $\alpha \in (\frac12,1)$. As we shall see later in this section, the leading order structure of the counterterm remains the same in any dimension $d \geq 1$ and for any $\alpha \in(0,1)$. Additionally, we work with the alternative model $\hat\Pi$ described in~\cref{rem:analyticity} such that our multiindices have no $e_0$ component but $\hat\Pi_{x\beta}$ is an analytic function of $a_0$. As mentioned above, the corresponding hierarchy of linear PDEs can then be written as in~\eqref{eq:pibetaanalytic} as
\begin{align}
\hat{L} \hat{\Pi}_{x\beta} = \nabla \cdot \hat{\Pi}_{x\beta}^{-} \, ,
\end{align}
where the operator $\hat{L}:=(\partial_0 + (1-a_0)\Delta^2)$ depends on $a_0$. 
To shorten some expressions below we define $m_0:= 1- a_0$. We will show in this section that the counterterms in~\eqref{eq:tfe-intro} behave like $(\sqrt[8]{\tau})^{2\alpha-2}$  to leading order, when the noise $\xi$ is regularised to $\xi_\tau$, by mollifying with some smooth $\varphi$ (to be chosen in \cref{prop:scaling_c}) at length scale $\sqrt[8]{\tau}>0$. 

\medskip

Note that by stationarity \cref{ass}~(i), 
the law $\E$ of the tempered distribution $\xi$ is spatially homogeneous, 
i.e.~there exists a tempered distribution $C$ such that for any Schwartz functions $f,g$
\begin{equation}
\mathbb{E}[\langle \xi,f\rangle \langle \xi,g\rangle] 
= \langle C* f,g \rangle \, .
\end{equation}
In particular,
\begin{equation}
\E [\xi_\tau(x) \xi_\tau(y)] = F(x-y):= \langle C(x-y + \cdot)*\varphi_\tau,\varphi_\tau \rangle \, ,
\end{equation}
where $F$ is a Schwartz function which is even in space. We now state the main result of this section in which we will provide diverging lower bounds on the renormalisation constants of certain multiindices.  
\begin{theorem}\label{prop:scaling_c}
Let~\cref{ass} be satisfied with $\alpha \in (\frac12,1)\setminus \mathbb{Q}$, let $d=1$, 
and let $m_0=1-a_0\in[\lambda,\Lambda]$ for some $0<\lambda\leq\Lambda<\infty$.
Then, we have 
\begin{align}
\hat{c}_{e_1 + f_0 + f_1} (a_0)
= \int_{\R^2} \dx{k} \frac{(2\pi k_1)^4}{(2 \pi k_0)^2 + m_0^2(2 \pi k_1)^8} \left(\frac{4 m_0^2(2\pi k_1)^8}{(2 \pi k_0)^2 + m_0^2(2 \pi k_1)^8}-2 \right) \mathcal{F}F(k) \,  ,
\label{eq:value}
\end{align}
\begin{align}
\hat{c}_{2f_1} (a_0)
= \int_{\R^2} \dx{k} \frac{k_1 }{(2 \pi k_0)^2 + m_0^2(2 \pi k_1)^8} \left(-2\pi i k_0 + m_0 (2\pi k_1)^4\right) \partial_{k_1}\mathcal{F}F(k) \, ,
\label{eq:value2}
\end{align}
and
\begin{align}
\hat{c}_{2e_1 + 2f_0} (a_0)
= -3    \int_{\R^2} \dx{k} \frac{m_0(2 \pi k_1)^{12}}{((2\pi k_0)^2 + m_0^2 (2\pi k_1)^8)^2} \mathcal{F}F(k) \, ,
\label{eq:value3}
\end{align}
where we denote the operation of taking the Fourier transform by $\mathcal{F}$. 
All other renormalisation constants are zero. Assume furthermore that
\begin{equation}
 \mathcal{F}C(k)= \frac{1}{((2\pi k_0)^2 + (2\pi k_1)^8)^{\frac{\alpha -\frac12}{4}}} \, 
 \end{equation} 
and that $\varphi_\tau$ is such that 
\begin{equation}
|\mathcal{F}\varphi_\tau (k)|^2 = e^{- \tau (2\pi k_1)^8 - \tau^\eta (2\pi k_0)^2 } 
\label{eq:mollifier2}
\end{equation}
for some $\eta\geq 1$. Then, if $\eta>1$
\begin{equation}
\hat{c}_{e_1 + f_0 + f_1}(a_0)
= C_{\alpha, 1}(m_0) m_0^{-1}(\sqrt[8]{\tau})^{2\alpha -2}  
+  O\left((\sqrt[8]{\tau})^{2\alpha -2 +  (\eta-1) (3 + 2\alpha)}\right) \, ,
 \end{equation}
 where, $C_{\alpha,1}(m_0)$ is independent of $\tau$ and satisfies 
 \begin{equation}
 \lim_{\alpha \downarrow \frac12} C_{\alpha,1}(m_0) = 0 \, .
 \label{eq:exactvalue2}
 \end{equation}
 Similarly, we have that
 \begin{equation}
\hat{c}_{2f_1}(a_0)
= C_{\alpha,2}(m_0) (\sqrt[8]{\tau})^{2\alpha -2}  
+ O\left( m_0 (\sqrt[8]{\tau})^{2\alpha -2 +  (\eta-1) (3 + 2\alpha)}\right) \, ,
 \end{equation}
 where $C_{\alpha,2}(m_0)$ is independent of $\tau$ and satisfies 
 \begin{equation}
  \lim_{\alpha \downarrow \frac12} C_{\alpha,2}(m_0) 
  = - \frac{\Gamma\left(\frac98\right)}{2\pi} \, ,
 \label{eq:exactvalue22}
 \end{equation}
 and 
  \begin{equation}
\hat{c}_{2e_1 + 2f_0}(a_0)
= C_{\alpha,3}(m_0) m_0^{-2}(\sqrt[8]{\tau})^{2\alpha -2}  
+ O\left( m_0 (\sqrt[8]{\tau})^{2\alpha-2+(\eta-1)(11+2\alpha) }\right) \, ,
 \end{equation}
  where $C_{\alpha,3}(m_0)$ is independent of $\tau$ and satisfies
 \begin{equation}
  \lim_{\alpha \downarrow \frac12} C_{\alpha,3}(m_0) 
  = - \frac{3 \Gamma \left(\frac{9}{8}\right)}{4 \pi} \, .
 \label{eq:exactvalue23}
 \end{equation}
\end{theorem}
We relegate the proof of the above theorem to~\cref{sec:explicit}.
\begin{remark}
In the specific case in which the ensemble $\xi$ is Gaussian, the choice of $C$ in the above theorem amounts to specifying the corresponding Cameron--Martin space as $\dot{H}^{-s}$ for $s= \alpha-1/2$,
where $\dot{H}^s$ are the $L$-dependent anisotropic Sobolev spaces defined in~\eqref{eq:Sobolev}.
\end{remark}
\begin{remark}\label{rem:spatial}
The choice of mollifier made in~\eqref{eq:mollifier2} may seem odd at first sight but it is quite natural considering the effect we are trying to capture. Setting $\eta=1$ would correspond to natural anisotropic parabolic scaling between space and time which would mean that the mollifier treats space and time on equal footing when acting on a given distribution. However if $\eta>1$, as we have chosen in~\eqref{eq:mollifier2}, the mollifier smooths out more in space than in time. Thus, this choice of mollifier mimics a spatial discretisation of the SPDE~\eqref{eq:spde-intro}. We will see that it will play a role in the next subsection.
\end{remark}

\subsubsection{Structure of the counterterm}\label{sec:structure}
In this subsection, we will discuss the form of the counterterm that arises from the choice of renormalisation constants we have obtained in~\cref{prop:scaling_c}. 
We know from the discussion in~\cref{rem:analyticity} that the function $h(\cdot)$ can be expressed as 
\begin{align}
h(u(x)) 
&= \hat{c}_{e_1+ f_0 +f_1}\big(a(u(x))\big) a'(u(x)) b(u(x)) b'(u(x)) 
+ \hat{c}_{2 f_1}\big(a(u(x))\big) (b'(u(x)))^2  \\& \,+ \hat{c}_{2 e_1 + 2 f_0}\big(a(u(x))\big) (a'(u(x)))^2 (b(u(x)))^2\, ,
\end{align}
where we have applied~\cref{prop:scaling_c}. For the specific case of the the thin-film equation, we have $a(u)=1-M(u)$ and $b(u)=M^{\frac12}(u)$  which leads us to
\begin{align}
 h(u)=&-\frac12 \hat{c}_{e_1+ f_0 +f_1}(a(u)) (M'(u))^2  +\frac{1}{4} \hat{c}_{2f_1}(a(u)) \frac{(M'(u))^2}{M(u)} \\
 &+ \hat{c}_{2 e_1 + 2 f_0}(a(u)) M(u) (M'(u))^2 \, .
\end{align} 
For the choice of mollifier with $|\mathcal{F}{\varphi}_\tau(k)|^2=e^{-\tau(2\pi k_1)^8 -\tau^\eta (2\pi k_0)^2}$ in~\eqref{eq:mollifier2} and $\eta>1$ (see the discussion in~\cref{rem:spatial}), we know from~\cref{prop:scaling_c} that 
\begin{align}
 h(u) &=- (\sqrt[8]{\tau})^{2 \alpha -2}\frac{C_{\alpha,1}(M(u))}{2} (M(u))^{-1}   (M'(u))^2 \\
 &\,+ (\sqrt[8]{\tau})^{2 \alpha -2}\frac{C_{\alpha,2}(M(u))}{4}    \frac{(M'(u))^2}{M(u)}\\ 
 & \, + (\sqrt[8]{\tau})^{2 \alpha -2} C_{\alpha,3}(M(u))  (M(u))^{-2} M(u) (M'(u))^2   \\ &\,+ O \left(\sqrt[8]{\tau}^{2\alpha -2 +(\eta-1)(3 + 2\alpha )}\right) (M'(u))^2  \\ 
 &\,+ O \left(\sqrt[8]{\tau}^{2\alpha -2 +(\eta-1)(11 + 2\alpha )}\right) (M(u))^2(M'(u))^2    \, .
\end{align}
Thus, to leading order, the counterterm is of the form
\begin{equation}
(\sqrt[8]{\tau})^{2 \alpha -2}\left(\frac{C_{\alpha,2}(M(u))}{4} 
+ C_{\alpha,3}(M(u)) 
-\frac{C_{\alpha,1}(M(u))}{2} \right)\partial_x\big((M(u))^{-1} (M'(u))^2 \partial_x u \big) \, .
 \end{equation} 
 Even though we cannot derive uniform estimates on the model as in~\cref{thm:main} for the case $\alpha=\frac12$ (see~\cref{rem:failureSchauder}), we can formally write down the leading order form of the counterterm in this case as 
 \begin{equation}
   -\frac{7\,\Gamma\left(\frac98\right)}{8\pi}(\sqrt[8]{\tau})^{-1} \partial_x\left( \frac{(M'(u))^2}{M(u)} \partial_x u \right)\,.
 \end{equation} 
In the case $M(u)=u^m, m \geq 0$ (which is only covered by Theorem~\ref{prop:scaling_c} provided one knew that $u$ is bounded and bounded away from zero) reduces up to an order one constant to
 \begin{equation}
  -(\sqrt[8]{\tau})^{-1}\partial_x( u^{m-2} \partial_x u ) \, .
  \label{eq:counter}
 \end{equation}
 We note that the above term shows up with a ``\emph{good}'' sign in the  renormalised SPDE, i.e.~it shows up as $(\sqrt[8]{\tau})^{-1}\partial_x( u^{m-2} \partial_x u )$ on the right hand side of the equation. This implies that it has a smoothing effect (at least for strictly positive $u$) which blows up as the regularisation parameter $\tau$ goes to $0$. For the case $m=2$, as can be seen from the expression in~\eqref{eq:counter}, the term takes an even simpler linear form and the counterterm can be formally thought of as $\infty \times \partial_x^2 u$.

Interestingly, the above term agrees exactly with the form of a correction term that shows up in the discretisation discussed in~\cite{GGKO22}. In~\cite{GGKO22}, the authors derive a spatial discretisation for the SPDE~\eqref{eq:spde-intro} for $\alpha=\frac12$, based on its formal gradient flow structure, which leaves invariant a discrete version of the thermodynamically correct invariant measure, the so-called conservative Brownian excursion. Representing the discretisation as an SDE leads to a correction term whose formal limit as $N$ (the number of lattice points) tends to $\infty$ is exactly of the form~\eqref{eq:counter}, at least for power mobilities $M(u)=u^m$.    We refer the reader to~\cite[Remark 9.1]{GGKO22} where the origin of this correction term and its formal limit are discussed in more detail. We also remark here that $\alpha=\frac12$ is the only physically relevant choice of noise in one spatial dimension. Indeed, the noise arises naturally as the limiting noise through a long wave/lubrication approximation of the Navier--Stokes equations driven by a fluctuating stress tensor which itself arises from thermal fluctuations in the fluid (see \cite{GMR06}). 

  The fact that the form of the counterterm seems to agree with the form of the correction term in~\cite{GGKO22} lends credence to the hypothesis that the discretisation has the counterterm ``\emph{built in}''. Indeed, numerical experiments suggest that the discretisation in~\cite{GGKO22} converges to a nontrivial limit as $N\to \infty$ (see, in particular,~\cite[Section 11.4]{GGKO22}).
In addition to showing up in the discretisation studied in \cite{GGKO22}, the counterterm we conjecture has shown up as an ad-hoc regularisation term in some works that study the equation with colored noise \cite{GK22}.

\section{\texorpdfstring{\bf Proof of~\cref{thm:main}}{Proof of the stochastic estimates}}\label{sec:proof}
\subsection{Strategy of the proof}\label{sec:strategy}
In this section, we give an overview of the proof of \cref{thm:main} 
and discuss the main steps involved, 
which are integration, reconstruction, algebraic-, and three-point arguments. 
We refer the reader to \cref{sec:inductive_proof} for the precise logical order in which we go through these steps 
in the inductive proof. We remind the reader that our approach follows closely the approach in~\cite{LOTT21}. Some of the intermediate arguments carry over mutatis mutandis to our setting in which case we omit the proofs and cite the original arguments from~\cite{LOTT21}.

\medskip

\subsubsection{Integration and semigroup convolution}
We start with a discussion on the estimate \eqref{estPi} on $\Pi_{x\beta}$. 
This will be a consequence of the corresponding estimate on $\Pi^-_{x\beta}$ via a Schauder-type argument, see \cref{cor:intI} (Integration~I), 
which we refer to as integration argument in the sequel. 
Since we expect $\Pi^-_{x\beta}$ to be a tempered distribution in the absence of any mollification of the noise, 
we test against a test function in order to be able to obtain a stable estimate as the mollification is removed.
It is convenient to express this weak estimate by testing against a semigroup $\psi_t$; 
more precisely, we choose $\psi_t$ to be the Green's function 
associated to the symmetric and uniformly elliptic operator 
$LL^*=-\partial_0^2 + \Delta^4$, i.e.~$\psi_t$ is the unique solution of 
\begin{equation}\label{kernel}
	\partial_t \psi_t + LL^* \psi_t =0 \, ,
\end{equation}
such that $\psi_{t=0} = \delta_{x=0}$. It is straightforward to check that $\psi_t$ is a Schwartz function and satisfies the following natural scaling invariance
\begin{equation}\label{eq:psiscaling}
	\psi_t(x) = \frac{1}{(\sqrt[8]{t})^D}
    \psi_1\left(\frac{x_0}{(\sqrt[8]{t})^{\mathfrak{s}_0 }}, \dots, 
    \frac{x_d}{(\sqrt[8]{t})^{\mathfrak{s}_d }}\right) \, ,
\end{equation} 
where $\mathfrak{s} \in \N^{1+d}$ is the scaling associated to $L$ and $D$ is the effective dimension, see \eqref{dim}. 
As $\psi_{t=1}$ is a Schwartz function, the following bound holds for all $\theta \in \R$
\begin{align}
\int \dx{z} |\partial^{\n}\psi_{1}(y-z)| (1 + \abss{x-y} + \abss{y-z})^{\theta}  \lesssim (1+ \abss{x-y})^\theta \, ,
\end{align} 
which by the scaling invariance of $\psi_t$ from~\eqref{eq:psiscaling}
implies the moment bound
\begin{align}
\int \dx{z} |\partial^{\n}\psi_{t}(y-z)| (\sqrt[8]{t} + \abss{x-y} + \abss{y-z})^{\theta}  
\lesssim (\sqrt[8]{t})^{-|\n|}(\sqrt[8]{t}+ \abss{x-y})^\theta \, . 
\label{eq:kerbound}
\end{align}
One can also check that $\psi_t$ satisfies the following semigroup property
\begin{equation}\label{eq:semi}
	\psi_t * \psi_s = \psi_{t+s} \, ,
\end{equation}
for all $s,t \geq 0$. 
Finally, given a tempered distribution $f$, we define
\begin{equation}
	f_t:= \psi_t *f  \,.
\end{equation}
With this notation in hand, the estimate on $\Pi^-_{x\beta}$ we aim for is 
\begin{equation}\label{estPiminus}
    \E^\frac{1}{p}| \Pi^-_{x\beta t}(y) |^p \lesssim (\sqrt[8]{t})^{\alpha-3} (\sqrt[8]{t} + |x-y|_\s)^{|\beta|-\alpha}.
\end{equation}
Note that the appearance of $\sqrt[8]{\cdot}$ in \eqref{estPiminus} 
is dictated by the scaling \eqref{eq:psiscaling}.

\medskip

\subsubsection{Reconstruction}
Estimating $\Pi^-_{x\beta}$ before $\Pi_{x\beta}$ is at the core of an inductive argument, 
as this allows to use estimates on $\Pi_{x\beta'}$ 
for $\beta'$ ``smaller''\footnote{in a sense to be made precise in \cref{sec:inductive_proof}} than $\beta$ to estimate $\Pi^-_{x\beta}$ 
via the hierarchy \eqref{piminuscomponents}. 
In case of $|\beta|>3$, this is indeed a rather straightforward task, 
and is carried out in \cref{rec1} (Reconstruction~I). 

\medskip

\subsubsection{Malliavin derivative and dualisation}\label{sec:application_SG}
The situation is much more complex in the case of $|\beta|<3$. It is here that we will leverage an improvement at the level of the Malliavin derivative as we shall explain now. For these multiindices, we apply the $p$-version of the spectral gap inequality \eqref{sgp} to $F=\Pi^-_{x\beta t}(y)$, which results in
\begin{equation}
\E^\frac{1}{p}| \Pi^-_{x\beta t}(y) |^p
\lesssim |\E\Pi^-_{x\beta t}(y)|
+ \E^\frac{1}{p}\Big\| \frac{\partial \Pi^-_{x\beta t}(y)}{\partial \xi} \Big\|^p_{\dot H^{-s}} \, .
\end{equation}
Although $\Pi^-_{x\beta t}(y)$ is not a cylindrical function, 
it can be approximated by such objects and so the application of \eqref{sgp} is justified; 
for a precise version of this approximation argument we refer to \cite[Section~7]{LOTT21}. 

\medskip

To estimate the first term on the right hand side, 
we will fix the counterterm $c$ by the so-called BPHZ-choice of renormalisation. 
We give a detailed account of the choice of $c$ and how to use it to estimate $\E\Pi^-_{x\beta t}(y)$ by the right hand side of \eqref{estPiminus} in \cref{sec:bphz}, 
see in particular \eqref{eq:expbound}. 

\medskip

To estimate the Malliavin derivative, we actually establish the stronger 
\begin{equation}\label{eq:rhsweakbounddPi}
\E^{\frac{1}{q'}} |\delta \Pi_{x \beta t}^-(y)|^{q'} 
\lesssim (\sqrt[8]{t})^{\alpha-3} 
(\sqrt[8]{t} + \abss{x-y})^{|\beta|-\alpha} \bar w 
\end{equation} 
for all $1 < q' < q \leq 2$,  where we have introduced the notation 
\begin{equation}\label{w_bar} 
\bar w := \Big(\int_{\R^{1+d}} \dx{z}\, \E^\frac{2}{q}\big| 
(LL^*)^{\frac{s}{2|L|}} \, \delta\xi(z) \big|^q\Big)^{\frac{1}{2}} \, . 
\end{equation}
Note that by $q\leq2$ we can appeal to Minkowski's inequality to see
\begin{equation}
\bar w \leq \E^\frac{1}{q} \| \delta\xi \|^q_{\dot H^s} \, ,
\end{equation}
which together with 
$|\E\delta\Pi^-_{x\beta t}(y)| 
\leq \E^\frac{1}{q'}|\delta\Pi^-_{x\beta t}(y)|^{q'}$ 
shows that \eqref{eq:rhsweakbounddPi} is by duality indeed a strengthening of 
\begin{equation}
\E^\frac{1}{p}\Big\| \frac{\partial \Pi^-_{x\beta t}(y)}{\partial \xi} \Big\|^p_{\dot H^{-s}} 
\lesssim (\sqrt[8]{t})^{\alpha-3} 
(\sqrt[8]{t} + \abss{x-y})^{|\beta|-\alpha} \, ,
\end{equation}
with $p\geq2$ being the H\"older-conjugate exponent of $q\leq2$. 
The reason for introducing $1<q'$ is that we will appeal to H\"older's 
inequality within the proof of \eqref{eq:rhsweakbounddPi}, 
where one factor will involve a Malliavin derivative $\delta$, 
and the other factor(s) are controlled in probabilistic $L^p$-norms for $p<\infty$. 
Thus, the implicit constants in estimates like \eqref{eq:rhsweakbounddPi} 
on Malliavin derivatives depend in addition to $\alpha$, $\beta$, $p$, and $\|\psi\|_{L^1}$, also on $1<q'<q$. 

\begin{remark}[Qualitative smoothness II]\label{rem:smoothness_Mall}
Analogous to \cref{rem:smoothness} we have qualitative smoothness of the Malliavin derivative of $\Pi_x$ and $\Pi^-_x$. 
More precisely, we have boundedness of up to fourth-order derivatives of $\delta\Pi_x$, 
\begin{equation}\label{boundedness_deltaPi}
\E^\frac{1}{p}|\partial^\n \delta\Pi_{x\beta}(y)|^p
\lesssim (\sqrt[8]{\tau})^{\alpha-|\n|}
(\sqrt[8]{\tau}+|x-y|_\s)^{|\beta|-\alpha} \bar w
\quad\textnormal{for all }1\leq|\n|\leq4 ,
\end{equation}
the following annealed and weighted $C^{4,\alpha}$-estimate on $\delta\Pi_{x}$, 
\begin{align}\label{smoothness_deltaPi}
&\E^\frac{1}{p}|\partial^\n \delta\Pi_{x\beta}(y)-\partial^\n\delta\Pi_{x\beta}(z)|^p \\
&\lesssim (\sqrt[8]{\tau})^{-|\n|}
(\sqrt[8]{\tau}+|x-y|_\s+|x-z|_\s)^{|\beta|-\alpha}
|y-z|_\s^{\alpha} \bar w
\quad\textnormal{for all } |\n|\leq 4,
\end{align}
and the analogous annealed and weighted $C^{1,\alpha}$-estimate on $\delta\Pi_x^-$, 
\begin{align}\label{smoothness_deltaPi-}
&\E^\frac{1}{p}|\partial^\n\delta\Pi^-_{x\beta}(y)-\partial^\n\delta\Pi^-_{x\beta}(z)|^p \\
&\lesssim (\sqrt[8]{\tau})^{-3-|\n|}
(\sqrt[8]{\tau}+|x-y|_\s+|x-z|_\s)^{|\beta|-\alpha}
|y-z|_\s^{\alpha} \bar w
\quad\textnormal{for all }|\n|\leq 1.
\end{align}
By an application of Kolmogorov's continuity theorem, 
this justifies pointwise evaluation of derivatives of 
$\delta\Pi_x$ and $\delta\Pi^-_x$. 
The proof of these estimates follows the one of \cref{rem:smoothness}, 
which we therefore omit.
\end{remark}

\medskip

\subsubsection{Improved modeledness}
We now outline the proof of \eqref{eq:rhsweakbounddPi}. 
Note that when we pass from $\xi$ to the direction $\delta \xi$ 
we obtain a gain in regularity from $\alpha-3$ to $s=\alpha-3+D/2$,
cf.~\cref{rem:reg}. 
One may ask if a similar gain in regularity can be expected at the level of $\delta \Pi^-_{x\beta}$ for arbitrary $|\beta|<3$. 
This, however, is unreasonable as $\Pi_{x\beta}$ 
(and hence $\Pi^-_{x\beta}$) is multilinear in the noise $\xi$. 
What is reasonable, is a gain in modeledness of $\delta \Pi^-_{x\beta}$ 
of order $D/2$ around a secondary base point $z$, 
after it has been appropriately recentered by some $\d\G_{xz}$.
In fact, we will only track a gain of regularity of order $\kappa<D/2$, 
and claim that there exists a $\d\G_{xz}\in{\rm End}(\T^*)$ such that 
\begin{align}\label{eq:lhsstrongboundincrement}
&\mathbb{E}^\frac{1}{q'} 
|(\delta\Pi_{x}-\delta\Pi_x(z)-{\rm d}\Gamma^*_{xz}\Pi_{z})_{\beta}(y)|^{q'}
\\
&\lesssim \abss{y-z}^{\kappa+\alpha} 
(\abss{y-z}+\abss{x-z})^{|\beta|-\alpha} 
(w_x(y)+w_x(z)) \, ,
\end{align}
and the analogous estimate for $\delta\Pi^-_x$, 
\begin{align}\label{eq:rhsweakboundincrement}
&\mathbb{E}^{\frac{1}{q'}} 
|(\delta \Pi_x^- - \dx{\Gamma_{xz}^*} \Pi_z^-)_{\beta t} (y)|^{q'} \\
&\lesssim (\sqrt[8]{t})^{\alpha -3} 
(\sqrt[8]{t} + \abss{y-z})^{\kappa} 
(\sqrt[8]{t} + \abss{y-z} + \abss{x-z})^{|\beta|-\alpha} 
(w_x(y) + w_x(z)) \, .
\end{align} 
We choose to work with $L^\infty$-based norms, 
as they behave well under multiplication, 
whereas the gain of regularity we observe at the level of $\delta\xi$ 
is on $L^2$-based norms. 
The price to pay is to include the weights 
\begin{equation}\label{weight}
w_x(z):=|x-z|_\s^{-\kappa}\bar w + w(z) \, ,
\end{equation}
where
\begin{equation}\label{eq:weightednorm}
w(z) := \Big( \int_{\R^{1+d}} \dx{y}\, |y-z|_\s^{-2\kappa} \, 
\E^\frac{2}{q}\big| 
(LL^*)^{\frac{s}{2|L|}} \, \delta\xi(y) 
\big|^q\Big)^{\frac{1}{2}} \, .
\end{equation} 
Importantly, $w(z)$ behaves well under (square) averaging, 
\begin{align}
\Big(\int_{\R^{1+d}} \dx{z}\, |\psi_t(y-z)| \, w^2(z) \Big)^\frac{1}{2}
&\lesssim \min \big( w(y), (\sqrt[4]{t})^{-\kappa} \bar{w} \big) \, ,
\label{average_psi_w}
\end{align}
which is a consequence of the moment bound 
\begin{equation}
\int_{\R^{1+d}} \dx{z}\, |\psi_t(y-z)| |x-z|_\s^{-2\kappa} 
\lesssim (\sqrt[8]{t}+|x-y|_\s)^{-2\kappa} 
\end{equation}
and relies on $\kappa<D/2$. 
Furthermore, as a consequence of \eqref{average_psi_w} and the bound of negative moments we have 
\begin{equation}\label{average_psi_wx}
\Big(\int_{\R^{1+d}}\dx{z}\, |\psi_t(y-z)|\, w_x^2(z) \Big)^\frac{1}{2}
\lesssim \min \big( w_x(y), (\sqrt[8]{t})^{-\kappa} \bar w \big) \, .
\end{equation}
These weights could be avoided by working with Besov norms, 
e.g.~as done in \cite{HS23} and \cite{BOTT23}. 

\medskip

By averaging in the secondary base point and 
using \eqref{average_psi_wx}, 
we show in \cref{lem:averaging} that \eqref{eq:rhsweakboundincrement}
implies \eqref{eq:rhsweakbounddPi}. 
This involves estimates on $\d\G_{xz}$, 
which we shall establish along the way, 
along with estimates on $\G_{xy}$ and $\delta\G_{xy}$, 
that are also used in several other places. 
A discussion of $\G$, $\delta\G$, and ${\rm d}\G$ will follow in \cref{sec:group,sec:dGamma}.
Before that, we shall explain how we derive the estimates \eqref{eq:lhsstrongboundincrement} and \eqref{eq:rhsweakboundincrement}.

\subsubsection{Integration and Reconstruction for increments}

As earlier for $\Pi_x$ and $\Pi^-_x$, 
we will first establish \eqref{eq:rhsweakboundincrement} and 
obtain \eqref{eq:lhsstrongboundincrement} from a Schauder-type argument based on 
\begin{equation}
L (\delta\Pi_x-\delta\Pi_x(z)-\d\G_{xz}\Pi_z)_\beta 
= \nabla \cdot (\delta\Pi^-_x - \d\G_{xz}\Pi^-_z)_\beta \, ,
\label{eq:prerepresentation}
\end{equation}
see \cref{lem:integrationIII} (Integration~III).

\medskip

Estimating increments of $\delta\Pi^-_x$ before increments of $\delta\Pi_x$ 
allows again for an inductive argument, 
where the hierarchy \eqref{piminuscomponents} used in the case $|\beta|>3$ 
is now replaced by the identity 
\begin{align}\label{magic}
    Q(\delta\Pi^-_x-\d\G_{xz}\Pi^-_z)(z) 
    &= Q\sum_k \a_k\Pi_x^k(z)\nabla\Delta(\delta\Pi_x-\d\G_{xz}\Pi_z)(z) \\
    &\,+ Q\sum_\ell \b_\ell \Pi_x^\ell(z) \delta\xi_\tau(z) \, . 
\end{align}
Here, $Q$ denotes the projection of a power series $\sum_\beta\pi_\beta\z^\beta$ to $\sum_{|\beta|<3}\pi_\beta\z^\beta$, 
meaning that in \eqref{magic} we are only interested in $\beta$-components 
with $|\beta|<3$. 
On the one hand, the right hand side of \eqref{magic}
involves only $\beta'$ components of $\Pi_x$ and $\delta\Pi_x-\d\G_{xz}\Pi_z$ for $\beta'$ ``smaller'' than $\beta$.
On the other hand, the improved vanishing \eqref{eq:lhsstrongboundincrement}
at the secondary base point $z$ 
and the improved regularity of $\delta\xi$ 
allow for a reconstruction argument, 
which requires $\alpha+(\kappa+\alpha-3)>0$.
This is carried out in \cref{rec2} (Reconstruction~II), 
establishing \eqref{eq:rhsweakboundincrement}.

At this point, we mention two further (artificial) restrictions on $\kappa$. 
To avoid case distinctions, 
it is convenient to not recenter to unnecessarily high order, 
and we will therefore assume $\kappa+\alpha<3$. 
Similarly, to simplify some of the estimates later on, 
it is convenient to also assume $\kappa+2\alpha<\min \A\cap(3,\infty)$. 
Altogether, this imposes
\begin{subnumcases}
{
    3 < \kappa+2\alpha < \label{kappa}
}
    \tfrac{D}{2}+2\alpha, \label{kappa1} \\
    \min \A\cap(3,\infty) \label{kappa3}.
\end{subnumcases}
By the restriction $\alpha>\frac{3}{2}-\frac{D}{4}$ in Assumption~\ref{ass}~(iii), 
it is possible to choose $\kappa$ satisfying \eqref{kappa1}, 
while since $\A$ is locally finite it is also possible to choose $\kappa$ satisfying at the same time (the artificial) \eqref{kappa3}.
Since $3<3+\alpha\in \A$, \eqref{kappa3} implies $\kappa+\alpha<3$.

\medskip

\subsubsection{The structure group}\label{sec:group}
We turn to a discussion of $\G_{xy}$. 
To simplify some of the proofs, 
it will be convenient to strengthen $\G_{xy}\in{\rm End}(\T^*)$ 
as stated in \cref{thm:main} 
to $\G_{xy}\in{\rm Alg}(\RR)\cap{\rm End}(\T^*)$. 
By this we mean that $\G_{xy}$ is a well-defined linear map from $\RR$ to itself, 
compatible with its algebra structure in the sense that for $\pi,\pi'\in\RR$
\begin{equation}\label{mult}
    \G_{xy}(\pi\pi')=(\G_{xy}\pi)(\G_{xy}\pi'), 
\end{equation}
and it preserves $\T^*\subset\RR$ in the sense that 
\begin{equation}\label{GammaPreservesT}
\G_{xy} \T^*\subset \T^*.
\end{equation}
This deviates slightly from \cite[Section~2.5]{LOTT21} where $\G_{xy}$ is only defined\footnote{
at least, it is not mentioned that it actually is well-defined on the larger $\RR$} on the smaller $\T^*$. 
The reason for defining it on the larger $\RR$ is, that this allows to apply $\G_{xy}$ to $c$ which, due to the constraint \eqref{pop2}, is not an element of $\T^*$. 

\medskip

As in \cite[Section~2.5]{LOTT21} we start from a purely algebraic map $\{\pi^{(\n)}\}_\n\mapsto\G$ given by 
\begin{equation}\label{exp}
    \G = \sum_{j\geq0}\tfrac{1}{j!}\sum_{\n_1,\dots,\n_j}\pi^{(\n_1)}\cdots\pi^{(\n_j)}D^{(\n_1)}\cdots D^{(\n_j)},
\end{equation}
where $D^{(\0)}$ is given by \eqref{D0} and $D^{(\n)}$ for $\n\neq\0$ is the derivation on $\RR$ defined by 
\begin{equation}\label{Dn}
    D^{(\n)}:=\partial_{\p_\n}.
\end{equation}
For later use we note that 
\begin{equation}\label{Dncomp}
    (D^{(\n)})_\beta^\gamma = \gamma(\n)\delta_\beta^{\gamma-g_\n} \quad\text{for }\n\neq\0,
\end{equation}
hence
\begin{equation}\label{Dnprops}
    (D^{(\n)})_{\beta}^{\gamma}\neq 0  
    \hspace{1ex}\implies\hspace{1ex}\left\{
    \begin{array}{l}
    \beta_a(k) = \gamma_a(k) \text{ for all } k, \\
    \beta_b(\ell) \,= \gamma_b(\ell) \, \text{ for all } \ell, \\
    \beta_p(\n') = \gamma_p(\n')-\delta^\n_{\n'} \text{ for all }\n'\neq\0.
    \end{array}\right.
\end{equation}
Despite the two infinite sums in \eqref{exp}, the following lemma shows that $\G$ is well-defined 
for a suitable choice of $\{\pi^{(\n)}\}_\n$. 
\begin{lemma}\label{lem:defineGamma}
    Let $\{\pi^{(\n)}\}_\n \subset \T^*$ satisfying
    \begin{equation}\label{pin}
        \pi^{(\n)}_\beta\neq 0 \quad\implies\quad |\beta|>|\n|.
    \end{equation}
    Then \eqref{exp} defines $\G\in{\rm Alg}(\RR)\cap{\rm End}(\T^*)$.
\end{lemma}
\begin{proof}
We start by arguing that the matrix coefficients 
\begin{equation}\label{coeffGamma}
(\G)_\beta^\gamma = \sum_{j\geq0} \sum_{\n_1,\dots,\n_j} \sum_{\beta_1+\cdots+\beta_{j+1}=\beta}
\pi^{(\n_1)}_{\beta_1}\cdots\pi^{(\n_j)}_{\beta_j} (D^{(\n_1)}\cdots D^{(\n_j)})_{\beta_{j+1}}^\gamma
\end{equation}
are well-defined for all $\beta,\gamma$.
From \eqref{D0props}, \eqref{Dnprops} and \eqref{pin} we see that if a summand is non-vanishing, then 
\begin{equation}\label{rg05}
    [\gamma]=[\beta_{j+1}]-j,\quad
    |\gamma|_p = |\beta_{j+1}|_p+\sum_{i=1}^j|\n_i|,\quad\text{and}\quad
    |\beta_i|>|\n_i| \text{ for } i=1,\dots,j.
\end{equation}
This implies $[\gamma]+|\gamma|_p \leq [\beta_{j+1}]-j+|\beta_{j+1}|_p+\sum_{i=1}^j|\beta_i|$, 
and since $\beta_1+\dots+\beta_{j+1}=\beta$ and $|\cdot|-\alpha$ is additive
\begin{equation}\label{rg04}
    [\gamma]+|\gamma|_p \leq [\beta_{j+1}]-j+|\beta_{j+1}|_p+|\beta|-|\beta_{j+1}|+j\alpha.
\end{equation}
As $\beta$ is fixed, we obtain for a $\beta$-dependent constant $C$ that $[\gamma]+|\gamma|_p\leq C - j(1-\alpha)$. 
By $0\leq[\cdot]+|\cdot|_p$, which follows from the definition \eqref{homogeneity2}, 
and by $1-\alpha>0$, we conclude that $j$ is bounded.
Hence the sum over $j\geq0$ in \eqref{coeffGamma} is finite, 
and by $|\n_i|<|\beta_i|$ also the sum over $\n_1,\dots,\n_j$ is finite. 
Thus the coefficient $(\G)_\beta^\gamma$ is well-defined. 

\medskip

To guarantee that these coefficients $\{(\G)_\beta^\gamma\}_{\beta,\gamma}$ define a linear map from $\RR$ to itself, 
we have to show that for fixed $\beta$ there are only finitely many $\gamma$ with $(\G)_\beta^\gamma\neq0$.
In case $(\G)_\beta^\gamma\neq0$, we learn from \eqref{rg04} that $[\gamma]+|\gamma|_p$ is bounded. 
This forces $\gamma$ to assign only finitely many values to $k\neq0$, $\ell\neq0$, $\n\neq\0$, 
and to vanish for all but finitely many $k$, $\ell$, $\n\neq\0$.
It remains to argue that also $\gamma_a(k=0)$ and $\gamma_b(\ell=0)$ can take only finitely many values. 
This follows from $(D^{(\n_1)}\cdots D^{(\n_j)})_{\beta_{j+1}}^\gamma\neq0$, which by \eqref{D0comp} and \eqref{Dnprops} implies
$\gamma_a(k=0)\leq\beta_{a,j+1}(k=0)+j\leq\beta_a(k=0)+j$ and 
$\gamma_b(\ell=0)\leq\beta_{b,j+1}(\ell=0)+j\leq\beta_b(\ell=0)+j$.

\medskip

The proof of multiplicativity of $\G$ follows from the derivation property of $D^{(\n)}$ and does not rely on the domain of $\G$ at all. We therefore refer to \cite[Lemma~3.12~(v)]{LO22} for a proof. 

\medskip

We finally show that $\G$ preserves $\T^*$. 
For this we shall argue that if $\gamma$ is populated and $(\G)_\beta^\gamma\neq0$, then $\beta$ is populated. 
For purely polynomial $\gamma=g_\n$ we observe that \eqref{exp} immediately yields
\begin{equation}\label{Gamma_pn}
    \G \p_\n = \p_\n + \pi^{(\n)},
\end{equation}
hence $(\G)_\beta^{g_\n} = \delta_\beta^{g_\n} + \pi^{(\n)}_\beta$. 
Since $\pi^{(\n)}\in \T^*$, this is only non-vanishing for populated $\beta$. 
We turn to multiindices $\gamma$ that are populated and not purely polynomial. 
Recall from \eqref{D0props} and \eqref{Dnprops} that $(D^{(\n_1)}\cdots D^{(\n_j)})_{\beta_{j+1}}^\gamma\neq0$ implies
$\sum_\ell \gamma_b(\ell)=\sum_\ell \beta_{j+1,b}(\ell)$.
Since by assumption $\sum_\ell\gamma_b(\ell)>0$, we obtain $\sum_\ell\beta_b(\ell)\geq\sum_\ell\beta_{j+1,b}(\ell)>0$. 
Similarly, since $\gamma$ is populated, we obtain from the first item of \eqref{rg05} that 
$[\beta_{j+1}] = [\gamma] + j = \sum_\ell \gamma_b(\ell) -1+j = \sum_\ell \beta_{j+1,b}(\ell) -1+j$.
Hence $(\G)_\beta^\gamma\neq0$ yields 
\begin{align}
    [\beta] 
    = [\beta_1]+\cdots+[\beta_{j+1}]
    = \sum_\ell \beta_{1,b}(\ell)+\cdots+\beta_{j+1,b}(\ell)) -1
    = \sum_\ell \beta_b(\ell)-1,
\end{align}
where we used that $\beta_1,\dots,\beta_j$ in \eqref{coeffGamma} are populated. 
\end{proof}

From the proof of Lemma~\ref{lem:defineGamma} we obtain in addition
\begin{equation}\label{GammaPreservesTildeT}
    \G_{xy} \widetilde{\T}^* \subset \widetilde{\T}^*,
\end{equation}
which we shall use in the sequel.

\begin{remark}
We note that $\G$ defined here coincides\footnote{up to the fact that there are no $\b_\ell$ components in \cite{OST23}} with the one constructed in \cite[Lemma~3]{OST23}, 
since both maps are multiplicative and coincide on the coordinates $\a_k,\b_\ell,\p_\n$. 
Therefore, $\mathsf{G}^*:=\{\G\text{ as in Lemma~\ref{lem:defineGamma}}\}$ 
is a group (with respect to ~composition) and there exists a group $\mathsf{G}$, 
called the structure group, 
such that $\mathsf{G}^*$ is the pointwise dual of $\mathsf{G}$, cf.~\cite[Lemma~4]{OST23}. 
\end{remark}

\medskip

In \cref{sec:induction_proper} (item (4) of the case $|\beta|>3$) we shall argue that there is a choice of $\{\pi^{(\n)}_{xy}\}_\n$ such that the associated $\G_{xy}$ (see~\eqref{exp}) satisfies \eqref{recenter}, \eqref{transitive}, \eqref{Gtri}, and \eqref{recenterPoly}. 
To estimate $\G_{xy}$, 
we will mainly appeal to the exponential formula \eqref{exp}, 
see \cref{alg1} (Algebraic argument~I). 
This makes use of estimates on $\pi^{(\n)}_{xy}$ 
that we obtain in \cref{3pt1} (Three-point argument~I), 
based on \eqref{recenter} involving two base points and one active point. 
To estimate $\delta\G_{xy}$, 
which is the directional derivative of $\G_{xy}$ in the direction $\delta\xi$, 
we proceed similarly, see \cref{alg2} (Algebraic argument~II). 
It is based on estimates on $\delta\pi^{(\n)}_{xy}$, 
which is the directional derivative of $\pi^{(\n)}_{xy}$ in the direction $\delta\xi$, 
that we establish in \cref{3pt2} (Three-point argument~II). 


\subsubsection{Ansatz for \texorpdfstring{$\d\G$}{dGamma}}\label{sec:dGamma}
We now discuss $\d\G_{xz}$ and start by motivating an ansatz.  By $\alpha<1$, we infer from \eqref{kappa} that $\kappa+\alpha>2$, and hence \eqref{eq:lhsstrongboundincrement} implies on a purely qualitative level 
\begin{equation}
    \E^\frac{1}{q'} | (\delta\Pi_x-\delta\Pi_x(z)-\d\G_{xz}\Pi_z)_{\beta}(y)|^{q'} = o(|y-z|_\s^2).
\end{equation}
Since $\partial^\n\Pi_{x\beta}$ and $\partial^\n\delta\Pi_{x\beta}$ are continuous functions for $|\n|\leq2$, 
see \cref{rem:smoothness} and \cref{rem:smoothness_Mall}, 
this amounts to 
\begin{equation}\label{qualitativeIncrement}
    \partial^\n(\delta\Pi_x-\delta\Pi_x(z)-\d\G_{xz}\Pi_z)_\beta (z) = 0 \quad\text{for}\quad |\n|\leq2.
\end{equation}
Note that for $\n=\0$ this is automatically satisfied by 
\begin{equation}\label{Pixx}
\Pi_x(x)=0 \, ,
\end{equation}
which is a consequence of the estimate \eqref{estPi} since $|\cdot|\geq\alpha>0$. 
A first ansatz to obtain \eqref{qualitativeIncrement} for $|\n|=1,2$ as well could be $\d\G_{xz}=\delta\G_{xz}$. 
However, $\delta\G_{xz}$ is not rich enough: 
to achieve second order vanishing around $z$ we expect to need to recenter $\delta\Pi_x-\delta\Pi_x(z)$ by $(\cdot-z)^\n$ for $|\n|=1,2$.
By \eqref{polypart:poly}, this is only possible if $(\delta\G_{xz})_\beta^{g_\n}$ does not vanish for $|\n|=1,2$. 
As we will see in Lemma~\ref{lem:tri1}, $\delta\G_{xz}$ is triangular with respect to ~the homogeneity $|\cdot|$, meaning that 
$(\delta\G_{xz})_\beta^{g_\n}\neq0$ implies $|g_\n|<|\beta|$. 
Hence $\delta\G_{xz}$ only allows for the appropriate recentering for multiindices $|\beta|>2$. 
To achieve the recentering  for multiindices $|\beta|\leq2$ as well, we have to relax the population condition and give up the triangularity of $\d\G_{xz}$ with respect to ~the homogeneity, cf.~\eqref{tri6}. 
We therefore make the ansatz\footnote{note the structural similarity to $\delta\G_{xz}=\sum_\n \delta\pi^{(\n)}_{xz}\G_{xz}D^{(\n)}$}
\begin{equation}\label{dGamma}
    \d\G_{xz} = \sum_{|\n|\leq2} \d\pi^{(\n)}_{xz}\G_{xz}D^{(\n)} Q,
\end{equation}
where 
\begin{equation}\label{choice_dpi}
\d\pi^{(\0)}_{xz}:=Q\delta\Pi_x(z)\in Q\widetilde{\T}^* \quad\text{and}\quad 
\d\pi^{(\n)}_{xz}\in Q\widetilde{\T}^* \text{ for }|\n|=1,2 \text{ to be chosen}. 
\end{equation}
Recall, that $Q$ denotes the projection of a power series 
$\sum_\beta\pi_\beta\z^\beta$ to $\sum_{|\beta|<3}\pi_\beta\z^\beta$. 
The reason for including $Q$ in the definition of $\d\G$ will become clear in the proof of~\cref{lem:integrationIII}.
Using the population constraint \eqref{choice_dpi}, one can check that 
\begin{equation}\label{mapping_dGamma}
\d\G_{xz} \T^*\subset \widetilde{\T}^* \, ,
\end{equation}
the proof of which follows the same lines as the one of \cref{lem:defineGamma}.

\medskip

We will argue in \cref{sec:induction_proper} (item (10) of the case $|\beta|<3$)
that \eqref{qualitativeIncrement} 
indeed determines $\d\pi^{(\n)}_{xz}$ for $|\n|=1,2$. 
The estimate on $\d\G_{xz}$ is based on \eqref{dGamma}, 
see \cref{alg4} (Algebraic argument IV), 
which is based on estimates on $\d\pi^{(\n)}_{xz}$ that we establish in 
\cref{3pt4} (Three-point argument IV).
In addition to the plain estimate on $\d\G_{xz}$, 
when obtaining the improved vanishing \eqref{eq:rhsweakboundincrement} of increments of $\delta\Pi^-$ 
we will make use of an estimate on the increment 
$\d\G_{xy}-\d\G_{xz}\G_{zy}$. 
This estimate on the increment is obtained in \cref{alg3} (Algebraic argument III), 
based on the corresponding estimate on 
$\d\pi^{(\n)}_{xy}-\d\pi^{(\n)}_{xz}-\d\G_{xz}\pi^{(\n)}_{zy}$
obtained in \cref{3pt3} (Three-point argument III).

\medskip

We next argue that the ansatz \eqref{dGamma} 
allows for the crucial identity \eqref{magic} to hold true. 

\begin{lemma}
Let $\d\G_{xz}$ be given by \eqref{dGamma} 
with $\d\pi^{(\n)}_{xz}$ satisfying \eqref{choice_dpi}, 
and such that \eqref{qualitativeIncrement} holds true.
Then, \eqref{magic} holds true.
\end{lemma}

\begin{proof}
By \eqref{Pixx} we read off from \eqref{piminus} that 
\begin{equation}
    \Pi^-_z(z) = \a_0\nabla\Delta\Pi_z(z) + \b_0\xi_\tau(z)-\nabla\Pi_z(z) c.
\end{equation}
Since $|e_0+\beta|=|\beta|$, cf.~\eqref{homogeneity}, 
we have $Q(\a_0\nabla\Delta\Pi_z)=\a_0\nabla\Delta Q\Pi_z$, 
and by the derivation property of $D^{(\n)}$ and multiplicativity of $\G_{xz}$, this yields 
\begin{equation}
\d\G_{xz}(\a_0\nabla\Delta\Pi_z)
= (\d\G_{xz} \a_0)\nabla\Delta \G_{xz} Q\Pi_z 
+ (\G_{xz} \a_0)\nabla\Delta \d\G_{xz} \Pi_z \, .
\end{equation}
Furthermore, from the estimate \eqref{estPi} of $\Pi$ we learn 
\begin{equation}
\E^\frac{1}{p}|\nabla\Pi_{z\beta t}(z)|^p
\leq \int \, \dx{y} |\nabla\psi_t(y)| 
\E^\frac{1}{p}|\Pi_{z\beta}(z-y)|^p 
\lesssim (\sqrt[8]{t})^{|\beta|-1}, 
\end{equation}
where we have used the moment bound \eqref{eq:kerbound} in the last inequality,
which implies in particular $\nabla\Pi_{z\beta}(z)=0$ a.s.~for $|\beta|>1$. 
Together with the fact that $c_\beta$ is only non vanishing for $|\beta|<2+\alpha$, see \eqref{pop4}, 
we obtain $Q(\nabla\Pi_z(z) c)=\nabla\Pi_z(z) c$ and hence
\begin{equation}
\d\G_{xz}(\nabla\Pi_z(z) c)
= (\d\G_{xz}\nabla\Pi_z(z)) \G_{xz} c 
+ (\G_{xz}\nabla\Pi_z(z)) \d\G_{xz} c.
\end{equation}
Plugging into the exponential formula \eqref{exp} 
the definition \eqref{D0} of $D^{(\0)}$ 
and the choice $\pi^{(\0)}_{xz}=\Pi_x(z)$, see \eqref{pi0}, 
we see 
\begin{equation}
\G_{xz}\a_{k'} 
= \sum_{k\geq0} \tbinom{k+k'}{k} \Pi_x^k(z) \a_{k+k'} 
\quad\textnormal{and}\quad
\G_{xz}\b_{\l'} 
= \sum_{\l\geq0} \tbinom{\l+\l'}{\l} \Pi_x^\l(z) \b_{\l+\l'} \, .
\end{equation}
Using this, we can read off from the ansatz \eqref{dGamma} of $\d\G$ 
and the chain rule for the Malliavin derivative 
\begin{equation}
\d\G_{xz} \a_0 = \sum_{k\geq0} \a_k \delta(\Pi_x^k(z)) 
\quad\textnormal{and}\quad
\d\G_{xz} \b_0 = \sum_{\l\geq0} \b_\l \delta(\Pi_x^\l(z)) \, .
\end{equation}
Furthermore, since $c\in\R[[\a_k,\b_\l]]$, see \cref{thm:main}, 
the same arguments yield 
\begin{equation}
\G_{xz} c = \sum_{m} \tfrac{1}{m!} \Pi_x^m(z)(D^{(\0)})^m c
\quad\textnormal{and}\quad
\d\G_{xz} c = \sum_{m} \tfrac{1}{m!} \delta(\Pi_x^m(z))(D^{(\0)})^m c \, .
\end{equation}
Altogether we obtain 
\begin{align}
    \d\G_{xz}\Pi^-_z(z) 
    &= \sum_k \a_k \delta(\Pi_x^k(z)) \nabla\Delta \G_{xz} Q\Pi_z(z) 
     + \sum_k \a_k \Pi_x^k(z) \nabla\Delta \d\G_{xz} \Pi_z(z) \\
    &\,+ \sum_\ell \b_\ell \delta(\Pi_x^\ell(z)) \xi_\tau(z) 
    - (\d\G_{xz}\nabla\Pi_z(z)) \sum_m \tfrac{1}{m!} \Pi_x^m(z) (D^{(\0)})^m c \\
    &\,- (\G_{xz}\nabla\Pi_z(z)) \sum_m \tfrac{1}{m!} \delta(\Pi_x^m(z)) (D^{(\0)})^m c. 
\end{align}
On the other hand, applying the Malliavin derivative to \eqref{piminus} we get

\begin{align}
    \delta\Pi^-_x 
    &= \sum_k \a_k \delta(\Pi_x^k)\nabla\Delta\Pi_x
    + \sum_k \a_k \Pi_x^k\nabla\Delta\delta\Pi_x \\
    &\,+ \sum_\ell \b_\ell \delta(\Pi_x^\ell)\xi_\tau
    + \sum_\ell \b_\ell \Pi_x^\ell \delta\xi_\tau  \\
    &\,- \sum_m \tfrac{1}{m!} \delta(\Pi_x^m)\nabla\Pi_x (D^{(\0)})^m c
    - \sum_m \tfrac{1}{m!} \Pi_x^m \nabla\delta\Pi_x (D^{(\0)})^m c.
\end{align}
Thus
\begin{align}
    Q(\delta\Pi^-_x -\d\G_{xz}\Pi^-_z) (z) 
    &= Q \sum_k \a_k \delta(\Pi_x^k(z))\nabla\Delta(\Pi_x-\G_{xz} Q\Pi_z)(z) \\
    &\,+ Q  \sum_k \a_k \Pi_x^k(z)\nabla\Delta(\delta\Pi_x-\d\G_{xz} \Pi_z)(z) \\
    &\,+ Q \sum_\ell \b_\ell \Pi_x^\ell(z) \delta\xi_\tau(z)  \\
    &\,- Q \sum_m \tfrac{1}{m!} \delta(\Pi_x^m(z)) \nabla (\Pi_x-\G_{xz}\Pi_z)(z) (D^{(\0)})^m c \\
    &\,- Q \sum_m \tfrac{1}{m!} \Pi_x^m(z) \nabla(\delta\Pi_x-\d\G_{xz}\Pi_z)(z) (D^{(\0)})^m c. 
\end{align}
We shall argue that the first, fourth, and last right hand side term vanish.
For the first term we use that $Q(\a_k\pi_1\cdots\pi_{k+1})=Q(\a_k(Q\pi)\cdots(Q\pi_{k+1}))$, 
which follows from 
\begin{equation}
    e_k+\beta_1+\cdots+\beta_{k+1}=\beta \quad\implies\quad |\beta_1|+\cdots+|\beta_{k+1}|=|\beta|
\end{equation}
and non-negativity of the homogeneity $|\cdot|$, 
and that $Q\G_{xz}Q=Q\G_{xz}$, which follows from the triangularity \eqref{Gtri} of $\G_{xz}$ with respect to ~the homogeneity. 
Thus, by \eqref{recenter} the first right hand side term vanishes.
The fourth right hand side term vanishes by \eqref{recenter} as well. 
The last right hand side term vanishes by \eqref{qualitativeIncrement}, 
which finishes the proof of \eqref{magic}. 
\end{proof}


\subsection{Inductive structure of the proof}\label{sec:inductive_proof}
The whole argument outlined above is carried out inductively. A natural choice for the ordering needed for induction is the length of a multiindex $\beta$. We consider instead the following weighted length 
\begin{align}\label{order}
\length{\beta}:= \sum_k \beta_a(k) + \sum_\ell \beta_b(\ell) + \lambda \sum_{\n\neq\0} |\n|\beta_p(\n) 
\end{align}
with $0<\lambda<1/2$. For ease of notation, we introduce 
\begin{align}
\gamma&\prec\beta \quad \Longleftrightarrow \quad \length{\gamma}<\length{\beta}, \\
\gamma&\preccurlyeq\beta \quad \Longleftrightarrow \quad \gamma\prec\beta \textnormal{ or } \gamma=\beta.
\end{align}
\begin{remark}\label{rem:general_order}
The weight $\lambda$ is necessary for $\d\G$ to be triangular with respect to  this length, see \eqref{tri6}. 
More generally, if the sum in the definition \eqref{dGamma} of $\d\G$ is restricted to $|\n|\leq C$, 
then the upcoming Lemma~\ref{lem:tri1} remains true, provided $\lambda$ is restricted by $0<\lambda<1/C$
and $2$ is replaced by $C$ in the last item of \eqref{tri6}. 
\end{remark}
\begin{remark}
The weight $|\n|$ allows for the following finiteness property, 
which makes $\prec$ suitable for an inductive argument: For all $\beta$
\begin{equation}\label{coercive}
    \#\{\gamma \text{ populated}\,|\,\gamma\prec\beta\}<\infty. 
\end{equation}
Indeed, if $\length{\gamma}$ is bounded, then the term $\sum_{\n\neq\0}|\n|\gamma_p(\n)$ forces $\gamma$ 
to assign only finitely many values to $\n\neq\0$, and to vanish for all but finitely many $\n\neq\0$. 
In particular, there are finitely many purely polynomial $\gamma$. 
If $\gamma$ is populated and not purely polynomial, then by \eqref{pop1}
$$
\sum_k (k+1)\gamma_a(k)+\sum_\ell (\ell+1)\gamma_b(\ell)=-1+\sum_k\gamma_a(k)+2\sum_\ell \gamma_b(\ell) + \sum_{\n\neq\0}\gamma_p(\n).
$$
The right hand side of this expression is bounded by assumption, 
forcing $\gamma$ also to assign only finitely many values to $k,\ell$, and to vanish for all but finitely many $k,\ell$.
\end{remark}

Together with \eqref{coercive}, the following lemma provides all triangular dependencies that allow for an inductive proof.
\begin{lemma}[Triangularity]\label{lem:tri1} 
\begin{itemize}
    \item[] 
    \item[(i)] $\Pi_{x\beta}^-$ given by \eqref{piminuscomponents} does not depend on $\Pi_{x\beta'}$ unless $\beta'\prec\beta$.  Furthermore, if $\Pi_{x\beta}^{-}$ depends on $c_{\beta'}$, then we must have $\beta' +g_{\mathbf{n}^i} \prec \beta$, for all $1 \leq i \leq d$, or $\beta'+ g_{\mathbf{n}^i} = \beta$, for some $1 \leq i \leq d$, 
    where\footnote{Mind the difference of notation between subscripts $\n_1,\dots,\n_j$ used for enumeration and superscripts $\n^i$ denoting the unit vectors of $\mathbb{N}_0^{1+d}$.} $\n^i$ is the unit vector in the $i$-th direction. \label{prop:choosing}

    \item[(ii)] For $\G$ defined in \eqref{exp} and all $\gamma$ (not necessarily populated), 
    \begin{align}
    (\G)_\beta^\gamma &\text{ does not depend on } \pi_{\beta'}^{(\n)} \text{ unless } \beta'\preceq\beta, \label{tri2} \\
    \text{if }\sum_\ell\gamma_b(\ell)>0, \text{ then } (\G)_\beta^{\gamma} &\text{ does not depend on } \pi_{\beta'}^{(\n)} \text{ unless } \beta'\prec\beta. \label{tri1}
    \end{align}
    Moreover, 
    \begin{align}
    (\G-\id)_\beta^\gamma\neq0 &\quad\implies\quad \gamma\prec\beta \quad\text{and}\quad |\gamma|<|\beta|, \label{tri3}\\
    (\delta\G)_\beta^\gamma\neq0 &\quad\implies\quad \gamma\prec\beta \quad\text{and}\quad |\gamma|<|\beta|. \label{tri4}
    \end{align}
    \item[(iii)] For $\d\G$ defined in \eqref{dGamma} and $\gamma$ populated, 
    \begin{equation}
    (\d\G)_\beta^{\gamma\neq\pp} \text{ does not depend on } \d\pi^{(\n)}_{\beta'}, \G_{\beta'} \text{ unless } \beta'\prec\beta. \label{tri5}
    \end{equation}
    Moreover,
    \begin{align}
    (\d\G)_\beta^{\gamma\neq\pp}\neq0 \quad&\implies\quad |\beta|\geq2\alpha, \label{tri7}\\
    (\d\G)_\beta^\gamma\neq0 \quad&\implies\quad \gamma\prec\beta \quad\text{and}\quad |\gamma|\leq|\beta|+2-\alpha. \label{tri6}
    \end{align}
\end{itemize}
\end{lemma}

We provide the proof of \cref{lem:tri1} at the end of this section. 
In the following two subsections we outline the logical order of the induction. 

\subsubsection{Purely polynomial multiindices}\label{ind:pp}
Before we come to the induction proper, we construct and estimate
all purely polynomial components of all objects involved in the proof of \cref{thm:main}.
For such $\beta=g_\n$, the estimate \eqref{estPi} of $\Pi_{x g_\n}$ is satisfied trivially, 
since $\Pi_{x g_\n}$ is according to \eqref{polypart:poly} defined by $(\cdot-x)^\n$.
Because $\Pi^-_{xg_\n}=0$ by \eqref{eq:piminusttilde}, 
the estimate \eqref{estPiminus} of $\Pi^-_{x g_\n}$ is also true. 

Similarly, the estimate \eqref{estGamma} of $(\G_{xy})_{g_\n}^\gamma$ also holds:
By \eqref{GammaPreservesTildeT} we know that $(\G_{xy})_{g_\n}^\gamma$ is only non-vanishing if $\gamma=g_\m$ for some $\m\neq\0$, 
in which case it is, according to \eqref{recenterPoly}, defined by $\binom{\n}{\m}(y-x)^{\n-\m}$, 
with the implicit understanding that $\binom{\n}{\m}=0$ if the componentwise $\m\leq\n$ is violated. 
The estimate \eqref{estpin} of $\pi^{(\m)}_{xy g_\n}$ follows analogously, 
since the exponential formula \eqref{exp} yields 
$(\G_{xy})_{g_\n}^{g_\m}=(\p_\m+\pi^{(\m)}_{xy})_{g_\n}$, 
which because of the previous argument leads us to define 
$\pi^{(\m)}_{xyg_\n} = \binom{\n}{\m}(y-x)^{\n-\m}$ for\footnote{by $\m<\n$ we understand $\m\neq\n$ and componentwise $\m\leq\n$} $\m<\n$. Note that this choice is consistent with the population constraint \eqref{pin}.

The estimates \eqref{estdeltaPi}, \eqref{eq:rhsweakbounddPi}, \eqref{estdeltaGamma}, and \eqref{estdeltapin} on
$\delta\Pi_{x g_\n}$, $\delta\Pi^-_{x g_\n}$, $(\delta\G_{xy})_{g_\n}^\gamma$, and $\delta\pi^{(\m)}_{xy g_\n}$, respectively, 
hold true, since all these objects vanish as they are deterministic by the previous arguments.

From the mapping property \eqref{mapping_dGamma} of $\d\G_{xz}$, 
we know that $(\d\G_{xz})_{g_\n}^\gamma$ vanishes for populated $\gamma$, 
and so does $\d\pi^{(\m)}_{xz g_\n}$ since $\d\pi^{(\m)}_{xz}$ is an element of $\widetilde{\T}^*$ by \eqref{dGamma}. 
Thus, the estimates \eqref{eq:gammabound} and \eqref{estdpin} on $(\d\G_{xz})_{g_\n}^\gamma$ and $\d\pi^{(\m)}_{xz g_\n}$ also hold true trivially. 

Finally, the estimates \eqref{eq:lhsstrongboundincrement}, \eqref{eq:rhsweakboundincrement}, \eqref{estdGammaincrement}, and \eqref{estdpinincrement} on increments of $\delta\Pi_{g_\n}$, $\delta\Pi^-_{g_\n}$, $\d\G_{g_\n}$, and $\d\pi^{(\m)}_{g_\n}$ are trivially satisfied, as we have just argued that all these objects vanish. 

\subsubsection{Induction proper}\label{sec:induction_proper}

We turn to the proper induction, where we treat populated and not purely polynomial multiindices $\beta$. 
From \eqref{pop1} and the definition \eqref{order} of $\length{\cdot}$, we 
see that $\beta=g_\n$ with $|\n|=1$ can serve as the base case. 

In the induction step, we fix a populated and not purely polynomial $\beta$ and assume for all $\beta'\prec\beta$ that the estimates 
\eqref{estPi}, \eqref{estPiminus}, \eqref{estGamma}, and \eqref{estpin} on 
$\Pi_{\beta'}$, $\Pi^-_{\beta'}$, $\G_{\beta'}$, and $\pi^{(\n)}_{\beta'}$, 
 the estimates 
\eqref{estdeltaPi}, \eqref{eq:rhsweakbounddPi}, \eqref{estdeltaGamma}, and \eqref{estdeltapin} on 
$\delta\Pi_{\beta'}$, $\delta\Pi^-_{\beta'}$, $\delta\G_{\beta'}$ and $\delta\pi^{(\n)}_{\beta'}$, 
 the estimates
\eqref{eq:gammabound} and \eqref{estdpin} on 
$\d\G_{\beta'}$ and $\d\pi^{(\n)}_{\beta'}$, 
and the estimates 
\eqref{eq:lhsstrongboundincrement}, \eqref{eq:rhsweakboundincrement}, \eqref{estdGammaincrement}, and \eqref{estdpinincrement} on increments of $\delta\Pi_{\beta'}$, $\delta\Pi^-_{\beta'}$, $\d\G_{\beta'}$, and $\d\pi^{(\m)}_{\beta'}$ 
hold true, with the understanding that all these objects have been constructed. 
Furthermore, we assume that $c_{\beta'}$ has been constructed for all $\beta'+g_{\n^i}\prec\beta$. 
The aim is to construct and estimate the corresponding $\beta$-components, 
except for $c$ where we construct the $c_{\beta-g_{\n^i}}$ component. 
In the induction, we distinguish the case $|\beta|<3$ from $|\beta|>3$, 
and start by explaining the simpler case $|\beta|>3$. 
Note that by \eqref{homogeneity_integer_polynomial} the case $|\beta|=3$ has been dealt with in the previous subsection on purely polynomial multiindices. 
\begin{enumerate}
\item By the triangular property \eqref{tri1}, we construct and estimate 
$(\G)_\beta^{\gamma}$ for $\gamma$ not purely polynomial 
in \cref{alg1} (Algebraic argument~I). 
\item By the triangular property of \cref{lem:tri1}~(i), 
we define $\Pi^-_\beta$ by \eqref{piminuscomponents}, 
where we set $c_{\beta-g_{\n^i}}=0$. 
Furthermore, we estimate $\Pi^-_\beta$ in \cref{rec1} (Reconstruction~I). 
\item Based on a Liouville principle, we construct and estimate $\Pi_\beta$ in \cref{cor:intI} (Integration~I). 
\item We construct and estimate $\pi^{(\n)}_\beta$, 
which by \eqref{Gamma_pn} yields together with Item~(1) the construction of and estimates on $(\G)_\beta^\gamma$ for all populated $\gamma$. 
The only equation dependent ingredient in the construction of $\pi^{(\n)}_\beta$ 
is a Liouville principle, which we provide in \cref{lem:int1}; 
we therefore refer for the construction to \cite[Section~5.3]{LOTT21}. 
We only note for later that the choice
\begin{equation}\label{pi0}
\pi^{(\0)}_{xy\beta} = \Pi_{x\beta}(y)
\end{equation}
has to be made, 
and that the construction respects \eqref{pin} and yields that the 
$\beta$-components of \eqref{recenter} and \eqref{transitive} hold. 
Recall that also \eqref{Gtri} and \eqref{recenterPoly} hold, 
the former by \cref{lem:tri1} 
and the latter by \eqref{GammaPreservesTildeT} and the choice we made for purely polynomial multiindices in the previous subsection. 
The estimate on $\pi^{(\n)}_\beta$ is provided in \cref{3pt1} (Three-point argument~I). 
\end{enumerate}

In the case $|\beta|<3$ we proceed as follows. 
\begin{enumerate}
\item By the triangular properties \eqref{tri1} and \eqref{tri5}, 
we construct and estimate 
\begin{enumerate}
    \item $(\G)_\beta^{\gamma\neq\pp}$ in \cref{alg1} (Algebraic argument~I), 
    \item $(\delta\G)_\beta^{\gamma\neq\pp}$ in \cref{alg2} (Algebraic argument~II), 
    \item $(\d\G-\d\G\G)_\beta^{\gamma\neq\pp}$ in \cref{alg3} (Algebraic argument~III), 
    \item $(\d\G)_\beta^{\gamma\neq\pp}$ in \cref{alg4} (Algebraic argument~IV).
\end{enumerate}
\item By the triangular property of \cref{lem:tri1}~(i), 
we define $\Pi^-_\beta$ by \eqref{piminuscomponents}, 
where we define $c_{\beta-g_{\n^i}}$ according to the BPHZ-choice \eqref{eq:BPHZchoice}. 
If $\beta-g_{\n^i}$ is not a multiindex then there is no $c$-component to choose, 
but \eqref{eq:BPHZchoice} is still satisfied by the symmetry properties of \cref{prop:symmetries} (for more details see \cref{sec:bphz}). 
We show in \cref{prop:expectation} that this choice allows to estimate $\E\Pi^-_\beta$. 
\item Based on \eqref{magic}, we estimate $(\delta\Pi^- -\d\G\Pi^-)_\beta$ in \cref{rec2} (Reconstruction~II). 
\item Equipped with the estimates of Item~(1d) and Item~(3), 
we estimate $\delta\Pi^-_\beta$ in \cref{lem:averaging} (Averaging). 
\item As explained in \cref{sec:application_SG}, we estimate $\Pi^-_\beta$ by an application of the spectral gap inequality, 
based on the estimates of Items~(2) and (4). 
\item Exactly as in Item~(3) of the case $|\beta|>3$, we construct and estimate $\Pi_\beta$ in \cref{cor:intI} (Integration~I). 
\item Exactly as in Item~(4) of the case $|\beta|>3$ we construct $\pi^{(\n)}_\beta$ 
and provide its estimate in \cref{3pt1} (Three-point argument~I). 
As before, this finishes together with Item~(1a) 
the construction and estimate on $(\G)_\beta^\gamma$ for all populated $\gamma$. 
\end{enumerate}

\begin{figure}[h!]
\begin{center}
\begin{tikzpicture}[scale=.75]
\draw[-stealth, thick] (8.5,3.5) arc (45:-37:5cm) 
node [pos=.55, below, rotate=-87, yshift=-0pt] () 
{\scriptsize (3)}
node [pos=.5, right, rotate=-90, yshift=5pt, xshift=-30pt] () 
{\footnotesize Reconstruction II};
\draw[-stealth, thick] (7.8,0) arc (39:-2:4.5cm) 
node [pos=.5, below, sloped, yshift=-0pt] () 
{\scriptsize (3)}; 
\node[rectangle, draw, fill=white, align=center, inner sep = 2pt] () at (5,0)
{$\delta\Pi_x-\delta\Pi_{x}(z)-{\rm d}\Gamma^*_{xz}\Pi_{z}$};
\node[rectangle, draw, fill=white, text width = 1cm, align=center, inner sep = 3pt] () at (8.5,3.5)
{$\Pi_x$};
\node[rectangle, draw, fill=white, align=center, inner sep = 2pt] () at (8.5,-3.5)
{$\delta\Pi_x^- -{\rm d}\Gamma^*_{xz}\Pi_z^-$};

\draw[-stealth, thick] (8.5,-3.5) arc (-45:-127:5cm) 
node [pos=.5, below, yshift=-0pt] () 
{\footnotesize Averaging} 
node [pos=.5, above, yshift=0pt, xshift=-3pt] () 
{\scriptsize (4)};
\draw[-stealth, thick] (7.3,-3.5) arc (-90:-157:5cm) 
node [pos=.5, below, sloped, yshift=-0pt, xshift=2pt] ()
{\footnotesize Integration III} 
node [pos=.5, below, above, sloped, yshift=0pt] () 
{\scriptsize (11)};

\node[rectangle, draw, fill=white, align=center, inner sep = 2pt] () at (8.5,-3.5)
{$\delta\Pi_x^- -{\rm d}\Gamma^*_{xz}\Pi_z^-$};
\node[rectangle, draw, fill=white, text width = 1cm, align=center, inner sep = 2pt] () at (1.5,-3.5)
{$\delta\Pi_x^-$};
\draw[-stealth, thick] (1.5,-3.5) arc (-135:-217:5cm) 
node [pos=.55, above, sloped, yshift=0pt] () 
{\footnotesize BPHZ \& (SG)} 
node [pos=.55, below, sloped, yshift=-0pt] () 
{\scriptsize (2) \& (5)};
\node[rectangle, draw, fill=white, text width = 1cm, align=center, inner sep = 2pt] () at (1.5,-3.5)
{$\delta\Pi_x^-$};
\node[rectangle, draw, fill=white, text width = 1cm, align=center, inner sep = 2pt] () at (1.5,3.5)
{$\Pi_x^-$};
\draw[-stealth, thick] (1.5,3.5) arc (135:53:5cm) 
node [pos=.55, above, sloped, yshift=0pt] () 
{\footnotesize Integration I} 
node [pos=.54, below, sloped, yshift=-0pt] () 
{\scriptsize (6)};
\node[rectangle, draw, fill=white, text width = 1cm, align=center, inner sep = 2pt] () at (1.5,3.5)
{$\Pi_x^-$};
\end{tikzpicture}
\end{center}
\caption{Visualization of the main steps of the inductive structure of the proof for multiindices $|\beta|<3$.}
\end{figure}
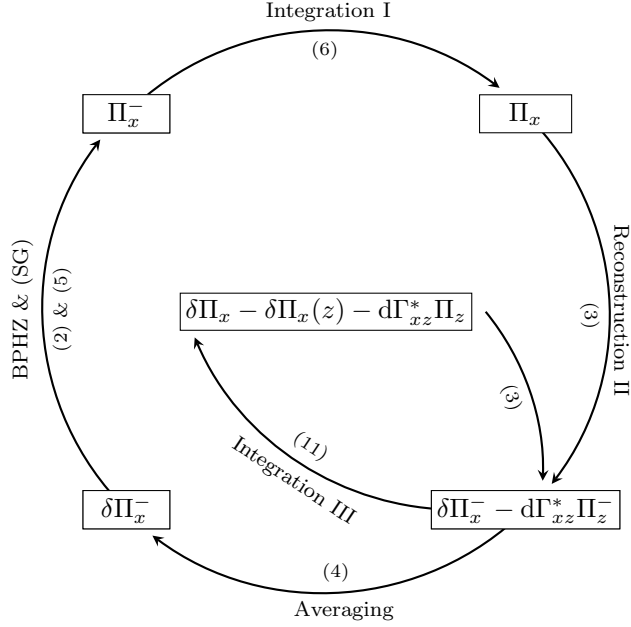

This finishes the construction and estimates on the $\beta$-components 
of all objects stated in \cref{thm:main}.
However, for later induction steps we have to construct and estimate a few more objects which we have made use of. 
\begin{enumerate}
\setcounter{enumi}{7}
\item Analogous to $\Pi_\beta$, we estimate its Malliavin derivative 
$\delta\Pi_\beta$ in \cref{cor:intII} (Integration~II) based on a Liouville principle.
\item Analogous to $\pi^{(\n)}_\beta$, we estimate its Malliavin derivative 
$\delta\pi^{(\n)}_\beta$ in \cref{3pt2} (Three-point argument~II). 
Applying $\delta$ to \eqref{Gamma_pn}, we see that this provides together with Item~(1b) 
the estimate on $(\delta\G)_\beta^\gamma$ for all populated $\gamma$. 
\item We construct $\d\pi^{(\n)}_\beta$ for $|\n|=1,2$ as follows:
By $D^{(\n)}\p_\n=1$ (see \eqref{Dn}) and $\G_{xz}1=1$ (see \eqref{exp}), 
we obtain from the ansatz \eqref{dGamma} that $\d\G_{xz}\p_\n=\d\pi^{(\n)}_{xz}$ for $|\n|=1,2$. 
Furthermore, \eqref{polypart:poly} implies $\tfrac{1}{\n!}\partial^\n ((1-P)\Pi_z) (z) = \p_\n$, 
hence \eqref{qualitativeIncrement} yields for $|\n|=1,2$ 
\begin{equation}
    \d\pi^{(\n)}_{xz} = \tfrac{1}{\n!} \partial^\n(\delta\Pi_x-\d\G_{xz}P\Pi_z)(z).
    	\allowdisplaybreaks
\end{equation}
By the triangular structure \eqref{tri6}, 
this serves as an inductive definition of $\d\pi^{(\n)}$ 
provided we are given $\delta\Pi_\beta$, $(\d\G)_\beta^{\gamma\neq\pp}$,
and $\Pi_{\prec\beta}$.
By $\d\G\p_\n=\d\pi^{(\n)}$, this serves, 
together with Item~(1d), as a construction of 
$(\d\G)_\beta^\gamma$ for all populated $\gamma$. 
\item Once more based on a Liouville principle, we estimate 
$(\delta\Pi-\delta\Pi-\d\G\Pi)_\beta$ in \cref{lem:integrationIII} (Integration~III). 
\item We estimate $(\d\pi^{(\n)}-\d\pi^{(\n)}-\d\G\pi^{(\n)})_\beta$ in \cref{3pt3} (Three-point argument~III), 
which by \eqref{Gamma_pn} and $\d\G_{xz}\p_\n=\d\pi^{(\n)}_{xz}$ provides together with Item~(1c) the estimate on $(\d\G-\d\G\G)_\beta^\gamma$ for all populated $\gamma$. 
\item Finally we estimate $\d\pi^{(\n)}_\beta$ in \cref{3pt4} (Three-point argument~IV), 
which by $\d\G_{xz}\p_\n=\d\pi^{(\n)}_{xz}$ provides, together with 
Item~(1d), the estimate on $(\d\G)_\beta^\gamma$ for all populated $\gamma$. 
\end{enumerate} 

\begin{proof}[Proof of Lemma~\ref{lem:tri1}]
We start with the proof of (i). 
For the first two sums in \eqref{piminuscomponents} it is enough to establish 
\begin{align}
e_{k}+\beta_1+\cdots+\beta_{k+1}=\beta
\quad&\implies\quad
\beta_1,\dots,\beta_{k+1}\prec\beta, \label{ord1} \\ 
f_\ell+\beta_1+\cdots+\beta_{\ell}=\beta
\quad&\implies\quad
\beta_1,\dots,\beta_{\ell}\prec\beta. \label{ord2} 
\end{align}
Since $\length{\cdot}$ is additive and non negative, this is an immediate consequence of $\length{e_k}=\length{f_\ell}=1$.
The last sum in \eqref{piminuscomponents} is a linear combination of terms of the form  
$$
\Pi_{x\beta_1}\cdots\Pi_{x\beta_m}\nabla\Pi_{x\beta_{m+1}} \sum_\gamma ((D^{(\0)})^m)_{\beta_{m+2}}^\gamma c_\gamma
$$
for $m\geq0$ and $\beta_1+\dots+\beta_{m+2}=\beta$, which completes the proof of the first part of (i). We now move on to the proof of the second part of (i). Note that \eqref{D0props} implies that for all $\beta',\gamma'$
\begin{equation}\label{triD0}
    (D^{(\0)})_{\beta'}^{\gamma'}\neq0\quad\implies\quad \length{\beta'}=\length{\gamma'}.
\end{equation}
By iteration, the same property carries over to $((D^{(\0)})^m)_{\beta'}^{\gamma'}$.

We know now that the expression for $\Pi_{x\beta}^-$ given in~\eqref{piminuscomponents} consists only of $\Pi_{x\beta'}$ such that $\beta' \prec \beta$. We assume by induction that the result of (i) holds true for all such $\Pi_{x\beta'}$. Thus, we necessarily have that the first two terms on the right hand side of~\eqref{piminuscomponents} depend only on $c_{\beta'}$ such that $\beta' + g_{n_i} \prec \beta$. Thus, for the proof of this proposition, we only have to consider the last sum on the right hand side of~\eqref{piminuscomponents} which consists of terms of the form
\begin{equation}
\Pi_{x\beta_1}\cdots\Pi_{x\beta_m}\nabla\Pi_{x\beta_{m+1}} \sum_\gamma ((D^{(\0)})^m)_{\beta_{m+2}}^\gamma c_\gamma \, ,
\label{eq:sumterm}
\end{equation}
such that $\beta_1 + \dots + \beta_{m+2}=\beta$ and $m \geq 0$. As in the proof of~\cref{lem:tri1}, we observe that
\begin{equation}
((D^{(\0)})^{m})^{\gamma'}_{\beta'} \neq 0 \implies  |\beta'|_\prec = |\gamma'|_\prec\, .
\end{equation}
It follows that if $\Pi^-_{x\beta}$ contains $c_{\beta'}$, then there must be a term of the form~\eqref{eq:sumterm} such that $|\beta_{m+2}|_{\prec}=|\beta'|_{\prec}$ and $\beta_1+ \dots + \beta_{m+1}+ \beta_{m+2}=\beta$. Consider first the case in which $\beta$ has no polynomial component. Then, by the additivity of the ordering~\eqref{order}, we have
\begin{equation}
|\beta_1|_\prec + \dots + |\beta_{m+1}|_{\prec} + |\beta'|_{\prec}=|\beta|_{\prec}\, . 
\end{equation} 
Since none of $\beta_1,\dots,\beta_{m+1}$ can contain a polynomial component and $m \geq 0$, we must have that $|\beta|_{\prec} \geq |\beta'|_{\prec} +1 $. Furthermore, since $\lambda \in (0,1/2)$, it follows that $\beta' + g_{\mathbf{n}^i} \prec \beta $, for all $1 \leq i \leq d$. 

Consider now the case in which $|\beta|_p \geq 2$. Again, we have
\begin{equation}
|\beta_1|_\prec + \dots + |\beta_{m+1}|_\prec + |\beta'|_{\prec}=|\beta|_{\prec} \, .
\end{equation}
We already know that $\beta'$ has no polynomial component. It follows that
\begin{equation}
|\beta_1|_\prec + \dots + |\beta_{m+1}|_\prec \geq  \lambda |\beta|_p  \geq 2 \lambda\, .
\end{equation} 
It thus follows that 
\begin{equation}
|\beta|_{\prec} 
\geq |\beta'|_{\prec} + 2 \lambda 
\geq |\beta' + g_{\mathbf{n}^i}|_{\prec}  + \lambda \, ,
\end{equation}
for all $1 \leq i \leq d$. 

We are now left to treat the final case, $|\beta|_p=1$.  We first treat the case in which  $\beta_1,\dots,\beta_{m+1}$ are not all purely polynomial. In this case,  we must have
\[
|\beta_1|_\prec + \dots + |\beta_{m+1}|_\prec \geq 1 + \lambda \, .
\]
It follows that $\beta'+ g_{\mathbf{n}^i} \prec \beta$ for all $1 \leq i \leq d$. If $\beta_1, \dots,\beta_{m+1}$ are purely polynomial, since $|\beta|_p=1$, we must have $m=0$ and $\beta_1 = g_{\mathbf{n}^i}$ for some $1 \leq i \leq d$. Thus, $\beta_{m+2}=\beta_2= \beta-g_{\mathbf{n}^i}$. Furthermore, since $m=0$, we must have $\beta_{2}=\beta'$ and so $\beta' + g_{\mathbf{n}^i}=\beta$. This completes the proof of (i).

\medskip

We turn to (ii). 
Recall from \eqref{coeffGamma} that $(\G)_\beta^\gamma$ is a linear combination of terms of the form 
\begin{equation}\label{rg01}
    \pi_{\beta_1}^{(\n_1)}\cdots\pi_{\beta_j}^{(\n_j)} (D^{(\n_1)}\cdots D^{(\n_j)})_{\beta_{j+1}}^\gamma, 
\end{equation}
where $j\geq0$, $\n_1,\dots,\n_j\in\N_0^{1+d}$, $\beta_1+\cdots+\beta_{j+1}=\beta$ and $|\beta_i|>|\n_i|$ for $i=1,\dots,j$. 
Clearly, $\beta_1,\dots,\beta_j\preceq\beta$, which establishes \eqref{tri2}. 
For \eqref{tri1} it is enough to argue that $\beta_{j+1}\neq0$. 
By assumption we have $\sum_\ell\gamma_b(\ell)>0$, and by \eqref{D0props} and \eqref{Dnprops} we learn that if \eqref{rg01} is not vanishing then $\sum_\ell\gamma_b(\ell)=\sum_\ell\beta_{j+1,b}(\ell)$, hence $\beta_{j+1}\neq0$.
We turn to \eqref{tri3} and note that also $(\G-\id)_\beta^\gamma$ is a linear combination of terms of the form \eqref{rg01}, with the difference to above that here $j$ is restricted to $j\geq1$. 
We observe that \eqref{Dnprops} implies for $\n\neq\0$ and for all $\beta',\gamma'$
\begin{equation}\label{triDn}
    (D^{(\n)})_{\beta'}^{\gamma'}\neq0\quad\implies\quad \length{\beta'}=\length{\gamma'}-\lambda|\n|.
\end{equation}
Together with \eqref{triD0} we obtain 
\begin{equation}
    (D^{(\n_1)}\cdots D^{(\n_j)})_{\beta_{j+1}}^\gamma\neq0\quad\implies\quad 
    \length{\beta_{j+1}}=\length{\gamma}-\lambda(|\n_1|+\cdots+|\n_j|).
\end{equation}
Hence if \eqref{rg01} is non-vanishing, then 
\begin{equation}\label{rg02}
\length{\beta}=\length{\beta_1}+\cdots+\length{\beta_j}+\length{\gamma}-\lambda(|\n_1|+\cdots+|\n_j|).
\end{equation}
From \eqref{homogeneity} and $1\geq\alpha\lambda$ we obtain $\length{\beta_i}\geq \lambda|\beta_i|$, 
which together with $|\beta_i|>|\n_i|$ yields $\length{\beta_i}>\lambda|\n_i|$, 
and hence $\length{\beta}>\length{\gamma}$.
For the second item of \eqref{tri3} we first note that by \eqref{D0props} and \eqref{Dnprops} we have 
for all $\n$
\begin{equation}\label{rg03}
    (D^{(\n)})_{\beta'}^{\gamma'}\neq0\quad\implies\quad |\beta'|=|\gamma'|+\alpha-|\n|.
\end{equation}
Hence, similarly as above we obtain that if \eqref{rg01} is non-vanishing, then 
\begin{equation}\label{rg06}
    |\beta|=|\beta_1|+\cdots+|\beta_j|+|\gamma|-(|\n_1|+\cdots+|\n_j|).
\end{equation}
By $|\beta_i|>|\n_i|$ for $i=1,\dots,j$ and $j\geq1$ we obtain $|\beta|>|\gamma|$, 
which finishes the proof of \eqref{tri3}. 
\eqref{tri4} is an immediate consequence of \eqref{tri3}.

\medskip

We come to (iii) and note that by \eqref{dGamma}, $(\d\G)_\beta^\gamma$ is a linear combination of terms of the form 
$$
\d\pi_{\beta_1}^{(\n)} \sum_{\beta'} (\G)_{\beta_2}^{\beta'} (D^{(\n)})_{\beta'}^\gamma, 
$$
where $|\n|\leq2$ and $\beta_1+\beta_2=\beta$. 
Since only populated $\beta_1$ are relevant we have $\beta_1\neq0$, 
in particular $\length{\beta_1}>0$ and thus $\beta_2\prec\beta$. 
If $\beta_2$ were $0$, then $\beta'=0$ by the already established \eqref{tri3}.
However, $(D^{(\n)})_0^\gamma\neq0$ implies $\gamma=g_\n$ by \eqref{D0props} and \eqref{Dnprops}, 
which contradicts the assumption $\sum_\ell\gamma_b(\ell)>0$. 
Hence $\beta_2\neq0$ and therefore $\beta_1\prec\beta$, 
which finishes the proof of \eqref{tri5}.
For \eqref{tri6}, we appeal to \eqref{tri3}, \eqref{triD0} and \eqref{triDn} to obtain 
\begin{equation}
    \length{\beta}=\length{\beta_1}+\length{\beta_2}\geq \length{\beta_1} +\length{\gamma}-\lambda|\n|.
\end{equation}
By \eqref{choice_dpi}, $\beta_1$ is populated and not purely polynomial, which implies $\length{\beta_1}\geq1$. 
Together with $|\n|\leq2$ we obtain $\length{\beta}\geq\length{\gamma}+1-2\lambda>\length{\gamma}$. 
Similarly, by \eqref{tri3} and \eqref{rg03} we obtain 
\begin{equation}\label{rg10}
    |\beta|=|\beta_1|+|\beta_2|-\alpha\geq|\beta_1|+|\gamma|-|\n|.
\end{equation}
Since $|\beta_1|\geq\alpha$ and $|\n|\leq2$, this finishes the proof of \eqref{tri6}. 

\medskip

We finally provide the argument for \eqref{tri7}. 
In view of \eqref{rg10} it remains to argue that $|\gamma|-|\n|\geq\alpha$. 
If $\n=\0$, this is clear.
If $\n\neq\0$, note that $(D^{(\n)})_{\beta'}^\gamma\neq0$ implies $\gamma_p(\n)\geq1$ by \eqref{Dncomp}, 
hence in particular $|\gamma|_p\geq|\n|$.
As $\gamma$ is populated, we have $|\gamma| = \alpha\sum_\l\gamma_b(\l) +|\gamma|_p$, 
which implies the desired $|\gamma|\geq\alpha+|\n|$ since 
$\gamma$ is not purely polynomial and hence $\sum_\l\gamma_b(\l)\geq1$. 
\end{proof}


\subsection{BPHZ-choice of the renormalisation constant}
\label{sec:bphz}
In this section, we will explain how we will choose the renormalisation constants $c_\beta$. 
Note that in~\cref{sec:counterterm} we have already considerably reduced the structure of the counterterm. 
We thus need to argue that at the level of $\Pi_{x\beta}^{-}$ a number of terms should require no renormalisation due to exactly the same symmetries we have used to derive a reduced form of the counterterm. 
This leads us the to the following proposition.

\begin{proposition}
Assume that~\cref{ass} is satisfied. Then, for all $x,y,h \in \R^{1+d}$, the following properties hold true
\begin{enumerate}
\item $\Pi_{x\beta}[\xi (\cdot +h)](y)= \Pi_{x+h \beta}[\xi ](y+h)$,
\newline
 $\Pi_{x\beta}^-[\xi (\cdot +h)](y)= \Pi_{x+h \beta}^-[\xi ](y+h)$, ~\label{symmetry1}
\item $\Pi_{x\beta}[-\xi(R \cdot)](y) =  (-1)^{|\beta|_p}\Pi_{R x \beta}[\xi](R y) $, 
\newline $\Pi_{x\beta}^-[-\xi(R \cdot)](y) =  (-1)^{1+ |\beta|_p}\Pi_{R x \beta}^-[\xi](R y) $, and \label{symmetry2}
\item $\Pi_{x\beta}[-\xi](y) =  (-1)^{\sum \beta_b(\ell)}\Pi_{x \beta}[\xi](y) $, 
\newline $\Pi_{x\beta}^-[-\xi](y) =  (-1)^{1+\sum \beta_b(\ell)}\Pi_{x \beta}^-[\xi](y) $. \label{symmetry3}
\end{enumerate}
Furthermore, let $\beta'$ be such that $\beta'_p(\mathbf{n})=0$ for all $\n\neq\0$ and denote by $\beta^{i}=\beta' + g_{\mathbf{n}^i}$ for $1 \leq i \leq d$ where $\mathbf{n}^i$ is the unit vector of $\mathbb{N}_0^{1+d}$ in the $i$-th direction. Then for $j=1,\dots,d$
\begin{enumerate}[resume]
\item $\sum_{i =1}^d\bar{O}_{ij}\Pi_{x\beta^i}[\bar{O}^T\xi(O\cdot)](y) =  \Pi_{O x \beta^j}[\xi](O y)$,\newline
$\sum_{i=1}^d\bar{O}_{ij} \Pi_{x\beta^i}^-[\bar{O}^T \xi(O\cdot)](y) = \bar{O}^T\Pi_{Ox \beta^j}^-[\xi](Oy)$. \label{symmetry4}
\end{enumerate}
\label{prop:symmetries}
\end{proposition}
\begin{proof}
We will provide a formal proof of these identities by using the power series expansion for the solution. 
This proof can easily be made fully rigorous by using an induction argument on the hierarchy of equations given by \eqref{hierarchy}. However, for the sake of brevity and clarity of the exposition, we will avoid doing this here.

For the symmetry in~\cref{symmetry1}, we note, as before, that if the tuple $[u,a,b,p,\xi]$ is a solution of~\eqref{spde}, then so is $[u(\cdot +h),a,b,p(\cdot+h),\xi(\cdot+h)]$. It follows from the power series expansion~\eqref{ansatz} and comparing coefficients of $u(\cdot +h)$ and $u$ that we must have $\Pi_{x\beta}[\xi (\cdot +h)](y)= \Pi_{x+h \beta}[\xi ](y+h)$ for all $x,y,h \in \R^{1+d}$. The equality at the level of $\Pi_{x\beta}^-$ follows by simply using~\eqref{piminuscomponents}.

Similarly, for~\cref{symmetry2}, we note that if the tuple $[u,a,b,p,\xi]$ is a solution of~\eqref{spde} then so is $[u(R\cdot), a ,b , p(R \cdot), -\xi(R\cdot)]$. Using the expansion~\eqref{ansatz}, we have
\begin{align}
&u(R y)-u(R x) \\ &=\sum_{\beta}\Pi_{x\beta}[- \xi (R \cdot)](y)\mathsf{z}^\beta[a(\cdot + u(R x)) ,b(\cdot + u(Rx)),p(R \cdot +  Rx)-p( Rx)] \\
&= \sum_{\beta}(-1)^{\sum|\n|_{\geq 1}\beta_p(\n)}\Pi_{x\beta}[- \xi (R \cdot)](y)\mathsf{z}^\beta[a(\cdot + u(R x)) ,b(\cdot + u(Rx)),p(\cdot +  Rx)-p( Rx)]  \\
&= \sum_{\beta}(-1)^{|\beta|_p}\Pi_{x\beta}[- \xi (R \cdot)](y)\mathsf{z}^\beta[a(\cdot + u(R x)) ,b(\cdot + u(Rx)),p(\cdot +  Rx)-p( Rx)] 
\end{align}
where $|\mathbf{n}|_{\geq 1}:=\sum_{i =1}^d \mathbf{n}_i$.
The symmetry then follows by comparing coefficients. We now apply~\eqref{piminuscomponents} to compute
\begin{align}
&\Pi_{x\beta}^-[- \xi (R \cdot)](y)\\
&= \sum_{\substack{k\\e_k+\beta_1+\cdots+\beta_{k+1}=\beta}}
    (-1)^{1+ \sum_{i=1}^{k+1} |\beta_i|_{p}} 
    \Pi_{Rx\beta_1}[\xi](Ry)\cdots\Pi_{Rx\beta_k}[\xi](Ry) 
    (\nabla\Delta\Pi_{R x\beta_{k+1}}[\xi]) (Ry) \\
&\, + \sum_{\substack{\ell \\ f_\ell+\beta_1+\cdots+\beta_\ell=\beta}} 
    (-1)^{1+ \sum_{i=1}^{\ell} |\beta_i|_{p}} \Pi_{R x\beta_1}[\xi]\cdots\Pi_{R x\beta_\ell}[\xi] \xi_\tau \\
&\, - \sum_{\substack{m \\ \beta_1+\cdots+\beta_{m+2}=\beta}} 
    \tfrac{1}{m!}
    (-1)^{1+ \sum_{i=1}^{m+1} |\beta_i|_{p}}\Pi_{Rx\beta_1}[\xi](Ry)\cdots\Pi_{Rx\beta_m}[\xi](Ry)(\nabla\Pi_{R x\beta_{m+1}}[\xi])(Ry) \\[-2ex] 
    &\hspace{70ex}\times
    ((D^{(\0)})^m c)_{\beta_{m+2}} \\
    &= (-1)^{1+ |\beta|_p} \Pi_{Rx \beta}^-[\xi](Ry)\, ,
\end{align}
where we have used~\cref{ass} and the fact that $c$ depends only on the law of $\xi$. This completes the proof of~\cref{symmetry2}. 

We now move on to~\cref{symmetry3} by noting that if $[u,a,b,p,\xi]$ is a solution of~\eqref{spde}, then so is $[u,a,-b,p,-\xi]$. Once again, appealing to the expansion~\eqref{ansatz}, we have
\begin{align}
u(y)-u(x) 
&= \sum_{\beta}\Pi_{x\beta}[- \xi ](y)\mathsf{z}^\beta[a(\cdot+u(x)) ,-b(\cdot+u(x)),p(\cdot +  x)-p( x)] \\
&= \sum_{\beta}(-1)^{\sum\beta_b(\ell)}\Pi_{x\beta}[- \xi ](y)\mathsf{z}^\beta[a(\cdot+u(x)) ,b(\cdot+u(x)),p(\cdot +  x)-p( x)] \, .
\end{align}
Comparing coefficients with the original power series,~\cref{symmetry3} follows at the the level of $\Pi_{x\beta}$. The proof for $\Pi_{x\beta}^-$ follows in an identical manner to that of~\cref{symmetry2} by using~\eqref{piminuscomponents}.

Finally, for~\cref{symmetry4} we note that if $[u,a,b,p,\xi]$ is a solution of~\eqref{spde}, then so  is \newline $[u(O\cdot), a ,b , p(O\cdot), \bar{O}^T\xi(O\cdot)]$. Note that of $\beta^i$ is of the form described in the statement of the proposition, then
\begin{align}
&\mathsf{z}^{\beta^i}[a(\cdot + u(O\cdot)),b(\cdot + u(O\cdot)),p(O\cdot +O x) -p(Ox) ] \\
&=\sum_{j}\bar{O}_{ij} \mathsf{z}^{\beta^j}[a(\cdot + u(O\cdot)),b(\cdot + u(O\cdot)),p(\cdot +O x) -p(Ox) ] \, .
\end{align}
Comparing coefficients as before, we obtain
\begin{equation}
\sum_{i=1}^d  \bar{O}_{ij} \Pi_{x\beta^i}[\bar{O}^T \xi(O\cdot)](y) = \Pi_{O x\beta^j}[\xi](Oy) \, .\label{orth1}
\end{equation}
Before we move on to $\Pi_{x\beta}^-$, we note that an essentially similar argument to the one above can be used to show that, if $\beta$ has no polynomial component, then
\begin{equation}
\Pi_{x\beta}[\bar{O}^T \xi(O\cdot)](y) = \Pi_{Ox \beta}[\xi](Oy)\, .
\label{orth2}
\end{equation}
For $\Pi_{x \beta^i}^-$, we consider the three terms on the right hand side of~\eqref{piminuscomponents} separately. if we attempt to compute $\sum_{i}\bar{O}_{ij}\Pi_{x\beta^i}^{-}[\bar{O}^T \xi(O\cdot)](y)$, the first term on the right hand side of~\eqref{piminuscomponents} is of the following form
\begin{align}
\sum_{i=1}^d\bar{O}_{ij}\Pi_{x\beta_1}[\bar{O}^T \xi(O\cdot)](y)\cdots\Pi_{x\beta_k}[\bar{O}^T \xi(O\cdot)](y)(\nabla\Delta\Pi_{ x\beta_{k+1}}[\bar{O}^T \xi(O\cdot)]) (y)
\end{align}
where one of $\beta_1, \dots ,\beta_{k+1}$  is of the form $\bar \beta+ g_{\mathbf{n}^i}$, for some $\bar \beta$ with no polynomial component, with all the other multiindices having no polynomial component. In the case that one of $\beta_1,\dots,\beta_k$ (say $\beta_1$) is of the form $\bar \beta+ g_{\mathbf{n}^i}$, we can apply~\eqref{orth1} and~\eqref{orth2} to obtain
\begin{align}
&\sum_{i=1}^d\bar{O}_{ij}\Pi_{x\beta_1}[\bar{O}^T \xi(O\cdot)](y)\cdots\Pi_{x\beta_k}[\bar{O}^T \xi(O\cdot)](y)(\nabla\Delta\Pi_{ x\beta_{k+1}}[\bar{O}^T \xi(O\cdot)]) (y) \\
&= \Pi_{Ox\bar{\beta} + g_{\mathbf{n}^j}}[\xi](Oy)\cdots\Pi_{Ox\beta_k}[\xi](Oy)(\bar{O}^T\nabla\Delta\Pi_{ Ox\beta_{k+1}}[\xi]) (Oy) \, .
\end{align}
Similarly, if $\beta_{k+1}=\bar \beta+ g_{\mathbf{n}^i}$, we can apply similar arguments to obtain
\begin{align}
&\sum_{i=1}^d\bar{O}_{ij}\Pi_{x\beta_1}[\bar{O}^T \xi(O\cdot)](y)\cdots\Pi_{x\beta_k}[\bar{O}^T \xi(O\cdot)](y)(\nabla\Delta\Pi_{ x\beta_{k+1}}[\bar{O}^T \xi(O\cdot)]) (y) \\
&= \Pi_{Ox \beta_1}[\xi](Oy)\cdots\Pi_{Ox\beta_k}[\xi](Oy)(\bar{O}^T \nabla\Delta\Pi_{ Ox\bar{\beta} + g_{\mathbf{n}^j}}[\xi]) (Oy) \, .
\end{align}
Applying similar arguments to the other two terms on the right hand side of~\eqref{piminuscomponents} and using the fact that $c$ only depends on the law of $\xi$ (and~\cref{ass}), we obtain
\begin{equation}
\sum_{i=1}^d\bar{O}_{ij} \Pi_{x\beta^i}^-[\bar{O}^T \xi(O\cdot)](y) = \bar{O}^T\Pi_{Ox \beta^j}^-[\xi](Oy)\, ,
\end{equation}
thus completing the proof.
\end{proof}
We are now finally in a position to choose the constants $c_\beta$. We want to choose the constants such that for all $|\beta|<3$, the following large scale average vanishes,
\begin{equation}
\lim_{t \to \infty}\E [\Pi_{x\beta t}^-(x)]=0 \, .
\label{eq:BPHZchoice}
\end{equation}
Using the result of~\cref{prop:symmetries} tells us that  $\E [\Pi_{x \beta t}^-(x)] $ is independent of $x$ and is non-zero only if
\begin{equation}
 1+ |\beta|_p \quad\textnormal{and}\quad 1 + \sum_{\ell}\beta_b(\ell) \quad\textnormal{are even} \, .
 \end{equation} 
 This tells us that we only need to concern ourselves with $\beta=\beta' + g_{\mathbf{n}^i}$ for $1\leq i\leq d$ where $\beta'$ has no polynomial component and $\sum_{\ell}\beta'_b(\ell)$ is odd, thus \eqref{pop4} is satisfied. 

Using the triangularity established in~\cref{lem:tri1}, we can choose the constants $c_\beta$ in a manner that is self-consistent with respect to the ordering $|\cdot|_{\prec}$ defined in~\eqref{order}. 
As described in \cref{sec:inductive_proof}, 
we will perform induction on the ordering $\prec$ assuming that, for a given $\Pi_{x\beta}^-$ we have constructed and estimated $\Pi_{x\beta'}$ for $\beta' \prec \beta$ and $c_{\beta'}$ for $\beta' + g_{\mathbf{n}^i}\prec\beta$ (for all $i=1,\dots,d$). Then, for such a $\Pi_{x\beta}^-$ which depends on some $c_{\beta'}$, by ~\cref{lem:tri1}~(i), either $\beta' + g_{\mathbf{n}^i} \prec \beta$ for all $1 \leq i \leq d$ in which case $c_{\beta'}$ has already been chosen or $\beta' + g_{\mathbf{n}^i} = \beta$. In the latter case, we note that we can rewrite $\Pi_{x\beta' + g_{\mathbf{n}^i}}^-$ from~\eqref{piminuscomponents} componentwise as follows
\begin{equation}
\Pi_{x \beta' + g_{\mathbf{n}^i}}^- = \tilde{\Pi}_{x \beta' + g_{\mathbf{n}^i}}^- - c_{\beta'} \mathbf{n}^i \, ,
\label{eq:Pitilde} 
\end{equation}
where $\tilde{\Pi}_{x\beta' + g_{\mathbf{n}^i}}^-$ just represents the remaining terms on the right hand side of~\eqref{piminuscomponents}. It follows from~\cref{lem:tri1} (i) that $\tilde{\Pi}_{x\beta' + g_{\mathbf{n}^i}}^-$ only depends on $c_{\bar{\beta}}$ for $\bar{\beta} + g_{\mathbf{n}^j} \preceq \beta' + g_{\mathbf{n}^i} $, for all $1 \leq j \leq d$, which have all been chosen already. We then make the following choice
\begin{equation}
 c_{\beta'} =  \lim_{t \to \infty}\E [\tilde{\Pi}^-_{x\beta' + g_{\mathbf{n}^i} t}(x)]_i \, .
 \label{eq:BPHZ2}
\end{equation} 
Note that due to~\cref{symmetry4} from~\cref{prop:symmetries}, we have that
\begin{equation}
\E [\tilde{\Pi}^-_{x\beta' + g_{\mathbf{n}^i} t}(x)]_j =0 \, ,
\label{consequence1}
\end{equation}
for $i \neq j$ and 
\begin{equation}
\E [\tilde{\Pi}^-_{x\beta' + g_{\mathbf{n}^j} t}(x)]_j = \E [\tilde{\Pi}^-_{x\beta' + g_{\mathbf{n}^i} t}(x)]_i \, ,
\label{consequence2}
\end{equation}
for all $1 \leq i,j \leq d$.  For fixed $1 \leq i< j \leq d$, ~\eqref{consequence1} follows by choosing 
\begin{align}
(\bar{O})_{m n}=
\begin{cases}
\delta_{mn} & \textnormal{for } m \in \{1,\dots,d\}\setminus\{i,j\}, \\
\delta_{jn} & \textnormal{for } m =i, \\
-\delta_{in} & \textnormal{for } m =j, 
\end{cases}
\end{align}
while \eqref{consequence2} follows by choosing
\begin{align}
(\bar{O})_{m n}=
\begin{cases}
\delta_{mn} & \textnormal{for } m \in \{1,\dots,d\}\setminus\{i,j\}, \\
\delta_{jn} & \textnormal{for } m =i, \\
\delta_{in} & \textnormal{for } m =j. 
\end{cases}
\end{align}
Thus, the choice~\eqref{eq:BPHZ2} is consistent and ensures that the BPHZ renormalisation condition~\eqref{eq:BPHZchoice} is satisfied for all multiindices.

The choice of renormalisation we have made in~\eqref{eq:BPHZchoice} by controlling this large scale average of $\Pi_{x\beta}^-$ in fact lets us control $\E \Pi_{x\beta t}^- (y)$ as we shall establish in the following proposition.
\begin{proposition}\label{prop:expectation}
Assume $|\beta|<3$ and that$~\eqref{estPiminus}_{\prec \beta}$ and$~\eqref{estGamma}_\beta^\gamma$ hold true for all $\gamma$ not purely polynomial. Then,
\begin{equation}
\int_T^\infty \dx{t} \, \left| \frac{\d}{\d t}\E \Pi_{x\beta t}^-(y)\right| 
\lesssim (\sqrt[8]{T})^{\alpha-3} 
(\sqrt[8]{T} + \abss{x-y} )^{|\beta|-\alpha} \, .
\label{eq:timederivative}
\end{equation}
Furthermore, by~\eqref{eq:BPHZchoice}, 
\begin{equation}
|\E \Pi_{x\beta t}^-(y)| \lesssim (\sqrt[8]{t})^{\alpha-3}
(\sqrt[8]{t} + \abss{x-y} )^{|\beta|-\alpha} \, .
\label{eq:expbound}
\end{equation}
\end{proposition}
\begin{proof}
We note that
\begin{align}
\frac{\d}{\d t}\E \Pi_{x\beta t}^-(y) 
&= \frac{\d}{\d t}  \E \int_{\R^{1+d}} \dx{z} \, \psi_{t-s}(y-z)  \Pi_{x\beta s}^-(z) \\
&=  \int_{\R^{1+d}} \dx{z} \, (LL^* \psi_{t-s})(y-z)  \E (\Gamma^*_{xz}\Pi_{z s}^-)_\beta(z)  \, ,
\end{align}
where we have used the definition of $\psi_t$ along with the fact that the remainder that shows up in~\eqref{recenterPi-} is a random polynomial of degree lesser than or equal to $|\beta|-3$. We now simply apply~\cref{prop:symmetries} along with the translation invariance in law of the ensemble $\xi$ from~\cref{ass} to rewrite the above expression as
\begin{align}
\frac{\d}{\d t}\E \Pi_{x\beta t}^-(y) &=   \int_{\R^{1+d}} \dx{z} \, (LL^* \psi_{t-s})(y-z)  \E ((\Gamma^*_{xz}- \mathrm{id})\Pi_{z s}^-)_\beta(z)  \, .
\label{eq:inter}
\end{align}
From the triangularity of $\Gamma^*$ established in~\cref{lem:tri1} (see~\eqref{tri3}), we know that $((\Gamma^*_{xz}- \mathrm{id})\Pi_{z s}^-)_\beta$ depends on $\Pi_{z\beta'}^-$ only for $\beta' \prec \beta$.   Additionally, from~\eqref{eq:piminusttilde}, we know that $\Pi_{z}^- \in \widetilde{\T}^*$ from which it follows that $((\Gamma^*_{xz}- \mathrm{id})\Pi_{z s}^-)_\beta$ contains only terms of the form $(\Gamma^*_{xz}- \mathrm{id})_\beta^{\beta'}$ for $\beta'$ not purely polynomial.

We now choose $s=\frac{t}{2}$ in~\eqref{eq:inter}. Note that, applying the Cauchy--Schwarz inequality, we have the following estimate
\begin{align}
\big|\E \big[(\Gamma^*_{xz}- \mathrm{id})_{\beta}^{\beta'}\Pi_{z \beta's}^-(z)\big] \big|
&\lesssim \abss{x-z}^{|\beta|-|\beta'|}(\sqrt[8]{t})^{|\beta'|-3} 
\lesssim  (\sqrt[8]{t})^{\alpha-3} (\sqrt[8]{t}+ \abss{x-z})^{|\beta|-\alpha}\, ,
\end{align}
where we have used the fact that $|\beta'|\geq\alpha$ and $|\beta|-|\beta'|\geq0$ by the triangularity \eqref{tri3} of $\G$. 
Integrating in $z$ and applying the moment bound~\eqref{eq:kerbound}, we have
\begin{equation}
\left|\frac{\d}{\d t}\E \Pi_{x\beta t}^-(y) \right| \lesssim (\sqrt[8]{t})^{\alpha-11}(\sqrt[8]{t}+ \abss{x-y})^{|\beta|-\alpha}\, .
\end{equation}
Integrating in $t$ from $T$ to $\infty$, we obtain by $|\beta|<3$ the bound
\begin{align}
\int_T^\infty \dx{t} \, \left| \frac{\d}{\d t}\E \Pi_{x\beta t}^-(y)\right| 
\lesssim (\sqrt[8]{T})^{\alpha-3}
(\sqrt[8]{T} + \abss{x-y})^{|\beta|-\alpha}\, , 
\end{align}
which completes the proof of~\eqref{eq:timederivative}. 
We now have using~\eqref{eq:BPHZchoice}
\begin{equation}
|\E \Pi_{x \beta T}^-(x)| 
\leq  \int_T^\infty \dx{t} \, \left| \frac{\d}{\d t}\E \Pi_{x\beta t}^-(x)\right| 
\lesssim (\sqrt[8]{T})^{|\beta|-3} \, .
\end{equation}
Using the above bound and $\Gamma^*$, we have
\begin{align}
|\E \Pi_{x \beta t}^-(y)| &\leq |\E \Pi_{y \beta t}^-(y)| + |\E ((\Gamma^*_{xy}- \mathrm{id})\Pi_{y t}^-)_\beta(y)|  
\lesssim (\sqrt[8]{t})^{|\beta|-3} + |\E ((\Gamma^*_{xy}- \mathrm{id})\Pi_{y t}^-)_\beta(y)| .\label{eq:expbound1}
\end{align}
For the last right hand side term, we use the triangularity of $\Gamma^*_{xy}-\mathrm{id}$ as before to estimate the resulting terms with $\beta'\prec \beta$ as follows
\begin{align}
\big|\E \big[(\Gamma^*_{xy}- \mathrm{id})_{\beta}^{\beta'}\Pi_{y \beta't}^-(y)\big] \big|
\lesssim \abss{x-y}^{|\beta|-|\beta'|}(\sqrt[8]{t})^{|\beta'|-3} 
\lesssim (\sqrt[8]{t})^{\alpha-3}
(\sqrt[8]{t} + \abss{x-y})^{|\beta|-\alpha}\, ,
\end{align}
where we have used the fact that $|\beta'| < |\beta|$ (from~\eqref{tri3}) and $|\beta'|\geq \alpha$. 
Putting it together with the previous estimate, 
we obtain \eqref{eq:expbound}.
\end{proof}

\begin{remark}[Alternative choice of $\Pi^-$]
The careful reader may have noticed that the renormalisation constant 
$c_{\beta}$ that we choose for the multiindices $\beta + g_{\n^i}$, 
$1 \leq i \leq d$ vanishes after the application of the divergence operator. 
Since we are ultimately interested in estimates on $\Pi_{x\beta}$, 
which follow by integration from a corresponding estimate on $\nabla\cdot\Pi^-_{x\beta}$, 
this may lead one to the conclusion that the constant is not necessary. 
However, this is not the case: 
counterterms chosen within the induction at some point 
will play an important role for some ``bigger'' multiindices that come up later in the induction. 
As an example, consider the multiindex $2f_1+g_{(0,2)}$ 
where
\begin{align}
\nabla\cdot\Pi^-_{x 2f_1+g_{(0,2)}}
&= \nabla\cdot \big(\Pi_{x f_1+g_{(0,2)}} \xi_\tau
- \nabla (\cdot-x)_1^2 c_{2f_1}
- \nabla\Pi_{x f_1+g_{(0,2)}} c_{f_1} \big) \\
&= \nabla\cdot \big(\Pi_{x f_1+g_{(0,2)}} \xi_\tau \big)
- c_{2f_1} \, ,
\label{eq:counterexample}
\end{align}
where we have used that $c_{f_1}=0$ by~\cref{prop:symmetries}. 
From this expression, the diverging lower bound on $c_{2f_1}$ (in $d=1$) from~\cref{prop:scaling_c}, and the estimate~\eqref{estPiminus} (which in turn implies an estimate on its divergence), it is clear that $c_{2f_1}$ cannot be chosen to be $0$.

Alternatively, we could have made a different, but equally valid, choice for $\Pi_{x}^-$, say $\check{\Pi}_{x}^-$, by including the divergence operator $\nabla \cdot$ in its definition. This would amount to solving the hierarchy of linear PDEs given by
\begin{equation}
L \check{\Pi}_{x\beta}= \check{\Pi}_{x\beta}^{-} \, .
\end{equation}
We note that in this setting we would have to perform the BPHZ renormalisation in a different way. Repeating the arguments of Subsection~\ref{sec:bphz} leading to \eqref{eq:BPHZchoice}, it is easy to check that for all $\beta$ such that $\check{c}_{\beta}$ is populated, we would make the choice
\begin{equation}
\check{c}_{\beta} = \frac12 \lim_{t \to \infty}\E [\tilde{\check\Pi}^-_{x\beta + g_{2\mathbf{n}^i} t}(x)] 
\end{equation}
for some (and indeed all) $1 \leq i \leq d$. Here, $\tilde{\check\Pi}_{x}^-$ is defined in the natural manner as before.\footnote{
Note that $\beta + 2g_{\mathbf{n}^i} $ is not populated} 
Note that, a priori, this gives us a different choice of the constants $\{\check{c}_{\beta}\}_{\beta}$ corresponding to a different functional form of the counterterm $\check{h}(u(\cdot))$. However, if our construction of the models is consistent, $h,\check{h}$ should coincide with each other. We already know from~\cref{coro} that the two models $(\Pi,\Gamma)$ and $(\check{\Pi},\check{\Gamma})$ (defined in the sense of~\cite[Definition 1.1]{Tem23}) agree with each other. By induction, we can then show that the families of constants $\{\check{c}_{\beta}\}_{\beta}$ and $\{c_{\beta}\}_{\beta}$  are the same. 
Indeed, let us assume that, for any given $\beta$, we know that $c_{\beta'}=\check{c}_{\beta'}$ for all $\beta' \prec \beta$. Then, if we look at the model equation for the multiindex $\beta + g_{2\mathbf{n}^i}$,  for any $1 \leq i \leq d$, for both $\Pi_{x}$ and $\check{\Pi}_x$ and subtract them, we have
\begin{equation}
 0 =  \nabla \cdot \tilde{{\Pi}}_{x\beta+ g_{2\mathbf{n}^i}}- \tilde{\check{\Pi}}_{x\beta + g_{2\mathbf{n}^i}} + 2c_{\beta}\mathbf{n}^i - 2\check{c}_{\beta} \mathbf{n}^i \, .
\end{equation} 
We already know from~\cref{lem:tri1} (i) that $\nabla \cdot \tilde{\Pi}_{x\beta+ g_{2\mathbf{n}^i}}$ depends on $c_{\beta'}$ for $\beta' + g_{\mathbf{n}^j} \prec\beta + g_{2\mathbf{n}^i} $ ($\beta' + g_{\mathbf{n}^i} =\beta + g_{2\mathbf{n}^i} $ is clearly not possible) for all $1 \leq j \leq d$ from which it follows that $\beta' \prec \beta$. The same holds true for $\tilde{\check{\Pi}}_{x\beta + g_{2\mathbf{n}^i}}$ since it has the same dependence on $\{c_\beta\}_\beta$ as 
$\nabla \cdot \tilde{{\Pi}}_{x\beta+ g_{2\mathbf{n}^i}}$. It follows then that $c_\beta=\check{c}_\beta$. The base case can be checked in a similarly straightforward manner.
\label{rem:divcount}
\end{remark}

\subsection{Annealed Schauder theory}\label{sec:schauder}

\subsubsection{Integration of the model}
	In this subsection, we discuss the basic integration argument needed for our estimates, i.e.~we discuss how to solve~\eqref{eq:lspde-intro}.

		\begin{lemma}[Annealed Schauder estimate]
		\label{lem:int1}
			Let $d\geq 1$, $\gamma>0$, $\eta \in [\gamma,\infty)\setminus \Z$, $p < \infty$, and $x \in \R^{1+d}$ be given.	Assume that $f \in (\mathcal{S}'(\R^{1+d}))^{ d}$ 	is a random vector-valued tempered distribution which satisfies 
				\begin{equation}
					\E^{\frac1p} |f_t (y)|^p  \leq  (\sqrt[8]{t})^{\gamma -3} (\sqrt[8]{t} + \abss{x-y})^{\eta -\gamma} \, ,
					\label{eq:rhsweakbound}
				\end{equation} 
			for all $t >0$ and $y \in \R^{1+d}$. Then, there exists a unique random function $u$ satisfying
				\begin{equation}
					\sup_{y \in \R^{1+d}}\frac{1}{\abss{x-y}^\eta}\E^{\frac1p}|u(y)|^p < \infty
					\label{eq:lhsstrongbound}
				\end{equation}
			and, in the sense of distributions,
				\begin{equation}
					Lu = \nabla \cdot f \, .
					\label{eq:linPDE}
				\end{equation}
			Furthermore, the constant in the bound~\eqref{eq:lhsstrongbound} depends only on $\gamma$, $\eta$, and $d$.
		\end{lemma}
		\begin{proof}
			We first notice that we can formally represent the fundamental solution associated to $L$ as
			 	\begin{equation}
			 		\int_0^\infty \dx{t} (L^* \psi_t) \, ,
			 	\end{equation}
			where $L^*$ is the adjoint of $L$. We thus propose the following solution formula for $u$ 
				\begin{equation}
					u= \int_0^\infty \dx{t} (\mathrm{id} -\mathrm{T}_x^\eta) (L^* \nabla \cdot f_t) \, ,
					\label{eq:solutionformula}
				\end{equation}
			where the operator $\mathrm{T}_x^\eta$ projects an arbitrary smooth function onto its Taylor polynomial centered at $x$ of order $\leq \eta$. We will argue that $u$, defined in this manner, makes sense, satisfies~\eqref{eq:lhsstrongbound}, and is a distributional solution of~\eqref{eq:linPDE}. We will see that subtracting the Taylor polynomial is necessary in order for the expression~\eqref{eq:solutionformula} to make sense.

			Given the bound~\eqref{eq:rhsweakbound} on the right hand side $f$, we can obtain the following estimate
				\begin{equation}
					\E^{\frac1p}|\partial^{\n}f_t(y)|^{p} \lesssim  (\sqrt[8]{t})^{\gamma -3 - |\n|} (\sqrt[8]{t} + \abss{x-y})^{\eta -\gamma} \, .
					\label{eq:bounda}
				\end{equation}
			To see this, we use the semigroup property~\eqref{eq:semi}, along with the bound \eqref{eq:rhsweakbound}
				\begin{align}
					\E^{\frac1p}|\partial^{\n}f_t(y)|^{p} &\lesssim \int \dx{z}  |\partial^{\n}\psi_{\frac t2}(y-z)|\E^{\frac1p}|f_{\frac t2}(z)|^p \\
					&\lesssim (\sqrt[8]{t})^{\gamma -3} \int \dx{z}  |\partial^{\n}\psi_{\frac t2}(y-z)| (\sqrt[8]{t} + 
     \abss{x-z})^{\eta -\gamma} \, .
					\label{eq:boundb}
				\end{align}
			Applying the moment bound \eqref{eq:kerbound} gives us~\eqref{eq:bounda}. Given~\eqref{eq:bounda}, we can now estimate~\eqref{eq:solutionformula} by splitting it into a \emph{far-field} and~\emph{near-field} component. Before we do this, we note that the Taylor remainder $(\mathrm{id} -\mathrm{T}_x^\eta) (L^* \nabla \cdot f_t)(y)$ can be expressed as a linear combination of terms of the form $(y-x)^{\n}\partial^{\n} L^* \nabla \cdot f_t(z)$ for $|\n|>\eta$ where $z$ is some point between $y$ and $x$. Using an essentially identical argument to~\eqref{eq:bounda}, we can estimate such a term by
				\[
					\abss{x-y}^{|\n|}(\sqrt[8]{t})^{\gamma -8 - |\n|} (\sqrt[8]{t} + \abss{x-y})^{\eta -\gamma} \, .
				\]
			Thus, for $\sqrt[8]{t} \geq \abss{x-y}$, i.e.~the far-field component, we have
				\begin{align}
					\E^{\frac1p}\big| \int_{\abss{x-y}^8}^\infty \dx{t} (\mathrm{id} -\mathrm{T}_x^\eta) (L^* \nabla \cdot f_t)(y) \big|^p  
                    \lesssim  \int_{\abss{x-y}^8}^\infty \dx{t} \abss{x-y}^{|\n|}(\sqrt[8]{t})^{\eta -8 - |\n|} 
					\lesssim \abss{x-y}^\eta \, ,
				\end{align}
            where we have used $\eta\geq\gamma$ and $|\n|>\eta$.
			For $\sqrt[8]{t} \leq \abss{x-y}$, i.e.~the near-field component, we argue as follows
				\begin{align}
					\E^{\frac1p}\big|(\mathrm{id} -\mathrm{T}_x^\eta) (L^* \nabla \cdot f_t)(y)\big|^p 
					&\leq\E^{\frac1p}\big| (L^* \nabla \cdot f_t)(y)\big|^p + \sum_{|\n| \leq \eta} \abss{x-y}^{|\n|}\E^{\frac1p}\big| \partial^{\n}(L^* \nabla \cdot f_t)(x)\big|^p  \\
					&\lesssim (\sqrt[8]{t})^{\gamma - 8 } \abss{x-y}^{\eta -\gamma} + \sum_{|\n| \leq \eta} \abss{x-y}^{|\n|} (\sqrt[8]{t})^{\eta- 8 -|\n|} \, ,
				\end{align} 
			where in the last step we have used~\eqref{eq:bounda}. Since $\eta$ is not an integer, the sum in the above expression can be limited to $|\n| <\eta$ and so all powers of $t$ in the above expression are greater than $-1$ giving us the desired integrability near $t=0$. We thus have
				\begin{equation}
					\E^{\frac1p}\big|\int_0^{\abss{x-y}^8} \dx{t} (\mathrm{id} -\mathrm{T}_x^\eta) (L^* \nabla \cdot f_t)(y)\big|^p \lesssim \abss{x-y}^\eta \, ,
				\end{equation}
			completing the proof of~\eqref{eq:lhsstrongbound}. The fact that $u$ is a solution follows from applying the operator $L$ to~\eqref{eq:solutionformula} with cut-off at $s$ and $T$ and then passing to the limit. The limit converges in the topology defined by the norm in~\eqref{eq:lhsstrongbound}. 
			Note that
				\begin{align}
					L \int_{s}^T \dx{t} (\mathrm{id}-\mathrm{T}_x^\eta) (L^* \nabla \cdot f_t) = (1 - \mathrm{T}_x^{\eta-4} ) \nabla \cdot f_s - (1 - \mathrm{T}_x^{\eta-4} ) \nabla\cdot f_T \, .
				\end{align}
			Now~\eqref{eq:bounda} implies that, as $s \to 0$, $\mathrm{T}_x^{\eta-4} \nabla \cdot f_s$ converges to $0$ almost surely. Using~\eqref{eq:bounda} again, we find that $(1-\mathrm{T}_x^{\eta-4}) \nabla \cdot f_T$ converges to $0$ almost surely as $T \to \infty$. Thus, $u$ is necessarily a solution. 

			Finally, we present a Liouville-type argument for uniqueness. Let $v$ be the difference of two distributional solutions of \eqref{eq:linPDE} satisfying \eqref{eq:lhsstrongbound}. Then,
				\begin{equation}
					Lv = 0 \, .
					\label{eq:diff}
				\end{equation}
			We now use the kernel bound~\eqref{eq:kerbound} along with~\eqref{eq:lhsstrongbound} to see that
				\begin{equation}
					\lim_{t \to \infty}\E^{\frac1p}|\partial^{\n} v_t|^p  =0 \, ,
					\label{eq:infradecay}
				\end{equation}
			as long as $|\n| > \eta$. Using~\eqref{eq:diff}, we also have
				\begin{equation}
					\partial_t \partial^{\n} v_t =  -\partial^{\n} LL^*v_t =0 \, . 
				\end{equation}
			It follows by integrating in time and using~\eqref{eq:infradecay}
				\begin{equation}
					\partial^{\n} v = 0 \, ,
				\end{equation}
			for all $|\n|> \eta$. It follows that $v$ is a polynomial of degree $|\m|\leq \eta$ and, in fact, $|\m|<\eta$ since $\eta \notin \Z$. But the estimate \eqref{eq:lhsstrongbound} tells us that $v$ must in fact be identically zero, thus completing the proof.
		\end{proof}

		\begin{remark}[Failure of integration for integer $\eta$]
  \label{rem:failureSchauder}
			The careful reader may have noticed that we excluded $\eta \in \Z$ from the statement of~\cref{lem:int1}. This is due to the fact that Schauder theory, and by extension an annealed estimate of the form~\eqref{eq:lhsstrongbound}, fails to hold true for integer exponents. To understand this one can look at the Poisson equation
				\begin{equation}
					\Delta u =f
				\end{equation}
			for $f \in C^0(\R^d), d\geq 2$ and ask if $u \in C^2(\R^d)$. This is in general not true. As a counterexample for  $d=2$, consider $u(x) =\varphi(x)  (x_1^2 -x_2^2) \log (-\log(|x|^2)) $  where $\varphi$ is a smooth bump function which is $1$ for $|x|\leq 1$ and $0$ for $|x|\geq 2$. One can check (see~\cite[Section 2.2]{FRRO22}) that $u$ has a bounded and continuous Laplacian but an unbounded Hessian. The same counterexample can be used to show that the Calder\'on--Zygmund estimate fails for $p =\infty$.
		\end{remark}

We will now apply~\cref{lem:int1} to two specific cases, estimating $\Pi_{x\beta}$ given an estimate on $\Pi_{x\beta}^{-}$ and estimating $\delta \Pi_{x\beta}$ given an estimate on $\delta \Pi_{x\beta}^-$.

\begin{corollary}[Integration I: {$\Pi^{-}$ to $\Pi$}]\label{cor:intI}
Assume that $\eqref{estPiminus}_\beta$ holds.
Then, $\eqref{estPi}_\beta$ holds, i.e.
\begin{equation}
\E^\frac{1}{q'}|\Pi_{x\beta}(y)|^{q'}\lesssim |x-y|_\s^{|\beta|}.
\end{equation}
\end{corollary}

\begin{corollary}[Integration II: {$\delta \Pi^{-}$ to $\delta \Pi$}]\label{cor:intII}
Assume that $\eqref{eq:rhsweakbounddPi}_\beta$ holds. Then, $\eqref{estdeltaPi}_\beta$ holds, i.e. 
\begin{equation}\label{estdeltaPi}
\E^\frac{1}{q'}|\delta\Pi_{x\beta}(y)|^{q'}\lesssim |x-y|_\s^{|\beta|}\bar w.
\end{equation}
\end{corollary}
The proofs of the above two corollaries follow immediately from applying~\cref{lem:int1} with $f$ chosen to be $\Pi_{x\beta}^{-}$ and $\bar{w}^{-1}\delta \Pi_{x\beta}^{-}$, respectively.

	\subsubsection{Integration of the rough path increment}
	We now present a weighted version of the integration argument in~\cref{lem:int1} that will help us pass from the increment $(\delta \Pi_x^- - \dx{\Gamma_{xz}^* \Pi_z^-})$ to $(\delta \Pi_x -\delta \Pi_x(z) - \dx{\Gamma_{xz}^*} \Pi_z)$. 	The crucial ingredient for this integration argument is the following representation formula which establishes the relationship between the two increments:
		\begin{align}
			(\delta \Pi_x -\delta \Pi_x(z) - \dx{\Gamma_{xz}^*} \Pi_z)_\beta = \int_0^\infty \dx{t} (\id - \mathup{T}_z^2)(L^* \nabla \cdot (\delta \Pi_x^- - \dx{\Gamma_{xz}^* \Pi_z^-})_{\beta t} )  \, .
			\label{eq:representation}
		\end{align}
		We will provide the proof of this identity in~\cref{lem:integrationIII}, which is the main result of this section.

			\begin{lemma}[Integration III: {$(\delta \Pi_x^- - \dx{\Gamma_{xz}^* \Pi_z^-})$ to $(\delta \Pi_x -\delta \Pi_x(z) - \dx{\Gamma_{xz}^*} \Pi_z)$}]
			Let $|\beta|<3$ and assume that $\eqref{estPiminus}_{\prec \beta}$, $\eqref{eq:rhsweakbounddPi}_\beta$, and $\eqref{eq:rhsweakboundincrement}_\beta$ hold true. Furthermore, assume that, for all $\gamma$ not purely polynomial, we have the bound $\eqref{eq:gammabound}_{\beta}^\gamma$. 	Then, $\eqref{eq:representation}_\beta$ and $\eqref{eq:lhsstrongboundincrement}_\beta$ hold true.
			\label{lem:integrationIII}
		\end{lemma}

		\begin{proof}
		We will assume for the time being that $\eqref{eq:representation}_\beta$ holds true.	The strategy of proof is to show that the right hand side of~\eqref{eq:representation} is estimated by the right hand side of \eqref{eq:lhsstrongboundincrement}. We split our argument into three ranges: a near-field range $\sqrt[8]{t} \leq \abss{y-z}$, a far-field range $\sqrt[8]{t} \geq \max(\abss{y-z},\abss{x-z})$, and an intermediate range $\abss{y-z} \leq \sqrt[8]{t} \leq \abss{x-z}$. We start with the near-field range. Applying essentially the same argument as in the proof of~\cref{lem:int1} along with the negative moment bound~\eqref{eq:kerbound} we have the bound
			\begin{align}
				&\mathbb{E}^\frac{1}{q'}|\partial^{\bf n}
				(\delta \Pi^-_{x}-\dx{\Gamma^*_{xz}}\Pi^-_{z})_{\beta t}(y)|^{q'}
				\\
				&\lesssim(\sqrt[8]{t})^{\alpha-3-|{\bf n}|} (\sqrt[8]{t}+\abss{y-z})^{\kappa}
				(\sqrt[8]{t}+\abss{y-z}+\abss{x-z})^{|\beta|-\alpha}(w_x(y) + w_x(z)) \, .
				\label{eq:nearfield1}
			\end{align}
		We use \eqref{eq:nearfield1} to derive two intermediate estimates, one by restricting to $y=z$ and the other by restricting to the near-field range:
			\begin{align}
				&\mathbb{E}^\frac{1}{q'}|\partial^{\bf n}
				(\delta \Pi^-_{x}-\dx{\Gamma^*_{xz}}\Pi^-_{z})_{\beta t}(z)|^{q'}
				\lesssim(\sqrt[8]{t})^{\alpha-3-|{\bf n}|+\kappa}
				(\sqrt[4]{t}+\abss{x-z})^{|\beta|-\alpha}w_x(z),\\
				&\mathbb{E}^\frac{1}{q'}|\partial^{\bf n}
				(\delta \Pi^-_{x}-\dx{\Gamma^*_{xz}}\Pi^-_{z})_{\beta t}(y)|^{q'}
				\nonumber\\
				&\lesssim(\sqrt[8]{t})^{\alpha-3-|{\bf n}|}\abss{y-z}^{\kappa}
				(\abss{y-z}+\abss{x-z})^{|\beta|-\alpha}(w_x(y)+w_x(z))\quad\mbox{, if}\;
				\sqrt[8]{t}\leq \abss{y-z}.
			\end{align}
		We use this to estimate the Taylor polynomial and the original term as follows:
			\begin{align} 
			&\mathbb{E}^\frac{1}{q'}\left|{\rm T}_z^2L^* \nabla \cdot
				(\delta \Pi^-_{x}-\dx{\Gamma^*_{xz}}\Pi^-_{z})_{\beta t}(y)\right|^{q'} 
				\\
				&\lesssim t^{-1}\sum_{|{\bf n}|\leq 2}\abss{y-z}^{|{\bf n}|}
				(\sqrt[8]{t})^{\alpha-|{\bf n}|+\kappa}
				(\sqrt[8]{t}+\abss{x-z})^{|\beta|-\alpha}w_x(z)  \label{eq:Taylordiff} \, ,
				\end{align}
				\begin{align}
				&\mathbb{E}^\frac{1}{q'}\left|L^* \nabla \cdot
				(\delta \Pi^-_{x}-\dx{\Gamma^*_{xz}}\Pi^-_{z})_{\beta t}(y)\right|^{q'}
				\\
				&\lesssim t^{-1}(\sqrt[8]{t})^{\alpha}\abss{y-z}^{\kappa}
				(\abss{y-z}+\abss{x-z})^{|\beta|-\alpha}(w_x(y)+w_x(z))\quad\mbox{, if}\;
				\sqrt[8]{t}\leq\abss{y-z} \, .
			\end{align}
		Integrating $\int_0^{\abss{y-z}^8} \dx{t}$ and using the fact that $\alpha-2+\kappa>0$, which itself follows from~\eqref{kappa} and $\alpha<1$, on the first integral, and $\alpha>0$ on the second, we obtain
			\begin{align}
			&\mathbb{E}^\frac{1}{q'}\left|\int_0^{\abss{y-z}^8}\dx{t}{\rm T}_z^2 L^* \nabla \cdot
			(\delta \Pi^-_{x}-\dx{\Gamma^*_{xz}}\Pi^-_{z})_{\beta t}(y)\right|^{q'}\\
			&\lesssim \abss{y-z}^{\alpha+\kappa}
			(\abss{y-z}+\abss{x-z})^{|\beta|-\alpha}w_x(z),\label{eq:intIII3}\\
			&\mathbb{E}^\frac{1}{q'}\left|\int_0^{\abss{y-z}^8}\dx{t} L^* \nabla \cdot
			(\delta \Pi^-_{x}-\dx{\Gamma^*_{xz}}\Pi^-_{z})_{\beta t}(y)\right|^{q'}\\
			&\lesssim \abss{y-z}^{\alpha+\kappa}
			(\abss{y-z}+\abss{x-z})^{|\beta|-\alpha}(w_x(y)+w_x(z)),\label{eq:intIII4}
			\end{align}
		which takes care of the near-field contribution.

		We now deal with the far-field contribution $\sqrt[8]{t}\ge\max\{\abss{y-z},\abss{x-z}\}$, by splitting it into the one coming from $\dx{\Gamma_{xz}^*} \Pi_z^{-}$ and the one from $\delta\Pi_x^{-}$. For the first one, we use the fact that $\Pi^{-}_z\in\widetilde{\T}^*$ (see \eqref{eq:piminusttilde}) and the strict triangularity of
		$\dx{\Gamma^*_{xz}}$ with respect to $\prec$ (see \eqref{tri6}) along with~\eqref{eq:gammabound} and~\eqref{estPiminus} for $\Pi_{z\beta}^-$ to establish
			\begin{align}
			&\mathbb{E}^\frac{1}{q'}\left|(\dx{\Gamma_{xz}^*}\Pi_z^{-})_{\beta t}(y)\right|^{q'}\\
			&\lesssim\sum_{|\gamma|\in\mathsf{A}\cap(-\infty,3)\cap[\alpha,\kappa+|\beta|]}
			\abss{x-z}^{\kappa+|\beta|-|\gamma|}w_x(z)(\sqrt[8]{t})^{\alpha-3}
			(\sqrt[8]{t}+\abss{y-z})^{|\gamma|-\alpha},
			\label{dGammaLiouville2}
			\end{align}
		which, using a similar argument as before, we can transform into 
			\begin{align}
			&\mathbb{E}^\frac{1}{q'}\left|\partial^{\bf n}
			(\dx{\Gamma_{xz}^*}\Pi_z^{-})_{\beta t}(y)\right|^{q'}\\
			&\lesssim\sum_{|\gamma|\in\mathsf{A}\cap(-\infty,3)\cap[\alpha,\kappa+|\beta|]}
			(\sqrt[8]{t})^{|\gamma|-3-|{\bf n}|}
			\abss{x-z}^{\kappa+|\beta|-|\gamma|}w_x(z)
			\quad\mbox{, if}\;\abss{y-z}\leq\sqrt[8]{t} \, .
			\label{eq:farfield1}
			\end{align}
		We now represent Taylor's remainder in a manner compatible with the natural scaling $\mathfrak{s}$ associated to the operator $L$ 
			\begin{align}\label{eq:taylorremainder}
			(\id-{\rm T}_z^2)f(y)&=\int_0^1 \dx{s} \frac{(1-s)^2}{2}\frac{\d^3}{\d s^3}h(s) \, ,\\
			 h(s)&=f(s^{\mathfrak{s}_0}y_0+(1-s^{\mathfrak{s}_0})z_0,\dots, s^{\mathfrak{s}_d}y_d+(1-s^{\mathfrak{s}_d})z_d) \, .
			\end{align}
		Applying this to $f= L^* \nabla \cdot(\dx{\Gamma_{xz}^*}\Pi_z^{-})_{\beta t}$ and using~\eqref{eq:farfield1}, we obtain
			\begin{align}
				&\mathbb{E}^\frac{1}{q'}\left|(\id-{\rm T}_z^2)L^* \nabla \cdot
				(\dx{\Gamma_{xz}^*}\Pi_z^{-})_{\beta t}(y)\right|^{q'}\\
				&\lesssim t^{-1} \sum_{\substack{|{\bf n}|\ge 3 \\ n_0+\cdots+n_d\le 3}}
				\sum_{|\gamma|\in\mathsf{A}\cap(-\infty,3)\cap[\alpha,\kappa+|\beta|]}
				\abss{y-z}^{|{\bf n}|}(\sqrt[8]{t})^{|\gamma|-|{\bf n}|}
				\abss{x-z}^{\kappa+|\beta|-|\gamma|}w_x(z) \, , \\
				&\textnormal{ if}\;\abss{y-z}\leq\sqrt[8]{t} \, . 
    \label{anotherref}
			\end{align}
		Integrating over $\int_{\max\{\abss{y-z}^8,\abss{x-z}^8\}}^\infty \dx{t}$ and noting that $|\gamma|-|{\bf n}|<3-3=0$, we obtain
			\begin{align}
				&\mathbb{E}^\frac{1}{q'}\left|\int_{\max\{\abss{y-z}^8,\abss{x-z}^8\}}^\infty \dx{t}
				(\id-{\rm T}_z^2)L^* \nabla \cdot
				(\dx{\Gamma_{xz}^*}\Pi_z^{-})_{\beta t}(y)\right|^{q'}\\
				&\lesssim \hspace{-2ex}\sum_{\substack{|{\bf n}|\ge 3 \\ n_0+\cdots+n_d\le 3}}
				\sum_{|\gamma|\in\mathsf{A}\cap(-\infty,3)\cap[\alpha,\kappa+|\beta|]}
				\abss{y\-z}^{|{\bf n}|}(\abss{y\-z}\+\abss{x\-z})^{|\gamma|-|{\bf n}|}
				\abss{x\-z}^{\kappa+|\beta|-|\gamma|}w_x(z) \label{eq:farfield2}\\
				&\lesssim 
				\abss{y-z}^{\kappa+\alpha}(\abss{y-z}+\abss{x-z})^{|\beta|-\alpha}w_x(z)
				\quad\mbox{by}\quad|{\bf n}|\ge 3\geq
				\kappa+\alpha\, .
			\end{align}
		For the second part of the far-field contribution, we can use~\eqref{eq:rhsweakbounddPi} to obtain 
			\begin{align}
				\mathbb{E}^\frac{1}{q'}\left|\partial^{\bf n}\delta\Pi^{-}_{x\beta t}(y)\right|^{q'}
				\lesssim(\sqrt[8]{t})^{\alpha-3-|{\bf n}|}(\sqrt[8]{t}+\abss{x-y})^{|\beta|-\alpha}\bar w \, ,
				\label{eq:PDEvalid}
			\end{align}
		which, by Taylor's theorem and $\abss{x-y}+\abss{x-z}\lesssim\abss{y-z}+\abss{x-z}$, implies
			\begin{align}
			&\mathbb{E}^\frac{1}{q'}\left|(\id-{\rm T}_z^2)L^* \nabla \cdot \delta\Pi^{-}_{x\beta t}(y)\right|^{q'}
			\\&\lesssim t^{-1}\sum_{\substack{|{\bf n}|\ge 3 \\ n_0+\cdots+n_d\le 3}}
			\abss{y-z}^{|{\bf n}|}(\sqrt[8]{t})^{\alpha-|{\bf n}|}(\sqrt[8]{t}+\abss{y-z}+\abss{x-z})^{|\beta|-\alpha}\bar w.
			\label{dPiminusLiouville}
			\end{align}
		Integrating over $\int_{\max\{\abss{y-z}^8,\abss{x-z}^8\}}^\infty \dx{t}$ and using 
		$|\beta|-|{\bf n}|<3-3=0$ we obtain
			\begin{align}
			&\mathbb{E}^\frac{1}{q'}\left|\int_{\max\{\abss{y-z}^8,\abss{x-z}^8\}}^\infty \dx{t}
			(\id-{\rm T}_z^2)L^* \nabla \cdot\delta\Pi^{-}_{x\beta t}(y)\right|^{q'}\\\
			&\lesssim\sum_{\substack{|{\bf n}|\ge 3 \\ n_0+\cdots+n_d\le 3}}
			\abss{y-z}^{|{\bf n}|}
			(\abss{y-z}+\abss{x-z})^{\alpha-|{\bf n}|}
			(\abss{y-z}+\abss{x-z})^{|\beta|-\alpha}\bar w \label{eq:farfield3}\\
			&\lesssim
			\abss{y-z}^{\kappa+\alpha}(\abss{y-z}+\abss{x-z})^{|\beta|-\alpha} w_x(z)  \, ,
			\end{align}
		where in the last step we have simply used the definition of $w_x(z)$ (see \eqref{weight}) and the fact that $|\bf n|> \kappa + \alpha$. We now treat the intermediate range.	To this end, we start by applying the semigroup property to~\eqref{eq:nearfield1} to obtain
			\begin{align}
				&\mathbb{E}^{\frac{1}{q'}}|\partial^{\bf n}(\delta \Pi_x^- - \dx{\Gamma_{xz}^*} \Pi_z^-)_{\beta t}(y)|^{q'}\\
				&\leq\int \dx{y'}|\psi_\frac{t}{2}(y-y')|
				\mathbb{E}^\frac{1}{q'}|\partial^{\bf n}(\delta \Pi_x^- - \dx{\Gamma_{xz}^*} \Pi_z^-)_{\beta \frac t2}(y')|^{q'} \\
				&\lesssim \int \dx{y'}|\psi_\frac{t}{2}(y-y')|(\sqrt[8]{t})^{\alpha-3-|{\bf n}|} (\sqrt[8]{t}+\abss{y'-z})^{\kappa}
				(\sqrt[8]{t}+\abss{y'-z}+\abss{x-z})^{|\beta|-\alpha} \\
    &\,\times (w_x(y') + w_x(z)) \\
				&\lesssim (\sqrt[8]{t})^{\alpha-3-|{\bf n}|+\kappa}
(\sqrt[8]{t}+\abss{x-z})^{|\beta|-\alpha}w_x(z)\quad\mbox{, if}\;\abss{y-z}\leq\sqrt[8]{t} \, .
			\end{align}
		where in the last step we have applied the Cauchy--Schwarz inequality and \eqref{average_psi_wx}. Applying Taylor's theorem, we have
			\begin{align}
				&\mathbb{E}^\frac{1}{q'}\left|(\id-{\rm T}_z^2)L^* \nabla \cdot
				(\delta \Pi^-_{x}-\dx{\Gamma^*_{xz}}\Pi^-_{z})_{\beta t}(y)\right|^{q'}
				\\
				&\lesssim t^{-1}\sum_{\substack{|{\bf n}|\ge 3 \\ n_0+\cdots+n_d\le 3}}
				\abss{y-z}^{|{\bf n}|}(\sqrt[8]{t})^{\alpha-|{\bf n}|+\kappa}
				\abss{x-z}^{|\beta|-\alpha}w_x(z)\, , \quad\mbox{if}\;
				\abss{y-z}\le\sqrt[8]{t}\le\abss{x-z}.
			\end{align}
		Finally, integrating over $\int_{\abss{y-z}^8}^{\abss{x-z}^8} \dx{t}$ as expected, we obtain
			\begin{align}\label{eq:intfield1}
				\mathbb{E}^\frac{1}{q'}\left|\int_{\abss{y-z}^8}^{\abss{x-z}^8}\dx{t}(\id-{\rm T}_z^2)
				L^* \nabla \cdot(\delta\Pi_{x}^- -\dx{\Gamma_{xz}^*}\Pi_z^-)_{\beta t}(y)\right|^{q'}
				\lesssim\abss{y-z}^{\kappa+\alpha}\abss{x-z}^{|\beta|-\alpha}w_x(z) \, .
			\end{align}

   \medskip
   
We are now left to argue that the representation in~\eqref{eq:representation} is justified. 
Note that by using~\eqref{eq:PDEPibeta} and taking the Malliavin derivative, \eqref{eq:prerepresentation} holds true, i.e
\begin{equation}
L(\delta \Pi_x -\delta \Pi_x(z) - \dx{\Gamma_{xz}^*} \Pi_z)_\beta 
= \nabla \cdot (\delta \Pi_x^- - \dx{\Gamma_{xz}^* \Pi_z^-})_{\beta}  \, .
\label{eq:PDEdGamma}
\end{equation} 
Furthermore, we know from~\eqref{qualitativeIncrement} that 
$(\delta \Pi_x -\delta \Pi_x(z) - \dx{\Gamma_{xz}^*} \Pi_z)_\beta(y)$ 
vanishes super-quadratically in $\abss{y-z}$, 
and by \eqref{estdeltaPi}, \eqref{estPi} and the fact that $\dx{\Gamma^*}$ has the projection $Q$ built-in 
it grows sub-cubically in $y$ at infinity. 
Provided the right hand side of \eqref{eq:representation} also satisfies equation \eqref{eq:PDEdGamma}, vanishes super-quadratically at $z$ and grows sub-cubically, 
it then follows that~\eqref{eq:representation} holds true by using exactly the same Liouville-type argument as in the proof of~\cref{lem:int1}. 

To see that the time integral on the right hand side of~\eqref{eq:representation} is a solution of~\eqref{eq:PDEdGamma}, 
we cut off the time integral as follows and apply the operator $L$ to note that
\begin{align}
&L \int_s^T \dx{t} (\id - \mathup{T}_z^2)(L^* \nabla \cdot (\delta \Pi_x^- - \dx{\Gamma_{xz}^* \Pi_z^-})_{\beta t} ) \\
&= \nabla \cdot (\delta \Pi_x^- - \dx{\Gamma_{xz}^* \Pi_z^-})_{\beta s} -\nabla \cdot (\delta \Pi_x^- - \dx{\Gamma_{xz}^* \Pi_z^-})_{\beta T} \, .
\end{align}
Using~\eqref{dGammaLiouville2} and \eqref{eq:PDEvalid}, 
the second term goes to $0$ as $T \to \infty$, 
while the first term converges to the required object as $s \to 0$.

For the vanishing behaviour at $z$, 
we note that above we have estimated the right hand side of~\eqref{eq:representation} 
by the right hand side of~$\eqref{eq:lhsstrongboundincrement}$. 
Along with the observation that $\kappa +\alpha>2$ (from~\eqref{kappa} and $\alpha \in (0,1)$), 
this implies that the right hand side of~\eqref{eq:representation}  vanishes super-quadratically in $\abss{y-z}$. 

For the growth at infinity, we split the $t$-integral of \eqref{eq:representation} into three regimes: $t \in [0,1], [1, \abss{y-z}^8], [\abss{y-z}^8,\infty)$. For $0 \leq t \leq 1$, we estimate the parts involving $\id$ and ${\rm T}_z^2$ separately. For the part involving $\id$, we can directly apply~\eqref{eq:PDEvalid} and~\eqref{dGammaLiouville2}, to obtain
	  \begin{align}
	  &\mathbb{E}^\frac{1}{q'}\left|\int_{0}^1 \dx{t} \, L^* \nabla \cdot(\delta \Pi_x^- - \dx{\Gamma_{xz}^* \Pi_z^-})_{\beta t}(y)\right|^{q'} \\
	  &\lesssim  \int_0^1 \dx{t} \, t^{-1}\bigg((\sqrt[8]{t})^{\alpha}(\sqrt[8]{t}+\abss{x-y})^{|\beta|-\alpha}\bar w  \\& + \sum_{|\gamma|\in\mathsf{A}\cap(-\infty,3)\cap[\alpha,\kappa+|\beta|]}
			\abss{x-z}^{\kappa+|\beta|-|\gamma|}w_x(z)(\sqrt[8]{t})^{\alpha}
			(\sqrt[8]{t}+\abss{y-z})^{|\gamma|-\alpha} \bigg) \\
			&\lesssim (1+ \abss{x-y})^{|\beta|-\alpha}\bar{w} + \sum_{|\gamma|\in\mathsf{A}\cap(-\infty,3)\cap[\alpha,\kappa+|\beta|]}
			\abss{x-z}^{\kappa+|\beta|-|\gamma|}w_x(z) (1+\abss{y-z})^{|\gamma|-\alpha} \,,
	  \end{align}
	  which grows sub-cubic in $\abss{y}$ as desired since both $|\beta|,|\gamma|<3$. For the Taylor polynomial, we use the bound~\eqref{eq:Taylordiff} to arrive at
	  \begin{align}
	    &\mathbb{E}^\frac{1}{q'}\left|\int_{0}^1 \dx{t} \,{\rm T}_z^2 L^* \nabla \cdot(\delta \Pi_x^- - \dx{\Gamma_{xz}^* \Pi_z^-})_{\beta t} (y) \right|^{q'} \\
	    & \lesssim \int_0^1 \dx{t}t^{-1}  \sum_{|{\bf n}|\leq 2}\abss{y-z}^{|{\bf n}|}
				(\sqrt[8]{t})^{\alpha-|{\bf n}|+\kappa}
				(\sqrt[8]{t}+\abss{x-z})^{|\beta|-\alpha}w_x(z) \\
			&	\lesssim  \sum_{|{\bf n}|\leq 2}\abss{y-z}^{|{\bf n}|} (1+\abss{x-z})^{|\beta|-\alpha}w_x(z) \, ,
	  \end{align}
	  which again is sub-cubic in $\abss{y}$. For the second regime, $1 \leq t \leq \abss{y-z}^8$, we estimate the $\id$ part with exactly the same estimates as before to obtain
	  \begin{align}
	  &\mathbb{E}^\frac{1}{q'}\left|\int_{1}^{\abss{y-z}^8} \dx{t} \, L^* \nabla \cdot(\delta \Pi_x^- - \dx{\Gamma_{xz}^* \Pi_z^-})_{\beta t} (y) \right|^{q'} \\
	  & \lesssim   \int_1^{\abss{y-z}^8} \dx{t} \, t^{-1}\bigg((\sqrt[8]{t})^{\alpha}(\sqrt[8]{t}+\abss{x-y})^{|\beta|-\alpha}\bar w  \\& + \sum_{|\gamma|\in\mathsf{A}\cap(-\infty,3)\cap[\alpha,\kappa+|\beta|]}
			\abss{x-z}^{\kappa+|\beta|-|\gamma|}w_x(z)(\sqrt[8]{t})^{\alpha}
			(\sqrt[8]{t}+\abss{y-z})^{|\gamma|-\alpha} \bigg) \\
			&\lesssim (1+ \abss{y-z}^\alpha) (\abss{y-z}+\abss{x-y})^{|\beta|-\alpha} \bar{w} \\& + \sum_{|\gamma|\in\mathsf{A}\cap(-\infty,3)\cap[\alpha,\kappa+|\beta|]}
			\abss{x-z}^{\kappa+|\beta|-|\gamma|}w_x(z) (1+\abss{y-z}^{\alpha}) \abss{y-z}^{|\gamma|-\alpha} \,,
	  \end{align}
   which grows again sub-cubically in $\abss{y}$ since $|\beta|,|\gamma|<3$. 
	  We estimate the Taylor polynomial ${\rm T}_z^2$, separately as follows: for the term involving $\delta \Pi_{x\beta}^-$, we use~\eqref{eq:PDEvalid} to arrive at
	  \begin{align}
	  &\mathbb{E}^\frac{1}{q'}\left|\int_{1}^{\abss{y-z}^8} \dx{t} \,{\rm T}_z^2 L^* \nabla \cdot \delta \Pi^-_{x\beta t}(y)\right|^{q'} \\
	  &\lesssim \int_1^{\abss{y-z}^8}  \dx{t}\, t^{-1} \sum_{|\mathbf{n}|\leq 2}\abss{y-z}^{|\bf n|}(\sqrt[8]{t})^{\alpha-|{\bf n}|}(\sqrt[8]{t}+\abss{x-z})^{|\beta|-\alpha}\bar w  \\
	  &\lesssim \int_1^{\abss{y-z}^8}  \dx{t}\, t^{-1} \sum_{|\mathbf{n}|\leq 2}\abss{y-z}^{|\bf n|}(\sqrt[8]{t})^{\alpha-|{\bf n}|}((\sqrt[8]{t})^{|\beta|-\alpha} +\abss{x-z}^{|\beta|-\alpha})\bar w  \\
	  &\lesssim \sum_{|{\bf n}| \leq 2}\abss{y-z}^{|\mathbf{n}|} 
   (1+ \abss{y-z}^{|\beta|-|\n|} + \abss{x-z}^{|\beta|-\alpha} + \abss{y-z}^{\alpha-|\n|} \abss{x-z}^{|\beta|-\alpha})  \, ,
	  \end{align}
   which grows sub-cubically in $\abss{y}$ by $|\beta|<3$, 
	  while for the term involving $\dx{\Gamma^*_{xz}}$, we proceed with~\eqref{eq:gammabound} and \eqref{estPiminus} to obtain
	  \begin{align}
	  &\mathbb{E}^\frac{1}{q'}\left|\int_{1}^{\abss{y-z}^8} \dx{t} \,{\rm T}_z^2 L^* \nabla \cdot(\dx{\Gamma_{xz}^* \Pi_z^-})_{\beta t} (y)  \right|^{q'} \\
	  &\lesssim \int_1^{\abss{y-z}^8} \dx{t} \, t^{-1}\sum_{|\gamma|\in\mathsf{A}\cap(-\infty,3)\cap[\alpha,\kappa+|\beta|]}
			\abss{x-z}^{\kappa+|\beta|-|\gamma|}w_x(z)
   \sum_{|\n|\leq2}\abss{y-z}^{|\n|}
			(\sqrt[8]{t})^{|\gamma|-|\n|} \\
			&\lesssim  \sum_{|\gamma|\in\mathsf{A}\cap(-\infty,3)\cap[\alpha,\kappa+|\beta|]} \sum_{|\n|\leq2} 
			\abss{x-z}^{\kappa+|\beta|-|\gamma|}w_x(z) (1+\abss{y-z}^{|\gamma|-|\n|})\, .
	  \end{align}
   Again by $|\gamma|<3$ this grows sub-cubically in $\abss{y}$.
	  We  now treat the final regime $\abss{y-z}^8 \leq t < \infty$. For the term involving $\delta \Pi_{x}^-$, we simply apply~\eqref{dPiminusLiouville} to estimate it as follows
	  \begin{align}
	  	& \mathbb{E}^\frac{1}{q'}\left|\int_{\abss{y-z}^8}^\infty \dx{t} \, (\id-{\rm T}_z^2)L^* \nabla \cdot \delta\Pi^{-}_{x\beta t}(y)\right|^{q'}
			\\&\lesssim \int_{\abss{y-z}^8}^\infty \dx{t} \, t^{-1}\sum_{\substack{|{\bf n}|\ge 3 \\ n_0+\cdots+n_d\le 3}}
			\abss{y-z}^{|{\bf n}|}(\sqrt[8]{t})^{\alpha-|{\bf n}|}(\sqrt[8]{t}+\abss{y-z}+\abss{x-z})^{|\beta|-\alpha}\bar w \\
			&\lesssim   			  \abss{y-z}^{|\beta|} +\abss{x-z}^{|\beta|-\alpha} \abss{y-z}^{\alpha} \, .
	  \end{align}
	 For the the term involving $\dx{\Gamma^*}$, we use~\eqref{anotherref} to obtain
	  \begin{align}
	  	&\mathbb{E}^\frac{1}{q'}\left|\int_{\abss{y-z}^8}^\infty  \,(\id - {\rm T}_z^2)L^*\nabla \cdot(\dx{\Gamma_{xz}^*}\Pi_z^{-})_{\beta t}(y)\right|^{q'}\\
			&\lesssim \int_{\abss{y-z}^8}^\infty  \dx{t}\, t^{-1}\hspace{-2ex}\sum_{\substack{|{\bf n}|\ge 3 \\ n_0+\cdots+n_d\le 3}}
				\sum_{|\gamma|\in\mathsf{A}\cap(-\infty,3)\cap[\alpha,\kappa+|\beta|]}
                \hspace{-5ex}
				\abss{y\-z}^{|{\bf n}|}(\sqrt[8]{t})^{|\gamma|-|{\bf n}|}
				\abss{x\-z}^{\kappa+|\beta|-|\gamma|}w_x(z) \\
			&\lesssim   \sum_{|\gamma|\in\mathsf{A}\cap(-\infty,3)\cap[\alpha,\kappa+|\beta|]} \abss{y-z}^{|\gamma|}
			\abss{x-z}^{\kappa+|\beta|-|\gamma|}w_x(z) \, . \qedhere
	  \end{align}
    \end{proof}


\subsection{Reconstruction}\label{sec:recon}

It will be crucial for reconstruction to control the dependence of $\Pi^-_x$ on its base point, 
we therefore start with the following observation. 
As a consequence of \eqref{recenter}, $\G_{xy}$ recenters also the negative part $\Pi^-_y$ of the model,
\begin{equation}\label{recenterPi-}
\Pi^-_x = \G_{xy}\Pi^-_y 
+ P \sum_k \a_k (\G_{xy}(\id-P)\Pi_y+\pi^{(\0)}_{xy})^k
\nabla\Delta(\G_{xy}(\id-P)\Pi_y+\pi^{(\0)}_{xy}) \, .
\end{equation}
In particular, by \eqref{polypart:poly} one can read off that 
$(\Pi^-_x-\G_{xy}\Pi^-_y)_\beta$ is a space-time polynomial of degree $\leq|\beta|-3$. 
The proof of \eqref{recenterPi-} is analogous to the one of \cite[Proposition~5.2]{LOTT21} (in particular, $\eqref{recenterPi-}_\beta$ is a consequence of $\eqref{recenter}_{\prec\beta}$). 

\begin{lemma}[Reconstruction I]\label{rec1}
Assume $|\beta|>3$, that $\eqref{estPi}_{\prec\beta}$, 
$\eqref{estpin}_{\prec\beta}$ and $\eqref{estPiminus}_{\prec\beta}$ hold, 
and that $\eqref{estGamma}_\beta^\gamma$ holds for all $\gamma$ populated and not purely polynomial.
Then $\eqref{estPiminus}_\beta$ holds.
\end{lemma}

\begin{proof}
    The estimate \eqref{estPiminus} follows by general reconstruction \cite[Lemma~4.8]{LO22} from ``vanishing at the base point'',
    \begin{equation}
        \lim_{t\to0} \E^\frac{1}{p}| \Pi^-_{x\beta t}(x) |^p = 0 
        \quad \text{provided } |\beta|>3,
    \end{equation}
    and ``continuity in the base point'',
    \begin{equation}\label{Pi-cont_basepoint}
        \E^\frac{1}{p}| (\Pi^-_y-\Pi^-_x)_{\beta t}(x)|^p 
        \lesssim (\sqrt[8]{t})^{\alpha-3}(\sqrt[8]{t}+|x-y|_\s)^{|\beta|-\alpha}.
    \end{equation}
    For the former, we note that by (annealed) continuity \eqref{smoothness_Pi-} of $\Pi^-_{x\beta}$, 
    it is enough to show that $\Pi^-_{x\beta}(x)=0$. 
    By Lemma~\ref{lem:tri1}~(i), we observe that on the right hand side of \eqref{piminuscomponents} only $\Pi_{x\beta'}$ with $\beta'\prec\beta$ come up. 
    We can therefore appeal to $\eqref{estPi}_{\prec\beta}$, 
    which implies $\Pi_{x\beta'}(x)=0$, to see 
    \begin{equation}
        \Pi^-_{x\beta}(x) = \nabla\Delta\Pi_{x \, \beta-e_0}(x) 
        + \delta_\beta^{f_0} \xi_\tau(x) 
        - \sum_{\beta_1+\beta_2=\beta} \nabla\Pi_{x\beta_1}(x) c_{\beta_2}.
    \end{equation}
    Note that $\eqref{estPi}_{\prec\beta}$ implies by the semigroup property and the moment bound \eqref{eq:kerbound} 
    \begin{equation}
        \E^\frac{1}{p}|\partial^\n\Pi_{x\beta' t}(x)|^p
        \lesssim (\sqrt[8]{t})^{|\beta'|-|\n|}.
    \end{equation}
    By $|\beta-e_0|=|\beta|>3$ and the (annealed) continuity \eqref{smoothness_Pi}, 
    we have $\nabla\Delta\Pi_{x\,\beta-e_0}(x)=0$. 
    Similarly, we have $\sum_{\beta_1+\beta_2=\beta}\nabla\Pi_{x\beta_1}(x)c_{\beta_2}=0$: 
    by \eqref{pop4} we have
    $|\beta_1|=|\beta|-|\beta_2|+\alpha>3-2-\alpha+\alpha=1$ and therefore $\nabla\Pi_{x\beta_1}(x)=0$. 
    By $|f_0|=\alpha<3<|\beta|$ we also have $\delta_\beta^{f_0}=0$, 
    which concludes the argument for $\Pi^-_{x\beta}(x)=0$. 

    \medskip
    
    We turn to the continuity in the base point \eqref{Pi-cont_basepoint}, 
    which relies crucially on \eqref{recenterPi-}.
    Since \eqref{recenterPi-} is identical to \cite[(2.63)]{LOTT21}, 
    except that the second derivative $\partial_1^2$ of \cite[(2.63)]{LOTT21} 
    is replaced by the third derivative $\nabla\Delta$ here, 
    the proof of \eqref{Pi-cont_basepoint} is identical to the one of \cite[(4.9)]{LOTT21}, 
    except that the exponent $-2$ has to be replaced by $-3$. 
\end{proof}

\begin{lemma}[Reconstruction II]\label{rec2}
Assume that $\eqref{estPi}_{\prec\beta}$, $\eqref{estGamma}_{\prec\beta}$, 
$\eqref{estPiminus}_{\prec\beta}$ and $\eqref{eq:lhsstrongboundincrement}_{\prec\beta}$ hold,
and that $\eqref{estdGammaincrement}_{\preceq\beta}^\gamma$ holds
for all $\gamma$ populated and not purely polynomial. 
Then $\eqref{eq:rhsweakboundincrement}_\beta$ holds.
\end{lemma}

\begin{proof}
We define 
\begin{equation}
F_{xz}:=(\delta\Pi^-_x-\d\G_{xz}\Pi^-_z) - \Big(
\sum_k \a_k \Pi^k_x(z) \nabla\Delta(\delta\Pi_x-\d\G_{xz}\Pi_z)
+ \sum_\l \b_\l \Pi^\l_x(z)\delta\xi_\tau \Big),
\end{equation}
and we shall establish
\begin{align}\label{bound1}
&\E^\frac{1}{q'} | F_{xz\beta t}(y) |^{q'} \\
&\lesssim (\sqrt[8]{t})^{\alpha-3} (\sqrt[8]{t}+|y-z|_\s)^\kappa 
(\sqrt[8]{t}+|y-z|_\s+|x-z|_\s)^{|\beta|-\alpha} (w_x(y)+w_x(z)), 
\end{align}
which together with 
\begin{align}\label{bound2}
&\E^\frac{1}{q'} \big| \big( 
\sum_k \a_k \Pi^k_x(z) \nabla\Delta(\delta\Pi_x-\d\G_{xz}\Pi_z)
+ \sum_\l \b_\l \Pi^\l_x(z)\delta\xi_\tau \big)_{\beta t}(y) \big|^{q'} \\
&\lesssim (\sqrt[8]{t})^{\alpha-3} (\sqrt[8]{t}+|y-z|_\s)^\kappa 
(\sqrt[8]{t}+|y-z|_\s+|x-z|_\s)^{|\beta|-\alpha} (w_x(y)+w_x(z))
\end{align}
yields the desired \eqref{eq:rhsweakboundincrement}. 
We start with the proof of \eqref{bound1}. 
As in the proof of Lemma~\ref{rec1}, 
as a consequence of general reconstruction \cite[Lemma~4.8]{LO22} 
we obtain from ``vanishing at the base point''
\begin{equation}\label{bound1a}
\lim_{t\to0} \E^\frac{1}{q'}| F_{xz\beta t}(z) |^{q'} = 0
\end{equation}
and ``continuity in the base point''
\begin{align}\label{bound1b}
&\E^\frac{1}{q'} | (F_{xy} - F_{xz})_{\beta t}(y) |^{q'} \\
&\lesssim (\sqrt[8]{t})^{\alpha-3} (\sqrt[8]{t}+|y-z|_\s)^{\kappa+\alpha} 
(\sqrt[8]{t}+|y-z|_\s+|x-z|_\s)^{|\beta|-2\alpha} (w_x(y)+w_x(z))
\end{align}
the following stronger version of \eqref{bound1},
\begin{align}
&\E^\frac{1}{q'} | F_{xz\beta t}(y) |^{q'} \\
&\lesssim (\sqrt[8]{t})^{\alpha-3} (\sqrt[8]{t}+|y-z|_\s)^{\kappa+\alpha} 
(\sqrt[8]{t}+|y-z|_\s+|x-z|_\s)^{|\beta|-2\alpha} (w_x(y)+w_x(z)).
\end{align}
Here, we have used that $\kappa+2\alpha-3>0$, which follows from \eqref{kappa},
and we have used that $|\beta|\geq2\alpha$ unless the left hand side of \eqref{bound1b} vanishes, 
for which we give the argument term by term in \eqref{bound1bi} -- \eqref{bound1biv} below.
For \eqref{bound1a} we note that \eqref{magic} amounts to
\begin{equation}
F_{xz\beta}(z) = 0, 
\end{equation}
which implies \eqref{bound1a}, provided we have (annealed) continuity 
of $F_{xz\beta}(y)$ in the active variable $y$. 
This continuity is a consequence of \cref{rem:smoothness} and \cref{rem:smoothness_Mall}. 

We turn to \eqref{bound1b}, and bound its left hand side by the triangle inequality by 
\begin{align}
&\E^\frac{1}{q'}| (\d\G_{xy}\Pi^-_y - \d\G_{xz}\Pi^-_z)_{\beta t}(y) |^{q'} \label{bound1bi} \\
&\,+ \sum_k \E^\frac{1}{q'} \big| \big( \a_k \Pi^k_x(y)\nabla\Delta(\d\G_{xz}\Pi_z-\d\G_{xy}\Pi_y)_t(y) \big)_\beta \big|^{q'} 
\label{bound1bii} \\
&\,+ \sum_k \E^\frac{1}{q'} \big| \big( \a_k (\Pi^k_x(y)-\Pi^k_x(z)) \nabla\Delta(\delta\Pi_x-\d\G_{xz}\Pi_z)_t(y) \big)_\beta\big|^{q'} \label{bound1biii} \\
&\,+ \sum_\l \E^\frac{1}{q'} | \big( \b_\l (\Pi^\l_x(y) - \Pi^\l_x(z)) \big)_\beta (\delta\xi_\tau)_{t}(y) |^{q'}. \label{bound1biv}
\end{align}
For \eqref{bound1bi}, we note that the presence of $Q$ in $\d\G_{xz}$ 
allows by \eqref{recenterPi-} to rewrite
\begin{equation}
\d\G_{xy}\Pi^-_y - \d\G_{xz}\Pi^-_z = (\d\G_{xy}-\d\G_{xz}\G_{zy})\Pi^-_y.
\end{equation}
Hence the left hand side of \eqref{bound1bi} is by H\"older's inequality, 
$\eqref{estdGammaincrement}_\beta^{\gamma\neq\pp}$ and $\eqref{estPiminus}_{\prec\beta}$ estimated by  
\begin{align}
1_{|\beta|\geq2\alpha} 
\sum_{|\gamma|\in \A\cap[\alpha,\kappa+2\alpha)} |y-z|_\s^{\kappa+2\alpha-|\gamma|} (|y-z|_\s+|x-z|_\s)^{|\beta|-2\alpha} (w_x(y)+w_x(z)) 
(\sqrt[8]{t})^{|\gamma|-3},
\end{align}
which is bounded by the right hand side of \eqref{bound1b}. 
Since $\nabla\Delta(\d\G_{xy}\Pi_y-\d\G_{xz}\Pi_z) = \nabla\Delta(\d\G_{xy}-\d\G_{xz}\G_{zy})\Pi_y$ by \eqref{recenter},  \eqref{bound1bii} is by H\"older's inequality, $\eqref{estdGammaincrement}_{\prec\beta}^{\gamma\neq\pp}$, $\eqref{estPi}_{\prec\beta}$ and the moment bound \eqref{eq:kerbound} estimated by 
\begin{align}
&\sum_k \sum_{e_k+\beta_1+\cdots+\beta_{k+1}=\beta} |x-y|_\s^{|\beta_1|+\cdots+|\beta_k|} \, 1_{|\beta_{k+1}|\geq2\alpha}\\
&\times \sum_{|\gamma|\in \A\cap[\alpha,\kappa+2\alpha)} |y-z|_\s^{\kappa+2\alpha-|\gamma|} 
(|y-z|_\s+|x-z|_\s)^{|\beta_{k+1}|-2\alpha} (w_x(y)+w_x(z)) 
(\sqrt[8]{t})^{|\gamma|-3},
\end{align}
which is again bounded by the right hand side of \eqref{bound1b}. 
To estimate \eqref{bound1biii}, 
we note that the same argumentation as for \cite[(4.81)]{LOTT21} shows that 
$\eqref{estPi}_{\prec\beta}$ and $\eqref{estGamma}_{\prec\beta}$ imply
\begin{equation}\label{mt05}
\sum_k \E^\frac{1}{p}\big|\big(\a_k\Pi^k_x(y)-\a_k\Pi^k_x(z)\big)_{\beta}\big|^p 
\lesssim |y-z|_\s^\alpha (|y-z|_\s+|x-z|_\s)^{|\beta|-2\alpha}. 
\end{equation}
Also here, the left hand side vanishes unless $|\beta|\geq2\alpha$: 
the $k=0$ term vanishes, and for $k\geq1$ the left hand side vanishes unless $\beta$ contains $e_k$ which implies $|\beta|\geq2\alpha$. 
Furthermore, the same argumentation as for \cite[(4.82)]{LOTT21} shows that 
$\eqref{eq:lhsstrongboundincrement}_{\beta}$ implies
\begin{align}
&\E^\frac{1}{q'} | \nabla\Delta(\delta\Pi_x-\d\G_{xz}\Pi_z)_{\beta t}(y) |^{q'}\\
&\lesssim (\sqrt[8]{t})^{\alpha-3} (\sqrt[8]{t}+|y-z|_\s)^{\kappa} 
(\sqrt[8]{t}+|y-z|_\s+|x-z|_\s)^{|\beta|-\alpha} (w_x(y)+w_x(z)).
\end{align}
Hence \eqref{bound1biii} is by H\"older's inequality, 
$\eqref{estPi}_{\prec\beta}$, $\eqref{estGamma}_{\prec\beta}$ 
and $\eqref{eq:lhsstrongboundincrement}_{\prec\beta}$ estimated by 
\begin{align}
&\sum_{\beta_1+\beta_2=\beta} |y-z|_\s^\alpha (|y-z|_\s+|x-z|_\s)^{|\beta_1|-2\alpha} \\
&\times (\sqrt[8]{t})^{\alpha-3} (\sqrt[8]{t}+|y-z|_\s)^{\kappa} 
(\sqrt[8]{t}+|y-z|_\s+|x-z|_\s)^{|\beta_2|-\alpha} (w_x(y)+w_x(z)),
\end{align}
which is once more bounded by the right hand side of \eqref{bound1b}. 
We turn to \eqref{bound1biv}, and note that the same 
argumentation as for \eqref{mt05} shows that $\eqref{estPi}_{\prec\beta}$ and $\eqref{estGamma}_{\prec\beta}$ imply
\begin{equation}\label{mt06}
\sum_\l \E^\frac{1}{p}\big|\big(\b_\l \Pi^\l_x(y)-\b_\l \Pi^\l_x(z)\big)_{\beta}\big|^p 
\lesssim |y-z|_\s^\alpha (|y-z|_\s+|x-z|_\s)^{|\beta|-2\alpha},
\end{equation}
as above with the understanding that the left hand side vanishes unless $|\beta|\geq2\alpha$.
Furthermore, by the semigroup property 
\begin{equation}
\delta\xi_t(y) 
= \int_{\R^{1+d}} \dx{z} |y-z|^{\kappa} (LL^*)^{-\frac{s}{2|L|}} 
\psi_t(y-z) |y-z|^{-\kappa} (LL^*)^{\frac{s}{2|L|}} \delta\xi(z) \, ,
\end{equation}
which by the triangle inequality and Cauchy--Schwarz yields
\begin{equation}
\E^\frac{1}{q}| \delta\xi_t(y)|^q 
\leq \left(\int_{\R^{1+d}}\dx{z} |y-z|^{2\kappa} 
| (LL^*)^{-\frac{s}{2|L|}} \psi_t(y-z) |^2 \right)^\frac{1}{2}
w(y) \, .
\end{equation}
Since $q'<q$, this implies by the scaling \eqref{eq:psiscaling} of $\psi_t$
\begin{equation}
\E^\frac{1}{q'}| (\delta\xi_\tau)_t(y) |^{q'} 
\lesssim (\sqrt[8]{t})^{\alpha-3+\kappa} w(y) \, .
\end{equation}
Thus \eqref{bound1biv} is by H\"older's inequality, $\eqref{estPi}_{\prec\beta}$ and $\eqref{estGamma}_{\prec\beta}$ estimated by 
\begin{equation}
|y-z|_\s^\alpha (|y-z|_\s+|x-z|_\s)^{|\beta|-2\alpha} (\sqrt[8]{t})^{\alpha-3+\kappa} w(y),
\end{equation}
which is bounded by the right hand side of \eqref{bound1b}. 

\medskip

It remains to establish \eqref{bound2}. 
We bound its left hand side by the triangle inequality by 
\begin{align}
&\sum_k \E^\frac{1}{q'} \big| \big( 
\a_k \Pi^k_x(z) \nabla\Delta(\delta\Pi_x-\d\G_{xz}\Pi_z)_t(y) \big)_\beta \big|^{q'}
+\sum_\l \E^\frac{1}{q'} \big| \big( \b_\l \Pi^\l_x(z)\big)_\beta (\delta\xi_\tau)_{t}(y) \big|^{q'}.
\end{align}
As above, we argue that this is by $\eqref{estPi}_{\prec\beta}$ and 
$\eqref{eq:lhsstrongboundincrement}_{\prec\beta}$ estimated by the right hand side of \eqref{bound2}.
\end{proof}
\subsection{Algebraic arguments}\label{sec:algebraic}

\begin{lemma}[Algebraic argument I]\label{alg1}
Assume that $\eqref{estpin}_{\prec\beta}$ holds. 
Then $\eqref{estGamma}_\beta^{\gamma}$ holds for all $\gamma$ populated and not purely polynomial. 
Furthermore, 
\begin{equation}\label{estGammaDn}
    \E^\frac{1}{p}|(\G_{xy}D^{(\n)})_\beta^\gamma|^p\lesssim |x-y|_\s^{|\beta|-|\gamma|-\alpha+|\n|} 
\end{equation}
holds for all $\gamma$ populated and for all $\n$.
\end{lemma}
\begin{proof}
    Recall from \eqref{coeffGamma} that $(\G_{xy})_\beta^\gamma$ is a linear combination of terms of the form 
    \begin{equation}
        \pi_{xy \beta_1}^{(\n_1)}\cdots\pi_{xy \beta_j}^{(\n_j)} (D^{(\n_1)}\cdots D^{(\n_j)})_{\beta_{j+1}}^\gamma, 
    \end{equation}
    where $j\geq0$, $\n_1,\dots,\n_j\in\N_0^{1+d}$ and $\beta_1+\cdots+\beta_{j+1}=\beta$. 
    Since by assumption $\sum_\ell\gamma_b(\ell)>0$, \eqref{tri1} yields $\beta_1,\dots,\beta_j\prec\beta$, 
    hence H\"older's inequality and $\eqref{estpin}_{\prec\beta}$ 
    imply that the stochastic norm $\E^\frac{1}{p}|\cdot|^p$ of the above expression is estimated by 
    \begin{equation}
        |x-y|_\s^{|\beta_1|-|\n_1|}\cdots|x-y|_\s^{|\beta_j|-|\n_j|}. 
    \end{equation}
    By \eqref{rg06}, the sum of the exponents equals $|\beta|-|\gamma|$. 
    We turn to the proof of \eqref{estGammaDn}. 
    If $\gamma$ is purely polynomial, then either the left hand side of \eqref{estGammaDn} vanishes, 
    or $\n\neq\0$ and $\gamma=g_\n$. In the latter case, $(\G_{xy}D^{(\n)})_\beta^\gamma = (\G_{xy})_\beta^0=\delta_\beta^0$, which trivially satisfies \eqref{estGammaDn} since $|\beta=0|=\alpha$ and $|\gamma=g_\n|=|\n|$.
    If $\gamma$ is not purely polynomial, then $\sum_\ell\gamma_b(\ell)>0$. It follows from \eqref{D0props} and \eqref{Dnprops} that $(\G_{xy}D^{(\n)})_\beta^\gamma=\sum_{\beta'}(\G_{xy})_{\beta}^{\beta'}(D^{(\n)})_{\beta'}^\gamma$ is a sum over multiindices $\beta'$ restricted to $\sum_\ell\beta'_b(\ell)>0$. 
    We can therefore appeal to the already established $\eqref{estGamma}_\beta^{\beta'}$ 
    (notice that above we only used $\sum_\ell\gamma_b(\ell)>0$, not that $\gamma$ is populated) 
    to estimate the $\E^\frac{1}{p}|\cdot|^p$-norm of every summand by 
    $|x-y|_\s^{|\beta|-|\beta'|}$. From \eqref{rg03} we obtain $|\beta'|=|\gamma|+\alpha-|\n|$, which establishes \eqref{estGammaDn}.    
\end{proof}
 
\begin{lemma}[Algebraic argument II]\label{alg2}
Assume that $\eqref{estpin}_{\prec\beta}$ and $\eqref{estdeltapin}_{\prec\beta}$ hold. 
Then for all $\gamma$ populated and not purely polynomial 
\begin{equation}\label{estdeltaGamma}
\E^\frac{1}{q'}|(\delta\G_{xy})_\beta^{\gamma}|^{q'}\lesssim |x-y|_\s^{|\beta|-|\gamma|} \bar w \, .
\end{equation}
\end{lemma}
\begin{proof}
    Applying $\delta$ to \eqref{exp} yields by the chain rule 
    \begin{equation}
        \delta\G_{xy} = \sum_{j\geq1}\tfrac{1}{(j-1)!} \sum_{\n_1,\dots,\n_j} \delta\pi^{(\n_1)}_{xy}\pi^{(\n_2)}_{xy}\cdots\pi^{(\n_j)}_{xy} D^{(\n_1)}\cdots D^{(\n_j)}.
    \end{equation}
    Clearly, \eqref{tri1} transfers from $\G$ to $\delta\G$, 
    hence the same argumentation as in Lemma~\ref{alg1} applies.
\end{proof}

\begin{lemma}[Algebraic argument III]\label{alg3}
Assume that $\eqref{estpin}_{\prec\beta}$, $\eqref{estdpinincrement}_{\prec\beta}$ and $\eqref{estdpin}_{\prec\beta}$ hold,
and that $\eqref{estGamma}_{\preceq\beta}^\gamma$ holds for all $\gamma$ populated and not purely polynomial. 
Then for all $\gamma$ populated and not purely polynomial 
\begin{align}\label{estdGammaincrement} 
    &\E^\frac{1}{q'}|(\d\G_{xy}-\d\G_{xz}\G_{zy})_\beta^\gamma|^{q'} \\
    &\lesssim |y-z|_\s^{\kappa+2\alpha-|\gamma|} (|y-z|_\s+|x-z|_\s)^{|\beta|-2\alpha} (w_x(y)+w_x(z)) \, ,
\end{align}
with the understanding that $\kappa+2\alpha-|\gamma|>0$ and $|\beta|\geq2\alpha$ unless the left hand side vanishes.
\end{lemma}
\begin{proof}
Due to the presence of the projection $Q$ in the definition \eqref{dGamma} of $\d\G$ together
with the triangularity \eqref{tri1} of $\G$ with respect to ~$|\cdot|$, 
the left hand side of \eqref{estdGammaincrement} vanishes unless $|\gamma|<3$, 
which by \eqref{kappa} implies $|\gamma|<\kappa+2\alpha$. 
Furthermore, by \eqref{GammaPreservesTildeT} and \eqref{tri7} we observe that $|\beta|\geq2\alpha$ 
unless the left hand side of \eqref{estdGammaincrement} vanishes.

\medskip

We turn to the proper estimate \eqref{estdGammaincrement}, 
for which we momentarily denote by $\d\Gtilde_{xy}$ the object defined in \eqref{dGamma} without the projection $Q$, i.e.
\begin{equation}
    \d\Gtilde_{xy} = \sum_{|\n|\leq2} \d\pi^{(\n)}_{xy} \G_{xy} D^{(\n)}.
\end{equation}
Then  
\begin{equation}
    (\d\G_{xy}-\d\G_{xz}\G_{zy}) = (\d\G_{xy}-\d\G_{xz}\G_{zy})Q
    = (\d\Gtilde_{xy}-\d\Gtilde_{xz}\G_{zy}) Q + (\d\Gtilde_{xz}-\d\G_{xz})\G_{zy} Q,
\end{equation}
where in the first equality we used $Q\G_{zy}Q=Q\G_{zy}$, which follows from \eqref{tri3}, 
and in the second equality we used $\d\G_{xy}Q=\d\Gtilde_{xy}Q$. 
We start by estimating 
\begin{align}
    ((\d\Gtilde_{xz}-\d\G_{xz})\G_{zy} Q)_\beta^\gamma 
    &= \big(\sum_{|\n|\leq2} \d\pi^{(\n)}_{xz} \G_{xz} D^{(\n)} (\id-Q) \G_{zy} Q\big)_\beta^\gamma \\
    &= \sum_{|\n|\leq2} \sum_{|\beta'|\geq3} \sum_{\beta_1+\beta_2=\beta} 
    \d\pi^{(\n)}_{xz\beta_1} (\G_{xz} D^{(\n)})_{\beta_2}^{\beta'} (\G_{zy})_{\beta'}^\gamma \, 1_{|\gamma|<3} .
\end{align}
By assumption $\gamma$ is populated and not purely polynomial, which carries over to $\beta'$ by \eqref{GammaPreservesTildeT}. 
By \eqref{tri5} and \eqref{tri6}, which also hold for $\d\Gtilde$ since we did not at all make use of the projection $Q$ in the proof, we have $\beta_1,\beta_2,\beta'\prec\beta$. 
Therefore we appeal to $\eqref{estdpin}_{\prec\beta}$, 
$\eqref{estGamma}_{\prec\beta}^{\gamma}$ and $\eqref{estGammaDn}_{\prec\beta}^{\beta'}$ to estimate the $\E^\frac{1}{q'}|\cdot|^{q'}$ norm of every summand by 
\begin{equation}
    |x-z|_\s^{\kappa+|\beta_1|-|\n|} w_x(z) |x-z|_\s^{|\beta_2|-|\beta'|-\alpha+|\n|} |y-z|_\s^{|\beta'|-|\gamma|}.
\end{equation}

Since $\beta'$ is populated and not purely polynomial, the condition $|\beta'|\geq3$ strengthens to $|\beta'|>3$, 
and from \eqref{kappa3} we obtain $|\beta'|\geq\kappa+2\alpha$.
Hence the above expession is further estimated by 
\begin{equation}
    |y-z|_\s^{\kappa+2\alpha-|\gamma|} (|y-z|_\s+|x-z|_\s)^{|\beta_1|+|\beta_2|-3\alpha} w_x(z), 
\end{equation}
which is estimated by the right hand side of \eqref{estdGammaincrement} since $\beta_1+\beta_2=\beta$ implies $|\beta_1|+|\beta_2|-\alpha = |\beta|$ by \eqref{homogeneity}.

\medskip

We turn to the estimate on $(\d\Gtilde_{xy}-\d\Gtilde_{xz}\G_{zy}) Q$. 
The same argumentation as in \cite[(4.43)]{LOTT21} reveals 
\begin{equation}
    (\d\Gtilde_{xy}-\d\Gtilde_{xz}\G_{zy})Q 
    = \sum_{|\n|\leq2} (\d\pi^{(\n)}_{xy}-\d\pi^{(\n)}_{xz}-\d\Gtilde_{xz}\pi^{(\n)}_{zy}) \G_{xy} D^{(\n)} Q,
\end{equation}
where we rewrite the right hand side as 
\begin{equation}
    \sum_{|\n|\leq2} (\d\pi^{(\n)}_{xy}-\d\pi^{(\n)}_{xz}-\d\G_{xz}\pi^{(\n)}_{zy}) \G_{xy} D^{(\n)} Q 
    + \sum_{|\n|\leq2} \big((\d\G_{xz}-\d\Gtilde_{xz})\pi^{(\n)}_{zy}\big) \G_{xy} D^{(\n)} Q. 
\end{equation}
The $(\cdot)_\beta^\gamma$-component of this first term equals
\begin{equation}\label{rg07}
    \sum_{|\n|\leq2} \sum_{\beta_1+\beta_2=\beta} 
    (\d\pi^{(\n)}_{xy}-\d\pi^{(\n)}_{xz}-\d\G_{xz}\pi^{(\n)}_{zy})_{\beta_1} (\G_{xy} D^{(\n)})_{\beta_2}^\gamma 1_{|\gamma|<3} .
\end{equation}
As in the proof of \eqref{tri5} we argue that $\beta_2\neq0$ and thus $\beta_1\prec\beta$. 
Clearly, we also have $\beta_2\preceq\beta$ by $\beta_1+\beta_2=\beta$. 
Since $\gamma$ is by assumption populated, we can therefore appeal to H\"older's inequality, 
$\eqref{estdpinincrement}_{\prec\beta}$ and $\eqref{estGammaDn}_{\preceq\beta}^{\gamma}$ to estimate the $\E^\frac{1}{q'}|\cdot|^{q'}$-norm of every summand of \eqref{rg07} by 
\begin{equation}
    |y-z|_\s^{\kappa+\alpha-|\n|}(|y-z|_\s+|x-z|_\s)^{|\beta_1|-\alpha} (w_x(y)+w_x(z)) |x-y|_\s^{|\beta_2|-|\gamma|-\alpha+|\n|}.
\end{equation}
Let us mention that only $\gamma$ with $(\G_{xy}D^{(\n)})_{\beta_2}^\gamma\neq0$ come up, 
which by \eqref{D0comp} and \eqref{Dncomp} satisfy $|\gamma|\geq|\n|+\alpha$. 
This expression is therefore further bounded by 
\begin{equation}
    |y-z|_\s^{\kappa+2\alpha-|\gamma|} 
    (|y-z|_\s+|x-z|_\s)^{|\beta_1|+|\beta_2|-3\alpha}(w_x(y)+w_x(z)) ,
\end{equation}
which coincides with the right hand side of \eqref{estdGammaincrement}.

\medskip

It remains to estimate 
\begin{align}
    &\big(\sum_{|\n|\leq2} \big((\d\Gtilde_{xz}-\d\G_{xz})\pi^{(\n)}_{zy}\big) \G_{xy} D^{(\n)} Q\big)_\beta^\gamma \\
    &= \sum_{|\n|\leq2} \sum_{\beta_1+\beta_2=\beta} 
    \big((\d\Gtilde_{xz}-\d\G_{xz})\pi^{(\n)}_{zy}\big)_{\beta_1} (\G_{xy} D^{(\n)})_{\beta_2}^\gamma 1_{|\gamma|<3} \\
    &=\sum_{|\n|\leq2} \sum_{\beta_1+\beta_2+\beta_3=\beta} \sum_{|\beta'|\geq3} \sum_{|\m|\leq2}
    \d\pi^{(\m)}_{xz\beta_1} (\G_{xz}D^{(\m)})_{\beta_2}^{\beta'} \pi^{(\n)}_{zy\beta'} 
    (\G_{xy} D^{(\n)})_{\beta_3}^\gamma 1_{|\gamma|<3} .
\end{align}
We shall argue now that $\beta_1,\beta_2,\beta_3,\beta'\prec\beta$. 
Since $\beta_1$ is populated, we have $\beta_1\neq0$ and therefore $\beta_2,\beta_3\prec\beta$. 
Since $\gamma$ is by assumption not purely polynomial, we also have $\beta_3\neq0$ and therefore $\beta_1\prec\beta$. 
From \eqref{tri6} we know $\beta'\prec\beta_1+\beta_2=\beta-\beta_3\preceq\beta$. 
We can therefore appeal to $\eqref{estdpin}_{\prec\beta}$, $\eqref{estGammaDn}_{\prec\beta}^{\beta',\gamma}$ and $\eqref{estpin}_{\prec\beta}$ to estimate every summand by 
\begin{equation}
    |x-z|_\s^{|\beta_1|-|\m|+\kappa}w_x(z)
    |x-z|_\s^{|\beta_2|-|\beta'|-\alpha+|\m|}
    |y-z|_\s^{|\beta'|-|\n|}
    |x-y|_\s^{|\beta_3|-|\gamma|-\alpha+|\n|}.
\end{equation}
By $|\beta'|\geq3$ we obtain $|\beta'|+\alpha>3$ and therefore \eqref{kappa3} yields $|\beta'|\geq\kappa+\alpha$, 
and the same argument as for \eqref{rg07} yields $|\gamma|\geq|\n|+\alpha$. 
The above expression is therefore further estimated by 
\begin{equation}
    |y-z|_\s^{\kappa+2\alpha-|\gamma|} 
    (|y-z|_\s+|x-z|_\s)^{|\beta_1|+|\beta_2|+|\beta_3|-4\alpha} w_x(z),
\end{equation}
which is bounded by the right hand side of \eqref{estdGammaincrement} since $\beta_1+\beta_2+\beta_3=\beta$ and $|\cdot|-\alpha$ is additive. 
\end{proof}

\begin{lemma}[Algebraic argument IV]\label{alg4}
Assume that $\eqref{estdeltaPi}_{\prec\beta}$ and $\eqref{estdpin}_{\prec\beta}$ hold, 
and that $\eqref{estGamma}_\beta^{\gamma}$ holds for all $\gamma$ populated and not purely polynomial. 
Then for all $\gamma$ populated and not purely polynomial
\begin{equation}
\mathbb{E}^{\frac{1}{q'}}|(\dx{\Gamma_{xz}^*})_{\beta}^\gamma|^{q'} 
\lesssim \abss{x-z}^{\kappa + |\beta|-|\gamma|} w_x(z) \, .
\label{eq:gammabound}
\end{equation}
\end{lemma}
The proof of Lemma~\ref{alg4} follows the same lines as the one in \cite[Proposition~4.16]{LOTT21}, 
which we therefore skip.


\subsection{Three-point arguments}\label{sec:3point}

\begin{lemma}[Three-point argument I]\label{3pt1}
Assume that $\eqref{estPi}_{\preceq\beta}$ holds, 
and that $\eqref{estGamma}_\beta^{\gamma}$ holds for all $\gamma$ populated and not purely polynomial. 
Then 
\begin{equation}\label{estpin}
    \E^\frac{1}{p}|\pi^{(\n)}_{xy\beta}|^p\lesssim |x-y|_\s^{|\beta|-|\n|} \, .
\end{equation}
\end{lemma}
The proof of Lemma~\ref{3pt1} follows the same lines as in \cite[Proposition~4.4]{LOTT21}, and relies on the three-point identity 
\begin{equation}
    \sum_\n \pi^{(\n)}_{xy}(z-y)^\n = \Pi_x(z)-\Pi_y(z)-(\G_{xy}-\id)P\Pi_y(z),
\end{equation}
which is a consequence of \eqref{recenter}, \eqref{polypart:poly}, \eqref{Gamma_pn} and \eqref{pi0}.

\begin{lemma}[Three-point argument II]\label{3pt2}
Assume that $\eqref{estPi}_{\prec\beta}$ and $\eqref{estdeltaPi}_{\preceq\beta}$ hold,
and that $\eqref{estGamma}_\beta^{\gamma}$ and $\eqref{estdeltaGamma}_\beta^{\gamma}$ hold for all 
$\gamma$ populated and not purely polynomial. 
Then 
\begin{equation}\label{estdeltapin}
\E^\frac{1}{q'}|\delta\pi^{(\n)}_{xy\beta}|^{q'}\lesssim |x-y|_\s^{|\beta|-|\n|}\bar w \, .
\end{equation}
\end{lemma}
The proof of Lemma~\ref{3pt2} is identical to the one of \cite[Proposition~4.10]{LOTT21}, 
and relies on 
\begin{equation}
    \sum_\n \delta\pi^{(\n)}_{xy}(z-y)^\n = \delta\Pi_x(z)-\G_{xy}P\delta\Pi_y(z)-\delta\G_{xy}P\Pi_y(z),
\end{equation}
which is seen to be true by applying $\delta$ to the three-point identity above.

\begin{lemma}[Three-point argument III]\label{3pt3}
Assume that $\eqref{estPi}_{\prec\beta}$ and $\eqref{eq:lhsstrongboundincrement}_\beta$ hold, 
and that $\eqref{estdGammaincrement}_\beta^{\gamma}$ holds for all $\gamma$ populated and not purely polynomial. 
Then for $|\n|\leq2$
\begin{align}
&\E^\frac{1}{q'}|(\d\pi^{(\n)}_{xy}-\d\pi^{(\n)}_{xz}-\d\G_{xz}\pi^{(\n)}_{zy})_\beta|^{q'}\label{estdpinincrement} \\
&\lesssim |y-z|_\s^{\kappa+\alpha-|\n|}(|y-z|_\s+|x-z|_\s)^{|\beta|-\alpha} (w_x(y)+w_x(z)) \, . 
\end{align}
\end{lemma}
\begin{proof}
    For $\n=\0$, the statement is a consequence of \eqref{choice_dpi}, \eqref{pi0} and \eqref{eq:lhsstrongboundincrement}.
    For $\n\neq\0$, we first prove the formula
    \begin{align}
        \sum_{|\n|\leq2} (\d\pi^{(\n)}_{xy} - \d\pi^{(\n)}_{xz} - \d\G_{xz}\pi^{(\n)}_{zy}) (\cdot-y)^\n \label{rg08} 
        &= (\delta\Pi_x-\delta\Pi_x(z)-\d\G_{xz}\Pi_z) \\[-2.5ex]
        &\, - (\delta\Pi_x-\delta\Pi_x(y)-\d\G_{xy}\Pi_y) \\
        &\, - (\d\G_{xy}-\d\G_{xz}\G_{zy})P\Pi_y. 
    \end{align}
    Indeed, \eqref{choice_dpi} and \eqref{recenter} yield
    \begin{align}
        &(\delta\Pi_x-\delta\Pi_x(z)-\d\G_{xz}\Pi_z) - (\delta\Pi_x-\delta\Pi_x(y)-\d\G_{xy}\Pi_y) 
        - (\d\G_{xy}-\d\G_{xz}\G_{zy})P\Pi_y \\
        &= \d\pi^{(\0)}_{xy} - \d\pi^{(\0)}_{xz} - \d\G_{xz} \Pi_z(y) + (\d\G_{xy}-\d\G_{xz}\G_{zy})(1-P)\Pi_y,
    \end{align}
    which by \eqref{pi0} and \eqref{polypart:poly} equals
    \begin{equation}
        \d\pi^{(\0)}_{xy} - \d\pi^{(\0)}_{xz} - \d\G_{xz} \pi^{(\0)}_{zy} 
        + (\d\G_{xy}-\d\G_{xz}\G_{zy}) \sum_{\n\neq\0} \p_\n (\cdot-y)^\n.
    \end{equation}
    From \eqref{dGamma}, \eqref{D0} and \eqref{Dn} we read off 
    \begin{equation}\label{rg09}
        \d\G_{xy}\sum_{\n\neq\0} \p_\n (\cdot-y)^\n = \sum_{\n\neq\0, |\n|\leq2} \d\pi^{(\n)}_{xy} (\cdot-y)^\n,
    \end{equation}
    and using in addition \eqref{Gamma_pn} we see 
    \begin{align}
        \d\G_{xz}\G_{zy}\sum_{\n\neq\0}\p_\n (\cdot-y)^\n 
        &= \d\G_{xz} \sum_{\n\neq\0} (\p_\n+\pi^{(\n)}_{zy}) (\cdot-y)^\n \\
        &= \sum_{\n\neq\0, |\n|\leq2} (\d\pi^{(\n)}_{xz} + \d\G_{xz} \pi^{(\n)}_{zy}) (\cdot-y)^\n ,
    \end{align}
    which establishes \eqref{rg08}. 
    The $\E^\frac{1}{q'}|\cdot|^{q'}$-norm of the $\beta$-component of the right hand side of \eqref{rg08} is 
    by $\eqref{eq:lhsstrongboundincrement}_\beta$, $\eqref{estdGammaincrement}_\beta^{\gamma\neq\pp}$ and $\eqref{estPi}_{\prec\beta}$ 
    estimated by 
    \begin{align}
        &|\cdot-z|_\s^{\kappa+\alpha} (|\cdot-z|_\s+|x-z|_\s)^{|\beta|-\alpha} (w_x(\cdot)+w_x(z)) \\
        &\,+ |\cdot-y|_\s^{\kappa+\alpha} (|\cdot-y|_\s+|x-y|_\s)^{|\beta|-\alpha} (w_x(\cdot)+w_x(y)) \\
        &\,+ |y-z|_\s^{\kappa+2\alpha-|\gamma|} (|y-z|_\s+|x-z|_\s)^{|\beta|-2\alpha} (w_x(y)+w_x(z)) |\cdot-y|_\s^{|\gamma|}.
    \end{align}
    Restricting the active variable to $|\cdot-y|_\s\leq|y-z|_\s$, this is further estimated by 
    \begin{equation}
        |y-z|_\s^{\kappa+\alpha} (|y-z|_\s+|x-z|_\s)^{|\beta|-\alpha} (w_x(\cdot)+w_x(y)+w_x(z)) \, ,
    \end{equation}
and we obtain 
\begin{align}\label{mt10}
&\E^\frac{1}{q'}\big|
\sum_{|\n|\leq2} (\d\pi^{(\n)}_{xy} - \d\pi^{(\n)}_{xz} - \d\G_{xz}\pi^{(\n)}_{zy}) (\cdot-y)^\n \big|^{q'} \\
&\lesssim |y-z|_\s^{\kappa+\alpha} (|y-z|_\s+|x-z|_\s)^{|\beta|-\alpha} (w_x(\cdot)+w_x(y)+w_x(z)) \, .
\end{align}
We now evaluate at $y+\lambda\m$ for $0\neq|\m|\leq2$ and average over $\lambda\leq|y-z|_\s/|\m|$ in order to recover $\eqref{estdpinincrement}_\beta$ for $0\neq|\n|\leq2$. 
Indeed, for the left hand side of \eqref{mt10} we appeal to the obvious $\fint_{\lambda\leq|y-z|_\s} \lambda \sim |y-z|_\s$. 
For the right hand side of \eqref{mt10}, by definition \eqref{weight} of $w_x$, it suffices to appeal to $\fint_{\lambda\leq|y-z|_\s} |y+\lambda\m-x|^{-\kappa}\lesssim|y-x|^{-\kappa}$. 
\end{proof}

\begin{lemma}[Three-point argument IV]\label{3pt4}
Assume that $\eqref{estPi}_{\prec\beta}$, $\eqref{estdeltaPi}_\beta$ and $\eqref{eq:lhsstrongboundincrement}_\beta$ hold,
and that $\eqref{eq:gammabound}_\beta^{\gamma}$ holds for all $\gamma$ populated and not purely polynomial. 
Then for $|\n|\leq2$
\begin{equation}\label{estdpin}
    \E^\frac{1}{q'}|\d\pi^{(\n)}_{xy\beta}|^{q'} \lesssim |x-y|_\s^{|\kappa+|\beta|-|\n|}w_x(y) \, .
\end{equation}
\end{lemma}
\begin{proof}
    For $\n=\0$, the statement is a consequence of \eqref{choice_dpi} and \eqref{estdeltaPi}. 
    For $\n\neq\0$ we observe that by \eqref{rg09}
    \begin{align}
        \sum_{\n\neq\0,|\n|\leq2} \d\pi^{(\n)}_{xy} (z\-y)^\n 
        = \delta\Pi_x(z) \- \delta\Pi_x(y) \- 
        (\delta\Pi_x \- \delta\Pi_x(y) \- \d\G_{xy}\Pi_y)(z) \- \d\G_{xy}P\Pi_y(z) .
    \end{align}
    By assumption we can therefore estimate 
    \begin{align}
        \E^\frac{1}{q'} \Big| \sum_{\n\neq\0,|\n|\leq2} \d\pi^{(\n)}_{xy\beta} (z-y)^\n \Big|^{q'} 
        &\lesssim |x-z|_\s^{|\beta|} \bar w + |x-y|_\s^{|\beta|} \bar w \\
        &\, + |y-z|_\s^{\kappa+\alpha}(|y-z|_\s+|x-y|_\s)^{|\beta|-\alpha} (w_x(y)+w_x(z)) \\
        &\, + \sum_{|\gamma|\in \A\cap(\alpha,|\beta|+2-\alpha]} |x-y|_\s^{\kappa+|\beta|-|\gamma|} w_x(y) |y-z|_\s^{|\gamma|}.
    \end{align}
    Restricting $z$ to $|y-z|_\s\leq|x-y|_\s$, the right hand side is further estimated by 
    \begin{equation}
        |x-y|_\s^{\kappa+|\beta|} (w_x(y)+w_x(z)),
    \end{equation}
    and as in the proof of Lemma~\ref{3pt3} we obtain \eqref{estdpin}.
\end{proof}


\subsection{Averaging}\label{sec:average}

\begin{lemma}[Averaging]\label{lem:averaging}
Assume that 
$\eqref{eq:rhsweakboundincrement}_\beta$, 
$\eqref{eq:rhsweakbounddPi}_{\prec\beta}$,
$\eqref{estPiminus}_{\prec\beta}$, 
$\eqref{estGamma}_{\beta}^{\gamma}$,
$\eqref{estdeltaGamma}_\beta^{\gamma}$ and 
$\eqref{eq:gammabound}_\beta^{\gamma}$ hold, 
all for $\gamma$ not purely polynomial.
Then $\eqref{eq:rhsweakbounddPi}_\beta$ holds.
\end{lemma}

\begin{proof}
We first establish $\eqref{eq:rhsweakbounddPi}_\beta$ for $x=y$.
For that, we use the semigroup property \eqref{eq:semi} and the triangle inequality to get 
\begin{align}
\E^\frac{1}{q'}| \delta\Pi^-_{x\beta t}(x) |^{q'} 
&\leq \int_{\R^{1+d}} \dx{y}\, |\psi_{\frac{t}{2}}(x-y)| \,
\E^\frac{1}{q'}| (\delta\Pi^-_x -\d\G_{xy}\Pi^-_y)_{\beta \frac{t}{2}}(y)|^{q'} \\
&\, + \int_{\R^{1+d}} \dx{y}\, |\psi_{\frac{t}{2}}(x-y)| \,
\E^\frac{1}{q'}| (\d\G_{xy}\Pi^-_y)_{\beta \frac{t}{2}}(y)|^{q'} \, .
\end{align}
The first right hand side term is by $\eqref{eq:rhsweakboundincrement}_\beta$ estimated by 
\begin{equation}
\int_{\R^{1+d}} \dx{y}\, |\psi_{\frac{t}{2}}(x-y)| \,
(\sqrt[8]{t})^{\alpha-3} (\sqrt[8]{t})^\kappa (\sqrt[8]{t}+|x-y|_\s)^{|\beta|-\alpha}
w_x(y),
\end{equation}
which by \eqref{average_psi_wx} and the moment bound \eqref{eq:kerbound} is bounded by the desired $(\sqrt[8]{t})^{|\beta|-3} \bar w$.
For the second right hand side term we appeal to
$\eqref{eq:gammabound}_\beta^{\gamma\neq\pp}$, 
the triangularity \eqref{tri6} of $\d\G$ with respect to $\prec$, 
and $\eqref{estPiminus}_{\prec\beta}$, 
which by H\"older's inequality imply an estimate by 
\begin{equation}
\int_{\R^{1+d}} \dx{y} \, |\psi_\frac{t}{2}(x-y)| \, 
\sum_{|\gamma|\in \A\cap[\alpha, |\beta|+2-\alpha] }
|x-y|_\s^{\kappa+|\beta|-|\gamma|} w_x(y) (\sqrt[8]{t})^{|\gamma|-3} \, .
\end{equation}
Again by \eqref{average_psi_wx} and the moment bound \eqref{eq:kerbound} this is estimated as desired by $(\sqrt[8]{t})^{|\beta|-3}\bar w$. 
To get rid of the restriction $x=y$ we apply the Malliavin derivative to \eqref{recenterPi-}, 
where we note that due to $|\beta|<3$ the correction is not present, 
which yields
\begin{equation}
\delta\Pi^-_{x\beta t}(y) 
= \big( \delta\G_{xy} \Pi^-_y + \G_{xy}\delta\Pi^-_y\big)_{\beta t}(y) \, .
\end{equation}
Applying $\E^\frac{1}{q'}|\cdot|^{q'}$ and H\"older's inequality, 
we use on the first right hand side term 
$\eqref{estdeltaGamma}_\beta^{\gamma\neq\pp}$ and 
$\eqref{estPiminus}_{\prec\beta}$ which is sufficient by the triangularity \eqref{tri4}, 
and we use on the second right hand side term 
$\eqref{estGamma}_\beta^{\gamma\neq\pp}$ and the just established 
$\eqref{eq:rhsweakbounddPi}_{\preceq\beta}$ for $x=y$ which is sufficient by the triangularity \eqref{tri3}, 
to obtain 
\begin{align}
\E^\frac{1}{q'}| \delta\Pi^-_{x\beta t}(y)|^{q'}
&\lesssim \sum_{|\gamma|\in \A\cap[\alpha,|\beta|]} \big( 
|x-y|_\s^{|\beta|-|\gamma|} \bar w \, (\sqrt[8]{t})^{|\gamma|-3}
+ |x-y|_\s^{|\beta|-|\gamma|} (\sqrt[8]{t})^{|\gamma|-3} \bar w \big) \\
&\lesssim (\sqrt[8]{t})^{\alpha-3} (\sqrt[8]{t}+|x-y|_\s)^{|\beta|-\alpha} \bar w \, .
\qedhere
\end{align}
\end{proof}


\section{\texorpdfstring{\bf Proof of~\cref{prop:scaling_c}}{Proof of the form of the counterterm}}\label{sec:explicit}

We remind the reader that in this section we are working with the model $\hat\Pi_{\hat\beta}$ and counterterm $\hat{c}_{\hat\beta}$ from \cref{rem:analyticity}, 
with multiindices $\hat\beta$ that do not contain any $e_0$ components, i.e.~$\hat{\beta}_a(k=0)=0$. 
For notational convenience we drop the hat on the multiindex $\hat\beta$ in the remainder of this section. 
Furthermore, we work with $\hat{L}=(\partial_0-(1-a_0)\Delta^2)$ which depends on $a_0$ and we define $m_0=1-a_0$. 
As argued in the first paragraph of Appendix~\ref{app:analytic}, the estimates obtained in Theorem~\ref{thm:main} remain valid for $\hat\Pi$, locally uniformly for $a_0<1$. We will denote the $m_0$-dependent semigroup associated to the operator $\hat{L}\hat{L}^*$ by $\hat{\psi}_{t}$ (see~\eqref{kernel}). Note however, that for the large scale average which shows up in BPHZ choice  of renormalisation made in~\eqref{eq:BPHZchoice}, we will continue to use the $m_0$-independent semigroup $\psi_{t}$.  Given~\cref{prop:symmetries} and the BPHZ choice of renormalisation made in~\eqref{eq:BPHZchoice}, we have the following result that gives us a natural restriction on when $\hat{c}_\beta$ can be chosen to be zero.
\begin{lemma}
Let~\cref{ass} be satisfied with $\alpha \in (\frac12,1)\setminus \mathbb{Q}$, 
and let $d=1$. 
Then, for all $x \in \mathbb{R}^{1+d}$ and $|\beta|< 3$ such that $ \beta \notin \{e_1 + f_0 + f_1 + g_{(0,1)}, 2f_1 + g_{(0,1)}, 2e_1 +  2f_0 + g_{(0,1)}\}$, we have
\begin{equation}
\lim_{t \to \infty}\E [\hat{\tilde\Pi}_{x \beta t}^-(x)] =0 \, ,
\end{equation}
where $\hat{\tilde\Pi}^-$ is defined analogously to \eqref{eq:Pitilde} by 
$\hat\Pi^-_{x\hat\beta}=\hat{\tilde\Pi}^-_{x\hat\beta}-\hat c_{\hat\beta-g_\n} \n$.
\label{lem:reducingterms}   
\end{lemma}

\begin{proof}
Since the value of $\E [\hat{\tilde\Pi}_{x \beta t}^-(x)]$ depends only on the law of $\xi$, the symmetries of~\cref{prop:symmetries} tell us that  $\E [\hat{\Pi}_{x \beta t}^-(x)] \neq 0$ only if
$[\beta] + |\beta|_p \textnormal{ and } 1 + [\beta] \textnormal{ are even}$.
 This restriction, together with the fact that $\alpha \in (\frac12, 1)$ and $|\beta|<3$ tell us that  $\E [\hat{\tilde\Pi}_{x \beta t}^-(x)] \neq 0$ only if 
 \begin{equation}\label{badMultiindices}
 \beta \in \left\{ e_1 + f_0 +f _1 , f_0 + f_2 , 2f_1  , 2 e_1 + 2 f_0 , e_2 + 2 f_0 \right\} + g_{(0,1)} \, .
 \end{equation}
 Note now that by \cref{rem:hierarchy}, 
 and since $\hat{c}_{f_1}=0$ as a consequence of $f_1$ not being present in the above set, we have 
 \begin{equation}
 \hat{\tilde\Pi}_{x (f_0 +f_2 + g_{(0,1)}) }^-(y) =  2(y_1 - x_1) \hat{\Pi}_{x f_0}(y) \xi_{\tau}(y) \, .
 \end{equation}
By~\cref{prop:symmetries}~(1) and \cref{ass} we can choose $x=0$ without loss of generality and compute using the integral representation \eqref{eq:solutionformula}
\begin{align}
\E \left[\hat{\tilde\Pi}_{0 (f_0 +f_2 + g_{(0,1)}) t }^-(0)\right] 
&= 2\E \int_{\R^2} \dx{y} \psi_{t}(y) y_1 \hat{\Pi}_{0 f_0}(y) \xi_\tau(y) \\
&= 2 \E \int_{\R^2}\int_0^\infty \dx{y} \dx{s} \psi_{t}(y) y_1 \left[(\mathrm{id} - {\rm T}_0^0) \hat{L}^* \partial_1 \xi_\tau*\hat{\psi}_s(y)\right]  \xi_\tau(y) \\
&=  \E \int_{\R^2} \int_{\R^2}\int_0^\infty \dx{y} \dx{z} \dx{s} \psi_{t}(y) (\hat{L}^* \partial_1\hat{\psi}_s) (y-z) y_1  \xi_\tau(z)  \xi_\tau(y) \\
&\,- \E \int_{\R^2} \int_{\R^2}\int_0^\infty \dx{y} \dx{z} \dx{s} \psi_{t}(y) (\hat{L}^* \partial_1 \hat{\psi}_s) (-z) y_1  \xi_\tau(z)  \xi_\tau(y) \\
&=  \int_{\R^2} \int_{\R^2}\int_0^\infty \dx{y} \dx{z} \dx{s} \psi_{t}(y) (\hat{L}^* \partial_1\hat{\psi}_s) (y-z) y_1  F(y-z) \\
&\,-  \int_{\R^2} \int_{\R^2}\int_0^\infty \dx{y} \dx{z} \dx{s} \psi_{t}(y) (\hat{L}^* \partial_1 \hat{\psi}_s) (-z) y_1  F(y-z) \\
&= -  \int_{\R^2} \int_{\R^2}\int_0^\infty \dx{y} \dx{z} \dx{s} \psi_{t}(y) (\hat{L}^* \partial_1 \hat{\psi}_s) (-z) y_1  F(y-z)\, ,
\end{align}
where in the last equality we have used the fact that $\int_{\R^2}\dx{y}\psi_t (y) y_1 =0$.  For the remaining term, we proceed as follows 
\begin{align}
&\left| \int_{\R^2} \int_{\R^2}\int_0^\infty \dx{y} \dx{z} \dx{s} \psi_{t}(y) (\hat{L}^* \partial_1 \hat{\psi}_s) (-z) y_1  F(y-z) \right| \\
&=\left|
\int_{\R^2}\int_0^\infty \dx{z} \dx{s} 
\big((\cdot)_1\psi_t * F\big) (z) (\hat{L}^*\partial_1\hat{\psi}_s)(z)
\right| \\
&\lesssim  \int_0^\infty \int_{\R^2} \dx{s} \dx{k} t |k_1|^7 |k_1|^4|k_1| |\mathcal{F}F(k)||\mathcal{F}\hat{\psi}_s(k)|  |\mathcal{F}\psi_t(k)| \\
&\lesssim t\int_0^\infty \int_{\R^2} \dx{s} \dx{k}  \abss{k}^{12} |\mathcal{F}F(k)|
e^{-\abss{k}^8 (t+s)} \\
&= t \int_{\R^2}  \dx{k}  \abss{k}^{4} |\mathcal{F}F(k)|
e^{-\abss{k}^8 t} \lesssim  (\sqrt[8]{t})^{-1}
\stackrel{ t \to \infty}{\to} 0\, \,,\label{eq:byhand1}
\end{align}
where the implicit constants depend on $\lambda,\Lambda$.
Note that we have used the fact that $\abss{k}^8 \lesssim(k_0^2 +k_1^8) \lesssim \abss{k}^8$, that $\mathcal{F}F$ is a Schwartz function, and the explicit form of the Fourier transforms of $\psi_t,\hat{\psi}_s$. 

Finally, we treat the term $\hat{\tilde\Pi}_{x(e_2 +2f_0 + g_{(0,1)} )}^-$ using again \cref{rem:hierarchy} and $\hat{c}_{e_1+f_0}=0$ as follows 
\begin{align}
 &\left|\E \hat{\tilde\Pi}_{0(e_2 +2f_0 + g_{(0,1)})t}^- (0)\right|\\
 &= 2\left|\E\int_{\R^2} \dx{y} \psi_{t}(y) y_1 \hat{\Pi}_{0 f_0}(y) (\partial_1^3 \hat{\Pi}_{0f_0}) (y)\right| \\
 &\leq \left| \E\int_{\R^2} \dx{y} \psi_{t}(y) y_1  \partial_1^3(\hat{\Pi}_{0 f_0}^2) (y)\right| 
 + 3 \left|\E \int_{\R^2} \dx{y} \psi_{t}(y) y_1  \partial_1((\partial_1\hat{\Pi}_{0 f_0})^2) (y)\right| \\
 &\lesssim (\sqrt[8]{t})^{2\alpha -2} \\
 &\,+ \left| \int_{\R^2} \int_0^\infty \int_0^\infty \dx{y} \dx{s_1}\dx{s_2} \psi_{t}(y) y_1 \partial_1 \int_{\R^2} \int_{\R^2} \dx{z_1} \dx{z_2} \hat{\psi}_{s_1}(y-z_1) \hat{\psi}_{s_2}(y-z_2) F(z_1-z_2)  \right| \\
 &=  (\sqrt[8]{t})^{2\alpha -2} \\
 &\,+ \left| \int_{\R^2} \int_0^\infty \int_0^\infty \dx{y} \dx{s_1}\dx{s_2} \psi_{t}(y) y_1 \partial_1 \int_{\R^2} \int_{\R^2} \dx{z_1} \dx{z_2} \hat{\psi}_{s_1}(z_1) \hat{\psi}_{s_2}(z_2) F(z_1-z_2)  \right| \\&=   (\sqrt[8]{t})^{2\alpha -2} \stackrel{t \to \infty}{\to} 0 \, ,
\end{align} 
where we have used the fact that $\alpha<1$ and the implicit constants depend on $\lambda,\Lambda$. 
\end{proof}

\begin{proof}[Proof of~\cref{prop:scaling_c}]
The results of~\cref{lem:reducingterms} tell us that in $d=1$ and for $\alpha \in (\frac12,1)$ there are three multiindices in need of renormalisation, which we deal with in the following three steps.

\textbf{Step 1} ($\beta=e_1+f_0+f_1+g_{(0,1)}$)\textbf{.} 
Choosing $x=0$ without loss of generality due to~\cref{prop:symmetries}, $\hat{\tilde\Pi}_{0(e_1 + f_0 + f_1 + g_{(0,1)})}^-$ can be expressed as 
\begin{align}
\hat{\tilde\Pi}_{0(e_1 + f_0 + f_1 + g_{(0,1)})}^- (y) &=  \underbrace{y_1 \partial_1^3 \hat{\Pi}_{0(f_0 + f_1)}(y)}_{\rm (i)}  + \underbrace{ \partial_1^3(\hat{\Pi}_{0 f_0} \hat{\Pi}_{0(f_1 + g_{(0,1)})})(y)}_{\rm (ii)} \\& \underbrace{-3 \partial_1(\partial_1  \hat{\Pi}_{0 (f_1 + g_{(0,1)})} \partial_1 \hat{\Pi}_{0f_0})(y)}_{\rm (iii)}  + \underbrace{\hat{\Pi}_{0(e_1 + f_0 + g_{(0,1)})}(y) \xi_\tau(y)}_{\rm (iv)} \, .
\end{align}
We now convolve with $\psi_t$ and evaluate at $0$, take the expectation, and treat the four terms on the right hand side of the above expression separately. For (i) and (ii), we can simply apply the bounds from~\cref{thm:main}, noting that $|f_0+f_1|=2\alpha$, $|f_0|=\alpha$ and $|f_1+g_{(0,1)}|=\alpha+1$ and that~\cref{thm:main} holds true locally uniformly in $a_0<1$, to obtain
\begin{equation}
\left|\E {\rm (i)}_t(0)\right| 
+ \left|\E {\rm (ii)}_t(0)\right| 
\lesssim (\sqrt[8]{t})^{2 \alpha -2 } 
\stackrel{t \to \infty}{\to }0\, , 
\end{equation}
where we used $\alpha<1$ and the implicit constants depend on $\lambda,\Lambda$. We now treat the term (iv). We  know from ~\eqref{badMultiindices} that $\hat{c}_{e_1 + f_0}=0$. Thus, using~\cref{rem:hierarchy} and the solution formula along with the fact that $|e_1+f_0+g_{(0,1)}|=\alpha+1<2$ we have that
\begin{align}
&\E {\rm (iv)}_t(0) \\
&= \E \int_{\R^2}\dx{y} \psi_t(y) \xi_{\tau}(y) \int_0^\infty \dx{s_1} (\mathrm{id} - \mathrm{T}_0^1) \hat{L}^* \partial_1 \hat{\psi}_{s_1}* ((\cdot)_1 \partial_1^3\hat{\Pi}_{0 f_0})(y)  \\
&=  \E \int_{\R^2}\dx{y} \psi_t(y) \xi_{\tau}(y) \int_0^\infty \dx{s_1} (\mathrm{id}-\mathrm{T}_0^1)\int_{\R^2} \dx{z} (\hat{L}^* \partial_1\hat{\psi}_{s_1})(y-z)  z_1 \\
&\,\times \int_0^\infty \int_{\R^2}\dx{s_2}\dx{v} \partial_1^4\hat{L}^*\hat{\psi}_{s_2}(z-v) \xi_{\tau}(v)   \\
&= \int_{\R^2}\int_{\R^2}\int_{\R^2}\int_0^\infty \int_0^\infty \dx{y}\dx{z}\dx{v}\dx{s_1}\dx{s_2} \psi_t(y) 
(\hat{L}^* \partial_1 \hat{\psi}_{s_1})(y\-z)  z_1  \partial_1^4\hat{L}^*\hat{\psi}_{s_2}(z\-v) F(y\-v)   \\
&\,+ \int_{\R^2}\int_{\R^2}\int_{\R^2}\int_0^\infty \int_0^\infty \dx{y}\dx{z}\dx{v}\dx{s_1}\dx{s_2} \psi_t(y)   (\hat{L}^* \partial_1 \hat{\psi}_{s_1})(z)  z_1  \partial_1^4\hat{L}^*\hat{\psi}_{s_2}(z\-v) F(y\-v)   \\
&\,-  \int_{\R^2}\int_{\R^2}\int_{\R^2}\int_0^\infty \int_0^\infty \dx{y}\dx{z}\dx{v}\dx{s_1}\dx{s_2} \psi_t(y)   y_1  (\partial_1^2 \hat{L}^*\hat{\psi}_{s_1})(z)  z_1 \partial_1^4\hat{L}^*\hat{\psi}_{s_2}(z\-v) F(y\-v) \\
&=   \underbrace{-\int_{\R^2}\int_0^\infty \int_0^\infty \dx{z}\dx{s_1}\dx{s_2}   (\hat{L}^* \partial_1 \hat{\psi}_{s_1})(z)  z_1  (\partial_1^4\hat{L}^*\hat{\psi}_{s_2})*F(z)}_{\rm (iv)_a}  \\
&\,+ \underbrace{\int_{\R^2}\int_0^\infty \int_0^\infty \dx{y}\dx{s_1}\dx{s_2} \psi_t(y)  \big( (\hat{L}^* \partial_1 \hat{\psi}_{s_1} (\cdot)_1)*(\partial_1^4\hat{L}^*\hat{\psi}_{s_2})*F\big)(y) }_{\rm (iv)_b}\\
&\,\underbrace{-  \int_{\R^2}\int_0^\infty \int_0^\infty \dx{y}\dx{s_1}\dx{s_2} y_1 \psi_t(y)  \big( (\hat{L}^* \partial_1^2 \hat{\psi}_{s_1} (\cdot)_1)*(\partial_1^4\hat{L}^*\hat{\psi}_{s_2})*F\big)(y)}_{\rm (iv)_c} \, ,
\end{align}
where in the last equality we have used that $F$ and $\psi$ are even. 
We now deal with $\rm (iv)_b$ and $\rm (iv)_c$. Assuming $m_0=1$  without loss of generality (since the estimates below depend only on $\lambda,\Lambda$), for $\rm (iv)_b$, we have
\begin{align}
|{\rm (iv)_b}|&=
\Big|\int_{\R^2}\int_0^\infty \int_0^\infty \dx{y}\dx{s_1}\dx{s_2} \psi_t(y)  \big( (\hat{L}^* \partial_1 \hat{\psi}_{s_1} (\cdot)_1)*(\partial_1^4\hat{L}^*\hat{\psi}_{s_2})*F\big)(y) \Big|  \\
&= \Big|\int_0^\infty \int_0^\infty \int_{\R^2} \dx{s_1}\dx{s_2}\dx{k} \mathcal{F}\psi_t(k)  i (2\pi  k_1)^5 (- 2\pi i k_0 + (2 \pi k_1)^4 )^2 \mathcal{F}\hat{\psi}_{s_2}(k) \\
&\,\times \mathcal{F}F(k) \mathcal{F}\big(\hat{\psi}_{s_1}(\cdot)_1\big)(k)\Big|\\
&= \Big| \int_0^\infty \int_0^\infty \int_{\R^2} \dx{s_1}\dx{s_2}\dx{k} \mathcal{F}\psi_t(k)  i (2\pi  k_1)^5 (- 2\pi i k_0 + (2 \pi k_1)^4 )^2 \mathcal{F}\hat{\psi}_{s_2}(k) \\
&\,\times \mathcal{F}F(k) \frac{i}{2\pi}\partial_{k_1}\mathcal{F}\hat{\psi}_{s_1}(k)\Big| \\
&\lesssim \int_0^\infty \int_0^\infty \int_{\R^2} \dx{s_1}\dx{s_2}\dx{k} \left|\mathcal{F}\psi_t(k)   k_1^5 (k_0^2 + k_1^8 ) \mathcal{F}\hat{\psi}_{s_2}(k)  \partial_{k_1}\mathcal{F}\hat{\psi}_{s_1}(k) \mathcal{F}F(k) \right| \\
&\lesssim \int_0^\infty \int_0^\infty \int_{\R^2} \dx{s_1}\dx{s_2}\dx{k} \left|s_1  k_1^{12} (k_0^2 + k_1^8 ) e^{-(( 2\pi  k_0)^2 + (2 \pi k_1)^8 )(t+s_1+s_2)} \mathcal{F}F(k) \right| 
\end{align}
where in the last inequality we have used the explicit form of $\mathcal{F}\psi$. 
Integrating in $s_1$ and $s_2$ and using that for fixed $\tau>0$, 
$F$ is a Schwartz function and thus bounded, we obtain 
\begin{equation}
|{\rm (iv)_b}|
\lesssim \int_{\R^2} \dx{k} \frac{1}{\abss{k}^4} e^{-\abss{k}^8t} \left| \mathcal{F}F(k) \right| 
\lesssim  \int_{\R^2} \dx{k} \frac{1}{\abss{k}^4} e^{-\abss{k}^8t} 
\lesssim (\sqrt[8]{t})^{-1} \stackrel{t \to \infty}{\to} 0\, .
\end{equation}
The term $\rm (iv)_c$ can be treated in a similar manner as follows
\begin{align}
|{\rm (iv)_c}|
&=\left| \int_{\R^2}\int_0^\infty \int_0^\infty \dx{y}\dx{s_1}\dx{s_2} y_1 \psi_t(y) \big( (\hat{L}^* \partial_1^2 \hat{\psi}_{s_1} (\cdot)_1)*(\partial_1^4\hat{L}^*\hat{\psi}_{s_2})*F\big)(y)\right| \\
&\lesssim \int_{\R^2} \dx{k} t\abss{k}^5 e^{-\abss{k}^8 t}  
\lesssim (\sqrt[8]t)^{-1} \stackrel{t \to \infty}{\to}0 \,,
\end{align}
where the implicit constant depends on $\lambda,\Lambda$.
We leave $\rm (iv)_a$ aside for the time being and move on to the term $\rm (iii)$ which we deal with as follows: 
By the integral representation \eqref{eq:solutionformula} (note that $|f_1+g_{(0,1)}|=1+\alpha$ and $|f_0|=\alpha$) and using that $\psi$ is even we have 
\begin{align}
\E {\rm (iii)}_t(0) 
&=   3\E \Big[\int_{\R^2}\dx{y}\partial_1 \psi_t(y) 
\int_0^\infty \dx{s_1} \hat{L}^* \partial_1^2  \hat{\psi}_{s_1}*(\xi_\tau)(y) \\
&\,\times \int_0^\infty \dx{s_2} \big( \hat{L}^* \partial_1^2 \hat{\psi}_{s_2}*((\cdot)_1 \xi_\tau)(y) - \hat{L}^* \partial_1^2 \hat{\psi}_{s_2}* ((\cdot)_1 \xi_\tau)(0) \big) \Big] \\
&=  3 \int_{\R^2}\int_{\R^2}\int_{\R^2}\int_0^\infty \int_0^\infty\dx{y} \dx{z}\dx{x}\dx{s_1}\dx{s_2}\partial_1 \psi_t(y) (\hat{L}^* \partial_1^2 \hat{\psi}_{s_1})(y-z) \\
&\,\times\Big( ( \hat{L}^* \partial_1^2\hat{\psi}_{s_2})(y-x) 
- ( \hat{L}^* \partial_1^2\hat{\psi}_{s_2})(-x) \Big) 
x_1 F(z-x) \, .
\end{align}
Using that $\psi$ and $F$ are even, we obtain
\begin{align}
&\E {\rm (iii)}_t(0) \\
&= 3 \int_{\R^2}\hspace{-.5ex}
\int_{\R^2}\hspace{-.5ex}
\int_{\R^2}\hspace{-.5ex}
\int_0^\infty\hspace{-.75ex}
\int_0^\infty\hspace{-.5ex} 
\dx{y} \dx{z}\dx{x}\dx{s_1}\dx{s_2}y_1 
\partial_1 \psi_t(y) 
(\hat{L}^* \partial_1^2 \hat{\psi}_{s_1})(z) 
(\hat{L}^* \partial_1^2 \hat{\psi}_{s_2})(x)  
F(z-x) \\
&\,- 3 \int_{\R^2} \int_0^\infty \int_0^\infty\dx{y} \dx{s_1}\dx{s_2}\partial_1 \psi_t(y)  \big( (\hat{L}^* \partial_1^2 \hat{\psi}_{s_1}) *( \hat{L}^* \partial_1^2\hat{\psi}_{s_2} (\cdot)_1) *F \big)(y) \\
&= \underbrace{-3  \int_{\R^2}\int_0^\infty \int_0^\infty\dx{z} \dx{s_1}\dx{s_2}   (\hat{L}^* \partial_1^2 \hat{\psi}_{s_1})(z) 
(\hat{L}^* \partial_1^2 \hat{\psi}_{s_2} * F) (z)}_{\rm (iii)_a} \\
&\,- \underbrace{3 \int_{\R^2} \int_0^\infty \int_0^\infty 
\dx{y} \dx{s_1}\dx{s_2}
\partial_1 \psi_t(y)  \big( 
(\hat{L}^* \partial_1^2 \hat{\psi}_{s_1}) *
(\hat{L}^* \partial_1^2 \hat{\psi}_{s_2} (\cdot)_1) * F \big)(y)}_{\rm (iii)_b} \, ,
\end{align}
where in the last equality we have used the fact that $\int_{\R^2} \dx{y} y_1\partial_1\psi(y) = -1$. 
We treat the term $\rm (iii)_b$ in a similar manner to $\rm (iv)_b$  as follows
\begin{equation}
|{\rm (iii)_b}|
\lesssim \int_{\R^2}\dx{k} \frac{1}{\abss{k}^{4}} e^{-\abss{k}^8 t} 
\lesssim (\sqrt[8]{t})^{-1} \stackrel{t \to \infty}{\to} 0 \, \, ,
\end{equation}
where the implicit constant again depends on $\lambda,\Lambda$.
We are now left to deal with the terms $\rm (iii)_a$ and $\rm (iv)_a$. We integrate by parts and after some tedious computations obtain
\begin{align}
{\rm (iii)_a + (iv)_a} 
&= -2  \int_{\R^2}\int_0^\infty \int_0^\infty\dx{z} \dx{s_1}\dx{s_2}   (\hat{L}^*  \hat{\psi}_{s_1})(z) ( \hat{L}^* \partial_1^4\hat{\psi}_{s_2}*F)(z) \\
&\,+ \int_{\R^2}\int_0^\infty \int_0^\infty\dx{z} \dx{s_1}\dx{s_2} z_1  (\hat{L}^*  \hat{\psi}_{s_1})(z) ( \hat{L}^* \partial_1^5\hat{\psi}_{s_2}*F)(z) \\
&=-2 \int_{\R^2}\dx{k}  \frac{(2 \pi k_1)^4}{(2 \pi  k_0)^2 + m_0^2(2\pi k_1)^8} \mathcal{F}F(k) \\
&\,+ 4 \int_{\R^2}\dx{k}  m_0^2\frac{(2 \pi k_1)^{12}}{((2 \pi  k_0)^2 + m_0^2(2\pi k_1)^8)^2} \mathcal{F}F(k)\, ,
\end{align}
which completes the proof of~\eqref{eq:value}. 

We now consider the choice of mollifier made in the statement of the theorem. After rescaling appropriately, we are left with
\begin{align}
&\hat{c}_{e_1+f_0+f_1}(a_0)\\
&= \lim_{t \to \infty }\E \hat{\tilde\Pi}_{0(e_1 + f_0 + f_1 + g_{(0,1)}) t}^-(0) \\
&= \frac{(\sqrt[8]{\tau})^{2\alpha -2}}{m_0(2\pi)^2} \int_{\R^2} \dx{k} \frac{ k_1^4}{(k_0^2 + k_1^8)(m_0^2k_0^2 + k_1^8)^{\frac{2\alpha -1}{8}}} \left(\frac{4 k_1^8}{k_0^2 + k_1^8}-2 \right)e^{-m_0^2k_0^2 \tau^{\eta-1} - k_1^8} \\
&= \frac{(\sqrt[8]{\tau})^{2\alpha -2}}{m_0} \underbrace{\frac{1}{(2\pi)^2}\int_{\R^2} \dx{k} \frac{ k_1^4}{(k_0^2 + k_1^8)(m_0^2 k_0^2 + k_1^8)^{\frac{2\alpha -1}{8}}} \left(\frac{4 k_1^8}{k_0^2 + k_1^8}-2 \right)e^{ - k_1^8} }_{=:C_{\alpha,1}(m_0)}\\
&\,+ \frac{(\sqrt[8]{\tau})^{2\alpha -2}}{m_0(2\pi)^2} \int_{\R^2} \dx{k} \frac{ k_1^4}{(k_0^2 + k_1^8)(m_0^2k_0^2 + k_1^8)^{\frac{2\alpha -1}{8}}} \left(\frac{4 k_1^8}{k_0^2 + k_1^8}-2 \right)e^{ - k_1^8}\big(e^{-k_{0}^2m_0^2 \tau^{\eta-1}}-1 \big) .
\end{align}
The first term is as desired, with~\eqref{eq:exactvalue2} following by exactly computing the integral for $\alpha=\frac{1}{2}$. The remainder we control as follows for $\eta>1$
\begin{align}
 &\frac{(\sqrt[8]{\tau})^{2\alpha -2}}{m_0(2\pi)^2} \left|\int_{\R^2} \dx{k} \frac{ k_1^4}{(k_0^2 + k_1^8)(m_0^2k_0^2 + k_1^8)^{\frac{2\alpha -1}{8}}} \left(\frac{12 k_1^8}{k_0^2 + k_1^8}-2 \right)e^{ - k_1^8} (e^{-k_{0}^2m_0^2 \tau^{\eta-1}}-1 )  \right|  \\
&\lesssim \frac{(\sqrt[8]{\tau})^{2\alpha -2}}{m_0^{\frac{2 \alpha + 3}{4}}} \int_{\R^2} \dx{k} \frac{ k_1^4}{(k_0^2)^{1+\frac{2\alpha -1}{8}}} e^{ - k_1^8}\left|e^{-k_{0}^2m_0^2 \tau^{\eta-1}}-1 \right|   \\
&\lesssim \frac{(\sqrt[8]{\tau})^{2\alpha -2}}{m_0^{\frac{2 \alpha + 3}{4}}} \int_{\R} \dx{k_0} \frac{\big| e^{-k_{0}^2m_0^2 \tau^{\eta-1}}-1 \big|}{(k_0^2)^{1+\frac{2\alpha -1}{8}}} 
\lesssim  
(\sqrt[8]{\tau})^{2\alpha -2 +  (\eta-1) (3 + 2\alpha)} \, . 
\end{align} 

\textbf{Step 2} ($\beta=2f_1+g_{(0,1)}$)\textbf{.}
We now move on to $\hat{\Pi}_{x(2f_1 + g_{(0,1)})}^-$ and note that again we have, by \cref{rem:hierarchy} and by $\hat{c}_{f_1}=0$,  
\begin{align}
\E \left[\hat{\tilde\Pi}_{0 (2 f_1 + g_{(0,1)}) t }^-(0)\right] 
&= \E \int_{\R^2} \dx{y} \psi_{t}(y) \hat{\Pi}_{0(f_1 + g_{(0,1)})}(y) \xi_\tau(y) \\
&= \E \int_{\R^2}\int_0^\infty \dx{y} \dx{s} \psi_{t}(y) \left[(\mathrm{id} - {\rm T}_0^1) \hat{L}^* \partial_1\hat{\psi}_s*((\cdot)_1\xi_\tau)(y)\right]  \xi_\tau(y) \\
&= \E \int_{\R^2}\int_{\R^2} \int_0^\infty \dx{y}\dx{z} \dx{s} \psi_{t}(y)  (\hat{L}^* \partial_1\hat{\psi}_s)(y-z) z_1\xi_\tau(z)  \xi_\tau(y) \\
&\,-  \E \int_{\R^2}\int_{\R^2} \int_0^\infty \dx{y}\dx{z} \dx{s} \psi_{t}(y)  (\hat{L}^* \partial_1\hat{\psi}_s)(-z) z_1\xi_\tau(z)  \xi_\tau(y) \\
&\,- \E \int_{\R^2}\int_{\R^2} \int_0^\infty \dx{y}\dx{z} \dx{s} \psi_{t}(y)  y_1 
( \hat{L}^* \partial_1^2\hat{\psi}_s)(-z) z_1\xi_\tau(z)  \xi_\tau(y) \\
&= \underbrace{\int_{\R^2}\int_{\R^2} \int_0^\infty \dx{y}\dx{z} \dx{s} \psi_{t}(y)  (\hat{L}^* \partial_1\hat{\psi}_s)(y-z) z_1 F(y-z)}_{\rm (v)} \\
&\,-  \underbrace{ \int_{\R^2}\int_{\R^2} \int_0^\infty \dx{y}\dx{z} \dx{s} \psi_{t}(y)  (\hat{L}^* \partial_1\hat{\psi}_s)(-z) z_1 F(y-z)}_{\rm (vi)} \\
&\,- \underbrace{\int_{\R^2}\int_{\R^2} \int_0^\infty \dx{y}\dx{z} \dx{s} \psi_{t}(y)  y_1 ( \hat{L}^* \partial_1^2\hat{\psi}_s)(-z) z_1F(y-z)}_{\rm (vii)} \, .
\end{align}
 We leave (v) as it is and now treat the terms (vi) and (vii) individually, explictly bounding them as we did in~\eqref{eq:byhand1}. For (vi), we have, after applying Plancherel's identity and using the assumptions on $m_0$,
\begin{align}
|\mathrm{(vi)}|
&= \left| \int_{\R^2}\int_0^\infty \dx{y} \dx{s}
\psi_t(y) (((\cdot)_1 \hat{L}^*\partial_1\hat{\psi}_s)* F)(y) \right| \\
&\lesssim \int_{\R^2}\int_0^\infty \dx{k}\dx{s} e^{-|k|_\s^8 t} 
\left| \partial_{k_1} \big(k_1(-k_0+k_1^4)e^{-|k|_\s^8}\big) \mathcal{F}F(k)\right| \\
 &\lesssim \int_{\R^2} \int_0^\infty \dx{k}\dx{s} e^{-\abss{k}^8 t}  |\mathcal{F}F(k)| \Big((|k_1|^4 \+ |k_0| ) e^{-\abss{k}^8s} \+ s|k_1|^7 (|k_0||k_1| \+ |k_1|^5)e^{-\abss{k}^8s}\Big) \, ,
\end{align}
where we have arrived at the above expression by using the explicit form of the Fourier transform of $z_1$ and brutally estimating the terms that show up. As before, the implicit constants depend on $\lambda,\Lambda$. Integrating in $s$ and using that $F$ is a Schwartz function, 
we see that the above term can be bounded as follows

\begin{equation}
|\mathrm{(vi)}| 
\lesssim  \int_{\R^2} \dx{k} \frac{1}{\abss{k}^4}e^{-\abss{k}^8 t} 
\lesssim (\sqrt[8]{t})^{-1}\stackrel{t\to \infty}{\to} 0 \, .
\end{equation}
For (vii), we proceed similarly to obtain
\begin{align}
&|\mathrm{(vii)}| 
=\left| \int_{\R^2}\int_0^\infty \dx{y}\dx{s} \psi_t(y) y_1
(((\cdot)_1 \hat{L}^*\partial_1^2\hat{\psi}_s)*F)(y) \right| \\
&\lesssim \int_{\R^2}\int_0^\infty \dx{k}\dx{s} \left|
\partial_{k_1} e^{-|k|_\s^8 t} 
\partial_{k_1}\big( k_1^2(-k_0+k_1^4)e^{-|k|_\s^8}\big) \mathcal{F}F(k) \right| \\
&\lesssim \int_{\R^2} \int_{0}^\infty \dx{k} \dx{s} |\mathcal{F}F(k)||k_1|^7t e^{-\abss{k}^8 t}\left( (|k_0||k_1| \+ |k_1|^5) \+ |k_1|^7 s  (|k_0||k_1|^2 \+ |k_1|^6) \right) e^{-\abss{k_1}^8 s} \, .
\end{align}
Integrating in $s$ and using that $\mathcal{F}F$ is a Schwartz function, 
we bound the above quantity as follows
\begin{equation}
|\mathrm{(vii)}| 
\lesssim \int_{\R^2} \dx{k} \abss{k}^4t e^{-\abss{k}^8 t} 
\lesssim (\sqrt[8]{t})^{-1} \stackrel{t\to \infty}{\to}0 \, .
\end{equation}
We are thus left only with (v) which we treat in the following manner 
\begin{align}
{\rm (v)}
= \- \int_{\R^2} \int_0^\infty \dx{z}\dx{s} 
(\hat{L}^* \partial_1\hat{\psi}_s)(z) z_1 F(z)\, ,
\end{align}
where we have used the fact that $\int_{\R^2}\dx{y}y_1 \psi_t(y) =0$. Applying Plancherel's identity and using the explicit form of the Fourier transform of $z_1$, we obtain
\begin{align}
{\rm (v)}
&= - \int_{\R^2} \int_0^\infty \dx{k} \dx{s} 
(-i2\pi k_0  + m_0(2\pi k_1)^4 )(i 2\pi k_1) e^{-((2\pi k_0)^2+ m_0^2(2\pi k_1)^8) s } \frac{i}{2\pi}\partial_{k_1} \mathcal{F}F(k) \\
&=  \int_{\R^2} \dx{k} \frac{k_1 }{(2 \pi k_0)^2 + m_0^2(2 \pi k_1)^8} 
\left(-i 2\pi k_0 + m_0 (2\pi k_1)^4\right) \partial_{k_1}\mathcal{F}F(k) \, ,
\end{align}
which implies~\eqref{eq:value2}. 

Again, we now choose $C$ and $\varphi_\tau$ to be as in the statement of the theorem and rescale as we did before to obtain the following
\begin{align}
&\hat{c}_{2f_1}(a_0)\\
&=\lim_{t \to \infty} \E \hat{\tilde\Pi}_{x(2f_1 + g_{(0,1)})}^- \\
&=(\sqrt[8]{\tau})^{2\alpha-2}\underbrace{
\frac{1}{(2\pi)^2}\int_{\R^2} \dx{k} \frac{k_1^4}{(k_0^2 + k_1^8)(m_0^2k_0^2 + k_1^8)^{\frac{2\alpha -1 }{8}}} \left(\frac{8k_1^8}{k_0^2 + k_1^8} -5\right) e^{ -k_1^8}}_{=:C_{\alpha,2}(m_0)}\\
&\,+ \frac{(\sqrt[8]{\tau})^{2\alpha-2}}{(2\pi)^2}\int_{\R^2} \dx{k} \frac{k_1^4}{(k_0^2 + k_1^8)(m_0^2k_0^2 + k_1^8)^{\frac{2\alpha -1 }{8}}} \left(\frac{8k_1^8}{k_0^2 + k_1^8} -5\right) e^{ -k_1^8} (e^{-m_0^2 k_0^2 \tau^{\eta-1}} -1)  \, .
\end{align}
The first term is as desired, with~\eqref{eq:exactvalue22} following by exactly computing the integral at $\alpha=\frac12$.
We control the second term as follows
\begin{align}
&\frac{(\sqrt[8]{\tau})^{2\alpha-2}}{(2\pi)^2}\left|\int_{\R^2} \dx{k} \frac{k_1^4}{(k_0^2 + k_1^8)(m_0^2k_0^2 + k_1^8)^{\frac{2\alpha -1 }{8}}} \left(\frac{8k_1^8}{k_0^2 + k_1^8} -5\right) e^{ -k_1^8} (e^{-m_0^2 k_0^2 \tau^{\eta-1}} -1) \right|\\
&\lesssim \frac{(\sqrt[8]{\tau})^{2\alpha-2}}{m_0^{\frac{2\alpha-1}{4}}} \int_{\R} \dx{k_0} \frac{|e^{-m_0^2 k_0^2 \tau^{\eta-1}} -1|}{(k_0^2)^{1 +\frac{2\alpha -1 }{8}}}  \lesssim m_0 (\sqrt[8]{\tau})^{2\alpha-2 + (\eta-1)(2 \alpha + 3)} \, . 
\end{align}

\textbf{Step 3} ($\beta=2e_1+2f_0+g_{(0,1)}$)\textbf{.}
We move on to treat $\hat{\Pi}_{x(2e_1 + 2 f_0 + g_{(0,1)})}^-$, where we note that by \cref{rem:hierarchy} and $\hat{c}_{e_1+f_0}=0$ as a consequence of $e_1+f_0$ not being an element of the set in \eqref{badMultiindices}, we have 
\begin{align}
&\E \left[\hat{\tilde\Pi}_{0 (2 e_1 +2f_0+ g_{(0,1)}) t }^-(0)\right]  \\
&=\E \int_{\R^2} \dx{y} \psi_{t}(y) y_1  (\partial_1^3 \hat{\Pi}_{0 (e_1 + 2f_0)}) (y) \\
&\,+  \E \int_{\R^2} \dx{y} \psi_{t}(y) \left( \hat{\Pi}_{0 (e_1+f_0+ g_{(0,1)})}(y) (\partial_1^3 \hat{\Pi}_{0f_0}) (y) + \hat{\Pi}_{0f_0}(y) (\partial_1^3 \hat{\Pi}_{0 (e_1+f_0+ g_{(0,1)})} ) (y)\right)  \, .
\end{align}
For the first term on the right hand side, we simply apply the bounds from~\cref{thm:main} and the scaling of $\psi_t$ (see~\eqref{eq:psiscaling}), 
where we note that $|e_1+2f_0|=2\alpha$ and $\alpha<1$, to obtain
\begin{equation}
 \left|\E \int_{\R^2} \dx{y} \psi_{t}(y) y_1  (\partial_1^3 \hat{\Pi}_{0 (e_1 + 2f_0)}) (y)\right|  
\lesssim (\sqrt[8]{t})^{2\alpha -2} \stackrel{ t \to \infty}{\to} 0 \, .
\end{equation}
We deal with the second term by using the solution formula~\eqref{eq:solutionformula} as follows
\begin{align}
&\E \int_{\R^2} \dx{y} \psi_{t}(y) \left( \hat{\Pi}_{0 (e_1+f_0+ g_{(0,1)})}(y) (\partial_1^3 \hat{\Pi}_{0f_0}) (y) + \hat{\Pi}_{0f_0}(y) (\partial_1^3 \hat{\Pi}_{0 (e_1+f_0+ g_{(0,1)})} ) (y)\right)  \\
&= \E \int_{\R^2} \dx{y} \psi_{t}(y) \partial_1^3 \left(\hat{\Pi}_{0f_0} \hat{\Pi}_{0 (e_1+f_0+ g_{(0,1)})} \right)  (y)  \\
&\,- 3 \E \int_{\R^2} \dx{y} \psi_{t}(y) \partial_1 \left( \partial_1 \hat{\Pi}_{0f_0} \partial_1\hat{\Pi}_{0 (e_1+f_0+ g_{(0,1)})} \right)  (y)  \\
&= \E \int_{\R^2} \dx{y} \psi_{t}(y) \partial_1^3 \left(\hat{\Pi}_{0f_0} \hat{\Pi}_{0 (e_1+f_0+ g_{(0,1)})} \right)  (y) 
\\& \,  + 3 \E \int_{\R^2} \dx{y} \partial_1 \psi_{t}(y) 
\int_0^\infty \dx{s_1} \hat{L}^* \partial_1^2 \hat{\psi}_{s_1} * \xi_\tau (y) \\  
&\,\times  \left(\int_0^\infty \dx{s_2} \hat{L}^* \partial_1^2 \hat{\psi}_{s_2} * \left((\cdot)_1\partial_1^3\hat{\Pi}_{0f_0}\right)(y) - \hat{L}^* \partial_1^2 \hat{\psi}_{s_2} * \left((\cdot)_1\partial_1^3\hat{\Pi}_{0f_0}\right)(0)\right)  \\
&=  \E \int_{\R^2} \dx{y} \psi_{t}(y) \partial_1^3 \left(\hat{\Pi}_{0f_0} \hat{\Pi}_{0 (e_1+f_0+ g_{(0,1)})} \right)  (y) \\
&\, +3 \int_{\R^2} \int_0^\infty \int_0^\infty \int_0^\infty \dx{y} \dx{s_1}\dx{s_2} \dx{s_3} \partial_1  \psi_{t}(y)  \int_{\R^2} \int_{\R^2} \int_{\R^2} \dx{z} \dx{x}\dx{v} x_1  \\ 
& \,\times \bar{\psi}_{s_1}(y-z) \big( \bar{\psi}_{s_2}(y-x)- \bar{\psi}_{s_2}(-x)\big) \tilde{\psi}_{s_3}(x-v) F(z-v)   \, ,
\end{align}
where $\bar \psi_s = \partial_1^2 \hat{L}^* \hat{\psi}_s$ and $\tilde{\psi}=\partial_1^4 \hat{L}^* \hat{\psi}_s$. The first term on the right hand side goes to $0$ as $t\to\infty$ by $\alpha<1$ by applying the bounds from~\cref{thm:main} (note that $|f_0|=\alpha$ and $|e_1+f_0+g_{(0,1)}|=\alpha+1$). We now deal with the second term which we can rewrite as follows
\begin{align}
&3 \int_{\R^2} \int_0^\infty \int_0^\infty \int_0^\infty \dx{y} \dx{s_1}\dx{s_2} \dx{s_3} \partial_1  \psi_{t}(y)  \int_{\R^2} \int_{\R^2} \int_{\R^2} \dx{z} \dx{x}\dx{v} x_1  \\ 
& \, \times \bar{\psi}_{s_1}(y-z) \big( \bar{\psi}_{s_2}(y-x)- \bar{\psi}_{s_2}(-x)\big) \tilde{\psi}_{s_3}(x-v) F(z-v) \\
&= 3 \int_{\R^2} \int_0^\infty \int_0^\infty \int_0^\infty \dx{y} \dx{s_1}\dx{s_2} \dx{s_3} \partial_1  \psi_{t}(y)  \int_{\R^2} \int_{\R^2} \int_{\R^2} \dx{z} \dx{x}\dx{v} x_1 \\
& \, \times \bar{\psi}_{s_1}(y-z) \bar{\psi}_{s_2}(y-x) \tilde{\psi}_{s_3}(x-v) F(z-v) \\
&\, - 3 \int_{\R^2} \int_0^\infty \int_0^\infty \int_0^\infty \dx{y} \dx{s_1}\dx{s_2} \dx{s_3} \partial_1  \psi_{t}(y)  \int_{\R^2} \int_{\R^2} \int_{\R^2} \dx{z} \dx{x}\dx{v} x_1  \\
&\,\times \bar{\psi}_{s_1}(y-z) \bar{\psi}_{s_2}(-x) \tilde{\psi}_{s_3}(x-v) F(z-v) \\
&= 3 \int_{\R^2} \int_0^\infty \int_0^\infty \int_0^\infty \dx{y} \dx{s_1}\dx{s_2} \dx{s_3} \partial_1  \psi_{t}(y)  \int_{\R^2} \int_{\R^2} \int_{\R^2} \dx{z} \dx{x}\dx{v} (x_1 -y_1 +y_1) \\
&\,\times \bar{\psi}_{s_1}(y-z) \bar{\psi}_{s_2}(y-x) \tilde{\psi}_{s_3}(x-v) F(z-v)  \\
& \, - 3  \int_{\R^2} \int_0^\infty \int_0^\infty \int_0^\infty \dx{y} \dx{s_1}\dx{s_2} \dx{s_3} \partial_1  \psi_{t}(y)  \int_{\R^2} \int_{\R^2} \int_{\R^2} \dx{z} \dx{x}\dx{v} x_1 \\
&\,\times \bar{\psi}_{s_1}(y-z) \bar{\psi}_{s_2}(-x) \tilde{\psi}_{s_3}(x-v) F(z-v) \\
&= \underbrace{3 \int_{\R^2} \int_0^\infty \int_0^\infty \int_0^\infty \dx{y} \dx{s_1}\dx{s_2} \dx{s_3} y_1  \partial_1  \psi_{t}(y)   \int_{\R^2} \dx{v}  \left( \bar{\psi}_{s_1}*F\right)(v) \left(\bar{\psi}_{s_2}* \tilde{\psi}_{s_3}\right) (v)}_{\rm (viii)}   \\
& \, +  \underbrace{3  \int_{\R^2} \int_0^\infty \hspace{-.5ex}
\int_0^\infty \hspace{-.5ex} \int_0^\infty \hspace{-.5ex} \int_{\R^2} \dx{y} \dx{s_1}\dx{s_2} \dx{s_3} \dx{z} \partial_1  \psi_{t}(y)   \bar{\psi}_{s_1}(y-z) \left(((\cdot)_1\bar{\psi}_{s_2})*\tilde{\psi}_{s_3}*F\right) (-z) }_{\rm (ix)} \, ,
\end{align}
where for the term (viii) we have used the fact that the term involving $(x_1-y_1)$ is independent of $y$ and that $\int_{\R^2} \dx{y} \partial_1 \psi_t (y)=0$. We bound the term (ix) as follows
\begin{align}
 |{\rm (ix)}|
 &=\left|3  \int_{\R^2} \int_0^\infty \hspace{-.5ex}
 \int_0^\infty \hspace{-.5ex} \int_0^\infty \hspace{-.5ex} \int_{\R^2} \dx{y} \dx{s_1}\dx{s_2} \dx{s_3} \dx{z} (\partial_1  \psi_{t} * \bar{\psi}_{s_1}) (z) \left(((\cdot)_1\bar{\psi}_{s_2} )*\tilde{\psi}_{s_3}*F\right) (z) \right|  \\
&\lesssim \int_{\R^2} \int_0^\infty \int_0^\infty \int_0^\infty  \dx{k} \dx{s_1}\dx{s_2} \dx{s_3} e^{-\abss{k}^8t} |k_1|^7(|k_0| + |k_1|^4)^3 \\ 
&\,\times (s_2|k_1|^9+|k_1| )e^{-\abss{k}^8 (s_1+s_2+s_3)}|\mathcal{F}F(k)|  \\
& \lesssim   \int_{\R^2} \frac{e^{-\abss{k}^8t}}{\abss{k}^4}  \lesssim (\sqrt[8]{t})^{-1} \stackrel{t \to \infty}{\to}0  \, ,
\end{align} 
where we have again used the explicit forms of the Fourier transforms of $\psi$ and $x_1$ along with the fact that $F$ is a Schwartz function. The implicit constants depend on $\lambda,\Lambda$. We are thus left to treat the term (viii), which we do as follows
\begin{align}
\mathrm{(viii)} 
&= 3 \int_{\R^2} \int_0^\infty \int_0^\infty \int_0^\infty \dx{y} \dx{s_1}\dx{s_2} \dx{s_3} y_1  \partial_1  \psi_{t}(y)   \int_{\R^2} \dx{v}  \left( \bar{\psi}_{s_1}*F\right)(v) \left(\bar{\psi}_{s_2}* \tilde{\psi}_{s_3}\right) (v) \\
&= -3    \int_{\R^2} \dx{k} \frac{1}{((2\pi k_0)^2 + m_0^2 (2\pi k_1)^8)^2}(i 2 \pi k_0 + m_0 (2\pi k_1)^4 ) (2 \pi k_1)^8 \mathcal{F}F(k) \\
&= -3    \int_{\R^2} \dx{k} \frac{m_0(2 \pi k_1)^{12}}{((2\pi k_0)^2 + m_0^2 (2\pi k_1)^8)^2} \mathcal{F}F(k) \, ,
\end{align}
which completes the proof of~\eqref{eq:value3}. 

As before, choosing $C$ and $\varphi_\tau$ to be as in the statement of the theorem and rescaling we obtain
\begin{align}
\hat{c}_{2e_1 + 2f_0}(a_0)&=\lim_{t\to \infty}\E \left[\hat{\tilde\Pi}_{0 (2 e_1 +2f_0+ g_{(0,1)}) t }^-(0)\right]\\
&= \frac{(\sqrt[8]{\tau})^{2\alpha-2}}{m_0^{2}}\underbrace{\frac{-3}{(2\pi)^2}    \int_{\R^2} \dx{k} \frac{k_1^{12}}{(k_0^2 +k_1^8)^2(m_0^2 k_0^{2} +  k_1^8)^{ \frac{2\alpha-1}{8}}} e^{- k_1^8}}_{=:C_{\alpha,3}(m_0)} \\
& \,  - \frac{3(\sqrt[8]{\tau})^{2\alpha-2}}{(2\pi)^2m_0^{2}}    \int_{\R^2} \dx{k} \frac{k_1^{12}}{(k_0^2 +k_1^8)^2(m_0^2 k_0^{2} +  k_1^8)^{ \frac{2\alpha-1}{8}}} e^{-k_1^8} (e^{- m_0^2k_0^2\tau^{\eta-1}} -1) \, ,
\end{align}
where the value of $C_{\alpha,3}$ given in~\eqref{eq:exactvalue23} follows from computing the integral explicitly for $\alpha=\frac12$. The error term can be controlled as follows
\begin{align}
& \frac{3(\sqrt[8]{\tau})^{2\alpha-2}}{(2\pi)^2m_0^{2}} \left| \int_{\R^2} \dx{k} \frac{k_1^{12}}{(k_0^2+k_1^8)^2(k_0^{2} +  m_0^2k_1^8)^{\frac{2\alpha-1}{8}}} e^{-k_1^8} 
(e^{- m_0^2k_0^2\tau^{\eta-1}} -1)\right| \\
&\lesssim 
\frac{(\sqrt[8]{\tau})^{2\alpha-2}}{m_0^{2+\frac{2\alpha-1}{4}}} 
\int_{\R} \dx{k_0} \frac{|e^{- m_0^2k_0^2\tau^{\eta-1}} -1|}{(k_0^{2})^{2+ \frac{2\alpha-1}{8}}} 
\lesssim
m_0 (\sqrt[8]{\tau})^{2\alpha-2+(\eta-1)(11+2\alpha)} \, . \qedhere
\end{align}
\end{proof}


\appendix
\section{\texorpdfstring{\bf Proof of qualitative smoothness}{Proof of qualitative smoothness}} 
\label{app:smooth}

\begin{proof}[Proof of \cref{rem:smoothness}]
The estimates \eqref{boundedness_c} -- \eqref{smoothness_Pi-} are clear for purely polynomial mutltiindices. 
The remaining multiindices we treat by induction with respect to ~$\prec$.
The base case amounts to $\eqref{smoothness_Pi-}_{\beta=f_0}$, 
which is contained in Step~1 below. 
The induction step we split over the following four steps.
We show in Step~1 that 
$\eqref{boundedness_c}_{\prec\beta-g_{\n^i}}$ for all $i=1,\dots,d$ \& 
$\eqref{boundedness_Pi}_{\prec\beta}$ \& 
$\eqref{smoothness_Pi}_{\prec\beta}$ imply 
$\eqref{smoothness_Pi-}_{\beta}$. 
In Step~2 we prove that 
$\eqref{smoothness_Pi-}_{\beta}$ implies
$\eqref{smoothness_Pi}_\beta$. 
Step~3 establishes that 
$\eqref{smoothness_Pi}_\beta$ implies 
$\eqref{boundedness_Pi}_\beta$, 
and finally in Step~4 we obtain that 
$\eqref{boundedness_c}_{\prec\beta-g_{\n^i}}$ for all $i=1,\dots,d$ \& 
$\eqref{boundedness_Pi}_{\prec\beta}$ imply 
$\eqref{boundedness_c}_{\beta-g_{\n^i}}$ for all $i=1,\dots,d$. 

\medskip

{\bf Step 1.} 
We show 
$\eqref{boundedness_c}_{\prec\beta-g_{\n^i}}$ \& 
$\eqref{boundedness_Pi}_{\prec\beta}$ \& 
$\eqref{smoothness_Pi}_{\prec\beta}$ together with 
$\eqref{estPi}_{\prec\beta}$ imply 
$\eqref{smoothness_Pi-}_{\beta}$.
We only give the proof for $\n=\0$, 
the proof for $|\n|=1$ is analogous by using Leibniz rule.
To obtain \eqref{smoothness_Pi-}, we estimate the individual components of $\Pi^-_x$ separately, and start with $\sum \a_k \Pi_x^k\nabla\Delta\Pi_x$. 
We rewrite the $\beta$-component of its increment as
\begin{align}
&\sum_{k\geq0} \sum_{\beta_1+\beta_2=\beta} 
\big( \a_k\Pi_x^k(y) - \a_k\Pi_x^k(z) \big)_{\beta_1} \nabla\Delta\Pi_{x\beta_2}(y) 
\\ 
&\,+ \sum_{k\geq0}\sum_{\beta_1+\beta_2=\beta} 
(\a_k\Pi_x^k(z))_{\beta_1} \big( \nabla\Delta\Pi_x(y)-\nabla\Delta\Pi_x(z) \big)_{\beta_2},
\end{align}
where we note that $\beta_1,\beta_2\prec\beta$.
Thus we can estimate the $\E^\frac{1}{p}|\cdot|^p$-norm of the first line
as in \eqref{mt05} with $\eqref{estPi}_{\prec\beta}$ and  $\eqref{smoothness_Pi}_{\prec\beta}$, and with $\eqref{boundedness_Pi}_{\prec\beta}$
by 
\begin{equation}
\sum_{\beta_1+\beta_2=\beta} 
|y-z|_\s^\alpha (|x-y|_\s+|x-z|_\s)^{|\beta_1|-2\alpha} 
(\sqrt[8]{\tau})^{\alpha-3} (\sqrt[8]{\tau}+|x-y|_\s)^{|\beta_2|-\alpha}, 
\end{equation}
which since $|\cdot|-\alpha$ is additive is estimated by the right hand side of \eqref{smoothness_Pi-}.
Similarly, the $\E^\frac{1}{p}|\cdot|^p$-norm of the second line is with 
$\eqref{estPi}_{\prec\beta}$ and $\eqref{smoothness_Pi}_{\prec\beta}$ estimated by 
\begin{equation}
\sum_{\beta_1+\beta_2=\beta} 
|x-z|_\s^{|\beta_1|-\alpha} (\sqrt[8]{\tau})^{-3} 
(\sqrt[8]{\tau}+|x-y|_\s+|x-z|_\s)^{|\beta_2|-\alpha} |y-z|_\s^\alpha,
\end{equation}
which is as before estimated by the right hand side of \eqref{smoothness_Pi-}.

\medskip

We turn to $\sum_\ell \b_\ell \Pi_x^\ell \xi_\tau$, 
and rewrite the $\beta$-component of its increment as
\begin{equation}
\sum_{\ell\geq0} \big(\b_\ell\Pi_x^\ell(y)-\b_\ell\Pi_x^\ell(z)\big)_\beta 
\xi_\tau(y) 
+ \sum_{\ell\geq0} (\b_\ell\Pi_x^\ell(z))_\beta (\xi_\tau(y)-\xi_\tau(z)).
\end{equation}
The $\E^\frac{1}{p}|\cdot|^p$-norm of the first sum is as in \eqref{mt06} with 
$\eqref{estPi}_{\prec\beta}$ and $\eqref{smoothness_Pi}_{\prec\beta}$, 
and with \eqref{bound_xi} estimated by 
\begin{equation}
|y-z|_\s^\alpha (|x-y|_\s+|x-z|_\s)^{|\beta|-2\alpha} (\sqrt[8]{\tau})^{\alpha-3} ,
\end{equation}
which is estimated by the right hand side of \eqref{smoothness_Pi-}.
For the second sum we first note that 
\begin{equation}
\E^\frac{1}{p}|\xi_\tau(y)-\xi_\tau(z)|^p 
\lesssim (\sqrt[8]{\tau})^{-3} |y-z|_\s^\alpha,
\end{equation}
which follows from the mean-value theorem and \eqref{bound_xi}. 
Together with $\eqref{estPi}_{\prec\beta}$ we therefore obtain a bound 
of the $\E^\frac{1}{p}|\cdot|^p$-norm of the second sum by 
\begin{equation}
|x-z|_\s^{|\beta|-\alpha} (\sqrt[8]{\tau})^{-3} |y-z|_\s^\alpha,
\end{equation}
which is once more estimated by the right hand side of \eqref{smoothness_Pi-}.

\medskip

We turn to $\sum_m \tfrac{1}{m!} \Pi_x^m \nabla\Pi_x (D^{(\0)})^m c$, 
and rewrite the $\beta$-component of its increment as 
\begin{align}
&\sum_{m\geq0} \sum_{\beta_1+\beta_2+\beta_3=\beta} 
\big(\Pi_x^m(y)-\Pi_x^m(z)\big)_{\beta_1} \nabla\Pi_{x\beta_2}(y) 
((D^{(\0)})^m c)_{\beta_3} \\ 
&\,+ \sum_{m\geq0} \sum_{\beta_1+\beta_2+\beta_3=\beta} 
(\Pi_x^m(z))_{\beta_1} \big(\nabla\Pi_x(y)-\nabla\Pi_x(z)\big)_{\beta_2} 
((D^{(\0)})^m c)_{\beta_3} .
\end{align}
We note that only $c_\gamma$-components with $\gamma\prec\beta-g_{\n^i}$ can appear due to \cref{lem:tri1}~(i), 
and by \eqref{D0props} in this case $|\gamma|=|\beta_3|-m\alpha$.
We thus estimate the $\E^\frac{1}{p}|\cdot|^p$-norm of the first line as in \eqref{mt05} (with $\beta$ replaced by $\beta_1+e_m$) with $\eqref{estPi}_{\prec\beta}$ and 
$\eqref{smoothness_Pi}_{\prec\beta}$, 
and with 
$\eqref{boundedness_Pi}_{\prec\beta}$ and $\eqref{boundedness_c}_{\prec\beta-g_{\n^i}}$ by a linear combination of terms of the form 
\begin{align}
\sum_{\beta_1+\beta_2+\beta_3=\beta}
|y-z|_\s^{\alpha} (|x-y|_\s+|x-z|_\s)^{|\beta_1+e_m|-2\alpha} 
(\sqrt[8]{\tau})^{\alpha-1} (\sqrt[8]{\tau}+|x-y|_\s)^{|\beta_2|-\alpha} \\
\times (\sqrt[8]{\tau})^{|\beta_3|-m\alpha-\alpha-2} .
\end{align}
By $|\beta_1+e_m|=|\beta_1|+m\alpha$, and since $|\beta_3|-m\alpha=|\gamma|\geq0$, 
this is further bounded by 
\begin{equation}\label{mt07}
\sum_{\beta_1+\beta_2+\beta_3=\beta} 
|y-z|_\s^\alpha (\sqrt[8]{\tau})^{-3}
(\sqrt[8]{\tau}+|x-y|_\s+|x-z|_\s)^{|\beta_1|+|\beta_2|+|\beta_3|-3\alpha} ,
\end{equation}
which by additivity of $|\cdot|-\alpha$ is estimated by the right hand side of \eqref{smoothness_Pi-}.
For the $\E^\frac{1}{p}|\cdot|^p$-norm of the second line we proceed similarly, 
and use $\eqref{estPi}_{\prec\beta}$, $\eqref{smoothness_Pi}_{\prec\beta}$ and 
$\eqref{boundedness_c}_{\prec\beta-g_{\n^i}}$ to obtain an estimate by a linear combination of terms of the form 
\begin{align}
\sum_{\beta_1+\beta_2+\beta_3=\beta}
|x-z|_\s^{|\beta_1|+(m-1)\alpha} 
(\sqrt[8]{\tau})^{-1} (\sqrt[8]{\tau}+|x-y|_\s+|x-z|_\s)^{|\beta_2|-\alpha} |y-z|_\s^\alpha \\
\times (\sqrt[8]{\tau})^{|\beta_3|-m\alpha-\alpha-2}, 
\end{align}
which as before is estimated by \eqref{mt07} and therefore by the right hand side of \eqref{smoothness_Pi-}.

\medskip

{\bf Step 2.}
We show 
$\eqref{smoothness_Pi-}_{\beta}$ together with 
$\eqref{estPi}_\beta$ \& 
$\eqref{estGamma}_\beta$ \&
$\eqref{estPiminus}_\beta$ imply 
$\eqref{smoothness_Pi}_\beta$.
For the rest of Step~2 we fix $\n$ with $|\n|\leq4$. 
First, we claim that \eqref{smoothness_Pi} follows from 
\begin{equation}\label{campanato}
\E^\frac{1}{p}| \partial^\m \partial^\n \Pi_{x\beta t}(y) |^p 
\lesssim (\sqrt[8]{\tau})^{-|\n|} 
(\sqrt[8]{t}+\sqrt[8]{\tau}+|x-y|_\s)^{|\beta|-\alpha} 
(\sqrt[8]{t})^{\alpha-|\m|} 
\textnormal{ for all }\m\neq\0.
\end{equation}
Indeed, rewriting $\partial^\n\Pi_{x\beta}(y)-\partial^\n\Pi_{x\beta}(z)$ as 
\begin{equation}
\big( \partial^\n\Pi_{x\beta}(y) - \partial^\n\Pi_{x\beta t}(y) \big)
+ \big( \partial^\n\Pi_{x\beta t}(y) - \partial^\n\Pi_{x\beta t}(z) \big)
+ \big( \partial^\n\Pi_{x\beta t}(z) - \partial^\n\Pi_{x\beta}(z) \big) ,
\end{equation}
we can estimate the $\E^\frac{1}{p}|\cdot|^p$-norm of the first and the third terms by using \eqref{kernel} and \eqref{campanato} by 
\begin{equation}
\int_0^t \dx{s} \, \E^\frac{1}{p}| L L^* \partial^\n \Pi_{x\beta s}(y) |^p 
\lesssim \int_0^t \dx{s} \, (\sqrt[8]{\tau})^{-|\n|} 
(\sqrt[8]{s}+\sqrt[8]{\tau}+|x-y|_\s)^{|\beta|-\alpha} 
(\sqrt[8]{s})^{\alpha-8} .
\end{equation}
By $\alpha>0$ this expression is integrable at $0$ and with the choice $\sqrt[8]{t}=|y-z|_\s$ thus estimated by the right hand side of \eqref{smoothness_Pi}.
For the $\E^\frac{1}{p}|\cdot|^p$-norm of the second term we obtain by the mean-value theorem (mind the anisotropy) and \eqref{campanato} an estimate by 
\begin{equation}
(\sqrt[8]{\tau})^{-|\n|} 
(\sqrt[8]{t}+\sqrt[8]{\tau}+|x-y|_\s+|x-z|_\s^{|\beta|-\alpha}
\big( (\sqrt[8]{t})^{\alpha-4} |y-z|_\s^4 
+ (\sqrt[8]{t})^{\alpha-1} |y-z|_\s \big) ,
\end{equation}
which again by the choice $\sqrt[8]{t}=|y-z|_\s$ is estimated by the right hand side of \eqref{smoothness_Pi}. 

\medskip

We further claim that it is enough to establish \eqref{campanato} along the diagonal $y=x$ in form of 
\begin{equation}\label{campanato_diagonal}
\E^\frac{1}{p}| \partial^\m \partial^\n \Pi_{x\beta t}(x) |^p 
\lesssim (\sqrt[8]{\tau})^{-|\n|} 
(\sqrt[8]{t}+\sqrt[8]{\tau})^{|\beta|-\alpha} 
(\sqrt[8]{t})^{\alpha-|\m|} 
\quad\textnormal{for all }\m\neq\0.
\end{equation}
Indeed, using the recentering \eqref{recenter}, the estimate $\eqref{estGamma}_\beta$ of $\G_{xy}$ and \eqref{campanato_diagonal} we obtain 
\begin{equation}
\E^\frac{1}{p}|\partial^\m\partial^\n\Pi_{x\beta t}(y)|^p
\lesssim \sum_{|\gamma|\in \A\cap[\alpha,|\beta|]} |x-y|_\s^{|\beta|-|\gamma|} 
(\sqrt[8]{\tau})^{-|\n|} 
(\sqrt[8]{t}+\sqrt[8]{\tau})^{|\gamma|-\alpha} 
(\sqrt[8]{t})^{\alpha-|\m|} \, ,
\end{equation}
which is estimated by the right hand side of \eqref{campanato}.

\medskip

Before we prove \eqref{campanato_diagonal}, we note that it is enough to establish \eqref{campanato_diagonal} in the regime $t<\tau$.
Indeed, we obtain from $\eqref{estPi}_\beta$, 
the semigroup property \eqref{eq:semi} and the moment bound \eqref{eq:kerbound}
the estimate 
\begin{equation}\label{mt08}
\E^\frac{1}{p}| \partial^\m \partial^\n \Pi_{x\beta t}(x) |^p 
\lesssim (\sqrt[8]{t})^{|\beta|-|\n|-|\m|} \, , 
\end{equation}
which for $t\geq\tau$ is stronger than \eqref{campanato_diagonal}.

\medskip

We now turn to the proof of \eqref{campanato_diagonal} for $t<\tau$, 
where we distinguish the two cases $|\beta|<1+|\n|$ and $|\beta|\geq1+|\n|$. 
For the latter, we appeal again to \eqref{mt08} and use that $|\beta|-\alpha-|\n|\geq|\beta|-1-|\n|\geq0$ to see 
\begin{equation}
\E^\frac{1}{p}| \partial^\m \partial^\n \Pi_{x\beta t}(x) |^p 
\lesssim (\sqrt[8]{t})^{\alpha-|\m|}
(\sqrt[8]{\tau})^{|\beta|-\alpha-|\n|} \, ,
\end{equation}
which is estimated by the right hand side of \eqref{campanato_diagonal}.
For the former case, we appeal to the integral representation \eqref{eq:solutionformula} 
where we note that due to the presence of $\partial^\m\partial^\n$ the Taylor polynomial drops out, hence 
\begin{equation}
\E^\frac{1}{p}| \partial^\m\partial^\n\Pi_{x\beta t}(x) |^p
\lesssim \int_t^\infty \dx{s}\, \E^\frac{1}{p}|\partial^\m\partial^\n L^* \nabla\cdot\Pi^-_{x\beta s}(x)|^p \, .
\end{equation}
We split the integral from $t$ to $\tau$ and from $\tau$ to $\infty$. 
For the latter, we appeal to the semigroup property \eqref{eq:semi}, 
the estimate $\eqref{estPiminus}_\beta$ of $\Pi_{x\beta}^-$ and the moment bound \eqref{eq:kerbound} to obtain 
\begin{equation}
\int_\tau^\infty \dx{s}\, \E^\frac{1}{p}|\partial^\m\partial^\n 
L^* \nabla\cdot\Pi^-_{x\beta s}(x)|^p
\lesssim \int_\tau^\infty \dx{s}\, 
(\sqrt[8]{s})^{|\beta|-8-|\n|-|\m|} \, .
\end{equation}
Since $|\beta|-|\n|-|\m|\leq|\beta|-|\n|-1<0$, 
the integral is convergent at $s=\infty$ and bounded by 
\begin{equation}
(\sqrt[8]{\tau})^{|\beta|-|\n|-|\m|}
\lesssim
(\sqrt[8]{\tau})^{-|\n|}
(\sqrt[8]{t}+\sqrt[8]{\tau})^{|\beta|-\alpha}
(\sqrt[8]{\tau})^{\alpha-|\m|} \, .
\end{equation} 
Since $\alpha-|\m|<0$ and $t<\tau$, 
this is estimated by the right hand side of \eqref{campanato_diagonal}. 
For the integral from $t$ to $\tau$ we make the convolution with $\psi_s$ explicit 
\begin{equation}
\partial^\m\partial^\n L^* \nabla\cdot\Pi^-_{x\beta s}(x) 
= \int \dx{y}\, \partial^\m\partial^\n L^* \psi_s(x-y) 
\big( \nabla\cdot\Pi^-_{x\beta}(y) - \nabla\cdot\Pi^-_{x\beta}(x) \big) \, ,
\end{equation}
where we could smuggle in the term $\nabla\cdot\Pi^-_{x\beta}(x)$ since the integral over derivatives of $\psi_s$ vanishes. 
We thus obtain from $\eqref{smoothness_Pi-}_\beta$ and the moment bound \eqref{eq:kerbound} 
\begin{equation}
\int_t^\tau \dx{s}\, \E^\frac{1}{p}|\partial^\m\partial^\n 
L^* \nabla\cdot\Pi^-_{x\beta s}(x)|^p
\lesssim \int_t^\tau \dx{s}\, 
(\sqrt[8]{s})^{-|\n|-|\m|-4}
(\sqrt[8]{\tau})^{-4}
(\sqrt[8]{s}+\sqrt[8]{\tau})^{|\beta|-\alpha}
(\sqrt[8]{s})^{\alpha} \, ,
\end{equation}
which can be estimated by 
\begin{equation}
(\sqrt[8]{\tau})^{-4} 
(\sqrt[8]{t}+\sqrt[8]{\tau})^{|\beta|-\alpha}
\big( (\sqrt[8]{t})^{\alpha-|\n|-|\m|+4}
+ (\sqrt[8]{\tau})^{\alpha-|\n|-|\m|+4} \big) \, .
\end{equation}
Since $|\n|\leq4$ and $t<\tau$ we have $(\sqrt[8]{\tau})^{-4} \leq (\sqrt[8]{\tau})^{-|\n|} (\sqrt[8]{t})^{|\n|-4} $, hence the above expression is estimated by 
\begin{equation}
(\sqrt[8]{t}+\sqrt[8]{\tau})^{|\beta|-\alpha}
\big( (\sqrt[8]{\tau})^{-|\n|} (\sqrt[8]{t})^{\alpha-|\m|}
+ (\sqrt[8]{\tau})^{\alpha-|\n|-|\m|} \big) \, , 
\end{equation}
which since $\alpha-|\m|<0$ and $t<\tau$ is estimated by the right hand side of \eqref{campanato_diagonal}.

\medskip

{\bf Step 3.}
We show 
$\eqref{smoothness_Pi}_\beta$ together with 
$\eqref{estPi}_{\preccurlyeq\beta}$ \& $\eqref{estGamma}_\beta$ imply 
$\eqref{boundedness_Pi}_\beta$. 
We fix $\n$ with $1\leq|\n|\leq4$ and rewrite 
\begin{equation}
\partial^\n\Pi_{x\beta}(y) 
= \partial^\n\Pi_{x\beta\tau}(y) 
+ \int \dx{z}\, \psi_\tau(y-z) \big( \partial^\n\Pi_{x\beta}(y) 
- \partial^\n\Pi_{x\beta}(z)\big) \, .
\end{equation}
Since $\n\neq\0$ and the integral over derivatives of $\psi$ vanish, 
the first right hand side term equals 
\begin{equation}
\int \dx{z}\, \partial^\n\psi_\tau(y-z) \big(\Pi_{x\beta}(z) - \Pi_{x\beta}(y) \big) \, ;
\end{equation}
its $\E^\frac{1}{p}|\cdot|^p$-norm is by the recentering \eqref{recenter}, 
the estimate $\eqref{estGamma}_\beta$ on $\G_{xy}$ 
and the estimate $\eqref{estPi}_{\preccurlyeq\beta}$ on $\Pi_x$ 
estimated by 
\begin{equation}
\sum_{|\gamma|\in \A\cap[\alpha,|\beta|]}
\int \dx{z}\, |\partial^\n\psi_\tau(y-z)| |x-y|_\s^{|\beta|-|\gamma|} 
|y-z|_\s^{|\gamma|} \, ,
\end{equation}
which by the moment bound \eqref{eq:kerbound} is estimated by the right hand side of \eqref{boundedness_Pi}. 
For the second right hand side term we appeal to $\eqref{smoothness_Pi}_\beta$
to bound its $\E^\frac{1}{p}|\cdot|^p$-norm by 
\begin{equation}
\int \dx{z}\, |\psi_\tau(y-z)| 
(\sqrt[8]{\tau})^{-|\n|}
(\sqrt[8]{\tau}+|x-y|_\s+|x-z|_\s)^{|\beta|-\alpha}
|y-z|_\s^{\alpha} \, ,
\end{equation}
which by the moment bound \eqref{eq:kerbound} is again bounded by the right hand side of \eqref{boundedness_Pi}. 

\medskip

{\bf Step 4.}
We show 
$\eqref{boundedness_c}_{\prec\beta-g_{\n^i}}$ \& 
$\eqref{boundedness_Pi}_{\prec\beta}$ together with 
$\eqref{estPi}_{\prec\beta}$ \&
$\eqref{estGamma}_\beta^{\gamma\neq\pp}$ \& 
$\eqref{estPiminus}_{\prec\beta}$ imply 
$\eqref{boundedness_c}_{\beta-g_{\n^i}}$. 
By \eqref{pop4} we can restrict to multiindices $\beta$ with $|\beta|<3$. 
Since for such multiindices $\E\Pi^-_{x\beta s}(x)\to0$ as $s\to\infty$ by the BPHZ-choice \eqref{eq:BPHZchoice}, 
we have 
\begin{equation}
c_{\beta-g_{\n^i}} \n^i 
= \int \dx{y}\, \psi_t(x-y) \E\big(\Pi^-_{x\beta}(y) + c_{\beta-g_{\n^i}}\n^i\big) 
+ \int_t^\infty \dx{s}\, \partial_s \E \Pi^-_{x\beta s}(x) \, ,
\end{equation}
where the choice $t=\tau$ will turn out to be convenient. 
The first term on the right hand side can be estimated by the same arguments as we estimated $\Pi^-_{x\beta}(y)-\Pi^-_{x\beta}(z)$ in Step~1, 
where now we are in the simpler setting of not dealing with increments, 
and were $\eqref{boundedness_c}_{\prec\beta-g_{\n^i}}$,  $\eqref{boundedness_Pi}_{\prec\beta}$ and $\eqref{estPi}_{\prec\beta}$ 
are sufficient due to \cref{lem:tri1}~(i). 
More precisely, the first right hand side term can be estimated by 
\begin{equation}
\int dy\, \psi_t(x-y) (\sqrt[8]{\tau})^{-3} (\sqrt[8]{\tau}+|x-y|_\s)^{|\beta|} \, ,
\end{equation}
which by the moment bound \eqref{eq:kerbound} and the choice $t=\tau$ is bounded by $(\sqrt[8]{\tau})^{|\beta|-3}$. 
To estimate the second right hand side term we appeal to \eqref{kernel} 
and the semigroup property \eqref{eq:semi} and rewrite 
\begin{equation}
\int_t^\infty \dx{s}\, \partial_s \E\Pi^-_{x\beta s}(x) 
= - \int_t^\infty \dx{s}\, \int \dx{y} \, LL^*\psi_{s/2}(x-y) \E\Pi^-_{x\beta s/2}(y) \, .
\end{equation}
By $(\G_{xy}\Pi^-_y)_\beta = \Pi^-_{x\beta}$, 
which is a consequence of \eqref{recenterPi-} and $|\beta|<3$, 
and since $\E\Pi_{y\beta s/2}(y)$ does not depend on $y$ and integrals over derivatives of $\psi$ vanish, we have furthermore 
\begin{equation}
\int_t^\infty \dx{s}\, \partial_s \E\Pi^-_{x\beta s}(x) 
= - \int_t^\infty \dx{s}\, \int \dx{y} \, LL^*\psi_{s/2}(x-y) 
\E\big( (\G_{xy}-\id) \Pi^-_{y s/2}(y) \big)_\beta \, .
\end{equation}
Using H\"older's inequality together with 
$\eqref{estGamma}_\beta^{\gamma\neq\pp}$ and 
$\eqref{estPiminus}_{\prec\beta}$, 
which is sufficient by the triangularity \eqref{tri3} of $\G_{xy}-\id$,
this expression is estimated by 
\begin{equation}
\int_t^\infty \dx{s}\, \int \dx{y}\, |LL^*\psi_{s/2}(x-y)| 
\sum_{|\gamma|\in \A\cap[\alpha,|\beta|)} 
|x-y|_\s^{|\beta|-|\gamma|} (\sqrt[8]{s})^{|\gamma|-3}
\lesssim 
\int_t^\infty \dx{s}\, (\sqrt[8]{s})^{|\beta|-3-8} \, 
\end{equation}
where we have used the moment bound \eqref{eq:kerbound} in the last inequality. 
Since $|\beta|-3<0$, this integral is convergent at $s=\infty$, 
and is bounded by $(\sqrt[8]{t})^{|\beta|-3}$. 
Again by the choice $t=\tau$ we obtain altogether
\begin{equation}
|c_{\beta-g_{\n^i}}| \lesssim (\sqrt[8]{\tau})^{|\beta|-3} \, .
\end{equation}
Relabelling $\tilde\beta=\beta-g_{\n^i}$ yields $|c_{\tilde\beta}|\lesssim (\sqrt[8]{\tau})^{|\tilde\beta+g_{\n^i}|}$, 
which by $|\tilde\beta+g_{\n^i}|=|\tilde\beta|+1-\alpha$ yields the desired \eqref{boundedness_c}.
\end{proof}


\section{\texorpdfstring{\bf Proof of analyticity}{Proof of analyticity}}
\label{app:analytic}

\begin{proof}[Proof of \eqref{eq:analytic}]
First note that \cref{thm:main} still holds true in the $\hat\cdot\,$-setting as well as in the $\bar\cdot\,$-setting, 
and the estimates \eqref{estPi} and \eqref{estGamma} on $\hat\Pi_{x\hat\beta}, \bar\Pi_{x\beta}$ and $(\hat\Gamma^*_{xy})_{\hat\beta}^{\hat\gamma}, (\bar\Gamma^*_{xy})_\beta^\gamma$ as well as the estimates \eqref{boundedness_c} and \eqref{boundedness_Pi} on $\hat c_{\hat\beta}, \bar c_{\beta}$ and $\partial^\n\hat\Pi_{x\hat\beta}, \partial^\n \bar\Pi_{x\beta}$ hold locally uniformly in $a_0$ for $\mathrm{Re}(a_0)<1$. 
This is an immediate consequence of the fact that the ``heat'' kernel $\hat\psi$ associated to 
$(\partial_0+(1-a_0)\Delta^2)(\partial_0+(1-a_0)\Delta^2)^*
=-\partial_0^2+|1-a_0|^2\Delta^4$ 
satisfies 
\begin{equation}
\hat\psi_t(a_0, x)=\psi_t \Big(x_0, \frac{x_1}{\sqrt[4]{|1-a_0|}},\dots,\frac{x_d}{\sqrt[4]{|1-a_0|}}\Big),
\end{equation}
where $\psi$ is the ``heat'' kernel associated to $(\partial_0+\Delta^2)(\partial_0+\Delta^2)^*$ from \eqref{kernel}. 
Hence the moment bound \eqref{eq:kerbound} holds also for $\hat\psi$, 
locally uniformly for $\mathrm{Re}(a_0)<1$.

\medskip

The analyticity expressed by \eqref{eq:analytic} is clear for purely polynomial $\hat\beta$, 
and we proceed by induction with respect to ~$\prec$ in the remaining multiindices. 
We show in Step~1 that $\eqref{eq:analytic_c}_{\prec\hat\beta-g_\n}$ \& 
$\eqref{eq:analytic_Pi}_{\prec\hat\beta}$ imply 
$\eqref{eq:analytic_c}_{\hat\beta-g_\n}$, 
and in Step~2 that $\eqref{eq:analytic_c}_{\preccurlyeq\hat\beta-g_\n}$ \&
$\eqref{eq:analytic_Pi}_{\prec\hat\beta}$ imply 
$\eqref{eq:analytic_Pi}_{\hat\beta}$. 
The base case amounts to establishing $\eqref{eq:analytic_Pi}_{\hat\beta=f_0}$, 
which is covered by Step~2 as we argue now. 
The proof of $\eqref{eq:analytic_Pi}_{\hat\beta}$ in Step~2 makes only use of \eqref{mt03}, which is true for $\hat\beta=f_0$ as can be seen from the componentwise form \eqref{piminuscomponents} of $\Pi^-$:
for $\hat k=0$ we have $\hat\Pi^-_{f_0} = \xi_\tau = \bar\Pi^-_{f_0}$; 
for $\hat k\geq1$ we have $\partial_{a_0}^{\hat k}\hat\Pi^-_{f_0}=0$, 
as well as $\bar\Pi^-_{f_0+\hat ke_0} - \nabla\Delta\bar\Pi_{x\hat\beta+(\hat k-1)e_0} = 0$. 

\medskip

\textbf{Step 1.}
We show $\eqref{eq:analytic_c}_{\prec\hat\beta-g_\n}$ \& $\eqref{eq:analytic_Pi}_{\prec\hat\beta}$ imply 
$\eqref{eq:analytic_c}_{\hat\beta-g_\n}$.
By \eqref{pop2} we may assume $\hat\beta=\hat\beta'+g_\n$ where $\hat\beta'$ has no polynomial components.
Fur such $\hat\beta$, we define $\hat{\tilde\Pi}^-_{x\hat\beta}$ by 
\begin{equation}\label{def_Pi_hat_tilde_minus}
\hat\Pi^-_{x\hat\beta}=\hat{\tilde\Pi}^-_{x\hat\beta}-\hat c_{\hat\beta-g_\n} \n.
\end{equation}
By Leibniz rule and using the notation $\hat\partial^{k} := \tfrac{1}{k!}\partial_{a_0}^k$, we obtain 
\begin{align}
\hat\partial^{\hat k} \hat{\tilde\Pi}^-_{x\hat\beta} 
&= \sum_{k\geq1} \sum_{\substack{e_k+\beta_1+\cdots+\beta_{k+1}=\hat\beta\\ \kay_1+\cdots+\kay_{k + 1}=\hat k}}
\hat\partial^{\kay_1} \hat\Pi_{x\beta_1} \cdots 
\hat\partial^{\kay_k} \hat\Pi_{x\beta_k} 
\hat\partial^{\kay_{k+1}} \nabla \Delta \hat\Pi_{x\beta_{k+1}} \\
&\, + \sum_{\ell\geq0}\sum_{\substack{f_\ell+\beta_1+\cdots+\beta_\ell=\hat\beta \\ \kay_1+\cdots+\kay_\ell=\hat k}} 
\hat\partial^{\kay_1} \hat\Pi_{x\beta_1} \cdots 
\hat\partial^{\kay_\ell}\hat\Pi_{x\beta_\ell}\xi_\tau \\
&\, - \sum_{m\geq1} \tfrac{1}{m!}
\sum_{\substack{\beta_1+\cdots+\beta_{m+2}=\hat\beta\\ \kay_1+\cdots+\kay_{m+2}=\hat k}} \hspace{-1ex}
\hat\partial^{\kay_1}\hat\Pi_{x\beta_1} \cdots 
\hat\partial^{\kay_m}\hat\Pi_{x\beta_m} 
\hat\partial^{\kay_{m+1}}\nabla\hat\Pi_{x\beta_{m+1}}
\hat\partial^{\kay_{m+2}}((\hat D^{(\0)})^m \hat c)_{\beta_{m+2}} \\
&\, - \sum_{\substack{\beta_1+\beta_2=\hat\beta\\ \beta_1\neq g_\n\\ \kay_1+\kay_2=\hat k}} \hat\partial^{\kay_1}\nabla\hat\Pi_{x\beta_1} 
\hat\partial^{\kay_2}\hat c_{\beta_2} \, . 
\end{align}
Note that \eqref{eq:analytic_Pi} implies 
$\hat\partial^{\hat k} \partial^\n \hat\Pi_{x\hat\beta} = \partial^\n \Pi_{x\hat\beta + \hat k e_0}$ for $1\leq|\n|\leq4$ with respect to the norm 
\begin{equation}
\sup_{y,t} (\sqrt[8]{t})^{-(\alpha-|\n|)} (\sqrt[8]{t}+|x-y|_\s)^{-(|\hat\beta|-\alpha)} 
\E^\frac{1}{p} | \partial^\n\hat \Pi_{x\hat\beta t}(y) |^p, 
\end{equation}
which follows from the locally uniform (in $a_0$) \eqref{estGamma}.
Together with the triangularity properties \cref{prop:choosing} and \cref{lem:tri1}~(i), we obtain from 
$\eqref{eq:analytic_c}_{\prec\hat\beta-g_\n}$ \& $\eqref{eq:analytic_Pi}_{\prec\hat\beta}$ 
\begin{align}
&\hat\partial^{\hat k} \hat{\tilde\Pi}^-_{x\hat\beta} \\
&= \sum_{k\geq1} \sum_{\substack{e_k+\beta_1+\cdots+\beta_{k+1}=\hat\beta\\ \kay_1+\cdots+\kay_{k + 1}=\hat k}}
\bar\Pi_{x \beta_1+\kay_1 e_0} \cdots 
\bar\Pi_{x\beta_k+\kay_k e_0} 
\nabla \Delta \bar\Pi_{x\beta_{k+1}+\kay_{k+1}e_0} \\
&\, + \sum_{\ell\geq0}\sum_{\substack{f_\ell+\beta_1+\cdots+\beta_\ell=\hat\beta \\ \kay_1+\cdots+\kay_\ell=\hat k}} 
\bar\Pi_{x\beta_1+\kay_1e_0} \cdots 
\bar\Pi_{x\beta_\ell+\kay_\ell e_0}\xi_\tau \\
&\, - \sum_{m\geq1} \tfrac{1}{m!}\sum_{\substack{\beta_1+\cdots+\beta_{m+2}=\hat\beta\\ \kay_1+\cdots+\kay_{m+2}=\hat k}} 
\bar\Pi_{x\beta_1+\kay_1 e_0} \cdots 
\bar\Pi_{x\beta_m+\kay_m e_0} 
\nabla\bar\Pi_{x\beta_{m+1}+\kay_{m+1}e_0}
\hat\partial^{\kay_{m+2}}((\hat D^{(\0)})^m \hat c)_{\beta_{m+2}} \\
&\, - \sum_{\substack{\beta_1+\beta_2=\hat\beta\\ \beta_1\neq g_\n\\ \kay_1+\kay_2=\hat k}} \nabla\bar\Pi_{x\beta_1+\kay_1 e_0} 
\bar c_{\beta_2+\kay_2 e_0} \, , 
\end{align}
with respect to 
\begin{equation}\label{eq:norm_Pi-_smooth_analytic}
\sup_{y,t} (\sqrt[8]{t})^{-(\alpha-3)} (\sqrt[8]{t}+|x-y|_\s)^{-(|\hat\beta|-\alpha)} 
\E^\frac{1}{p} | \hat \Pi^-_{x\hat\beta t}(y) |^p. 
\end{equation}
This establishes 
\begin{equation}\label{mt02}
\hat\partial^{\hat k} \hat{\tilde\Pi}^-_{x\hat\beta}
= \bar\Pi^-_{x\,\hat\beta+\hat k e_0} 
- \nabla\Delta\bar\Pi_{x\,\hat\beta+(\hat k-1) e_0}
+ \bar c_{\hat\beta-g_\n + \hat k e_0} \n 
\end{equation}
with respect to \eqref{eq:norm_Pi-_smooth_analytic} and with the understanding that the second right hand side term vanishes for $\hat k=0$, 
provided we show that for all $\kay\geq0$ and $m\geq0$ 
\begin{equation}\label{mt01}
\hat\partial^\kay ((\hat D^{(\0)})^m \hat c)_{\beta_{m+2}} 
= ((D^{(\0)})^m \bar c)_{\beta_{m+2}+\kay e_0} \, ,
\end{equation}
which we shall establish now by induction in $m$ and for $\beta_{m+2}$ 
replaced by an arbitrary $\beta\prec\hat\beta-g_\n$ with $\beta_a(k=0)=0$. 
This captures $\beta_{m+2}$, 
since by $m\geq1$ and $\length{\cdot}\geq\lambda$ we have 
$|\beta_{m+2}|_\prec 
= \length{\hat\beta}-\length{\beta_1}-\cdots-\length{\beta_{m+1}}
\leq \length{\hat\beta} - 2\lambda 
< \length{\hat\beta-g_\n}$.
The base case $m=0$ follows from $\eqref{eq:analytic_c}_{\prec\hat\beta-g_\n}$. 
For the induction step $m\leadsto m+1$ we argue as follows.
On the one hand, we have 
\begin{align}
&\hat\partial^\kay ((\hat D^{(\0)})^{m+1} \hat c)_{\beta} \\
&= \hat\partial^\kay \sum_\gamma (\hat D^{(\0)})_\beta^\gamma \big((\hat D^{(\0)})^m \hat c\big)_\gamma \\
&= \hat\partial^\kay \sum_\gamma \Big( 
\partial_{a_0} \delta_{\beta}^{\gamma+e_1} 
\+ \sum_{k\geq1} (k\+1) \gamma_a(k) \delta_{\beta}^{\gamma-e_k+e_{k+1}}
\+ \sum_{\ell\geq0} (\ell\+1) \gamma_b(\ell) \delta_{\beta}^{\gamma-f_\ell+f_{\ell+1}}\Big) 
\big((\hat D^{(\0)})^m \hat c\big)_\gamma \\
&= (\kay +1)\hat\partial^{\kay+1} \big((\hat D^{(\0)})^m \hat c\big)_{\beta-e_1} \\
&\, + \sum_\gamma \Big( \sum_{k\geq1} (k+1) \gamma_a(k) \delta_{\beta}^{\gamma-e_k+e_{k+1}}
+ \sum_{\ell\geq0} (\ell+1) \gamma_b(\ell) \delta_{\beta}^{\gamma-f_\ell+f_{\ell+1}}\Big) 
\hat\partial^\kay \big((\hat D^{(\0)})^m \hat c\big)_\gamma \, , 
\end{align}
which by the induction hypothesis (note that 
$\beta-e_1\prec\beta\prec\hat\beta-g_\n$
and by \eqref{triD0}
$|\gamma|_\prec=|\beta|_\prec<|\hat\beta-g_\n|_\prec$) 
yields
\begin{align}
&\hat\partial^\kay ((\hat D^{(\0)})^{m+1} \hat c)_{\beta} \\
&= (\kay+1) \big((D^{(\0)})^m \bar c\big)_{\beta+(\kay+1)e_0-e_1} \\
&\, + \sum_\gamma \Big( \sum_{k\geq1} (k+1)\gamma_a(k) 
\delta_{\beta}^{\gamma-e_k+e_{k+1}} 
+ \sum_{\ell\geq0} (\ell+1)\gamma_b(\ell) \delta_{\beta}^{\gamma-f_\ell+f_{\ell+1}} \Big) \big((D^{(\0)})^m \bar c\big)_{\gamma+\kay e_0} \, .
\end{align}
On the other hand, 
\begin{align}
&((D^{(\0)})^{m+1} \bar c)_{\beta+\kay e_0} \\
&= \sum_\gamma (D^{(\0)})_{\beta+\kay e_0}^\gamma \big( (D^{(\0)})^m \bar c\big)_\gamma \\
&= \sum_\gamma \Big(\sum_{k\geq0} (k+1)\gamma_a(k)\delta_{\beta+\kay e_0}^{\gamma-e_k+e_{k+1}} 
+\sum_{\ell\geq0} (\ell+1)\gamma_b(\ell)\delta_{\beta+\kay e_0}^{\gamma-f_\ell+f_{\ell+1}} \Big)
\big((D^{(\0)})^m \bar c\big)_\gamma \\
&= (\kay+1) \big((D^{(\0)})^m \bar c\big)_{\beta+(\kay+1)e_0-e_1} \\
&\, + \sum_\gamma \Big(\sum_{k\geq1} (k+1)\gamma_a(k)\delta_{\beta+\kay e_0}^{\gamma-e_k+e_{k+1}} 
+\sum_{\ell\geq0} (\ell+1)\gamma_b(\ell)\delta_{\beta+\kay e_0}^{\gamma-f_\ell+f_{\ell+1}} \Big)
\big((D^{(\0)})^m \bar c\big)_\gamma, 
\end{align}
where we used in the last equality that $\beta$ doesn't contain $e_0$ components, i.e.~$\beta_a(k=0)=0$. 
Since the last sum over $\gamma$ vanishes if $\gamma$ does not contain at least $\kay e_0$, 
and for $k\geq1$ and $\ell\geq0$ we have $(\gamma+\kay e_0)(k)=\gamma_a(k)$ and $(\gamma+\kay e_0)(\ell)=\gamma_b(\ell)$, 
we obtain by resummation
\begin{align}
&((D^{(\0)})^{m+1} \bar c)_{\beta+\kay e_0} \\
&= (\kay+1) \big((D^{(\0)})^m \bar c\big)_{\beta+(\kay+1)e_0-e_1} \\
&\, + \sum_\gamma \Big(\sum_{k\geq1} (k+1)\gamma_a(k)\delta_{\beta}^{\gamma-e_k+e_{k+1}} 
+\sum_{\ell\geq0} (\ell+1)\gamma_b(\ell)\delta_{\beta}^{\gamma-f_\ell+f_{\ell+1}} \Big)
\big((D^{(\0)})^m \bar c\big)_{\gamma+\kay e_0} \, ,
\end{align}
which finishes the argument for \eqref{mt01} and hence \eqref{mt02}.

\medskip

In the following, we pass from \eqref{mt02} to $\eqref{eq:analytic_c}_{\hat\beta-g_\n}$. 
Recall from \eqref{eq:BPHZchoice} that $\hat c_{\hat \beta-g_\n}$ is only non-vanishing, 
if $|\hat\beta|<3$. 
We therefore restrict to such multiindices $\hat\beta$. 
By the locally uniform (in $a_0$) estimate \eqref{estPiminus} of $\hat\Pi^-_{x\hat\beta}$ together with $|\hat\beta|<3$ and the definition \eqref{def_Pi_hat_tilde_minus} of $\hat{\tilde\Pi}^-_{x\hat\beta}$, 
we know that 
$\hat c_{\hat\beta-g_\n} \n = \lim_{t\to\infty} \E \hat{\tilde \Pi}^-_{x\hat\beta t}(x)$,
locally uniformly in $a_0$. 
Since the analyticity \eqref{mt02} of $\hat{\tilde\Pi}^-_{x\hat\beta}$ with respect to ~\eqref{eq:norm_Pi-_smooth_analytic} implies analyticity of $\E\hat{\tilde\Pi}^-_{x\hat\beta t}(x)$, 
this yields analyticity of $\hat c_{\hat\beta-g_\n}$.
Hence we have by \eqref{mt02}
\begin{equation}
\hat\partial^{\hat k} \hat c_{\hat\beta-g_\n} \n 
= \lim_{t\to\infty} \Big( 
\E\bar\Pi^-_{x\,\hat\beta+\hat k e_0 \, t}(x) 
- \E\nabla\Delta\bar\Pi_{x\,\hat\beta+(\hat k-1) e_0 \, t}(x)
+ \bar c_{\hat\beta-g_\n + \hat k e_0} \n \Big).
\end{equation}
Since $|\cdot|$ is degenerate in $e_0$ and $|\hat\beta|<3 $, we obtain from the estimates \eqref{estPiminus} of $\bar\Pi^-_x$ and \eqref{estPi} of $\bar\Pi_x$ that 
$\lim_{t\to\infty}\E\bar\Pi^-_{x\,\hat\beta+\hat ke_0\,t}(x) = 0$
and $\lim_{t\to\infty}\E\nabla\Delta\bar\Pi_{x\,\hat\beta+(\hat k-1)e_0\,t}(x)=0$, 
which implies $\eqref{eq:analytic_c}_{\hat\beta-g_\n}$. 

\medskip

\textbf{Step 2.}
We show $\eqref{eq:analytic_c}_{\preccurlyeq\hat\beta-g_\n}$ \&
$\eqref{eq:analytic_Pi}_{\prec\hat\beta}$ imply 
$\eqref{eq:analytic_Pi}_{\hat\beta}$.
By definition of $\hat{\tilde\Pi}^-_{x\hat\beta}$ and the just established \eqref{mt02}, we have
\begin{equation}
\hat\partial^{\hat k} \hat\Pi^-_{x\hat\beta} 
= \hat\partial^{\hat k} \big( \hat{\tilde\Pi}^-_{x\hat\beta} 
- \hat c_{\hat\beta-g_\n} \n \big)
= \bar\Pi^-_{x\hat\beta+\hat ke_0}
- \nabla\Delta\bar\Pi_{x\hat\beta+(\hat k-1)e_0}
+ \bar c_{\hat\beta-g_\n+\hat ke_0} \n 
- \hat\partial^{\hat k}\hat c_{\hat\beta-g_\n} \n,
\end{equation}
which by $\eqref{eq:analytic_c}_{\hat\beta-g_\n}$ yields
\begin{equation}\label{mt03}
\hat\partial^{\hat k} \hat\Pi^-_{x\hat\beta} 
= \bar\Pi^-_{x\hat\beta+\hat ke_0}
- \nabla\Delta\bar\Pi_{x\hat\beta+(\hat k-1)e_0} \, ,
\end{equation}
with respect to \eqref{eq:norm_Pi-_smooth_analytic} 
and again with the understanding that the second right hand side term vanishes if $\hat k=0$. 
We now perform an integration argument to pass from \eqref{mt03} to $\eqref{eq:analytic_Pi}_{\hat\beta}$.
We do so by induction in $\hat k\geq0$, and start with the base case $\hat k=0$. 
We define 
\begin{align}
R^-_{x\hat\beta} 
&:= \hat\Pi^-_{x\hat\beta}(a_0') - \hat\Pi^-_{x\hat\beta}(a_0), \\
R_{x\hat\beta} 
&:= \hat\Pi_{x\hat\beta}(a_0') - \bar\Pi_{x\hat\beta}(a_0). 
\end{align}
Then by the equations \eqref{eq:pibetaanalytic} and \eqref{eq:barPi} for $\hat\Pi$ and $\bar\Pi$, we obtain 
\begin{equation}
(\partial_0+(1-a_0)\Delta^2)R_{x\hat\beta}
= \nabla\cdot \hat\Pi^-_{x\hat\beta}(a_0') 
- (1-a_0')\Delta^2\hat\Pi_{x\hat\beta}(a_0')
+ (1-a_0)\Delta^2\hat\Pi_{x\hat\beta}(a_0')
- \nabla\cdot\bar\Pi^-_{x\hat\beta}(a_0) .
\end{equation}
Using $\bar\Pi^-_{x\hat\beta} = \hat\Pi^-_{x\hat\beta}$ from \eqref{mt03}, 
we obtain 
\begin{equation}
(\partial_0+(1-a_0)\Delta^2)R_{x\hat\beta}
= \nabla\cdot \big( R^-_{x\hat\beta} 
+ (a_0'-a_0) \nabla\Delta\hat\Pi_{x\hat\beta}(a_0') \big) .
\end{equation}
Since $R_{x\hat\beta}$ inherits from $\hat\Pi_{x\hat\beta}$ and $\bar\Pi_{x\hat\beta}$ the vanishing and growth conditions, 
we obtain an integral representation of $R_{x\hat\beta}$ in terms of 
$R^-_{x\hat\beta}+(a_0'-a_0)\nabla\Delta\hat\Pi_{x\hat\beta}(a_0')$, 
analogous to the proof of Lemma~\ref{lem:int1}. 
By the exact same argumentation as in the proof of Lemma~\ref{lem:int1}, 
we therefore obtain 
\begin{equation}
\|R_{x\hat\beta}\|_{\eqref{eq:norm_Pi_analytic}} 
\lesssim \|R^-_{x\hat\beta} \|_{\eqref{eq:norm_Pi-_smooth_analytic}}
+ |a_0'-a_0| \|\hat\Pi_{x\hat\beta}(a_0') \|_{\eqref{eq:norm_Pi_analytic}}.
\end{equation}
As the right hand side of this expression vanishes for $a_0'\to a_0$, 
we obtain as desired $\hat\Pi_{x\hat\beta}(a_0) = \bar\Pi_{x\hat\beta}(a_0)$.

\medskip

In the induction step $0,\dots,\hat k\leadsto \hat k+1$ we proceed similarly.
We define 
\begin{align}
R^-_{x\hat\beta} &:= \hat\Pi^-_{x\hat\beta}(a_0')
- \sum_{j=0}^{\hat k+1} (a_0'-a_0)^j \hat\partial^j \hat\Pi^-_{x\hat\beta}(a_0), \\
R_{x\hat\beta} &:= \hat\Pi_{x\hat\beta}(a_0')
- \sum_{j=0}^{\hat k} (a_0'-a_0)^j \hat\partial^j \hat\Pi_{x\hat\beta}(a_0) 
- (a_0'-a_0)^{\hat k+1} \bar\Pi_{x\hat\beta+(\hat k+1)e_0}(a_0). 
\end{align}
Then by the equations \eqref{eq:pibetaanalytic} and \eqref{eq:barPi} for $\hat\Pi$ and $\bar\Pi$, 
and by the induction hypothesis $\hat\partial^j\hat\Pi_{x\hat\beta} = \bar\Pi_{x\hat\beta+je_0}$ for $j=0,\dots,\hat k$, 
we obtain 
\begin{align}
(\partial_0+(1-a_0)\Delta^2)R_{x\hat\beta}
&= \nabla\cdot \hat\Pi^-_{x\hat\beta}(a_0') 
- (1-a_0')\Delta^2\hat\Pi_{x\hat\beta}(a_0')
+ (1-a_0)\Delta^2\hat\Pi_{x\hat\beta}(a_0') \\
&\,- \sum_{j=0}^{\hat k+1} (a_0'-a_0)^j \nabla\cdot\bar\Pi^-_{x\hat\beta+je_0}(a_0) .
\end{align}
Using \eqref{mt03} to rewrite $\bar\Pi^-_{x\hat\beta+je_0} = \hat\partial^j\hat\Pi^-_{x\hat\beta} + \nabla\Delta\bar\Pi_{x\hat\beta+(j-1)e_0}$, 
and using once more the induction hypothesis in form of 
$\hat\partial^{j-1}\hat\Pi_{x\hat\beta} = \bar\Pi_{x\hat\beta+(j-1)e_0}$ for $j-1=0,\dots,\hat k$, 
we obtain 
\begin{equation}
(\partial_0+(1-a_0)\Delta^2)R_{x\hat\beta}
= \nabla\cdot R^-_{x\hat\beta} 
+ (a_0'-a_0) \Delta^2 \big( \hat\Pi_{x\hat\beta}(a_0')
- \sum_{j=0}^{\hat k} (a_0'-a_0)^j \hat\partial^j \hat\Pi_{x\hat\beta}(a_0) \big).
\end{equation}
By definition of $R_{x\hat\beta}$, this yields
\begin{equation}
(\partial_0+(1-a_0)\Delta^2)R_{x\hat\beta}
= \nabla\cdot \Big( R^-_{x\hat\beta} 
+ (a_0'-a_0) \nabla\Delta \big(R_{x\hat\beta} + (a_0'-a_0)^{\hat k+1} 
\bar\Pi_{x\hat\beta+(\hat k+1)e_0}(a_0) \big) \Big). 
\end{equation}
Again, since $R_{x\hat\beta}$ inherits growth and vanishing conditions from $\hat\Pi_{x\hat\beta}$ and $\bar\Pi_{x\hat\beta}$, 
the same argumentation as in the proof of Lemma~\ref{lem:int1} yields
\begin{equation}
\|R_{x\hat\beta}\|_{\eqref{eq:norm_Pi_analytic}}
\lesssim \|R^-_{x\hat\beta}\|_{\eqref{eq:norm_Pi-_smooth_analytic}}
+ |a_0'-a_0| \| R_{x\hat\beta}\|_{\eqref{eq:norm_Pi_analytic}}
+ |a_0'-a_0|^{\hat k+2} \|\bar\Pi_{x\hat\beta+(\hat k+1)e_0}(a_0) \|_{\eqref{eq:norm_Pi_analytic}}.
\end{equation}
For $|a_0'-a_0|$ sufficiently small, the right hand side term $|a_0'-a_0| \|R_{x\hat\beta}\|_{\eqref{eq:norm_Pi_analytic}}$ can be absorbed in the left hand side, 
establishing that $\|R_{x\hat\beta}\|_{\eqref{eq:norm_Pi_analytic}}=o(|a_0'-a_0|^{\hat k+1})$. 
Hence $\hat\partial^{\hat k+1}\hat\Pi_{x\hat\beta}(a_0) = \bar\Pi_{x\hat\beta+(\hat k +1)e_0}(a_0)$, 
which finishes the induction step 
and therefore the proof of $\eqref{eq:analytic_Pi}_{\hat\beta}$. 
\end{proof}

%
\section*{Declarations}

\paragraph{\bf Funding and/or Conflicts of interests/Competing interests}
The authors have no competing interests to declare that are relevant to the content of this article.

\bibliographystyle{alphaurl}
\bibliography{ref}
\end{document}